\documentclass[10pt,a4paper]{amsart}
\usepackage[toc,page]{appendix}
\usepackage{a4}
\usepackage[active]{srcltx}
\usepackage{amsmath,bbm,esint}
\usepackage{amssymb}
\usepackage{amsfonts}
\usepackage{mathrsfs} 
\usepackage{graphicx}
\usepackage{tikz}
\usepackage{color}
\usepackage[colorlinks=true,urlcolor=blue,
citecolor=red,linkcolor=blue,linktocpage,pdfpagelabels,bookmarksnumbered,bookmarksopen]{hyperref}

\newcounter{hyp}
\def\a{\alpha}
\def\b{\beta}
\def\d{\delta}
\def\g{\gamma}

\def\l{\lambda}

\def\s{\sigma}

\def\t{\tau}
\def\e{\varepsilon}
\def\Om{\Omega}
\def\om{\omega}

\newcommand{\qo}{{\overline q}}

\newcommand{\cA}{{\mathcal A}}
\newcommand{\cB}{{\mathcal B}}

\newcommand{\cH}{{\mathcal H}}
\newcommand{\cF}{{\mathcal F}}

\newcommand{\cL}{{\mathcal L}}
\newcommand{\cK}{{\mathcal K}}

\newcommand{\cP}{{\mathcal P}}
\newcommand{\cN}{{\mathcal N}}
\newcommand{\cU}{{\mathcal U}}
\newcommand{\cV}{{\mathcal V}}

\newcommand{\B}{{\mathbb B}}
\newcommand{\bB}{{\mathbb B}}
\newcommand{\bH}{{\mathbb H}}

\newcommand{\R}{{\mathbb R}}

\newcommand{\N}{\mathbb{N}}
\newcommand{\M}{{\mathbb M}}

\newcommand{\ov}{\overline}

\newcommand{\weakly}{\ensuremath{\rightharpoonup}}

\newcommand{\da}{\ensuremath{\downarrow}}

\newcommand{\mres}{\mathbin{\vrule height 1.6ex depth 0pt width
0.13ex\vrule height 0.13ex depth 0pt width 1.3ex}}

\newcommand{\sL}{\mathscr{L}}

\newcommand{\sP}{{\mathcal P}}

\newcommand{\ds}{\displaystyle}

\newcommand{\spt}{{\rm{spt}}}

\newcommand{\dd}{d} 

\newcommand{\id}{{\rm id}}

\newcommand{\bz}{\bold z}

\newcommand{\be}{\begin{equation}}
\newcommand{\ee}{\end{equation}}

\newcommand{\ba}{\begin{array}}
\newcommand{\ea}{\end{array}}

\newtheorem{theorem}{Theorem}[section]
\newtheorem{definition}[theorem]{Definition}
\newtheorem{lemma}[theorem]{Lemma}
\newtheorem{corollary}[theorem]{Corollary}
\newtheorem{proposition}[theorem]{Proposition}

\newtheorem{remark}[theorem]{Remark}
\newtheorem{property}[theorem]{\textbf{Property}}

\newtheorem{innercustomgeneric}{\customgenericname}
\providecommand{\customgenericname}{}
\newcommand{\newcustomtheorem}[2]{%
  \newenvironment{#1}[1]
  {%
   \renewcommand\customgenericname{#2}%
   \renewcommand\theinnercustomgeneric{##1}%
   \innercustomgeneric
  }
  {\endinnercustomgeneric}
}

\newcustomtheorem{customthm}{Framework}

\numberwithin{equation}{section}

\title[Deterministic displacement convex potential games]{Global well-posedness of Master equations for deterministic displacement convex potential mean field games}  

\author[W. Gangbo]{Wilfrid Gangbo} 
\address{Department of Mathematics, UCLA, California, USA}
\email{wgangbo@math.ucla.edu}

\author[A.R. M\'esz\'aros]{Alp\'ar R. M\'esz\'aros}  
\date{\today}
\address{Department of Mathematical Sciences, University of Durham, Durham DH1 3LE, England}
\email{alpar.r.meszaros@durham.ac.uk} 

\begin{document}
\maketitle

\begin{abstract} This manuscript constructs global in time solutions to {\em master equations} for potential Mean Field Games. The study concerns a class of Lagrangians and initial data functions, which are {\em displacement convex} and so, it may be in dichotomy with the class of so--called {\em monotone} functions,  widely considered in the literature. We construct solutions to  both the scalar and vectorial master equations in potential Mean Field Games, when the  underlying space is the whole space $\R^d$ and so, it is not compact.
\end{abstract}

\tableofcontents
%
%
\section*{Introduction} 
In this manuscript, we study a Hamilton--Jacobi equation on $\cP_2(\R^d)$, the set of Borel probability measures on $\R^d$ of finite second moments. This allows to make inferences on the {\em master equation} in  Mean Field Games, introduced by P.-L. Lions in \cite{Lions:course}. Our study relies on an special notion of convexity, the so--called {\em displacement convexity}, which is natural for functions $\cV: \cP_2(\R^d) \rightarrow \R$. It differs from the classical notion of convexity on the set of measures, which corresponds to the so--called {\em Lasry--Lions monotonicity} condition, central in most prior works aiming to study global in time solutions to the master equation. A comparison between the classical notion of convexity and displacement convexity can already be made by considering ways of interpolating Dirac masses. Given two Dirac masses $\delta_{q_0}$ and $\delta_{q_1}$ the paths 
$$
[0,1]\ni t \mapsto \sigma_t:=(1-t)\delta_{q_0}+t \delta_{q_1}, \quad [0,1]\ni t  \mapsto \sigma^*_t :=\delta_{(1-t)q_0+t q_1}
$$ 
provide two distinct interpolations, these two elements of  $\cP_2(\R^d)$.   The function $\cV$ is called convex in the classical sense if it is convex along classical interpolation, which in particular  implies $t \mapsto \cV( \sigma_t)$ is a convex function on $[0, 1]$. The function is called {\em displacement convex} \cite{McCann1997}  if its restriction to any $W_2$--geodesics is convex, which in particular means $t \mapsto \cV( \sigma_t^*)$ is a convex function on $[0, 1]$. 

A blatant example which shows that convexity and displacement convexity cannot be the same is when 
$$2\cV(\mu)=  \int_{\R^{2d}}|q-q'|^2 \mu(dq)\mu(dq'), \qquad \mu \in \cP_2(\R^d).$$ 
In this case, it has been long known  that  $\cV$ is concave in the classical sense while $\cV$ is obviously displacement convex. However, for the purpose of our study, we need to come up with a richer class of examples consistent with our analysis. For instance, let us consider two functions $\phi, \phi_1 \in C^2(\R^d)$ with bounded second derivatives and such that $\phi_1$ is even and define 
\[
2\cV(\mu):= \int_{\R^d}\Big( 2\phi(q)+ \phi_1 * \mu(q)\Big)\mu(dq), \qquad \mu \in \cP_2(\R^d).
\]
Let us recall that (see Lemma \ref{lem:fourier-of-phi1}) the function $\cV$ is convex in the classical sense if and only if  $\hat \phi_1$ -- the  Fourier transform of $\phi_1$ -- is nonnegative, independently of whether or not additional requirements are imposed on $\phi.$  Suppose for instance $\phi$ is $2\lambda$--convex for some $\lambda>0.$ If $\phi_1$ is $\lambda_1$--convex for some $2\lambda_1\in (-\lambda, \lambda)$ then  $\cV$ is displacement convex. As discussed in Subsections \ref{ex:example} and \ref{subsec:conv_displ},  we can choose $\phi_1$ such that $\hat \phi_1$ changes sign, so that $\cV$ fails to be convex in the classical sense.

The theory of well--posedness of the master equation in Mean Field Games is well developed on the set of probability measures \cite{CardaliaguetDLL} (for a probabilistic approach to study such equations we refer to \cite{CarmonaD-II}), under the Lasry-Lions monotonicity condition  \cite{Cardaliaguet} \cite{LasryL0} \cite{LasryL1} \cite{LasryL2}, for games where the individual and/or  common noises are essential mechanisms governing the games. In the same setting of monotone data, global solutions were also constructed in \cite{ChassagneuxCD}, where the authors can handle even degenerate diffusions in the equations. In the same context, \cite{MouZ} improves the regularity restrictions on the data, which need to be still monotone, and propose a notion of weak solutions for the master equation. When the monotonicity condition fails (even in the presence of the noise), only short time existence results for the scalar master equation were achieved (in the deterministic case we refer to \cite{Bessi, GangboS2015, Mayorga}; in the presence of noise we refer to \cite{ChassagneuxCD} and \cite{CarmonaD-II}). For classical mean field games systems the smallness of the time horizon sometimes can be replaced by a smallness condition on the data (see for instance \cite{Ambrose1,Ambrose2}). Via a ``lifting procedure'', it is possible to study master equations on a Hilbert space of square integrable random variables. The main benefit of this process is to instead use the more familiar Fr\'echet derivatives on flat spaces and bypass the differential calculus on the space of probability measures, which is a curved infinite dimensional manifold. Such analysis were carried out for a special class of mechanical Lagrangians and 
for potential games, either in deterministic setting \cite{BensoussanY} or in the presence of individual noise in \cite{BensoussanGY, BensoussanGY-2}. 
Furthermore, the authors needed to impose higher than second order Fr\'echet differentiability on the data functions. It turns out (see below) that this may sometimes be a too severe restriction. Therefore, from this point of view the Hilbert space approach has a serious drawback.

This manuscript constructs global solutions to potential mean field games  master equations, where the widely used Lasry--Lions monotonicity condition is replaced by  {\em displacement convexity}, a concept which appeared in optimal transport theory in the early 90's. The use of displacement convexity in mean field control problems and mean field games goes back to \cite{CarmonaD2015}, where the authors study control problems of McKean-Vlasov type. In the case of mean field game systems with common noise, we refer the reader to  \cite{Ahuja,ARY}, where their so-called {\it weak monotonicity condition} assumed on the data, is equivalent to displacement convexity in the potential game case. As mentioned before, in \cite{BensoussanGY, BensoussanGY-2}, this condition is used in the Hilbertian setting. In the study of master equation arising in control problems of McKean-Vlasov type (in the presence of individual noise), \cite{ChassagneuxCD} seems to be the first work in the literature which imposed displacement convexity on their data to obtain well-posedness of a master equation in the spirit of \cite{CarmonaD2015}.

In potential mean field games, one considers smooth enough real valued functions $\cU_0, \cF$ defined on $\cP_2(\R^d)$. We assume that there are smooth real valued functions $u_0, f$ defined on $\R^d \times \cP_2(\R^d)$ which are related to $\cU_0, \cF$ in the following sense: the Wasserstein gradient of $\cU_0$ at $\mu \in \cP_2(\R^d)$ equals the finite dimensional gradient $D_q u_0(\cdot, \mu)$ and  the Wasserstein gradient of $\cF$ at $\mu \in \cP_2(\R^d)$ equals the finite dimensional gradient $D_q f(\cdot, \mu)$. Given a Hamiltonian $H \in C^3(\R^{2d})$ the master equation consists in finding a real valued function $u$ defined on $[0,\infty) \times \R^d \times \cP_2(\R^d),$ solution to the non--local equation 
\[
\left\{
\begin{array}{ll}
\ds\partial_t u+H(q,D_q u)+\cN_\mu\big[D_q u(t, \cdot, \mu), \nabla_w {u}(t, q, \mu)(\cdot)\big]=f(x,\mu), & (0,T)\times\R^d\times\sP_2(\R^d),\\[5pt]
u(0, \cdot,\cdot)=u_0, & \R^d\times\sP_2(\R^d).
\end{array}
\right.
\] 
Here, $\cN_\mu:L^2(\mu)\times L^2(\mu)\to\R$ is the non--local operator defined as 
\begin{equation}\label{eq:defineNmu}
\cN_\mu[\eta, \theta]:= \int_{\R^d} D_p H(c, \eta(c)) \cdot \theta(c) \mu(dc). 
\end{equation}
Let $L(q, \cdot)$ be the Legendre transform of $H(q, \cdot)$ and assume $L$ is strictly convex and both functions have bounded second order derivatives.  Under the assumption that $\cU_0$ and $\cF$ are displacement convex (convex along the Wasserstein geodesics), we construct classical solutions and weak solutions to the master equation, depending on the regularity properties imposed on the data. Following \cite{GangboT2017}, the starting point of our study relies on the point of view that the differential structure on $(\cP_2(\R^d),W_2)$ is inherited from the differential structure on the flat space  $\bH:=L^2((0,1)^d, \R^d)$ and the former space can be viewed as the quotient space of the latter.   The functions $\cU_0, \cF$ are lifted to obtain functions $\tilde \cU_0, \tilde \cF$ defined on the Hilbert space  $\bH$, with the property that they are rearrangement invariant. What we mean by rearrangement invariant is that  $\tilde \cU_0(x)=\tilde \cU_0(y)$ whenever the push forward of Lebesgue measure restricted to $(0,1)^d$ by $x, y \in \bH$ coincide. In this case, we sometimes say that $x$ and $y$  have the same law. The Hamiltonian $H$ is used to define  on the co--tangent bundle $\bH^2,$ another Hamiltonian denoted  
\[
\tilde \cH(x, b):= \int_{(0, 1)^d} H(x(\omega), b(\omega))d\omega- \tilde \cF(x). 
\]
The corresponding Lagrangian $\tilde \cL$ is on  $\bH^2,$ the tangent bundle, and is 
\[
\tilde \cL(x, a):= \int_{(0, 1)^d} L(x(\omega), a(\omega))d\omega+ \tilde \cF(x).
\]

Both the Lagrangian and the Hamiltonian are invariant under the action of the group of bijections of $(0, 1)^d$ onto $(0, 1)^d,$ which preserve the Lebesgue measure.  We are interested in regularity properties of $\tilde \cU: (0, \infty) \times \bH \rightarrow \R$ solutions to the Hamilton--Jacobi equation 
\[
\left\{
\begin{array}{l}
\partial_t \tilde \cU+\tilde \cH\bigl(\cdot, \nabla_x \tilde \cU\bigr)=0,\quad\mbox{in}\,\,(0,\infty)\times \bH,
\\
\tilde  \cU(0,\cdot)=\tilde \cU_0\qquad \qquad \quad \;\; \mbox{on}\,\,\bH.
\end{array}
\right.
\]
The characteristics of this infinite dimensional PDE and the smoothness properties of $\tilde \cU$ will play an essential role in the application of our study to mean field games. They allow us to obtain an explicit representation formula of the solution to the master equation for arbitrarily large times. Similar observations were made also by P.-L. Lions during a recorded seminar talk \cite{Lions:course}. This lecture seems to suggest that is was not clear at all how far the displacement convexity assumptions on the data could be used to advance the study of the global in time well--posedness of master equations.

Under appropriate growth and convexity conditions on the data, the classical theory of Hamilton--Jacobi equations on Hilbert spaces ensures  that $\tilde \cU(t, \cdot)$ is of class $C^{1,1}_{\rm loc}(\bH)$. Our Hamiltonian and Lagrangian being rearrangement invariant, by the uniqueness theory of Hamilton--Jacobi equation, $\tilde \cU(t, \cdot)$ is rearrangement invariant. This allows to define a function $\cU(t, \cdot)$ on $\cP_2(\R^d)$ such that $\cU(t, \mu)=\tilde \cU(t, x)$ whenever $x \in \bH$ has $\mu$ as its law. In the same time, $\cU$ will be the unique classical solution to the corresponding Hamilton--Jacobi equation set on $\cP_2(\R^d)$.

By Lemma \ref{lem:c11_equiv}, a function $\cV: \cP_2(\R^d) \rightarrow \R$ is of class $C^{1,1}_{\rm loc}$ on the Wasserstein space if and only if its lift  $\tilde \cV: \bH \rightarrow \R$ is of class $C^{1,1}_{\rm loc}$ on the Hilbert space. Since the Hilbert space theory ensures that $\tilde \cU(t, \cdot)$  is of class $C^{1,1}_{\rm loc}$ on the Hilbert space, we obtain as a by--product that $\cU(t, \cdot)$ is of class $C^{1,1}_{\rm loc}$ on the Wasserstein space. This is how far one could push the Hilbert approach in terms of regularity theory if one would like to make useful inference in mean field games.  Indeed, imposing that a rearrangement invariant function $\tilde \cV:\bH \rightarrow \R$ is of class $C^2$ (twice Fr\'echet differentiable) is too  stringent for the purpose of mean field games. For instance, if $\phi \in C_c^\infty(\R^d),$ unless $\phi \equiv 0$, the function $\tilde \cV$ defined on $\bH$ by 
$$\tilde \cV(x):= \int_{(0,1)^d} \phi(x(\omega))d\omega,$$ does not belong to $C^2(\bH)$   (cf. Proposition \ref{prop:finite-d-regularity}). The reader should compare this to another  subtlety in \cite[Section 2]{BuckdahnLPR}. Similar conclusions can be drawn on other functionals with a local representation such as 
$$ 
 \bH\ni x \mapsto \tilde \cV(x):= \int_{(0,1)^{nd}} \phi(x(\omega_1), \cdots, x(\omega_n))d\omega_1 \cdots d\omega_n,
$$ 
when $\phi \in C^3(\R^{nd})$ is symmetric and has bounded second and third order derivatives (cf. Proposition \ref{prop:locald}). Pursuing a deeper analysis, we assume $\alpha \in (0,1],$ $\tilde \cV \in C^{2,\alpha}_{\rm loc}\big(\bH \big)$ is rearrangement invariant so that it is the lift of a function $\cV:\cP_2(\R^d) \rightarrow \R.$  We show in Lemma \ref{lem:weak-Holder} that if \eqref{eq:jan19.2020.1} holds for all $h, h_* \in \mathbb H$ then $D_{q}\big(\nabla_w \cV(\mu) \big)$ is constant function on ${\rm spt}( \mu)$. 

A final argument to support the fact that we need a new concept of higher order derivatives on the set of probability measures is the following. When $k \geq 3$, making assumptions on  $k$--order differentials of Hamiltonians $\tilde \cH: \bH^2 \rightarrow \R$ and treating them as continuous multi--linear forms on cartesian products of $\bH^2$ is too restrictive for a theory in mean field games. Indeed,  frequently used Hamiltonians in mean field games theory are of the form 
\[
\tilde \cH(x, b)= \tilde \cH_H(x, b)-\tilde \cF(x), \qquad \tilde \cH_H(x, b)\equiv \int_{(0,1)^d} H(x(\omega), b(\omega)) d\omega
\]
where $H \in C^{3}(\R^{2d})$ is such that $D^2H$ is bounded. Let $\alpha \in (0, 1]$. Even if $ C^{2,\alpha}_{\rm loc}(\bH^2)$  is an infinite dimensional space, its intersection with the set of functions which have a local representation is contained in a finite dimensional space. For instance,  
\begin{equation}\label{eq:finite-d}
{\rm dim}\Big(C^{2,\alpha}_{\rm loc}(\bH^2) \cap \big\{ \tilde \cH_H \; : \; H\in C^{2, \alpha}_{\rm loc}(\R^{2d}), \;\; D^2 H \;\; \text{is bounded}\big\}\Big)<\infty.
\end{equation}

In this manuscript, to write a meaningful master equation, we are interested in functions  $\cV: \cP_2(\R^d) \rightarrow \R$ which satisfy higher regularity properties than being of $C^{1,1}_{\rm loc}$. We assume at least that their lifts $\tilde \cV: \bH \rightarrow \R$ are such that $\nabla \cV$ is G\^ateaux differentiable with bounded second order differential is a sense to be made precise. Due to the rearrangement invariance property of $\tilde \cV$, $\nabla^2 \tilde \cV$ must have a special form. Given $x \in \bH$, there exist  matrix valued maps 
$$A_{{12}} ^* \in L^\infty((0,1)^d; \R^{d \times d}), \quad A_{{22}}^* \in L^\infty((0,1)^{2d}; \R^{d \times d})$$  
such that $A_{12}^*$ is symmetric almost everywhere, $A^*_{22}(\omega, o)=A^*_{22}(o, \omega)^\top$ almost everywhere and  the operator $ \bH\ni \zeta \mapsto \nabla^2 \tilde \cV(x) \zeta $ can be written as
\begin{equation}\label{eq:intro-frechet0}
\big(\nabla^2 \tilde \cV(x) \zeta\big)(\omega)= A_{12}^*(\omega)\zeta(\omega) + \int_{(0,1)^d} A^*_{22}(\omega, o)\zeta(o) do.
\end{equation} 
In fact, as observed in \cite{BuckdahnLPR} (cf. also \cite{CardaliaguetDLL,CarmonaD-I,CarmonaD-II, ChassagneuxCD, ChowGan}), there exists a matrix field  $A_{12}$ defined on $R(x)$, the range of $x$ and a matrix field $A_{22}$ defined on $R(x) \times R(x)$, such that the following factorization holds: 
\[
A_{12}^*(\omega)=A_{12}\big(x(\omega) \big), \quad A_{22}^*(\omega, o)=A_{22}\big(x(\omega) , x(o)\big).
\]
We argue in Remark \ref{rem:whessian} that $A_{12}$ can be interpreted as $D_q\big( \nabla_w \cV(\mu)(q)\big)$ and indicate the relation between $A_{22}$ and the Wasserstein gradient of $\nabla_w \cV.$ 

When $\cB \subseteq \cP_2(\R^d)$ is an open set, we introduce vector spaces of functions $C^{2, \alpha, w}(\cB),$ as  substitutes for the spaces $C^{2, \alpha}(\bH).$ These spaces are such that whenever $\cV \in C^{2, \alpha, w}(\cB)$, its restrictions 
$$ 
\R^{nd}\ni (q_1, \cdots, q_n) \mapsto \cV\bigg({1 \over n}\sum_{i=1}^n \delta_{q_i}\bigg)
$$ 
belong to $C^{2, \alpha}_{\rm loc}(\R^{nd})$. The precise definition of this space can be found in Definition \ref{def:c21_wasserstein}. At least we require that if $\cV \in C^{2, \alpha, w}(\cB)$, since the second order G\^ateaux differential of its lift $\tilde \cV$ exists, it must satisfy the property 
\begin{equation}\label{eq:intro-frechet}
\Big|\nabla \tilde \cV(y)(\omega) -\nabla \tilde \cV(x)(\omega) - \nabla^2 \tilde \cV(x)(\omega)\big((y(\omega)-x(\omega)\big)\Big| \leq C\Big( |y(\omega)-x(\omega)|^\alpha+\|x-y\|^\alpha\Big)
\end{equation} 
whenever $x, y \in \bH$, $x$ pushes $\cL^d_{(0, 1)^d}$ forward to $\mu,$ $y$ pushes $\cL^d_{(0, 1)^d}$ forward to $\nu$ and $\|x-y\|=W_2(\mu, \nu).$ In fact, spaces of type $C^{2,1}(\cP_2(\M))$ have already been considered in the framework of mean field models in \cite{BuckdahnLPR}, based on a construction very similar to ours in Definition \ref{def:c21_wasserstein}.

A discretization approach (which consists in restricting our study to the subsets of $\cP_2(\M)$ which are averages of Dirac masses)  greatly facilitates the task to show \eqref{eq:intro-frechet0}, with $\tilde\cV$ replaced by the solution to the Hamilton--Jacobi equation we constructed on the Hilbert space. This helps us show that 
$$ 
A_{12}\in L^\infty(R(x); \R^{d \times d}), \ A_{22}\in L^\infty(R(x)\times R(x); \R^{d \times d})
$$ 
and for $\varphi \in C_c^\infty(\R^d)$ and $h:= D \varphi \circ x$,  
\[
D^2 \tilde \cV(x)(h, h)= \int_{(0,1)^d} A_{12}(x(\omega)) h(\omega) \cdot h(\omega) d\omega+  \int_{(0,1)^{2d}} A_{22}(x(\omega_1), x(\omega_2)) h(\omega_1) \cdot h(\omega_2) d\omega_1 d\omega_2.
\] 
This allows us to make inference beyond an estimate such as
$$
\sup_{x, h \in \bH} \Big\{|D^2 \tilde \cU(t, x)(h, h)|\; : \; \|h\| \leq 1, \; \|x\| \leq r\Big\} <+\infty \qquad \forall r>0.
$$ 
Unlike studies of the master equation in compact  settings such as the periodic setting $\R^d/\mathbb Z^d$, the fact that the range of $\tilde \cU$ is certainly unbounded, is a source of additional complications in our study, 

When $\nabla \tilde \cH$ is Lipschitz, the characteristics of the Hamilton--Jacobi equation are the Hamiltonian flow $\Sigma=(\Sigma^1, \Sigma^2):[0,\infty) \times \bH^2 \rightarrow \bH^2,$ uniquely defined by the solution of
\begin{equation}\label{eq:Hamiltonian-Flow}
\left\{
\begin{array}{ll}
\dot \Sigma^1(t, \cdot )=& \hfill \nabla_b \tilde \cH \big(\Sigma(t, \cdot ) \big),\quad \mbox{in}\,\,(0,\infty)\times \bH^2,\\
\dot \Sigma^2(t, \cdot )=&- \nabla_x \tilde \cH \big(\Sigma(t, \cdot ) \big),\quad\mbox{in}\,\,(0,\infty)\times \bH^2,\\
 \Sigma(0, \cdot)=&  \;\;\; \id_{\bH^2}.
\end{array}
\right.
\end{equation} 
The vector field $\nabla^\perp\tilde \cH$ is the velocity in Eulerian coordinates for the trajectory $\Sigma$ on the cotangent bundle $\bH^2.$
We denote as $$(\tilde \xi, \tilde \eta):[0,\infty) \times \bH \rightarrow \bH^2$$ the restriction of $\Sigma$ to the graph of $\nabla \tilde \cU_0$, i.e.
\begin{equation}\label{eq:Hamiltonian-Flow2}
  (\tilde \xi, \tilde \eta) := \Sigma \big(\cdot, \cdot, \nabla \tilde \cU_0 \big). 
\end{equation} 
When $\tilde \cL$  and $\tilde \cU_0$ are convex, under appropriate standard conditions on $\tilde \cL$ and $\tilde \cH$, differentiability properties of $\tilde \cU$ are obtained by standard methods. A strict convexity property of $\tilde \cL$ ensures that for any fixed $t\geq 0$, $\tilde\xi(t, \cdot)$ is a bijection of $\bH$ onto $\bH.$ The trajectories 
\[
[0, t]\ni s\mapsto \tilde S_s^t[x]:= \tilde\xi\Big(s,  \tilde\xi^{-1}(t,x)\Big) \in \bH
\]
are useful to write the representation formula  
\[
\tilde \cU(t, x)= \tilde \cU_0(\tilde S_0^t[x])+ \int_0^t \tilde \cL\big( \tilde  S_s^t[x],   \partial_s \tilde S_s^t[x]\big) ds.
\] 
The identity  
\begin{equation}\label{grad:rep}
 \nabla \tilde \cU(t, \cdot) =\tilde \eta(t, \tilde S_0^t)   
\end{equation}
suggests that the smoothness properties of $\tilde \cU$ rest on the smoothness properties of $\tilde S_0^t$  and $\tilde \eta$. While strict convexity of $\tilde \cL$ is sufficient to get that the restriction of $\tilde\xi(t, \cdot)^{-1}$ to appropriate finite dimensional spaces is continuously differentiable, it becomes much harder to show that  $\tilde\xi(t, \cdot)^{-1}$ is continuous on the whole space $\bH$ unless appropriate convexity properties are imposed on the data. 

Let us consider the vector field 
\[B(t, \cdot):= \nabla_b \tilde \cH\big(\cdot,  \tilde \eta(t, \tilde S_0^t)   \big) \]
which helps to study  the second order derivatives of $\tilde \cU$ and which represents the velocity of the flow $\tilde \xi$ in physical space, since $\dot {\tilde \xi} =B(s, \tilde \xi)$. When $\tilde \cU(t, \cdot)$ is twice differentiable then $\nabla^2 \tilde \cU(t, x), \nabla B(t, x): \bH^2 \rightarrow \R$ are bilinear forms which satisfy the relation   
\[
\nabla B(t, x)(h, a)=\nabla^2 \tilde \cU(t, x)\Big(a, D_{pp}^2H\big(x, \nabla \tilde \cU(t, x)\big) h \Big)  +   
 \int_{(0, 1)^d} \Big(D_{qp}^2 H\big( x, \nabla \tilde \cU(t, x)\big)\; a\Big) \cdot h d\omega, \quad (\forall h, a \in \bH).
\]

%
%
\subsection*{Summary of our main results}
Coming back to the description of our main results, after having provided the $C^{1,1}_{\rm{loc}}$ regularity for the viscosity solutions $\cU$ to the corresponding Hamilton--Jacobi equations on $\cP_2(\R^d)$, we completely abandon the setting of the Hilbert space and via the mentioned discretization approach we show that $\cU(t,\cdot)$ is actually of class $C^{2,1,w}_{\rm{loc}}$. We note that our approach seems to be novel and, although similar in flavor, it is completely different from the ones developed in \cite{GangboS2015} and \cite{Mayorga}. It relies on fine quantitative derivative estimates with respect to $m\in\N$ on the Hamiltonian flow for $m$-particles, then these in turn translate to higher regularity estimates on $\cU$ by carefully differentiating the identity \eqref{grad:rep}, written for the restriction of $\cU$ to the set of averages of Dirac masses. Let us emphasize that this finite dimensional projection of the value function solves the corresponding optimization problem but driven by the finite dimensional projections of the cost coefficients (see Remark \ref{rem:discrete-cont}); this is in fact what allows for a preliminary analysis of the optimal trajectories of the mean field control problem when restricting initial states of the population to uniform finite distributions. A key point is then to obtain regularity estimates that are independent of the cardinality of those finite distributions. This is one crucial step where the convexity structure plays a key role. This idea is in fact the heart of our analysis and works only for deterministic mean field games; the approach in this manuscript is entirely different from the existing ones to tackle mean field games master equations: most of them consist in working directly at the level of PDE system of mean field games.

Having $\cU(t,\cdot)\in C^{2,1,w}_{\rm{loc}}(\cP_2(\R^d))$ allows us to obtain {\it weak solutions} (see in Theorem \ref{thm:main-vector}) $\cV:[0,T]\times\sP_2(\R^d)\times\R^d\to\R^d$ to the so-called {\it vectorial master equation},

\begin{equation}\label{intro:master_vector}
\left\{
\begin{array}{r}
\ds\partial_t\cV+D_q H(q,\cV(t,\mu,q))+D_q\cV(t,\mu,q)\nabla_p H(q,\cV(t,\mu,q))+\overline \cN_\mu\big[\cV, \nabla_{w}^\top \cV \big](t, \mu, q)\\ 
\ds =\nabla_w\cF(\mu)(q)\\[5pt]
\ds\cV(0,\mu,\cdot)=\nabla_w\cU_0(\mu)(\cdot),
\end{array}
\right.
\end{equation}  
where for $\cV:\cP_2(\R^d)\times\R^d\to\R^d$ we define 
\[
\overline \cN_\mu\big[\cV, \nabla_{w}^\top \cV \big](t, \mu, q):=\int_{\R^d}\nabla_{w}^\top \cV(t, \mu,q)(b)  D_pH\big(b,  \cV(t, \mu, b)\big) \mu(db)
\]

This equation can be seen as a vectorial conservation law on $(0,T)\times\cP_2(\R^d)\times\R^d$ and can be derived formally by taking the Wasserstein gradient of the Hamilton-Jacobi equation satisfied by $\cU$. Such method is possible in the setting of the Hilbert space as well (provided one has the sufficient regularity to justify the differentiation), and this is done for instance in \cite{BensoussanGY} and \cite{BensoussanY} for short time and special Hamiltonians. Let us emphasize that there is a subtlety in this derivation and in particular at a first glance the vectorial master equation in the setting of $\cP_2(\R^d)$  is satisfied pointwise only on $(0,T)\times\bigcup_{\mu\in\cP_2(\R^d)}\{\mu\}\times\spt(\mu)$. Therefore, we refer to such solution as {\it weak solution}. Thus, additional effort is needed to extend the vectorial master equation to $(0,T)\times\cP_2(\R^d)\times\R^d$ and actually, this is possible through the solution to the scalar master equation. One cannot observe this phenomenon in the setting of $\bH$, because $\nabla\tilde\cU(t,x)$, as an element of $\bH$, does not carry explicitly the dependence on the range of $x\in\bH$. 

Let us stress that even though there is a deep connection between the vectorial and scalar master equations, while formally speaking the former one is the Wasserstein gradient of a Hamilton-Jacobi equation, additional effort is needed to justify the well-posedness of the latter one. And in particular, this is not a simple consequence of the well-posedness of the vectorial equation at all. In the same time, while the vectorial master equation might have physical relevance as a vectorial conservation law, in the theory of mean field games the scalar master equation is the one which has profound significance. One of the reasons for this is that this equation deeply carries the features of $N$--player differential games. In particular, as we can see this in \cite{CardaliaguetDLL}, it provides an important tool to prove the convergence of Nash equilibria of $N$--player differential games to the mean field games system, as $N\to+\infty$. In the same time, typically it provides quantified rates on propagation of chaos. Therefore, such equations are very natural, and they were successfully used in the literature in the context of mean field limits of large particle system (see for instance in \cite{MischlerM, ChassagneuxST}).

The candidate for the solution of the scalar master equation is constructed as follows. Given $t \in [0,T]$, $q \in \R^d$ and $\mu \in \cP_2(\R^d)$ we define 
\begin{equation}\label{intro:u_alt}
u(t, q, \mu):= \inf_{\gamma} \biggl\{u_0(\gamma_0, \sigma_0^t[\mu]) +\int_0^t \Big(L(\gamma_s, \dot \gamma_s) + f(\gamma_s, \sigma_s^t[\mu])\Big)ds\; : \; \gamma \in W^{1, 2}([0,t], \R^d), \gamma_t=q \biggr\},
\end{equation}
where the curve $(\s^t_s[\mu])_{s\in[0,t]}$ is the projection of the Hamiltonian flow onto $\cP_2(\R^d)$. We underline the important fact that the previous formula defines $u(t,\cdot,\mu)$ for every $q\in\R^d$ (and not just for $q\in\spt(q)$).

After obtaining the sufficient regularity of the mapping $\mu\mapsto \s^t_s[\mu]$ (using also the fact that $\cU(t,\cdot)\in C^{2,1,w}_{\rm{loc}}(\cP_2(\R^d))$), we show that $u$ is of class $C^{1,1}_{\rm{loc}}([0,T]\times\R^d\times\cP_2(\R^d))$ (see Lemma \ref{lem:u_alt-reg}).
The connection between $u$ and $\cU$ is that $D_q u(t,\cdot,\mu)=\nabla_w\cU(t,\mu)(\cdot)$ on $\spt(\mu)$. This is an important remark, since it means that $D_q u(t,\cdot,\mu)$ provides the natural Lipschitz continuous extension for $\nabla_w\cU(t,\mu)(\cdot)$ to $\R^d$. By these arguments we can prove Theorem \ref{thm:main-scalar}, the main theorem of this manuscript, which states that under our standing assumptions $u$ defined in \eqref{intro:u_alt} is the unique classical solution to the scalar master equation which is of class $C^{1,1}_{\rm{loc}}([0,T]\times\R^d\times\cP_2(\R^d))$.

Theorem \ref{thm:main-scalar} has several implication. First, the obtained regularity of $u$ and the fact that 
$$D_q u(t,\cdot,\mu)=\nabla_w\cU(t,\mu)(\cdot)\qquad \text{on} \quad \spt(\mu),$$ 
allow us to deduce that $D_q u$ is a solution to the vectorial master equation and \eqref{intro:master_vector} is satisfied for all $(t,\mu)\in(0,T)\times\cP_2(\R^d)$ and for $\cL^d$--a.e. $q\in\R^d$. Second, since the scalar master equation, and in particular our definition \eqref{intro:u_alt} possess the features of $N$--player differential games, we could easily deduce that $u(t,\cdot,\cdot)$, when restricted to $\bigcup_{q\in\R^{Nd}}\mu_q^{(N)}\times\spt(\mu_q^{(N)})$, provides approximate solutions to a system of  Hamilton--Jacobi equation, characterizing the Nash equilibria of the associated $N$--player differential game (such a construction would be similar to the ones in \cite{CardaliaguetDLL}, \cite{DelarueLR:19,DelarueLR:20}, so we omit the details on this). In the same time, the regularity of $u$ would allow us to deduce the local convergence of Nash equilibria as $N\to+\infty$, provided we know that the $N$--player Nash system of Hamilton-Jacobi equations has a smooth enough classical solution. In such a fortunate scenario, the proof of this result, even in the deterministic setting, would follow similar ideas as the ones in \cite{CardaliaguetDLL}, \cite{DelarueLR:19,DelarueLR:20}. However, let us emphasize that the well-posedness question of systems of Hamilton-Jacobi equations, in the deterministic setting seems to be widely open in the literature. It worth mentioning the recent work \cite{FischerS} which studies this convergence question in the deterministic setting in a suitable weak sense, without relying on the well-posedness neither of the Nash system nor the master equation.


The structure of the rest of the paper is the following. In Section \ref{sec:preliminaries} we provide the first part of our standing assumptions, we present the discretization approach and show a direct argument which provides $C^{1,1}_{\rm{loc}}$ regularity for solutions to a class of Hamilton--Jacobi equations set on Hilbert spaces. 

Section \ref{sec:app-reg} contains the important quantitative estimates with respect to $m$ on the Hamiltonian flows of $m$--particle systems and the corresponding derivative estimates of the solutions to Hamilton--Jacobi equations set on $\R^{md}$.

In Section \ref{sec:reg_P2_H_Pm} we compare notions of convexity and regularity for functions defined on $\cP_2(\R^d)$, their lifts defined on $\bH$ and their restrictions to discrete measures. Here we also show how can we deduce regularity estimates for functions on $\cP_2(\R^d)$ from precise quantitative derivative estimates on their restrictions to discrete measures.

Section \ref{sec:master_equations} is the core of the manuscript where we investigate the well-posedness of both vectorial and scalar master equations. Additional assumptions need to be imposed to establish the well-posedness of the scalar master equation. These are listed in this section. 

In Section \ref{sec:further} we have collected an important implication of the scalar master equation. We use scalar master equations  to improve the notion of weak solution for the vectorial equations. 

To facilitate the reading of the main text, our manuscript has several appendices. In Appendix \ref{appendixE} we demonstrate the limitations of the Hilbert space approach, when studying or assuming $C^{2,\alpha}$ type regularity on rearrangement invariant functionals having local representations. 

In Appendix \ref{app:convex_dispalcement} we emphasize how our setting by imposing displacement convexity of the data can replace the more standard monotonicity assumptions imposed typically in the mean field games literature. Here we provide examples of functionals which produce non-monotone coupling functions and an example of a Hamilton--Jacobi equation on $\cP_2(\R^d)$, for which the data provides the standard monotonicity condition, yet its classical solution ceases to exist after finite time.

In Appendix \ref{appendixC} we have collected some standard results on Hamiltonian flows on Hilbert spaces and we explained how the regularity of these flows can be used to show regularity of solution to a Hamilton-Jacobi equations. 

\medskip
{\sc Acknowledgements} The research of WG was supported by NSF grant DMS--1700202. Both authors acknowledge the support of Air Force grant FA9550-18-1-0502. The authors would like to express their gratitude to P. Cannarsa for the discussions and for pointing out important references on regularity properties of solutions to Hamilton--Jacobi equations on $\R^d$. They wish to thank A. \'Swiech for the discussions on regularity of solutions to Hamilton--Jacobi equations on Hilbert spaces. The feedback of P. Cardaliaguet on the manuscript is also greatly appreciated. The authors wish to thank the anonymous referee for making pertinent suggestions which improved the  manuscript.

 



%
%
\section{Preliminaries}\label{sec:preliminaries}
We start this section with some well--known definitions in the Hilbert setting as well as in the Wasserstein space. We denote as $\Om:=(0,1)^d\subset \R^d$ the unit cube  and as $\cL^d_\Om$  the Lebesgue measure restricted to $\Om.$  We sometimes refer to any Borel map of $\Om$ to $\M$ as a random variable. We shall work on the Hilbert space 
$$\bH:= L^2(\Om;\R^d),$$
the set of square integrable Borel vector fields with respect to $\Om$. 

Since it is more convenient to write $\M^m$ instead of $(\R^d)^m$, we shall use write $\M$ in place of $\R^d.$ Letters $x,y$ are typically used for elements of $\bH$, while elements of $\M$ are typically denoted by $q,p,v$. Sometimes, we also use the notation $\R_+:=[0,+\infty)$.

Given two topological spaces $\mathbb S_1$ and $\mathbb S_2$, a Borel measure $\mu$ on $\mathbb S_1$  and a Borel map $X: \mathbb S_1 \rightarrow \mathbb S_2$, $X_\sharp \mu$ is the measure on $\mathbb S_2$ defined as  $X_\sharp \mu(B)= \mu\bigl(X^{-1}(B) \bigr)$ for $B \subset \mathbb S_2$. 

The canonical projections $\pi^1, \pi^2: \M \times \M \rightarrow \M$ are defined as 
$$\pi^1(q_1, q_2)=q_1, \quad \pi^2(q_1, q_2)=q_2 \qquad \forall q_1, q_2 \in \M.$$   

Given $\mu_0, \mu_1 \in \cP_2(\M)$, we denote as $\Gamma(\mu_0, \mu_1)$ the set of Borel probability measures $\gamma$ on $\M \times \M$ such that $\pi^1_\sharp \gamma=\mu_0$ and $\pi^2_\sharp \gamma=\mu_1$. We denote as $\Gamma_o(\mu_0, \mu_1)$ the set of $\gamma\in \Gamma(\mu_0, \mu_1)$  such that 
\[
W_2^2(\mu_0, \mu_1)= \int_{\R^{2d}} |q_1-q_2|^2 \gamma(dq_1, dq_2).
\]
The law of $x \in \bH$ is the Borel probability measure $\sharp(x):=x_\sharp \cL^d_\Om.$ The map $\sharp$ maps  $\bH$ onto $\cP_2(\M)$, the set of Borel probability measure on $\M$ of finite second moments. One basic result in measure theory is that as $\Om$ has no atoms, any Borel probability measure on $\R^d$ is the law of a Borel map $z:\Om \rightarrow \R^d.$

If $\mu \in \cP_2(\M)$, the set of Borel vector fields $\xi: \M \rightarrow \M$ which are square integrable is denoted as $L^2(\mu).$  The tangent space to $\cP_2(\M)$ at $\mu$ denoted as $T_\mu \cP_2(\M)$ is closure of $\nabla C_c^\infty(\M)$ in $L^2(\mu)$. 

If $\tilde\cU:\bH\to\R$ is differentiable at $x\in\bH$, we use the notations $\nabla\tilde\cU(x)$ or $\nabla_x \tilde\cU(x)$ to denote its Fr\'echet derivative at $x$ (as element of $\bH$). If $\tilde\cU$ is twice differentiable at $x$, we use the notations $\nabla^2\tilde\cU(x)$ or $\nabla_{xx}^2\tilde\cU(x)$ to denote its Hessian (as bi-linear form on $\bH\times\bH$). If $u:\M\to\R$ is differentiable at $q\in\M$, we use the notation $Du(q)$ or $D_q u(q)$ to denote its gradient at $q$. If it is twice differentiable at $q$, we use the notations $D^2 u(q)$ or $D^2_{qq}u(q)$ to denote its Hessian matrix at $q$.

For $r>0$, we define $\cB_r$ to be the closed ball in $(\sP_2(\M),W_2)$, centered at  $\delta_0$ and of radius $r.$ $\B_r(0)$ stands for the closed ball in $\bH$ centered at $0$ and of radius $r$.

For any integer $m>1$ we fix $(\Omega_i^m)_{i=1}^m$ to be a partition of $\Omega$ into Borel sets of same volume. Given 
\[
q:=(q_1, \cdots, q_m),  \; p:= (p_1, \cdots, p_m) \in \M^m, \quad 
\]
we set 
\begin{equation}\label{eq:average-q}
M^q:= \sum_{i=1}^m q_i \chi_{\Omega_i^m}, \quad  M^{mp}:= \sum_{i=1}^m (mp_i) \chi_{\Omega^m_i}\equiv m M^p \quad \text{and} \quad \mu^{(m)}_q:= {1\over m} \sum_{i=1}^m \delta_{q_i} .
\end{equation} 
We  set 
\[
\bB^m_r:=\Bigl\{q \in \M^m \; : \; m^{-1} \sum_{j=1}^m |q_j|^2 \leq r^2 \Bigr\}.
\]
and 
\[ 
\sP_{2}^{(m)}(\M):=\left\{{1\over m} \sum_{i=1}^m \delta_{q_i} \; : \; q \in \M^m \right\}.
\]

\vskip0.40cm
\subsection{Assumptions}\label{subsect:assumptions} 

Throughout this manuscript $N \geq 1$ is an integer, $m_*, \lambda_0 \in \R$  and $\kappa_0, \lambda_1, \kappa_3>0.$  We shall denote as $\ov \kappa$ a generic constant depending on $m_*, \kappa_0, r_2, \kappa_3>0.$ 

Let $-\infty< s<t<\infty$ and let $m>1$ be an integer. 

When $\mathbb S$ is a metric space, we denote as $AC_2(s, t; \mathbb S)$ the set of $S: [s, t] \rightarrow \mathbb S$ which are $2$--absolutely continuous. When $\tau \in [s, t]$, when convenient, we write $S_\tau$ in place of $S(\tau).$ We are imposing the following standing assumptions throughout the paper.

Suppose  
\begin{align}\label{ass:F-U0-C11new1} 
\tilde \cF, \; \tilde \cU_0 \in C^{1,1}(\bH), \quad \tilde \cF \geq 0,\; \tilde \cU_0 \geq m_*,  \tag{H\arabic{hyp}}
\end{align} 
and are rearrangement invariant in the sense that if $x, y \in \bH$ have the same law, then $\tilde \cF(x)=\tilde \cF(y)$ and $\tilde \cU_0(x)=\tilde \cU_0(y).$ Note that \eqref{ass:F-U0-C11new1} implies in particular that there exists $\kappa_0>0$ such that
and
\begin{equation}\label{ass:LipschitzUandF}
\nabla \tilde  \cF, \nabla \tilde  \cU_0 \quad \text{are}\;\; \kappa_0 \text{-Lipschitz continuous}.
\end{equation} 
We assume \stepcounter{hyp}
\begin{align}\label{ass:F-U0-C11new2} 
\tilde  \cU_0 \; \text{is convex.} \tag{H\arabic{hyp}}
\end{align} 

Let \stepcounter{hyp}
\begin{equation}\label{ass:Hamiltonian}
H, L \in C^{N+1}(\M \times \R^d), \quad L \geq 0,
\tag{H\arabic{hyp}}
\end{equation} 
such that  $L(q, \cdot)$ and $H(q, \cdot)$ are Legendre transforms of each other for any $q \in \M.$ We assume \stepcounter{hyp}
\begin{equation}\label{ass:Hessianv-v}
D^2_{vv} L \geq \kappa_3 I_d,\quad  D^2_{pp} H>0, \tag{H\arabic{hyp}}
\end{equation} 
and \stepcounter{hyp}
\begin{equation}\label{ass:LipschitzH}
D H,\;\; D L  \;\; \text{are}\;\; \kappa_0\text{-Lipschitz continuous}. \tag{H\arabic{hyp}}
\end{equation} 
We further assume \stepcounter{hyp}
\begin{equation}\label{ass:UpperboundonL}
\lambda_1|v|^2 + \lambda_0 \leq L(q, v). \tag{H\arabic{hyp}}
\end{equation} 

We set 
\[
\tilde \cL(x, a)= \int_\Omega L\bigl(x(\omega), a(\omega) \bigr)d\omega +\tilde \cF(x), \qquad \tilde \cH(x, b)= \int_\Omega H\bigl(x(\omega), b(\omega) \bigr)d\omega -\tilde \cF(x)
\]
for $x, a, b \in \bH$ and assume \stepcounter{hyp}
\begin{equation}\label{ass:convexity-on-tildeL}
\tilde \cL \quad \text{is jointly  strictly convex in both variables}. \tag{H\arabic{hyp}}
\end{equation} 
\stepcounter{hyp}
Observe that a sufficient condtion for \eqref{ass:convexity-on-tildeL} to be satisfied is to assume existence of a constant $\kappa_1>0$ such that $\tilde  \cF$ is $\kappa_1$-convex and  that there exists $\kappa_2>0$ such that 
\begin{equation}\label{ass:convexity-on-tildeLC2}
D^2 L(\overline q, \overline{v}) \begin{pmatrix}q\\ v \end{pmatrix}\cdot \begin{pmatrix}q\\ v \end{pmatrix}\geq \kappa_2 |v|^2 \qquad \forall q, \overline q, v, \overline v \in \R^d.
\end{equation} 
In  this case, the strict convexity of $\tilde\cL$ would follow from the fact that   
\begin{equation}\label{ass:convexity-on-tildeLC3}
{d^2 \over dt^2} \tilde \cL(\overline x+t x, \overline a+t a)\Big|_{t=0}\geq \kappa_1 \|x\|^2+\kappa_2 \|a\|^2 \qquad \forall x, a, \overline x, \overline a \in \bH.
\end{equation} 

The regularity assumptions \eqref{ass:F-U0-C11new1} and \eqref{ass:Hamiltonian} will be important to derive regularity estimates on the classical solution $\tilde\cU$ to the corresponding Hamilton-Jacobi equation. At a first glance these are sufficient to obtain well-known semi-concavity and Lipschitz estimates on this solution. The convexity of $\tilde\cL$ in \eqref{ass:convexity-on-tildeL} and of $\tilde\cU_0$ in \eqref{ass:F-U0-C11new2} will then imply that $\tilde\cU(t,\cdot)$ (as a value function in an optimal control problem) is convex. Together with the previous properties this will lead to the $C^{1,1}$ regularity on $\tilde\cU(t,\cdot)$. To be able to achieve higher regularity estimates on $\tilde\cU(t,\cdot)$ that will be necessary to derive the corresponding master equations, additional assumptions will be introduced in Section \ref{sec:master_equations}. The combination of  \eqref{ass:F-U0-C11new1} and \eqref{ass:LipschitzH} ensures that the underlying Hamiltonian flow is globally well--posed. We combine \eqref{ass:UpperboundonL} and \eqref{ass:convexity-on-tildeL} to obtain existence and uniqueness of solutions to the  optimal control problems associated to $\tilde\cU(t,\cdot)$. Finally, the strict convexity assumptions in \eqref{ass:Hamiltonian} will help us to deduce the invertibility of the Hamiltonian flow and by this linking it to the optimal curve in the definition of $\tilde\cU(t,\cdot)$.

For any $S \in AC_2(s, t; \bH)$ we set 
\[
\tilde \cA_s^t(S):= \int_s^t \tilde \cL(S, \dot S) d\tau.
\]
When $x, y \in \bH$ we set 
\[
\tilde C_s^t(x, y):= \inf_S \Bigl\{ \tilde \cA_s^t(S) : S(0)=x, S(t)= y, \; S \in AC_2(s, t; \bH) \Bigr\}
\]
and define for $t>0$, 
\begin{equation}\label{eq:defineUtilde}
\tilde \cU(t, y)= \inf_{z \in \bH}\left\{ \tilde C_0^t(z, y) +\tilde \cU_0(z)\right\}.
\end{equation}
We denote as $AC_2(0, t; \bH_y)$  the set of $S \in AC_2(0, t; \bH)$ such that $A_0^t(S)<\infty$ and $S(t)=y.$ Strict convexity of $\tilde \cA_s^t$ is ensured by \eqref{ass:convexity-on-tildeL}. 

%
\begin{remark}\label{rem:new-bound-onHL}  The following hold. 
\begin{enumerate}
\item[(i)] Using \eqref{ass:LipschitzH}, we obtain that $|H|$ and $|L|$ are bounded above by quadratic forms.
\item[(ii)] Note that by \eqref{ass:F-U0-C11new1} and \eqref{ass:UpperboundonL}, 
$$\tilde \cA_0^t(S) \geq \lambda_1 \int_0^t \|\dot S\|^2 d\tau + \lambda_0 t+m_*.$$ 
This ensures a pre--compactness property to the sub-level sets of $\tilde \cA_0^t$ when they are contained in  $AC_2(0, t; \bH_y)$ for some $y \in \bH.$  
\item[(iii)] The  functions $D L$, $D H$,  $\nabla \tilde \cU_0$ and $\nabla \tilde \cF$ being Lipschitz,  there is a constant $\ov \kappa$ such that  
\[
|D L(q, v)| \leq \ov {\kappa} (|v|+|q|+1), \; |D H(q, p)| \leq \ov {\kappa} (|p|+|q|+1), \;\; \|\nabla \tilde \cU_0(x)\|+ \|\nabla \tilde \cF(x)\| \leq \ov {\kappa} (\|x\|+1).
\]
\end{enumerate}
\end{remark}

The assumptions imposed on $H$ and $\tilde \cF$ ensure $\nabla \tilde \cH: \bH^2 \rightarrow \R$ is Lipschitz and so, there exists a unique Hamiltonian flow $\Sigma: \R \times \bH^2 \rightarrow \bH^2$ on the phase space, solution to the initial value problem \eqref{eq:Hamiltonian-Flow}.
By Remark \ref{rem:new-bound-onHL} (iii) there exists a constant $\tilde \kappa>\overline \kappa$ depending only on $\ov \kappa$ such that 
\begin{equation}\label{eq:Hamiltonian-Flow3}
\|\Sigma(t, x, b)\|+1 \leq \big(\|(x, b)\| +1 \big) e^{\tilde \kappa t}  
\end{equation} 
for any $t>0$ and $x, b \in \bH$. The restriction of $\Sigma$ to the graph of $\nabla \tilde \cU_0$ is the flow map denoted as $(\tilde \xi, \tilde \eta)$ (defined in \eqref{eq:Hamiltonian-Flow2}) on the spatial space, with values in the cotangent bundle. We combine \eqref{ass:LipschitzUandF} and \eqref{eq:Hamiltonian-Flow3} to find $c_5>0$ depending only on $\kappa_0$ and $\|\nabla \tilde \cU_0(0)\|$ such that 
\begin{equation}\label{eq:Hamiltonian-Flow4}
 \| (\tilde \xi, \tilde \eta)\| +1 \leq c_5 \big(\|x\| +1 \big) e^{\tilde \kappa t}. 
\end{equation} 
We discuss some more classical properties of the Hamiltonian flow in the setting of Hilbert spaces in Appendix \ref{appendixC}.

%
%
%
%
%
%
\vskip0.40cm
\subsection{Discretization}\label{subsec:discretization}
Fix a natural number $m>1$.   For $q, v, p \in \M^m$ we define 
\[
 L^{(m)}(q, v):= \int_\Omega L(M^q, M^v) d\omega= {1 \over m}\sum_{i=1}^m L(q_i, v_i), \quad F^{(m)}(q):= \tilde \cF \bigl(M^q \bigr) 
\]
and 
\[
 H^{(m)}(q, p):= \int_\Omega H(M^q, M^{mp}) d\omega={1 \over m}\sum_{i=1}^m H(q_i, m p_i). 
\]
Then we set 

\[
 \cL^m(q, v):= L^{(m)}(q, v) + F^{(m)}(q), \quad  \cH^m(q, p):=  H^{(m)}(q, p)-F^{(m)}(q), \quad U^{(m)}(t, q):=\tilde \cU(t, M^q). 
\]


One checks that for each $j \in \{1, \cdots, m\}$, $\nabla \tilde \cU(t, M^q)$ is constant on $\Omega_j^m$ and the following useful identities (see for instance \cite{CarmonaD-I, GangboS2015}) hold:  
\begin{equation}\label{eq:discretegra1}
D_{q_j}  U^{(m)}(t, q_1, \cdots, q_m)= {1\over m}  \nabla \tilde \cU(t, M^q)|_{\Omega^m_j}  .
\end{equation} 
Note this means in particular, 
\begin{equation}\label{eq:finite-maps} 
\nabla \tilde \cU_0 : \{M^q\; : \; q \in \M^m \} \rightarrow \{M^q\; : \; q \in \M^m \}.
\end{equation} 

We infer   
\begin{equation}\label{eq:discretegra1b}
\nabla \tilde \cU(t, M^q) = m \sum_{j=1}^m \chi_{\Omega_j^m} D_{q_j} U^{(m)}(t, q).
\end{equation} 
Observe 
\begin{equation}\label{eq:discretegra2}
D_{q_j} \cL^m( q, v)= {1\over m}  \nabla_x \tilde \cL(M^q, M^v)|_{\Omega_j^m}, \quad  D_{v_j} \cL^m(q, v)= {1\over m}  \nabla_a \tilde \cL(M^q, M^v)|_{\Omega_j^m}, 
\end{equation} 
and so,  
\begin{equation}\label{eq:discretegra2new}
\nabla_x \tilde \cL(M^q, M^v)=m  \sum_{j=1}^m \chi_{\Omega_j^m} D_{q_j} \cL^m( q, v), \quad   \nabla_a \tilde \cL(M^q, M^v)=m  \sum_{j=1}^m \chi_{\Omega_j^m} 
  D_{v_j} \cL^m(q, v).
\end{equation} 
Similarly,  
\begin{equation}\label{eq:discretegra5}
D_{q_j} \cH^m(q, p)= {1\over m}  \nabla_x \tilde \cH(M^q, M^{mp})|_{\Omega_j}^m, \quad  D_{p_j} \cH^m(q, p)=   \nabla_b \tilde \cH(M^q, M^{mp})|_{\Omega_j^m}.
\end{equation} 
Note that the fact that the coefficient in front of $\nabla_b \tilde \cH(M^q, M^{mp})$ is not divided by $m$ is not a misprint. However, we have 
\begin{equation}\label{eq:discretegra6}
D_{q_j} \cH^m\Bigl(q, D_q U^{(m)}(t, q)\Bigr)={1 \over m} \nabla_x \tilde \cH\Bigl(M^q, \nabla \tilde \cU(t, M^q)\Bigr){ |_{\Omega_j^m}}, 
\end{equation} 
and so, 
\begin{equation}\label{eq:discretegra7}
{1 \over m} \nabla_x  \tilde \cH\Bigl(M^q, \nabla \tilde \cU(t, M^q)\Bigr)= \sum_{j=1}^m D_{q_j} \cH^m\Bigl(q, D_q U^{(m)}(t, q)\Bigr) \chi_{\Omega_j^m}. 
\end{equation} 

For any natural number $m$ denote as $(\Sigma^m_1, \Sigma^m_2): \R \times \M^{2m} \rightarrow \M^{2m}$ the Hamiltonian flow for $\cH^m$. 
For $x\in \bH$ such that $\sharp(x)=\mu^{(m)}_q$ (i.e. $x=M^q$), 
we consider the spatially discretized flows 
\begin{equation}\label{eq:define-xi-eta}
\xi^m_i(s,q):=\tilde\xi_s[x]|_{\Om_i^m}, \quad \eta^m_i(s,q)={1\over m}\tilde\eta_s[x]|_{\Om_i^m}.
\end{equation}
Using the notation $(\xi^m, \eta^m)=(\xi^m_1, \cdots, \xi^m_m, \eta^m_1, \cdots,  \eta^m_m)$, these flows are uniquely defined to satisfy 
\begin{equation}\label{eq:app_Hamiltonian_sys-discrete} 
\left\{
\begin{array}{ll}
\dot \xi^m_i(s, q)&= \hfill D_{p_i}  \cH^m \big(\xi^m_i(s, q),  \eta^m_i(s, q) \big),\quad \;\, \mbox{for}\,\,(s, q) \in (0,\infty)\times \M^m,\\
\dot \eta^m_i(s,q)& =- D_{q_i}  \cH^m  \big(\xi^m_i(s, q),  \eta^m_i(s, q) \big),\;\;\; \mbox{for} \; (s, q) \in (0,\infty)\times \M^m,\\
\big(\xi^m(0, q),\eta^m(0, q) \big)&=  \;\;\; \big(q, D_q U_0^{(m)}(q)\big) ,\qquad \qquad \quad \, \;\, \mbox{for}\; q \in \M^m.
\end{array}
\right.
\end{equation}

%
%
%
%
%
\vskip0.40cm
\subsection{Direct arguments for $C^{1,1}_{\rm loc}$--regularity in Hilbert setting} \label{subsec:standardresults} Throughout this subsection, we impose \eqref{ass:F-U0-C11new1}-\eqref{ass:convexity-on-tildeL}. We rely on the theory of existence of solutions to Hamilton--Jacobi equations on Hilbert spaces  developed in  \cite{CrandallLions} and \cite{CL2}. The function $\tilde U$ defined in \eqref{eq:defineUtilde} is the unique viscosity solution to 
\begin{equation}\label{eq:Hamilton-JacobiHil}
\left\{
\begin{array}{ll}
&
\partial_t \tilde \cU+\tilde \cH\bigl(x, \nabla \tilde \cU\bigr)=0,\quad\mbox{in}\,\,(0,\infty)\times \bH,
\\
&
\tilde  \cU(0,\cdot)=\tilde \cU_0\qquad \qquad \quad \;\; \mbox{on}\,\,\bH.
\end{array}
\right.
\end{equation}

In this subsection, basic analytical tools are used to verify that $\tilde \cU$ is of class $C^{1,1}_{\rm loc}.$ We refer the reader to \cite{GomesN2014} for instance for the proof of the following proposition.
\begin{proposition}\label{prop:semi-convex-time_Lip1}  There exists $e_1 \in C(\R_+, \R_+)$ monotone nondecreasing  such that the following hold for $T>0$, and $r>0$.
\begin{enumerate}
\item[(i)] $\tilde \cU$ is $e_1\bigl(r(T+1) \bigr)$--Lipschitz on $[0, T] \times \bB_r(0).$ 
\item[(ii)] $\tilde \cU(t, \cdot)$ is $e_1\bigl(r(t+1) \bigr)$--semiconcave on $\bB_r(0)$ for $t \in [0, T].$ 
\end{enumerate}
\end{proposition} 
%

\begin{proposition}\label{prop:semi-convex1} There is an increasing function  $e_1 \in C(\R_+, \R_+)$ such that if $t>0$ then     
\begin{enumerate}
\item[(i)] $\tilde  \cU(t, \cdot)$ is rearrangement invariant.
\item[(ii)] $\tilde  \cU(t, \cdot)$ is convex and so, it is differentiable and $\nabla \tilde  \cU(t, \cdot)$ is $e_1\bigl(r(t+1) \bigr)$--Lipschitz on $ \bB_r(0).$
\end{enumerate}
\end{proposition} 
\proof{} (i) The invariance property imposed on $\tilde \cU_0$ and $\tilde \cF$ implies $\tilde \cL$ satisfies the invariance property 
\[
\tilde \cL(x, a)= \tilde \cL(x \circ E, a \circ E)
\]
for $x, a \in \bH$, $E:\Om\to\Om$ such that $E$ preserves Lebesgue measure. Since $\tilde \cL$ is further continuous, we conclude that $\tilde  \cU(t, \cdot)$ is rearrangement invariant for $t\geq 0$ (cf. \cite{GangboT2017}).

(ii) The convexity of $A_0^t$ on $AC_2(0, t; \bH)$ {and \eqref{ass:F-U0-C11new2}} yields the convexity of $\tilde  \cU(t, \cdot)$ on $\bH.$ This, together with Proposition \ref{prop:semi-convex-time_Lip1} (ii) completes the proof.\endproof
\begin{remark}\label{rem:discrete-cont}
Let $q \in \M^m.$  Note $\sigma \mapsto \int_0^t \cL^m(\sigma, \dot \sigma)d\tau+  U_0^{(m)}( \sigma(0))$ is strictly convex on $AC_2\bigl(0, t; q; \mathbb R^{md}\bigr),$  the set of paths $\sigma \in AC_2\bigl(0, t; \mathbb R^{md}\bigr),$ such that $\sigma(t)=q.$ Since $\cL^m$ is of class $C^2$ and satisfies the assumptions in Subsection \ref{subsect:assumptions},  standard  results of the calculus of variations ensure that $\int_0^t \cL^m(\sigma, \dot \sigma)d\tau+   U^{(m)}_0(\sigma(0))$ admits a unique minimizer $\sigma^m$ on $AC_2\bigl(0, t; q; \M^m\bigr).$ The minimizer is completely characterized by the Euler--Lagrange equations 
\begin{equation}\label{eq:discretegra1c}
{d \over d\tau} \Bigl(D_v  \cL^m(\sigma^m, \dot \sigma^m )  \Bigr) = D_q  \cL^m(\sigma^m, \dot \sigma^m ), \quad \sigma^m(t)=q, \quad D_{q} U^{(m)}_0(\sigma^m(0))= D_q  \cL^m(\sigma^m(0), \dot \sigma^m(0) ).
\end{equation} 
Define 
\[
U^m(t, q):= \int_0^t \cL^m(\sigma^m, \dot \sigma^m) d\tau+  U_0^{(m)}( \sigma^m(0)).
\]
Then  it is well--known that  $U^m$ is the unique continuous viscosity solution to 
\begin{equation}\label{eq:discrte-visc1}
\partial_t  U^m + \cH^m\bigl(q, D_q    U^m \bigr)=0, \qquad \text{on} \quad (0,\infty) \times \M^{m}, \qquad U^m(0, \cdot)=U_0^{(m)}.
\end{equation}

Setting $S:=M^{\sigma^m }$, we have $\dot S=M^{\dot \sigma^m }.$ We use \eqref{eq:discretegra1b} at $t=0$, then use \eqref{eq:discretegra2new} and \eqref{eq:discretegra1c} to obtain  
\[
{d \over d\tau} \Bigl(\nabla_{a}  \tilde \cL(S, \dot S )  \Bigr) = \nabla_x  \tilde \cL(S, \dot S ), \quad \nabla \tilde \cU_0(S(0))= \nabla_a  \tilde \cL(S(0), \dot S(0) ).
\]
This means $S$ is a critical point of $A_0^t$ over $AC_2(0, t; \bH_y)$ if we set $y:=M^q.$ Since $A_0^t$ is convex over $AC_2(0, t; \bH_y)$, we conclude that $S$ is a minimizer of $A_0^t$ over $AC_2(0, t; \bH_y).$ Thus,
\begin{equation}\label{eq:crucial}
 U^m(t, q)= A_0^t(S)= \tilde \cU(t, M^q)=  U^{(m)}(t, q).
\end{equation}
Consequently, $U^{(m)}$ is the unique viscosity solution to \eqref{eq:discrte-visc1}. We emphasize that the observation \eqref{eq:crucial} is crucial in our consideration and in fact represents the heart of our analysis. This is a feature of the deterministic setting and so, this approach might not be applicable to stochastic Hamiltonian systems.
\end{remark}
%

%
%

The proof of the following proposition will be provided in the Appendix \ref{subsec:theorem:hilbert-smooth2}.

\begin{proposition}\label{theorem:hilbert-smooth}  There exists $e_0:[0,\infty) \rightarrow [0,\infty)$, monotone non--decreasing  such that the following hold.       
\begin{enumerate}
\item[(i)] If $0\leq t_1<t_2 \leq T$ then  
\[
\tilde \cU(t_2, y)- \tilde \cU(t_1, y)= - \int_{t_1}^{t_2} \tilde \cH\bigl(y, \nabla \tilde \cU(\tau, y) \bigr) d\tau \qquad \forall y \in \bH.
\]
\item[(ii)] $\tilde \cU$ is continuously differentiable on $(0,\infty) \times \bH$ and $\partial_t \tilde \cU,$ $\nabla \tilde \cU$ are Lipschitz on  $[0,T] \times \bB_r(0)$. 
\item[(iii)] For any $y \in \bH,$ there exists a unique $S \in AC_2(0,t; \bH_y)$ such that $\tilde \cU(t, y)= \tilde \cA_0^t(S)+\tilde \cU_0(S(0)).$   
\item[(iv)]  Let $S$ be as in (iii) and set $P:= \nabla_a \tilde \cL(S, \dot S)$. Then $S, P \in C^2([0, t]; \bH)$,    
\begin{equation}\label{eq:Euler-Lagrange-L}
\dot S=\nabla_b \tilde \cH(S, P), \quad \dot P= \nabla_x \tilde \cL(S, \dot S)=-\nabla_x \tilde \cH(S, P), \quad  \nabla \tilde \cU(\cdot, S)= \nabla_{a} \tilde \cL(S, \dot S) \quad \text{on} \;\; [0,t].
\end{equation} 
In particular,  
\begin{equation}\label{eq:Euler-Lagrange-L-bdry}
\nabla \tilde \cU_0(S(0))= \nabla_a \tilde \cL(S(0), \dot S(0)). 
\end{equation} 
\item[(v)] We have 
\[
\tilde C_0^t(S(0), y), \;  \|\dot S(\tau)\| \leq  e_0\bigl((t+1)\|y\|\bigr), \quad \|S(\tau)\| \leq \|y\|+te_0\bigl((t+1)\|y\|\bigr) \qquad \forall \tau \in [0,t].
\] 

\end{enumerate}
\end{proposition} 
%
%
%
%
%
\begin{remark}\label{rem:uniquestrict-c1}  (i) We denote the unique  $S$ which appears in Proposition \ref{theorem:hilbert-smooth} (iii) as 
\[
\tilde S_s^t[y](\omega):=S(s, \omega), \qquad 0\leq s \leq t, \quad \omega \in \Omega.
\]
It is uniquely characterized by the equation 
\begin{equation}\label{eq:characterizeS} 
\tilde \cU(t, y)=\int_0^t \tilde \cL\Big(\tilde S_s^t[y], \partial_s \tilde S_s^t[y]\Big) ds +\tilde \cU_0\big(\tilde S_0^t[y]\big), \qquad S_t^t[y]=y.
\end{equation}
Defining 
\[
\tilde P_s^t[y]= \nabla_a \tilde \cL\big( \tilde S_s^t[y], \partial_s \tilde S_s^t[y]\big),
\]
we have 
\begin{equation}\label{eq:app_Hamiltonian_sys_1} 
\left\{
\begin{array}{ll}
\partial_s \tilde S_s^t[y] &= \hfill \nabla_b \tilde \cH \big(\tilde S_s^t[y], \tilde P_s^t[y] \big),\quad \; \mbox{for}\,\, (s,y) \in (0,t)\times \bH,\\
\partial_s \tilde P_s^t[y]  &=- \nabla_x \tilde \cH  \big(\tilde S_s^t[y], \tilde P_s^t[y] \big),\;\mbox{for}\,\, (s, y) \in (0,t)\times \bH\\
\big( \tilde S_t^t[y],  \tilde P_0^t[y]\big) &= \;\;  \big(y, \nabla \tilde \cU_0(y)\big),\qquad \quad \,\mbox{for}\,\,y \in \bH.
\end{array}
\right.
\end{equation} 


(ii) For any natural number $m$ and $q\in \M^m$,  we have 
\begin{equation}\label{eq:relation-m-infty}  
\tilde S_s^t\big[M^q\big]=M^{\sigma_s^{t, m}[q]},
\end{equation}
where $(\sigma_s^{t,m}[q])_{s\in(0,t)}$ is the optimizer discussed in Remark \ref{rem:discrete-cont}. Let us emphasize only in the deterministic hamiltonian systems like ours, \eqref{eq:relation-m-infty} provides us with characteristics not only the viscosity solutions of the Hamilton-Jacobi equation on $\bH$ but also the one on $\M^m$.

(iii) When the conditions in Remark \ref{rem:uniquestrict-c1} are satisfied,  we define the vector field  
\begin{equation}\label{eq:define-velocity} 
B(t, \cdot):= \nabla_b \tilde \cH\big(\cdot,  \tilde \eta(t, \tilde S_0^t)   \big).
\end{equation}
which will turn out to be the velocity in Eulerian coordinates for the trajectory $\tilde\xi.$ 
\end{remark}

%
%
\section{Regularity estimates for HJEs and Hamiltonian systems for systems of $m$ particles.}\label{sec:app-reg} 
In this section, we assume  that (\ref{ass:Hamiltonian}) - (\ref{ass:UpperboundonL}) hold. Let $u_0\in C^N(\M)$ be a convex function with bounded second derivatives. Let $F\in C^N(\M)$ and $L$ be such that the corresponding Lagrangian action, as in \eqref{ass:convexity-on-tildeL}, is strictly convex.  We fix $T>0.$ We shall show that classical solutions to Hamilton-Jacobi equations set on $\M^m,$ possess higher derivative estimates that we precisely quantify in terms of $m$. As we will see in the next sections, when $m\to+\infty$, these estimates will provide the necessary regularity estimates on $\cU$, the solution to the corresponding Hamilton-Jacobi equation set on $\cP_2(\M).$

\vskip0.40cm
\subsection{One particle Hamiltonian flow} \label{subsec:one-d} 
We study the regularity of viscosity solutions $u:[0,T]\times\M\to \R$ of Cauchy problems of the form
\begin{equation}\label{eq:hj_appendix}
\left\{
\begin{array}{ll}
\partial_t u + H(q,\nabla u)-F(q)=0, & (0,T)\times\M,\\
u(0,\cdot)=u_0, & \M.
\end{array}
\right.
\end{equation}
%
%
Given  $t\in(0,T]$, we consider the Hamiltonian system 
\begin{equation}\label{eq:app_Hamiltonian_sys_1new2}
\left\{
\begin{array}{ll}
\dot S(s,q)=D_p H(S(s,q),P(s,q)), & s\in(0,t), \; q\in \M,\\[5pt]
\dot P(s,q)=-D_q H(S(s,q),P(s,q))+D_q F(Q(s,q)), & s\in(0,t), \; q\in \M,\\[5pt]
S(t,q)=q,\ P(0,q)=D u_0(S(0,q)),  \; q\in \M.
\end{array}
\right.
\end{equation} 
Such a flow has been considered in a greater generality in Remark \ref{rem:uniquestrict-c1}.  
Recall $S$ is the unique optimizer in 
\begin{align}\label{prob:action-min}
u(t,x):=\inf\left\{ u_0(\g(0))+\int_0^{t} L(\g(s),\dot\g(s))+F(\g(s))\dd s:\ \g(t)=x \right\}.
\end{align}
Similarly, we shall use the flow 
\begin{equation}\label{eq:app_Hamiltonian_sys}
\left\{
\begin{array}{ll}
\dot\xi(s,z)=D_p H(\xi(s,z),\eta(s,z)), & s\in(0,t), \; z\in\M\\[5pt]
\dot \eta(s,z)=-D_q H(\xi(s,z),\eta(s,z))+\nabla_q F(\xi(s,z)), & s\in(0,t),  \; z\in\M\\[5pt]
\xi(0,z)=z,\ \eta(0,z)=D u_0(z), \qquad  \; z\in\M
\end{array}
\right.
\end{equation}
denoted as $(\tilde\xi, \tilde\eta)$ in \eqref{eq:Hamiltonian-Flow} when our Hilbert space reduces to $\M.$
\begin{lemma}\label{lem:app_flow-regularity} Let  $t\in[0,T]$. 
\begin{itemize}
\item[(1)] The map $\xi_t: \M \rightarrow \M$ is a homeomorphism $S_s:=\xi_s \circ\xi^{-1}_t$ and $P_s:=\eta_s\circ \xi^{-1}_t.$ 
We have $\xi_t,\eta_t \in C^{N-1}(\M).$ 
\item[(2)] If we further assume $N\ge 2,$ then $u\in C^{1,1}_{\rm{loc}}([0,T]\times\M)$ is classical solution to \eqref{eq:hj_appendix} and $z\mapsto\xi(t,z)$ is a $C^{N-1}$ diffeomorphism from $\M$ onto itself. 
\end{itemize}
\end{lemma}

\begin{proof}
(1) The existence and smooth dependence on the data of the solution of \eqref{eq:app_Hamiltonian_sys_1new2} is classical, Proposition \ref{prop:homeo1} ensures $\xi_t: \M \rightarrow \M$ is a homeomorphism and $S(s,\cdot):=\xi_s\circ \xi^{-1}_t,$ $ P(s,\cdot):=\eta_s\circ \xi^{-1}_t.$

(2) By Proposition \ref{theorem:hilbert-smooth},  $u\in C^{1,1}_{\rm{loc}}([0,T]\times\M)$ and is classical solution to \eqref{eq:hj_appendix}. Let us show that $z\mapsto\xi(t,z)$ is a global $C^{N-1}$ diffeomorphism. Recall that by Proposition \ref{prop:homeo1}, $\xi$ is a solution to
\begin{align*}
\left\{
\begin{array}{ll}
\dot\xi(s,z)=D_p H(\xi(s,z),D u(s,\xi(s,z))), & s\in(0,t),\\
\xi(0,z)=z,
\end{array}
\right.
\end{align*}
from where one has
\begin{align*}
\left\{
\begin{array}{ll}
\partial_s D_z\xi(s,z)=A(s,z)D_z\xi(s,z), & s\in(0,t),\\
D_z\xi(0,z)=I_d.
\end{array}
\right.
\end{align*}
Here we used the notation 
$$A(s,z):=D^2_{xp}H(\xi(s,z),D u(s,\xi(s,z)))+D^2_{pp}H(\xi(s,z),D u(s,\xi(s,z)))D^2u(s,\xi(s,z)).$$ 
Since $A(s,z)$ is locally uniformly bounded, we have that for $s>0$ small enough $D_z\xi(s,z)$ is invertible. Therefore, Jacobi's formula yields
$$\det(D_z\xi(s,z))=\exp\left(\int_0^s{\rm{tr}}(A(\t,z))\dd\t\right).$$
Since $A(\t,\cdot)\in L^\infty_{\rm{loc}}(\M)$, uniformly with respect to $\t\in[0,t]$, we have that $\det(D_z\xi(s,z))>0$ for all $z\in\M$, uniformly with respect to $s$. Therefore, $D_z\xi(s,z)$ is invertible for any $z\in\M$ and for any $s\in[0,t]$. Thus, by the fact that $\xi(t,\cdot)\in C^{N-1}(\M)$ and the that $\xi(t,\cdot)$ is bijective, we conclude that $z\mapsto\xi(t,z)$ is a global $C^{N-1}$ diffeomorphism of $\M$ onto itself.
\end{proof}

\vskip0.40cm
\subsection{$m$-particles Hamiltonian flow} \label{subsec:multi-d}  Throughout this subsection, we assume to be given a positive monotone nondecreasing function $C_0:(0,\infty) \rightarrow (0,\infty).$  Furthermore, we impose that in the assumption \eqref{ass:Hamiltonian} $N\ge 2$ and $F^{(m)},U_0^{(m)}\in C^3(\M^m)$.


As in Subsection \ref{subsec:discretization} we define 
$$U_0^{(m)}(q):=  \cU_0\bigg({1\over m} \sum_{i=1}^m \delta_{q_i} \bigg), \quad F^{(m)}(q) :=\cF\bigg({1\over m} \sum_{i=1}^m \delta_{q_i} \bigg) \qquad \forall q \in \M^m.$$

We assume to be given $U_0^{(m)},F^{(m)}:\M^m\to\R$ satisfying Property \ref{def:app_reg_estim}(2) with $C=C_0(r).$ We also consider viscosity solutions $U^{(m)}:[0,T]\times\M^m\to\R$ of the Hamilton-Jacobi equation
\begin{equation}\label{eq:app-HJ_finite_dim}
\left\{
\begin{array}{ll}
\partial_t U^{(m)}(t,q )+H^{(m)}(q ,D_{q}U^{(m)}(t, q))-F^{(m)}(q)=0, & {\rm{on\ }} (0,T)\times \M^m,\\
U^{(m)}(0,\cdot)=U^{(m)}_0, & {\rm{on\ }} \M^m.
\end{array}
\right.
\end{equation}
By Remark \ref{rem:discrete-cont} 
$$U^{(m)}(t, q) \equiv \tilde \cU(t,M^q) \qquad \forall (t, q) \in [0,\infty) \times \M^m.$$ 

Given $t\in(0,T)$ we consider the $m$ particles flows $S^{t, m}, P^{t, m}: \M^m \rightarrow \M^m$. In other words,  
\begin{equation}\label{eq:app_Hamiltonian_sys_1-m}
\left\{
\begin{array}{ll}
\dot S_i^{t,m}(s,q)=D_p H(S_i^{t,m}(s,q ),mP_i^{t,m}(s,q)), & (s, q) \in(0,t) \times \M^m,\\[5pt]
\dot P_i^{t,m}(s,q )=-\frac{1}{m}D_q H(S_i^{t,m}(s,q),mP_i^{t,m}(s,q))+D_{q_i} F^{(m)}(S^{t,m}(s,q )), & (s, q) \in(0,t) \times \M^m,\\[5pt]
S_i^{t,m}(t,q)=q_i,\ P_i^{t,m}(0,q)=D_{q_i} U^{(m)}_0(S^{t,m}(0,q)) &  q \in \M^m.
\end{array}
\right.
\end{equation} 
This is analogous to the flow $(S^{t, m}, P^{t, m})$ in Remark \ref{rem:uniquestrict-c1} where we have not displayed the $m$ and $t$ dependence to alleviate the notation. We also consider the $m$ particles flows $\xi^m, \eta^m: [0,\infty) \times \M^m \rightarrow \M^m$, similar to \eqref{eq:app_Hamiltonian_sys} (which also correspond to the discretized flow \eqref{eq:app_Hamiltonian_sys-discrete}). They are defined as 
\begin{equation}\label{eq:app_Hamiltonian_sys-m}
\left\{
\begin{array}{ll}
\dot\xi^m_i(s,z )=D_p H(\xi^m_i(s,z),m\eta^m_i(s,z)), & s\in(0,t),\\[5pt]
\dot \eta^m_i(s, z)=-\frac{1}{m}D_q H(\xi^m_i(s, z),m\eta^m_i(s, z))+D_{q_i} F^{(m)}(\xi^m(s, z)), & s\in(0,t),\\[5pt]
\xi^m_i(0,z)=z_i,\ \eta^m_i(0,z)=D_{q_i} U^{(m)}_0( z),
\end{array}
\right.
\end{equation}
for $i\in\{1,\dots,m\}$, where $ z=(z_1,\dots,z_m)\in\M^m$. 

We next introduce functions on $\M^m$ and list some of their special properties which are useful for our study.
\begin{property}\label{def:app_reg_estim}
For a permutation invariant function $G^{(m)}:\M^m\to\R$ we define the following properties by assuming for each $r>0$, there is $C \equiv C(r)$ increasing in $r$ such that the following hold.

\begin{itemize}
\item[(1)] 
\begin{itemize}
\item[(a)]$G^{(m)} \in C_{\rm{loc}}^{0,1}(\M^m)\cap C^1(\M^m)$ and for every $m\in\N$ and $q\in\B_r^m(0)$ we have
\begin{equation}\label{reg:estim_C01}
|D_{q_i}G^{(m)}(q)|\le Cm^{-1},\ \forall i\in\{1,\dots,m\}.  
\end{equation} 
\item[(b)]$G^{(m)} \in C_{\rm{loc}}^{0,1}(\M^m)\cap C^1(\M^m)$ and for every $m\in\N$ and $q\in\B_r^m(0)$ we have
\begin{equation}\label{reg:estim_C01-less_restr}
\sum_{i=1}^m m|D_{q_i}G^{(m)}(q)|^2\le C.
\end{equation}
\end{itemize}
\item[(2)] $G^{(m)} \in C_{\rm{loc}}^{1,1}(\M^m)\cap C^2(\M^m)$ and for every $m\in\N$ and $q\in\B_r^m(0)$ we have 
\begin{equation}\label{reg:estim_C11}
\ds
|D^2_{q_iq_j}G^{(m)}(q)|_{\infty}\le\left\{
\begin{array}{ll}
Cm^{-1}, & i=j; \; i\in\{1,\dots,m\} \\[5pt]
Cm^{-2}, & i\neq j;\; i,j\in\{1,\dots,m\}.
\end{array}
\right.
\end{equation}
Here for $A=(A_{ij})_{i,j=1}^m$, we use the notation $|A|_\infty:=\max_{(i,j)}|A_{ij}|$.
\item[(3)] $G^{(m)} \in C_{\rm{loc}}^{2,1}(\M^m)\cap C^3(\M^m)$ and for every $m\in\N$ and $q\in\B_r^m(0)$ we have
\begin{equation}\label{reg:estim_C21}
\ds
|D^3_{q_iq_jq_k}G^{(m)}(q)|_{\infty}\le\left\{
\begin{array}{ll}
Cm^{-1}, & i=j=k;\; i \in\{1,\dots,m\} \\[5pt]
Cm^{-2}, & (i=j\neq k) \; \text{or} \; (i\neq j=k) \; \text{or} \; (i=k\neq j); \;  i, j, k\in\{1,\dots,m\}\\[5pt]
Cm^{-2}, & i\neq j\neq k,\; i, j, k\in\{1,\dots,m\}.
\end{array}
\right.
\end{equation}
\end{itemize}
Here for $A=(A_{ijk})_{i,j,k=1}^m$, we use the notation $|A|_\infty:=\max_{(i,j,k)}|A_{ijk}|$.
\end{property}

We present now the main theorem of this section.

\begin{theorem}\label{thm:app_regularity-m}
Let $U^{(m)}:(0,T)\times \M^m\to\R$ be the unique viscosity solution of \eqref{eq:app-HJ_finite_dim}, which is constructed by the discretization approach described in Remark \ref{rem:discrete-cont}. Let $r>0$. Then for all $t\in(0,T)$ there exists $C(t,r)>0$ such that the following hold for all $m\in\N$. 
\begin{itemize}
\item[(1)]  $U^{(m)}(t,\cdot)$ satisfies the estimates in Property \ref{def:app_reg_estim}(2) in $\B_r^m(0)$ with constant $C(t,r)$.
\item[(2)] Further assume that $U_0^{(m)}$ and $F^{(m)}$ satisfy Property \ref{def:app_reg_estim}(3) and \eqref{hyp:H_3} takes place. Then $U^{(m)}(t,\cdot)$ satisfies the estimates in Property \ref{def:app_reg_estim}(3) in $\B_r^m(0)$ with constant $C(t,r)$.
\item[(3)] We assume that the assumptions from (1) and \eqref{hyp:D_qH} take place. Then $\partial_t U^{(m)}(t,\cdot)$ satisfies the estimates in Property \ref{def:app_reg_estim}(1)-(b) in $\B_r^m(0)$ with constant $C(t,r)$. 
\end{itemize}
\end{theorem}

\begin{remark}
Since the proof of the previous theorem is quite technical, we summarize its main ideas. First, as a consequence of the results in Section \ref{sec:preliminaries} (in particular in Proposition \ref{theorem:hilbert-smooth}), $U^{(m)}$ is actually a classical solution to \eqref{eq:app-HJ_finite_dim} which is of class $C^{1,1}_{\rm{loc}}$. Then classical results from the literature will imply that it is as smooth as the data $H, F^{(m)}$ and $U_0^{(m)}$ (cf. \cite{CannarsaS}). Therefore, is remains to obtain the precise uniform derivative estimates as claimed in the statement of the theorem. 

A key observation is the well-known representation formula for $D_qU^{(m)}$, i.e. 
$$
D_{q} U^{(m)}(t,q)=\eta^m(t, \cdot)\circ(\xi^m)^{-1}(t, q),
$$
where $(\xi^m,\eta^m)$ is the Hamiltonian flow, the solution to \eqref{eq:app_Hamiltonian_sys-m}. Therefore, the precise derivative estimates on $U^{(m)}$ can be obtained by differentiating the previous formula and relying on fine derivative estimates of the flow $(\xi^m,\eta^m)$ and of its inverse. We obtain these necessary estimates by studying the linearized system (and its derivative) associated to \eqref{eq:app_Hamiltonian_sys-m}. Since these computations will be quite delicate, we identify two simplified systems in Lemma \ref{lem:app_matrix-exponential} and Lemma \ref{lem:app_matrix_exp-2}, which carry the main structure of the original linearized systems. Estimates on these simpler systems will essentially be enough to deduce the estimates on the linearized systems we are aiming for. Finally, the derivative estimates on $\partial_t U^{(m)}$ are obtained by directly differentiating the Hamilton-Jacobi equation and using the previously established estimates on spacial derivatives of $U^{(m)}$.
\end{remark}

\begin{proof}[Proof of Theorem \ref{thm:app_regularity-m}]
We aim to obtain precise upper bounds on expressions depending on $m$ (with respect to $m$, when $m$ is large). For this, we use the standard big-O notation. For instance, if $\alpha$ is an integer and $A(m)$ is a real number depending on $m$, by 
$$A(m)=O(m^\alpha)$$ 
we mean that there exists $C>0$ independent of $m$ such that $|A(m)|\le Cm^{\alpha}$, for all $m$ large. If $A(m)=(a_{ij}(m))_{ij}$ is a matrix whose elements are real numbers depending on $m$, by the abuse of the notation, by $A(m)= O(m^\alpha)$ we mean that there exists a constant $C>0$ independent of $m$ such that $|a_{ij}(m)|\le Cm^\alpha$ for all $i,j$. When $A(m)=(a_{ij}(m))_{ij}$ and $B(m)=(b_{ij}(m))_{ij}$ are matrices, by $A(m)=O(B(m))$ we mean that $a_{ij}(m)=O(b_{ij}(m))$ for all $i,j$. To ease the notation, we sometimes write $A(m)\sim B(m)$ for $A(m)=O(B(m))$  and $B(m)=O(A(m))$.

First, let us notice that by Proposition \ref{theorem:hilbert-smooth}, $U^{(m)}$ is a $C^{1,1}_{\rm{loc}}((0,T)\times\M^m)$ classical solution of \eqref{eq:app-HJ_finite_dim}, therefore in particular any point $(t,q)\in(0,T)\times\M^m$ is regular and not conjugate (by the proof of Lemma \ref{lem:app_flow-regularity}) in the sense of Definition 6.3.4 of \cite{CannarsaS}. 

Furthermore, we notice that Lemma \ref{lem:app_flow-regularity} asserts that $\xi^m(s,\cdot)$ is a $C^N$ diffeomorphism and Theorem 6.4.11 from \cite{CannarsaS} yields that $U^{(m)}\in C^3((0,T)\times\M^m)$. In what follows we aim to obtain quantitative derivative estimates on $U^{(m)}$ with respect to the discretization parameter $m$.

{\bf Step 0.} {\it Basic bounds on $\xi^m(t, z)$ when $q:=\xi^m_t( z)\in \B_r^{m}(0).$ } 

By Proposition \ref{prop:homeo1},  $\xi^m(s, z)=S_s^{t, m}[q]$  since $ q=\xi^m(t, z)$. 
By the same proposition, for $i\in\{1,\dots,m\}$ and $ z\in\M^m$, we have  
\begin{equation}\label{eq:estim_first} 
\left\{
\begin{array}{ll}
\dot\xi_i^m(t, z)=D_p H(\xi_i^m(t, z),mD_{q_i} U^{(m)}(t,\xi^m(t, z))), & t\in(0,T),\\
\xi^m(0, z)=z,
\end{array}
\right.
\end{equation}
and 
\begin{equation}\label{eq:app_eta_u}
\eta_i^m(t, z)=D_{q_i} U^{(m)}(t,\xi^m(t, z))=D_{q_i}U^{(m)}(t, x),\ \ {\rm{and}}\ \ \eta_i^m(0, z)=D_{q_i} U^{(m)}_0(z).
\end{equation}
By Proposition \ref{theorem:hilbert-smooth}  there exists $\b(t,r)>0$ (independent of $m$) for any $q \in\B_{r}^m(0)$ we have 
\begin{equation}\label{eq:app_xi_bounded}
S_s^{t, m}[q]\equiv \xi^m(s, z)\in\B_{\b(t,r)}^m,\ \  {\rm{for\ all\ }}s\in [0,t]. 
\end{equation}
Proposition \ref{prop:semi-convex-time_Lip1} ensures $\tilde \cU$ is locally Lipschitz on $[0,\infty) \times \mathbb H$ and so, there exists $C_1(t,r)>0$ (depending on $\b(t,r)$) such that 
$\|\nabla \tilde \cU(t, \xi(t, M^z)\| \le C_1(t,r).$ Using the relation between $\nabla \tilde \cU$ and $\eta$ provided by Proposition \ref{prop:homeo1} (iv) we conclude  
\begin{equation}\label{eq:app_eta_bounded}
\sum_{i=1}^m m|\eta_i^m(t, z)|^2\le  C_1(t,r).
\end{equation}

We are now well equipped to start the proof of the assertion (1) of the theorem.

{\bf Step 1.} {\it Estimates on $(D_{z_j}\xi_i(t,\cdot),D_{z_j}\eta_i(t,\cdot))_{i,j=1}^m.$ }

{\it Claim 1.} There exists a constant $C_2(t,r)>0$ (independent of $m$) such that if $\xi(t,z)=q\in \B_r^{m}(0)$, then for all $i,j\in\{1,\dots, m\}$ we have 
\begin{align*} 
|D_{z_j}\xi_i^m(t,\cdot)|_{\infty}\le\left\{
\begin{array}{ll}
C_2(t,r), & i=j\\[5pt]
\frac{C_2(t,r)}{m}, & i\neq j
\end{array}
\right.
\end{align*}
and 
\begin{align}\label{eq:app_estim1eta}
|D_{z_j}\eta_i^m(t,\cdot)|_{\infty}\le\left\{
\begin{array}{ll}
\frac{C_2(t,r)}{m}, & i=j\\[5pt]
\frac{C_2(t,r)}{m^2}, & i\neq j
\end{array}
\right..
\end{align}

{\it Proof of Claim 1.}  By differentiating the Hamiltonian system \eqref{eq:app_Hamiltonian_sys-m} with respect to the $z_j$, we get
\begin{equation}\label{eq:app_first_linerized}
\left\{
\begin{array}{l}
\partial_tD_{z_j}\xi_i^m=D^2_{qp} H(\xi_i^m,m\eta_i)D_{z_j}\xi_i^m+mD^2_{pp}H(\xi_i^m,m\eta_i^m)D_{z_j}\eta_i^m,\\[5pt]
\partial_t D_{z_j}\eta_i^m=-\frac{1}{m}\left(D^2_{qq} H(\xi_i^m,m\eta_i)D_{z_j}\xi_i^m+mD^2_{pq}H(\xi_i^m,m\eta_i)D_{z_j}\eta_i^m\right)+\sum_{l=1}^m D^2_{q_lq_i} F^{(m)}(\xi^m)D_{z_j}\xi_l^m,\\[5pt]
D_{z_j}\xi_i^m(0,\cdot)=\left\{
\begin{array}{ll}
I_{d\times d}, & i=j,\\
0_{d\times d}, & i\neq j,
\end{array}
\right.,\ D_{z_j}\eta_i^m(0, z)=D^2_{q_jq_i} U^{(m)}_0(z).
\end{array}
\right.
\end{equation}
Let us set 
$$\ov C_2:=\max\{|\partial_q^a\partial_p^b H(q,p)|:\ (q,p)\in\R^d\times \R^d,\ |a|+|b|= 2\}.$$
If $\xi^m(t, z)=q\in\B_r^m(0),$ then
in the same way, there exists $\tilde C_2(t,r)>0$ (depending on $\b(t,r)$) such that $D^2_{q_lq_i} F^{(m)}(\xi_1,\dots,\xi_m)$ and $D_{q_jq_i}^2 U_0^{(m)}(z)$ satisfy the estimate \eqref{reg:estim_C11} with $\tilde C_2(t,r).$ Set
$$\hat C_2=\hat C_2(t,r):=\max\{\ov C_2,\tilde C_2(t,r)\}.$$
We plan to use the bounds 
\begin{align*}
&|D^2_{qp} H(\xi_i^m,m\eta_i)|_\infty, \quad |D^2_{pq}H(\xi_i^m,m\eta_i)|_\infty \quad \le \ov C_2,,\\
&|(1/m)D^2_{qq} H(\xi_i^m,m\eta_i)|_\infty\le \ov C_2/m,\quad |mD^2_{pp}H(\xi_i^m,m\eta_i^m)|_\infty\le \ov C_2 m,\\
\end{align*}
and 
\begin{align*}
&|D^2_{q_lq_i} F^{(m)}(\xi^m)|_\infty\le \left\{
\begin{array}{ll}
\tilde C_2(t,r)m^{-1}, & i=l  \\[5pt]
\tilde C_2(t,r)m^{-2}, & i\neq l
\end{array}, \qquad
\right. |D^2_{q_jq_i} U^{(m)}_0(z)|_\infty\le \left\{
\begin{array}{ll}
\tilde C_2(t,r)m^{-1}, & i=j  \\[5pt]
\tilde C_2(t,r)m^{-2}, & i\neq j
\end{array}
\right..
\end{align*}
Thus, to obtain the precise bounds (in terms of $m$) on the solution to the system \eqref{eq:app_first_linerized}, it is enough to obtain bounds on the solution $(\hat X(s),\hat Y(s))=\left((\hat X_{ij}(s))_{i,j=1}^m,(\hat Y_{ij}(s))_{i,j=1}^m\right)$ to 
$$
\left\{
\begin{array}{l}
\partial_t\hat X_{ij}=\hat C_2 \hat X_{ij}+m\hat C_2 \hat Y_{ij},\\[5pt]
\partial_t \hat Y_{ij}=(\hat C_2/m)\hat X_{ij}+\hat C_2\hat Y_{ij}+\sum_{l=1,l\neq i}^m(\hat C_2/m^2)\hat X_{lj},\\[5pt]
\hat X_{ij}(0)=\left\{
\begin{array}{ll}
1, & i=j\\
0, & i\neq j
\end{array}
\right.,\ \hat Y_{ij}(0)=\left\{
\begin{array}{ll}
\hat C_2m^{-1}, & i=j\\
\hat C_2m^{-2}, & i\neq j
\end{array}
\right..
\end{array}
\right.
$$
The constant $\hat C_2>0$ can be simply factorized out from the previous system, and since this is independent of $m$, when studying the solution, without loss of generality it is enough to study the modified system with coefficients $1$, instead of $\hat C_2$. Thus, when writing the system in a closed form, one can clearly identify the blocks $B_1,\dots,B_4$ defined in \eqref{def:special_block} and the system appearing in Lemma \ref{lem:app_matrix_exp-2}. Therefore, by the precise estimates on $(X_{ij},Y_{ij})_{i,j=1}^m$ in of Lemma \ref{lem:app_matrix_exp-2}, we conclude that there exists $C>0$ (independent of $m$) such that {\it Claim 1} follows by setting 
$$C_2(t,r):=e^{tC\hat C(t,r)}.$$

Now, let us denote by $\zeta^m=(\zeta_1^m(t,\cdot),\dots,\zeta_m^m(t,\cdot)):=S^{t,m}_0[q]$ the inverse of $\xi^m(t,\cdot)$, in particular, we have that if $\xi_i^m(t,z)=q_i$, then $\zeta_i^m(t,q)=z_i$. Next, we derive estimates for $D_{q_j}\zeta_i^m(t,\cdot)$.

\medskip

{\bf Step 2.} {\it Estimates on $(D_{q_j}\zeta_i^m)_{i,j=1}^m$.}

\medskip

{\it Claim 2.} There exists $C_3(t,r)>0$ (independent of $m$) such that for all $i,j\in\{1,\dots, m\}$ we have
\[
|D_{q_j}\zeta_i^m(t,\cdot)|_{\infty}\le\left\{
\begin{array}{ll}
C_3(t,r), & i=j,\\[5pt]
\frac{C_3(t,r)}{m}, & i\neq j,
\end{array}
\right. \ \ {\rm{in}}\ \B_r^m(0).
\]

Since $\xi^m(t,\cdot): \M \rightarrow \M$ is a diffeomorphism,  we have
\begin{equation}\label{eq:app_inverse_der}
D_{q}\zeta^m(t,q)=\left(D_{z}\xi^m(t,\cdot)\right)^{-1}\circ\zeta^m(t,q)
\end{equation}
Since we have a uniform lower bound on $\det(D_{z}\xi(t,\cdot))$ in $\M^m$, we can simply study the asymptotic behavior of $D_{q}\zeta^m(t,q)$ with respect to $m$ via the asymptotic behavior of $(D_{z}\xi^m(t,\cdot))^{-1}.$ By the previous uniform local estimates on $D_{z}\xi^m(t,\cdot)$ (from {\it Claim 1}), we have that there exists a constant $C(t,r)>0$ depending on $C_2(t,r)$ such that  
\begin{equation}\label{eq:app_inverse-asym}
D_{z}\xi^m(t,\cdot)\sim C(t,r)\left[
\begin{array}{lllll}
A_d & \frac{1}{m}A_d & \frac{1}{m}A_d &\dots & \frac{1}{m}A_d\\[5pt]
\frac{1}{m}A_d & A_d & \frac{1}{m}A_d & \dots & \frac{1}{m}A_d\\[5pt]
\dots & \dots & \dots & \ddots & \dots\\[5pt]
\frac{1}{m}A_d & \frac{1}{m}A_d & \frac{1}{m}A_d & \dots & A_d
\end{array}
\right], 
\end{equation}
for some invertible $(d\times d)$-blocks $A_d.$ Therefore, 
$$
(D_{z}\xi(t,\cdot))^{-1}\sim\frac{1}{C(t,r)}\left[
\begin{array}{lllll}
\frac{m}{m-\frac12}A_d^{-1} & \frac{-m}{(2m-1)(m-1)}A_d^{-1} & \frac{-m}{(2m-1)(m-1)}A_d^{-1} &\dots & \frac{-m}{(2m-1)(m-1)}A_d^{-1}\\[5pt]
\frac{-m}{(2m-1)(m-1)}A_d^{-1} & \frac{m}{m-\frac12}A_d^{-1} & \frac{-m}{(2m-1)(m-1)}A_d^{-1} & \dots & \frac{-m}{(2m-1)(m-1)}A_d^{-1}\\[5pt]
\dots & \dots & \dots & \ddots & \dots\\[5pt]
\frac{-m}{(2m-1)(m-1)}A_d^{-1} & \frac{-m}{(2m-1)(m-1)}A_d^{-1} & \frac{-m}{(2m-1)(m-1)}A_d^{-1} & \dots & \frac{m}{m-\frac12}A_d^{-1}
\end{array}
\right], 
$$
and so {\it Claim 2} follows by setting $C_3(t,r):=C(t,r)^{-1}.$

Going forward to conclude the proof of the assertion (1) of the theorem, we recall that by \eqref{eq:app_eta_u}, 
\begin{align*}
\eta_i^m(t,\zeta^m(t, q))=D_{q_i}U^{(m)}(t,q).
\end{align*}
Differentiating this expression with respect to $q_j$ yields
\begin{align*}
D_{q_jq_i}U^{(m)}(t,q)&=\sum_{l=1}^m D_{q_l}\Big(\eta_i(t,\zeta^m(t,q))\Big)D_{q_j}\zeta_l^m(t,q)\\
&=D_{q_j}\eta_i^m(t,\zeta^m(t,q))D_{q_j}\zeta_j^m(t,q)+D_{q_i}\eta_i^m(t,\zeta^m(t,q))D_{q_j}\zeta_i^m(t,q)\\
&+\sum_{l\neq i,l\neq j} D_{q_l}\eta_i^m(t,\zeta(t,q))D_{q_j}\zeta_l^m(t,q).
\end{align*}
The previous estimates  established in {\it Claim 1} and {\it Claim 2}, yields  assertion (1).

\medskip

{\bf Step 3.} {\it Estimates on $(D_{z_kz_j}\xi_i^m(t,\cdot),D_{z_kz_j}\eta_i^m(t,\cdot))_{i,j,k=1}^m.$}

{\it Claim 3.} There exists a constant $C_4(t,r)>0$ depending on all the previous ones, but independent of $m$ such that if $\xi(t, z)=q\in\B_r^m(0)$, then for all $i,j,k\in\{1,\dots,m\}$ we have 
$$|D^2_{z_kz_j}\xi_i^m(t,\cdot)|_{\infty}\le 
\left\{
\begin{array}{ll}
C_4(t,r), & i=j=k,\\[5pt]
\ds\frac{C_4(t,r)}{m}, & i=j\neq k,\ i\neq j=k,\ i=k\neq j,\\[5pt]
\ds\frac{C_4(t,r)}{m^2}, & i\neq j\neq k,
\end{array}
\right. 
$$ 
and 
$$|D^2_{z_kz_j}\eta_i^m(t,\cdot)|_{\infty}\le 
\left\{
\begin{array}{ll}
\ds\frac{C_4(t,r)}{m}, & i=j=k,\\[5pt]
\ds\frac{C_4(t,r)}{m^2}, & i=j\neq k,\ i\neq j=k,\ i=k\neq j,\\[5pt]
\ds\frac{C_4(t,r)}{m^3}, & i\neq j\neq k.
\end{array}
\right.$$

{\it Proof of Claim 3.}  Differentiating the system \eqref{eq:app_first_linerized} with respect to $z_k$, we obtain for the first equation
\begin{align}\label{eq:matrix1}
\partial_tD^2_{z_kz_j}\xi_i^m&=D_{z_k}\xi_i^mD^3_{qqp} H(\xi_i^m,m\eta_i)D_{z_j}\xi_i^m+mD_{z_k}\eta_i^m D^3_{pqp}H(\xi_i^m,m\eta_i)D_{z_j}\xi_i^m\nonumber\\ 
&+D^2_{qp} H(\xi_i^m,m\eta_i^m)D^2_{z_kz_j}\xi_i^m+mD_{z_k}\xi_iD^3_{qpp}H(\xi_i^m,m\eta_i)D_{z_j}\eta_i^m\nonumber\\ 
&+m^2 D_{z_k}\eta_i^m D^2_{ppp}H(\xi_i^m,m\eta_i)D_{z_j}\eta_i^m+mD^2_{pp}H(\xi_i^m,m\eta_i)D^2_{z_kz_j}\eta_i^m
\end{align}
together with the initial condition $D^2_{z_kz_j}\xi_i^m(0,\cdot)=0_{d\times d\times d}.$ From the differentiation  of the second equation with respect to $z_k$, we obtain
\begin{align}\label{eq:matrix2}
\partial_t D^2_{z_kz_j}\eta_i^m&=-\frac{1}{m}\left(D_{z_k}\xi_iD^3_{qqq} H(\xi_i^m,m\eta_i^m)D_{z_j}\xi_i^m+ mD_{z_k}\eta_iD^3_{pqq} H(\xi_i^m,m\eta_i^m)D_{z_j}\xi_i^m \right)\nonumber\\
&-\frac{1}{m}\left(D^2_{qq} H(\xi_i^m,m\eta_i^m)D^2_{z_kz_j}\xi_i^m D_{z_k}\xi_i^m +D^3_{qpq}H(\xi_i^m,m\eta_i^m)D_{z_j}\eta_i^m\right)\nonumber\\
&-\frac{1}{m}\left(m^2 D_{z_k}\eta_iD^3_{ppx}H(\xi_i^m ,m\eta_i^m)D_{z_j}\eta_i^m+mD^2_{pq}H(\xi_i,m\eta_i^m)D^2_{z_kz_j}\eta_i^m\right)\nonumber\\
&+\sum_{l_1,l_2=1}^m D_{z_k}\xi_{l_1}^mD^3_{q_{l_1}q_{l_2}q_i} F^{(m)}(\xi^m)D_{z_j}\xi_{l_2}^m+\sum_{l=1}^m D^2_{q_lq_i} F^{(m)}(\xi^m)D^2_{z_kz_j}\xi_l^m
\end{align}
with the initial condition 
\begin{align}\label{eq:matrix3}
D^2_{z_kz_j}\eta_i^m(0,z)&=D^3_{q_{k}q_{j}q_i} U^{(m)}_0(z)
\end{align}

Let us fix $k,j$. The asymptotic behavior of $(D_{z_kz_j}\xi_i^m(t,\cdot),D_{z_kz_j}\eta_i^m(t,\cdot))$, as the solution to the system \eqref{eq:matrix1}-\eqref{eq:matrix2}, can be studied in the same way as the one of \eqref{eq:app_first_linerized} in {\bf Step 1}. For this, one needs to identify the precise bounds on the coefficient matrices in \eqref{eq:matrix1}-\eqref{eq:matrix2}.
Let us set 
$$\ov C_4:=\max\{|\partial_q^\a\partial_p^\beta H(q,p)|:\ (q,p)\in\R^d\times \R^d,\ 2\le|\a|+|\b|\le 3\},$$
then we notice that by the assumptions on $H$, we have that if $\xi^m(t,z)=q\in\B_r^m(0),$ then
$$|\partial_q^\a\partial_p^\beta H(\xi_i^m(t, z),m\eta_i^m(t, z))|\le \ov C_4.$$
In the same way, there exists $\tilde C_4(t,r)>0$ (depending on $\b(t,r)$) such that $D^2_{q_kq_jq_i} F^{(m)}(\xi^m)$ and $D_{q_kq_jq_i}^2 U_0^{(m)}(q)$ satisfy the estimate \eqref{reg:estim_C21} with $\tilde C_4(t,r).$ Set
$$\hat C_4(t,r):=\max\{\ov C_4,\tilde C_4(t,t)\}\max\{C_2(t,r),1\}^2.$$
Now, system \eqref{eq:matrix1}-\eqref{eq:matrix2} has the same structure as \eqref{eq:approx-sys2}, where $(D^2_{q_kq_j}\xi_i^m,D^2_{q_k q_j}\eta_i^m)$ plays the role of $(X_i,Y_i)$. The blocks $B_1,\dots,B_4$ the coefficient blocks appearing in \eqref{eq:approx-sys2} can be identified in the same way as in {\bf Step 1.} It remains to study the bounds on the corresponding $A_1,A_2$ and $Y_0$ appearing in this system, where  
\begin{align*}
(A_1)_i&:=D_{z_k}\xi_iD^3_{qqp} H(\xi_i^m,m\eta_i)D_{z_j}\xi_i^m+mD_{z_k}\eta_i^m D^3_{pqp}H(\xi_i^m,m\eta_i^m)D_{z_j}\xi_i^m\\
\nonumber&+mD_{z_k}\xi_i^m D^3_{qpp}H(\xi_i^m,m\eta_i^m)D_{z_j}\eta_i^m+m^2 D_{z_k}\eta_i^mD^2_{ppp}H(\xi_i^m,m\eta_i^m)D_{z_j}\eta_i^m,
\end{align*}
\begin{align*}
(A_2)_i&:= -\frac{1}{m}\left(D_{z_k}\xi_iD^3_{qqq} H(\xi_i^m,m\eta_i^m)D_{z_j}\xi_i^m+ mD_{z_k}\eta_i^m D^3_{pqq} H(\xi_i^m,m\eta_i)D_{z_j}\xi_i^m \right)\\
\nonumber&-\frac{1}{m}\left(D_{z_k}\xi_i^m mD^3_{qpq}H(\xi_i^m,m\eta_i^m)D_{z_j}\eta_i^m+m^2 D_{z_k}\eta_i^mD^3_{ppq}H(\xi_i^m,m\eta_i^m)D_{z_j}\eta_i^m\right)\\
\nonumber&+\sum_{l_1,l_2=1}^m D_{z_k}\xi_{l_1}^mD^3_{q_{l_1}q_{l_2}q_i} F^{(m)}(\xi^m)D_{z_j}\xi_{l_2}^m
\end{align*}
and we set 
\begin{align*}
(Y_0)_i:=D^3_{q_{k}q_{j}q_i} U^{(m)}_0
\end{align*}

Using the obtained bounds on $(D_{z_j}\xi_i, D_{z_j}\eta_i)$ in {\bf Step 1} and the  assumptions on $U^{(m)}_0$ in \eqref{reg:estim_C21}, one checks the following asymptotic properties with respect to $m$.

\medskip
{\it Sub-claim 3.}
\begin{itemize}
\item[(1)] If $k=j=i$, then $(A_1)_i = O(\hat C_4(t,r))$, $(A_2)_i=O(\frac{\hat C_4(t,r)}{m})$ and $(Y_0)_i=O(\frac{\hat C_4(t,r)}{m}).$
\item[(2)] If $k=j\neq i$ then $(A_1)_i= O(\frac{\hat C_4(t,r)}{m^2})$, $(A_2)_i = O(\frac{\hat C_4(t,r)}{m^2})$ and $(Y_0)_i=O(\frac{\hat C_4(t,r)}{m^2}).$
\item[(3)] If $k=i\neq j$ or $i=j\neq k$, $(A_1)_i=O( \frac{\hat C_4(t,r)}{m})$, $(A_2)_i = O(\frac{\hat C_4(t,r)}{m^2})$ and $(Y_0)_i=O(\frac{\hat C_4(t,r)}{m^2}).$
\item[(4)] If $k\neq j\neq i$, then $(A_1)_i=O(\frac{\hat C_4(t,r)}{m^2})$, $(A_2)_i=O(\frac{\hat C_4(t,r)}{m^3})$ and $(Y_0)_i=O(\frac{\hat C_4(t,r)}{m^3}),$
\end{itemize}
Now, one considers two cases when studying the desired properties. Let us recall that $k,j$ are fixed. 

{\it Case 1.} If $k=j$, (1)-(2) of {\it Sub-claim 3} can be combined with Lemma \ref{lem:app_matrix-exponential}(1) to conclude the proof of the Claim.

{\it Case 2.} If $k\neq j$, (3)-(4) of {\it Sub-claim 3} can be combined with Lemma \ref{lem:app_matrix-exponential}(2)  to conclude the proof of the Claim.

Therefore there exists a constant $C>0$ such that {\it Claim 3} holds for $C_4(t,r):=e^{tC\hat C_4(t,r)}.$

\medskip

{\bf Step 4.} {\it Estimates on $(D_{q_kq_j}\zeta_i(t,\cdot))_{i,j,k=1}^m.$}

{\it Claim 4.} There exists a constant $C_5(t,r)>0$ depending on all the previous ones, but independent of $m$ such that for all $i,j,k\in\{1,\dots,m\}$, we have 
$$|D^2_{x_kx_j}\zeta_i(t,\cdot)|_{\infty}\le 
\left\{
\begin{array}{ll}
C_5(t,r), & i=j=k,\\[5pt]
\ds\frac{C_5(t,r)}{m}, & i=j\neq k,\ i\neq j=k,\ i=k\neq j,\\[5pt]
\ds\frac{C_5(t,r)}{m^2}, & i\neq j\neq k,
\end{array}
\right. \ \ {\rm{in}}\ \B_r^m.
$$ 

{\it Proof of Claim 4.} It is enough to differentiate the expression \eqref{eq:app_inverse_der} and use all the previous estimates on $(D^2_{z_kz_j}\xi_i)_{i,j,k=1}^m$ and on $(D_{q_j}\zeta_i)_{i,j=1}^m$ from {\bf Step 3} and {\bf Step 2}, respectively.

We have
\begin{align*}
D^2_{q q}\zeta(t,q)=-\left\{\left[\left(D_{z}\xi(t,\cdot)\right)^{-1}D^2_{zz}\xi(t,\cdot)D_q\zeta(t,q)\left(D_{z}\xi(t,\cdot)\right)^{-1}\right]\circ\zeta(t,q)\right\}.
\end{align*}
The previous writing is used for the following short hand notation: for $k\in\{1,\dots,m\}$, we have
\begin{align*}
D_{q_k}D_{q}\zeta(t,q)=-\left\{\left[\left(D_{z}\xi(t,\cdot)\right)^{-1}\left(\sum_{l=1}^mD_{z_l}D_{z}\xi(t,\cdot)D_{q_k}\zeta_l(t,q)\right)\left(D_{z}\xi(t,\cdot)\right)^{-1}\right]\circ\zeta(t,q)\right\},
\end{align*}
and in particular for $i,j\in\{1,\dots,m\}$, we have 
$$
\left(\sum_{l=1}^mD_{z_l}D_{z}\xi(t,\cdot)D_{q_k}\zeta_l(t,q)\right)_{ij}=\sum_{l=1}^mD^2_{z_lz_j}\xi_i(t,\cdot)D_{q_k}\zeta_l(t,q)=: A_{ij}.
$$
For $k\in\{1,\dots, m\}$ fixed, by the definition of $A_{ij}$ and by {\bf Steps 2-3}, this last matrix can be bounded as follows: by setting $\tilde C_5(t,r):=C_4(t,r)C_3(t,r),$ we have
\begin{align*}
|A_{ij}|_\infty\le 
\left\{
\begin{array}{ll}
\tilde C_5(t,r), & i=j=k,\\[5pt]
\ds\frac{\tilde C_5(t,r)}{m}, & i=j\neq k,\ i\neq j=k,\ i=k\neq j,\\[5pt]
\ds\frac{\tilde C_5(t,r)}{m^2}, & i\neq j\neq k.
\end{array}
\right.
\end{align*} 
Now, using the bounds on $\left(D_{\bz}\xi(t,\cdot)\right)^{-1}$ from \eqref{eq:app_inverse-asym}, by setting $C_5(t,r):=\tilde C_5(t,r)C(t,r)^2$, we conclude the statement of {\it Claim 4.}

\medskip
{\bf Final Step.} Let us recall that from \eqref{eq:app_eta_u} that we have
\begin{align*}\label{eq:app_eta_u}
\eta_i(t,\zeta(t,q))=D_{q_i}U^{(m)}(t,q).
\end{align*}
Differentiating this expression with respect to $q_j$ and $q_k$ we obtain 
\begin{align*}
D^3_{q_kq_jq_i}U^{(m)}(t,\cdot)=\sum_{l_1,l_2=1}^mD_{q_k}\zeta_{l_2}(t,\cdot)D^2_{z_{l_2}z_{l_1}}\eta_i(t,\zeta(t,\cdot))D_{q_j}\zeta_{l_1}(t,\cdot)+\sum_{l=1}^m D_{z_l}\eta_i(t,\zeta(t,\cdot))D^2_{q_k q_j}\zeta_l(t,\cdot)
\end{align*}
from where by using the estimates from {\bf Steps 1-4}, we obtain
\begin{align*}
\Big{|}D^3_{q_kq_jq_i}U^{(m)}(t,\cdot)\Big{|}_\infty &\le\frac{1}{m}\left(|D_{q_k}\zeta_{i}|_{\infty}|D_{q_j}\zeta_{i}|_{\infty}+|D^2_{q_kq_j}\zeta_i|_\infty\right)\\
&+\frac{1}{m^2}\left(\sum_{l=1,l\neq i}^m|D_{q_k}\zeta_{l}|_{\infty}|D_{q_j}\zeta_{i}|_{\infty} +\sum_{l=1,l\neq i}^m |D_{q_k}\zeta_{i}|_{\infty}|D_{q_j}\zeta_{i}|_{\infty}+\sum_{l=1,l\neq i}^m|D^2_{q_k q_j}\zeta_l|_{\infty}\right)\\
&+\frac{1}{m^3}\sum_{\begin{subarray}{l}
l_1,l_2=1\\
l_1\neq l_2\neq i
\end{subarray}
}^m|D_{q_k}\zeta_{l_1}|_{\infty}|D_{q_j}\zeta_{l_2}|_{\infty}
\end{align*}
Using again the estimates from the previous steps, we obtain (1) and (2) of the theorem.

\smallskip

The statement in (3) can be easily shown by differentiating the Hamilton-Jacobi equation satisfied by $U^{(m)}$ with respect the variable $q_j$ and by using the estimates on $U^{(m)}$ provided in (1)-(2). Indeed, we have
\begin{align*}
|D_{q_j}\partial_t U^{(m)}|&\le \frac{1}{m}|D_qH(q_j,mD_{q_j}U^{(m)})|+\frac{1}{m}|D_pH(q_j,mD_{q_j}U^{(m)})|m|D^2_{q_jq_j}U^{(m)}|\\
&+\sum_{i\neq j}\frac{1}{m}|D_pH(q_i,m D_{q_i}U^{(m)})|m|D^2_{q_jq_i}U^{(m)}|+|D_{q_j}F^{(m)}|\\
&\le\frac{1}{m}|D_qH(q_j,mD_{q_j}U^{(m)})|+\frac{1}{m}|D_pH(q_j,mD_{q_j}U^{(m)})| +\frac{C}{m} + |D_{q_j}F^{(m)}|. 
\end{align*}
Thus 
\begin{align*}
& \sum_{j=1}^m m |D_{q_j}\partial_t U^{(m)}|^2\\
&\le\sum_{j=1}^m \frac{1}{m}|D_qH(q_j,mD_{q_j}U^{(m)})|^2+\sum_{j=1}^m\frac{1}{m}|D_pH(q_j,mD_{q_j}U^{(m)})|^2+C+\sum_{j=1}^m m|D_{q_j}F^{(m)}|^2\le C,
\end{align*}
where we used the assumption on $F^{(m)}$, \eqref{hyp:D_qH} and the fact that since $\cU\in C^{1,1}_{\rm{loc}}([0,T]\times\cP_2(\M))$ and $D_pH$ is Lipschitz, we have $\sum_{j=1}^m\frac{1}{m}|D_pH(q_j,mD_{q_j}U^{(m)})|^2\le C$. The claim follows, which concludes the proof of the theorem.
\end{proof}

\begin{lemma}\label{lem:app_matrix-exponential}
Let $[X\ Y]^\top=[X_1\ \dots\ X_m\ Y_1\ \dots \ Y_m]^\top\in\R^{2m}$ be the solution of the ODE system 
\begin{equation}\label{eq:approx-sys2}
\partial_t\left[
\begin{array}{l}
X\\
Y
\end{array}
\right]=\left[
\begin{array}{l}
A_1\\
A_2
\end{array}
\right]+
\left[
\begin{array}{ll}
B_1 & B_2\\
B_3 & B_4
\end{array}
\right]\left[
\begin{array}{l}
X\\
Y
\end{array}
\right],\ \
\left[
\begin{array}{l}
X(0)\\
Y(0)
\end{array}
\right] = \left[
\begin{array}{l}
0_m\\
Y_0
\end{array}
\right], 
\end{equation}
where $A_1,A_2, Y_0\in\R^m$, $0_m\in\R^m$ is the zero vector and the $(m\times m)$-dimensional blocks $B_i$ are such that 
\begin{align}\label{def:special_block}
B_1=B_4=I_{m};\ B_2=m I_{m}\ \ {\rm{and}}\ \ B_3=\left[
\begin{array}{llll}
\frac{1}{m} & \frac{1}{m^2} & \dots & \frac{1}{m^2}\\
\frac{1}{m^2} & \frac{1}{m} & \dots & \frac{1}{m^2}\\
\dots & \dots & \ddots & \dots\\
\frac{1}{m^2} & \dots & \frac{1}{m^2} & \frac{1}{m}
\end{array}
\right].
\end{align}

Then there exists a constant $C>0$ (independent of $m$), such that 
\begin{itemize}
\item[(1)] If for $i_0\in\{1,\dots, m\}$ fixed
$$(A_1)_{i_0}=1,\ \ (A_1)_{i}=\frac{1}{m},\ \forall\ i\neq i_0$$ 
and
$$(A_2)_{i_0}=(Y_0)_{i_0}=\frac{1}{m},\ \ (A_2)_{i}=(Y_0)_i=\frac{1}{m^2},\ \forall\ i\neq i_0,$$
then  
\begin{align*}
|X_i(t)|\le
\left\{
\begin{array}{ll}
e^{tC}, & i=i_0,\\[5pt]
\frac{e^{tC}}{m}, & i\in\{1,\dots,m\},\ i\neq i_0,\\[5pt]
\end{array}
\right. \ \ {\rm{and}}\ \ 
|Y_i(t)|\le
\left\{
\begin{array}{ll}
\frac{e^{tC}}{m}, & i=i_0,\\[5pt]
\frac{e^{tC}}{m^2}, & i\in\{1,\dots, m\},\ i\neq i_0.
\end{array}
\right.
\end{align*}
\item[(2)]If for some $k,j\in\{1,\dots,m\}$ fixed, $k\neq j$, we have 
$$(A_1)_{j}=(A_1)_k=\frac{1}{m},\ \ (A_1)_{i}=\frac{1}{m^2},\ \forall\ i\neq j,\ i\neq k$$ 
and
$$(A_2)_{j}=(A_2)_k=(Y_0)_{j}=(Y_0)_k=\frac{1}{m^2},\ \ (A_2)_{i}=(Y_0)_i=\frac{1}{m^3},\ \forall\ i\neq j,\ i\neq k$$
then  
\begin{align*}
|X_i(t)|\le
\left\{
\begin{array}{ll}
\frac{e^{tC}}{m}, & i=j,\ i=k,\\[5pt]
\frac{e^{tC}}{m^2}, & i\in\{1,\dots,m\},\ i\neq j,\ i\neq k\\[5pt]
\end{array}
\right. \ \ {\rm{and}}\ \ 
|Y_i(t)|\le
\left\{
\begin{array}{ll}
\frac{e^{tC}}{m^2}, & i=j,\ i=k,\\[5pt]
\frac{e^{tC}}{m^3}, & i\in\{1,\dots, m\},\ i\neq j,\ i\neq k.
\end{array}
\right.
\end{align*}
\end{itemize}
\end{lemma}

\begin{proof}
We analyse the representation formula for \eqref{eq:approx-sys2} in the different cases. Since we are only interested in the asymptotic properties of the solution with respect to $m$, first let us study the asymptotic behavior of the exponential and the inverse of the coefficient matrix.

Let $B:=\left[
\begin{array}{ll}
B_1 & B_2\\
B_3 & B_4
\end{array}
\right]$ and for $n\in\N$, let us denote the powers of $B$ as
$B^n:=\left[
\begin{array}{ll}
B_{1,n} & B_{2,n}\\
B_{3,n} & B_{4,n}
\end{array}
\right]$. 

{\it Claim.} We have the following properties for the blocks $B_{i,n}$ for all $n\in\N$ and for $i,j\in\{1,\dots,m\}$
\begin{itemize}
\item[(1)] $\ds(B_{1,n})_{ii}= O(1)$, $(B_{1,n})_{ij}=O(\frac1m),$ if $i\neq j$.
\item[(2)] $(B_{2,n})_{ii}=O(m)$, $(B_{2,n})_{ij}=O(1),$ if $i\neq j$.
\item[(3)] $(B_{3,n})_{ii}=O( \frac1m)$, $(B_{3,n})_{ij}=O(\frac{1}{m^2}),$ if $i\neq j$.
\item[(4)] $(B_{4,n})_{ii}=O(1)$, $(B_{4,n})_{ij}=O(\frac1m),$ if $i\neq j$.
\end{itemize}
{\it Proof of the Claim.} This follows from a mathematical induction argument in $n$.

Since we have a characterization of the asymptotic properties in terms of $m$ of the elements of the powers $n\in\N$ of the block matrix (which are uniform in $n$), the property from the {\it Claim} will also hold true for the blocks of the matrix exponential of $B$. Setting $A:=[A_1^\top\ A_2^\top]^\top$, the representation formula for the solutions of \eqref{eq:approx-sys2} reads as 
\[
\left[
\begin{array}{l}
X(t)\\
Y(t)
\end{array}
\right]=\exp(tB)\left([0_m^\top\ Y_0^\top]^\top+ B^{-1} A\right) - B^{-1}A.
\]

It remains to compute $B^{-1}$ (which exists, since $B$ is nonsingular), for which we have the formula (using the blocks from \eqref{def:special_block})
\begin{align*}
B^{-1}&=\left[
\begin{array}{ll}
(I_m-m B_3)^{-1} & -m(I_m-mB_3)^{-1}\\
-B_3(I_m-m B_3)^{-1} & I_m+ mB_3(I_m-m B_3)^{-1}
\end{array}
\right]=\left[
\begin{array}{ll}
M & -mM\\
-B_3M & I_m+ mB_3M
\end{array}
\right],
\end{align*}
where, we have used the notation
\begin{align*}
M:=(I_m-m B_3)^{-1}=m \left[
\begin{array}{llll}
0 & -1 & \dots & -1\\
-1 & 0 & \dots & -1\\
\dots & \dots & \ddots & \dots\\
-1 & \dots & -1 & 0
\end{array}
\right]^{-1} = \left[
\begin{array}{llll}
m\frac{m-2}{m-1} & \frac{-m}{m-1} & \dots & \frac{-m}{m-1}\\[5pt]
\frac{-m}{m-1} & m\frac{m-2}{m-1} & \dots & \frac{-m}{m-1}\\[5pt]
\dots & \dots & \ddots & \dots\\[5pt]
\frac{-m}{m-1} & \dots & \frac{-m}{m-1} & m\frac{m-2}{m-1}
\end{array}
\right]
\end{align*}

Now, in the case of (1), we have that $(B^{-1}A)_{i}=0$, if $i\in\{1,\dots,m\}$, and $(B^{-1}A)_{m+i_0}=\frac{1}{m}$ and $(B^{-1}A)_i=\frac{1}{m^2}$, if $i\in\{m+1,\dots, 2m$, $i\neq m+i_0$.


Furthermore, there exists a constant $C>0$ (independent of $m$) such that
\begin{align*}
\left(\exp(tB)[0_m^\top\ Y_0^\top]^\top\right)_{i}\sim
\left\{
\begin{array}{ll}
e^{tC}, & i=i_0,\\[5pt]
\frac{e^{tC}}{m}, & i\in\{1,\dots,m\},\ i\neq i_0,\\[5pt]
\frac{e^{tC}}{m}, & i=m+i_0,\\[5pt]
\frac{e^{tC}}{m^2}, & i\in\{m+1,\dots, 2m\},\ i\neq m+i_0.
\end{array}
\right.
\end{align*}
(1) from the thesis of the lemma follows.

\medskip

In the case on (2), we compute similarly $(B^{-1}A)_i=0$, if $i\in\{1,\dots, m\}$, $(B^{-1}A)_i=\frac{1}{m^2}$ if $i=m+j$ or $j=m+k$ and $(B^{-1}A)_i=\frac{1}{m^3}$ otherwise.

Furthermore, there exists a constant $C>0$ (independent of $m$) such that
\begin{align*}
\left(\exp(tB)[0_m^\top\ Y_0^\top]^\top\right)_{i}\sim
\left\{
\begin{array}{ll}
\frac{e^{tC}}{m}, & i=j,\ i=k,\\[5pt]
\frac{e^{tC}}{m^2}, & i\in\{1,\dots,m\},\ i\neq j,\ i\neq k\\[5pt]
\frac{e^{tC}}{m^2}, & i=m+j,\ i=m+k,\\[5pt]
\frac{e^{tC}}{m^3}, & i\in\{m+1,\dots, 2m\},\ i\neq m+j,\ i\neq m+k.
\end{array}
\right.
\end{align*}
And finally, (2) from the thesis of the lemma follows.
\end{proof}

\begin{lemma}\label{lem:app_matrix_exp-2}
Let $X=(X_{ij})_{i,j=1}^m$ and $Y=(X_{ij})_{i,j=1}^m$ be such that $[X\ Y]^\top\in\R^{2m\times m}$ is the solution of the ODE system  
\begin{equation}\label{eq:approx-sys3}
\partial_t\left[
\begin{array}{l}
X\\
Y
\end{array}
\right]=
\left[
\begin{array}{ll}
B_1 & B_2\\
B_3 & B_4
\end{array}
\right]\left[
\begin{array}{l}
X\\
Y
\end{array}
\right],\ \
\left[
\begin{array}{l}
X(0)\\
Y(0)
\end{array}
\right] = \left[
\begin{array}{l}
I_m\\
Y_0
\end{array}
\right], 
\end{equation}
where $Y_0\in\R^{m\times m}$, is set to $Y_0:=B_3$ and the $(m\times m)$-dimensional blocks $B_i$ are defined in \eqref{def:special_block}. Then, there exists $C>0$ (independent of $m$) such that
\begin{align*}
|X_{ij}(t)|\le
\left\{
\begin{array}{ll}
e^{tC}, & i=j,\\[5pt]
\frac{e^{tC}}{m}, & i\neq j,\\[5pt]
\end{array}
\right. \ \ {\rm{and}}\ \ 
|Y_{ij}(t)|\le
\left\{
\begin{array}{ll}
\frac{e^{tC}}{m}, & i=j,\\[5pt]
\frac{e^{tC}}{m^2}, &  i\neq j.
\end{array}
\right.
\end{align*}
\end{lemma}

\begin{proof}
This result is a consequence of the asymptotic behavior of the matrix exponential $\exp(tB)$, where $B:=\left[
\begin{array}{ll}
B_1 & B_2\\
B_3 & B_4
\end{array}
\right].$  Using the asymptotic result from the {\it Claim} in Lemma \ref{lem:app_matrix-exponential} and from the representation formula
\begin{equation}\label{eq:ODE_representation}
\left[
\begin{array}{l}
X(t)\\
Y(t)
\end{array}
\right]=\exp(tB)[I_m\ Y_0]^\top,
\end{equation}
the result follows.
\end{proof}

%
%
%
%
%
%
%
%
\section{Comparing regularity properties of functions defined on $\sP_2(\M)$, $\bH$ and $\M^m$}\label{sec:reg_P2_H_Pm}
Throughout this section, we lift any given function $\cU:\sP_2(\M)\to\R$ to $\bH$ to obtain the function $\tilde\cU:\bH\to\R$ defined as $\tilde\cU(x):=\cU(\sharp(x)).$ Recall $(\Om_j)_{j=1}^m$ is the Borel partition in Section \ref{sec:preliminaries}. We set  
$$U^{(m)}(q):=\cU(\mu^{(m)}_q)=\tilde\cU(M^q).$$
\vskip0.40cm
 \subsection{Semi-convex and semi-concave functions on Hilbert spaces}

\begin{definition}[Semi-convexity and semi-concavity on $\bH$]\label{def:semi_conv_hilbert}
Let $\B\subseteq\bH$ be a convex open set. We say that $\tilde \cU:\B\to\R$ is semi-convex (or $\l$-convex) on $\B$, if there exists $\l\in\R$ and for all $x\in \B$ there exists a continuous linear form $\theta_x$ on $\bH$ such that 
$$\tilde \cU(y)\ge \tilde \cU(x)+\theta_x(y-x) + \frac{\l}{2}\|x-y\|^2,\ \ \forall\ y\in \B.$$
We say that a function $\tilde \cU:\B\to \R$ is $\l$-concave, if $-\tilde\cU$ is $(-\l)$-semi-convex.
\end{definition}

\begin{remark}
The previous definition has an equivalent reformulation. Let $\B\subseteq\bH$ be a convex open set. Then $\tilde \cU:\B\to\R$ is $\l$-convex if and only if 
$$\tilde \cU((1-t)x+ty)\le (1-t)\tilde \cU(x)+t\tilde \cU(y)-\frac{\l}{2} t(1-t)\|x-y\|^2, \ \forall t\in[0,1],\ \forall x, y\in \B.$$
\end{remark} 

\begin{definition}[$C^{1,1}$ functions]
We say that $\tilde \cU: \B\to\R$ is $C^{1,1}$ on an open set $\B\subseteq\bH$, if it is Fr\'echet differentiable on $\B$ and its Fr\'echet differential is Lipschitz continuous, i.e. there exists $C>0$ such that
$$
\|\nabla \tilde \cU(x)-\nabla \tilde \cU(y)\|\le C\|x-y\|, \ \forall\ x, y\in \B.
$$
\end{definition}

Inspired from similar results on finite dimensional smooth manifold (see for instance in \cite{Fat:weakKAM}), we can state the following characterization of $C^{1,1}$ functions defined on subsets of $\bH$.

\begin{remark}\label{rem:C11onH} In fact $\tilde \cU:\B\to\R$ is $C^{1,1}$ on a convex set $\B\subseteq\bH$ if and only if it is Fr\'echet differentiable on $\B$ and there exists $K\ge 0$ such that
\begin{align*}
|\tilde \cU(y)-\tilde \cU(x)-\nabla \tilde \cU(x)(y-x)|\le K\|x-y\|^2, \ \forall\ x, y\in\B.
\end{align*}
\end{remark}

%
%
%
%
\vskip0.40cm
\subsection{Notions of convexity on $(\sP_2(\M),W_2)$} There are various notions of convexity for functionals defined on the Wasserstein space. The concept of  so-called {\it displacement convexity} \cite{AGS,McCann1997} is expressed in terms of $W_2$--geodesics. Recall that given $\mu_0,\mu_1\in\sP_2(\M)$, for any geodesics $[0,1]\ni t\mapsto \mu_t\in\sP_2(\M)$, of constant speed connecting $\mu_0$ to $\mu_1$ in $\sP_2(\M)$ is of the form $\mu_t= \mu_t:=((1-t)\pi^1+t\pi^1)_\sharp\g$ for some $\g\in\Gamma_o(\mu_0,\mu_1)$, then 

\begin{definition}[Semi-convexity and semi-concavity on $(\sP_2(\M),W_2)$] Let $\cU:\sP_2(\M)\to\R$.
\begin{enumerate}
\item[(1-i)] We say that $\cU$ is \emph{semi-convex} (or $\l$-convex) in the classical sense if there is   $\l\in\R$ such that 
$$\cU((1-t)\mu_0+t\mu_1)\le (1-t)\cU(\mu_0)+t\cU(\mu_1)-\frac{\lambda}{2} t(1-t)W_2^2(\mu_0,\mu_1),\ \ \forall\ \mu_0,\mu_1\in\sP_2(\M),\ \ \forall\ t\in[0,1].$$
\item[(1-ii)]We say that $\cU:\sP_2(\M)\to\R$ is semi-concave (or $\l$-concave) in the classical sense if $-\cU$ is $(-\l)$-convex. We refer to $0$-convex and $0$-concave functions simply as convex and concave functions, respectively.  
\item[(2-i)] We say $\cU:\sP_2(\M)\to\R$ is \emph{displacement semi-convex} (or displacement $\lambda$-convex) if there exists $\l\in\R$ such that for any $[0,1]\ni t\mapsto \mu_t\in\sP_2(\M)$ stands for any geodesic of constant speed connecting $\mu_0$ to $\mu_1$ we have 
$$\cU(\mu_t)\le (1-t)\cU(\mu_0)+t\cU(\mu_1)-\frac{\lambda}{2} t(1-t)W_2^2(\mu_0,\mu_1),\ \ \forall\ \mu_0,\mu_1\in\sP_2(\M),\ \ \forall\ t\in[0,1].$$
\item[(2-ii)] We say that $\cU:\sP_2(\M)\to\R$ is \emph{displacement semi-concave} (or displacement $\l$-concave) if $-\cU$ is displacement $(-\l)$-convex. We refer to displacement $0$-convex and  displacement $0$-concave as simply \emph{displacement convex} and \emph{displacement concave}, respectively.
\end{enumerate}
\end{definition}

The following results link $\l$-convexity on the Wasserstein, the Hilbert and the finite dimensional  space $\M^m$. This is a generalization of Proposition 5.79 from \cite{CarmonaD-I}.
\begin{lemma}\label{lem:con_conv} Let $\cU:\sP_2(\M)\to\R$ be a continuous function and let $\tilde \cU:\bH\to\R$ be defined as $\tilde\cU:=\cU\circ\sharp$ so that $\tilde \cU$ is continuous. As above consider for a natural number $m$ consider $U^{(m)}:\M^m\to\R$. Finally, fix $\l\in\R$. Then the followings are equivalent. 
\begin{itemize}
\item[(1)] $\tilde\cU$ is $\l$-convex on $\bH$;
\item[(2)] $\cU$ is displacement $\l$-convex on $(\sP_2(\M),W_2)$
\item[(3)] For any natural number $m$, we have that $U^{(m)}$ is $\frac{\l}{m}$-convex on $\M^m$.
\end{itemize}
\end{lemma}

\begin{proof}  (1)$\Rightarrow$(2). Let us suppose $\tilde\cU$ is $\l$-convex, 
let $\mu,\nu\in\sP(\M)$ and let $\gamma \in \Gamma_o(\mu, \nu).$ Then, there exist $x,y\in\bH$ such that $(x, y)_\sharp \cL^d_\Omega= \gamma. $ In particular, we have $\sharp(x)=\mu$, $\sharp(y)=\nu$ and $W_2(\mu,\nu)=\|x-y\|.$  For $[0,1]\ni t\mapsto \mu_t:=\left[(1-t)\pi^1+t\pi^2\right]_\sharp \g$  is a geodesic of constant speed connecting $\mu$ to $\nu$. Actually, any geodesic between $\mu$ and $\nu$ has this representation. By the $\l$-convexity of $\tilde\cU$ we have
\begin{align*}
\cU(\mu_t)&=\cU\left(\sharp\left[(1-t)x+ty\right]\right)=\tilde\cU((1-t)x+ty)\\
&\le (1-t)\tilde\cU(x) + t\tilde\cU(y)-\frac{\l}{2} t(1-t)\|x-y\|^2\\
&=(1-t) \cU(\mu)+t U(\nu)-\frac{\l}{2}t(1-t)W_2^2(\mu,\nu).
\end{align*}
Thus, $\cU$ is displacement $\l$-convex.


(2)$\Rightarrow$(3). Let us suppose that $\cU$ is displacement $\l$-convex and we show that $U^{(m)}$ is $\frac{\l}{m}$-convex on $\M^m$. Let us fix $(q_1,\dots, q_m)\in \M^m$. It is enough to show the $\frac{\l}{m}$-convexity of $U^{(m)}$ in a small neighborhood of this fixed point. Therefore, let $(q'_1,\dots,q'_m)\in\M^m$ be such that $\max\{|q_i-q'_i|:\ i\in\{1,\dots,m\}\}$ is small so that  $W_2^2(\mu^{(m)}_q,\mu^{(m)}_{q'})=\frac{1}{m}\sum_{i=1}^m|q_i- q'_i|^2.$ By this assumption, we also have that the constant speed geodesic connecting $\mu^{(m)}_q$ to $\mu^{(m)}_{q'}$ in a unit time is given by $[0,1]\ni t\mapsto \mu_t^{(m)}=\frac{1}{m}\sum_{i=1}^m\d_{(1-t)q_i + t q'_i}.$

By this construction, for $t\in[0,1]$ we have 
\begin{align*}
U^{(m)}((1-t)q + t { q'})=\cU(\mu_t^{(m)})
& \leq (1-t) \cU(\mu^{(m)}_q)+ t \cU(\mu^{(m)}_{q'})-{\l \over 2}t(1-t)  W_2\big( \mu^{(m)}_q, \mu^{(m)}_{q'}\big)\\ 
&=(1-t) U^{(m)}(q)+t U^{(m)}(q')-\frac{\l}{2m}t(1-t)\sum_{i=1}^m |q_i-q'_i|^2.
\end{align*}
Therefore, the $\frac{\l}{m}$-convexity of $U^{(m)}$ in a small neighborhood of $q$ follows.

\medskip
(3)$\Rightarrow$(1) We suppose $U^{(m)}$ is $\frac{\l}{m}$-convex for all natural number $m.$ We plan to show the $\l$-convexity of $\tilde\cU$ on $\bH$. Note the $\frac{\l}{m}$-convexity of $U^{(m)}$ is equivalent to the $\l$-convexity of the restriction of $\tilde \cU$ to $\{M^q\; : \; q \in \R^{md}\} \subset \bH.$ In particular, the local Lipschitz constants of these restrictions are bounded from above by a number which is independent of $m.$ These finite dimensional functions then have a unique extension $\tilde \cV$ on $\bH$, which is $\l$--convex and coincides with $\tilde \cU$ on a dense subset of $\bH.$ It suffices to know that $\tilde \cU$ is continuous to conclude that it is nothing but $\tilde \cV.$ 

\end{proof}

\vskip0.40cm
\subsection{ $C^{1,1}$ functions on $(\sP_2(\M),W_2)$ versus $C^{1,1}$ functions on $\bH$}
Given a differentiable function  $\cU:\sP_2(\M)\to \R$ (cf. \cite{AGS}), we denote as $\nabla_w \cU$ the Wasserstein gradient field of $\cU.$ This subsection exploits the connection between the differential of $\cU:\sP_2(\M)\to \R$ and the differential of its lift $\tilde \cU:\bH\to\R$ (\cite{GangboT2017}). More precisely, we have the following result. 
\begin{remark}\label{rem:factorization} 
Let $x\in\bH$ and set $\mu:=\sharp(x)$. Then $\cU$ is differentiable at $\mu$ if and only if $\tilde \cU$ is differentiable at $x$ and it this case, we have the factorization $\nabla\tilde \cU(x)=\nabla_w \cU(\mu)\circ x$. 
\end{remark}   

\begin{definition}\label{def:c11_wasserstein}
Let $\cB\subseteq \sP_2(\M)$ be open and geodesically convex. Let $\a\in(0,1]$. We say that $\cU\in C^{1,\a}(\cB)$, if it is continuously differentiable on $\cB$ and there exists a constant $C\ge 0$ such that
\begin{itemize}
\item[(1)] $\spt(\mu)\ni q_1\mapsto\nabla_w \cU(\mu)(q_1)$ is $\a$--H\"older continuous (or simply Lipschitz continuous if $\a=1$) with constant $C$ for any $\mu\in\cB$.\vspace{10pt}
\item[(2)]$\ds\left|\cU(\nu)-\cU(\mu)-\int_{\M^2}\nabla_w \cU(\mu)(q_1)\cdot(q_2-q_1)\dd\g(q_1, q_2)\right|\le CW_2^{1+\a}(\mu,\nu),\ \forall\ \mu,\nu\in \cB,\; \forall \g\in\Gamma_o(\mu,\nu).$
\end{itemize}
\end{definition}

\begin{definition}\label{def:c11_M_times_wasserstein}
Similarly to the previous definition, let $\cB\subseteq \sP_2(\M)$ be open and geodesically convex and let $K\subseteq\M$ be a convex open set. Let $\a\in(0,1]$. We say that $u\in C^{1,\a}(K\times\cB)$, if it is continuously differentiable on $K\times \cB$ and there exists a constant $C\ge 0$ such that
\begin{itemize}
\item[(1)] $\spt(\mu)\ni q_1\mapsto\nabla_w u(q,\mu)(q_1)$ is $\a$--H\"older continuous (or simply Lipschitz continuous if $\a=1$) with constant $C$ for any $(q,\mu)\in K\times\cB$.\vspace{10pt}
\item[(2)]\begin{align*}
\ds\Big|u(\ov q,\nu)-u(q,\mu)-D_q u(q,\mu)&\cdot(\ov q-q)-\int_{\M^2}\nabla_w u(q,\mu)(q_1)\cdot(q_2-q_1)\dd\g(q_1, q_2)\Big|\\
&\le C\left(|\ov q - q|^{1+\a}+ W_2^{1+\a}(\mu,\nu)\right), \ \ \forall\ \ov q, q\in K, \mu,\nu\in \cB,\; \forall \g\in\Gamma_o(\mu,\nu).
\end{align*}
\end{itemize}
\end{definition}

\begin{remark}\label{rmk:c11_wasserstein}
\begin{enumerate}
\item[(i)] 
Let us notice that Definition \ref{def:c11_wasserstein}(2) implies that $\nabla_w \cU$ is `$\a$--H\"older continuous' in the following sense. We have
$$\Bigg{|}\int_{\M^2}\nabla_w \cU(\mu)(q_1)\cdot(q_1-q_2)\dd\g(q_1, q_2)- \int_{\M^2}\nabla_w \cU(\nu)(q_2)\cdot(q_1- q_2)\dd\tilde\g(q_2, q_1)\Bigg{|}\le 2CW_2^{1+\a}(\mu,\nu),$$
for any $\mu,\nu\in \cB$ and $\g\in\Gamma_o(\mu,\nu),$ $\tilde\g\in\Gamma_o(\nu,\mu)$.
\item[(ii)] Let us underline that the inequality in Definition \ref{def:c11_wasserstein}(2) naturally encodes also the fact that $\cU$ is locally Lipschitz continuous. Indeed, that inequality, implies that
\begin{align*}
|\cU(\nu)-\cU(\mu)| & \le CW_2^{1+\a}(\mu,\nu)+\int_{\M^2}|\nabla_w \cU(\mu)(q_1)|\cdot|q_2-q_1|\dd\g(q_1, q_2)\\
&\le CW_2^{1+\a}(\mu,\nu)+\|\nabla_w\cU(\mu)\|_{L^2(\mu)}W_2(\mu,\nu)= \Big(CW_2^\a(\mu,\nu) + \|\nabla_w\cU(\mu)\|_{L^2(\mu)}\Big)W_2(\mu,\nu),
\end{align*}
so the local Lipschitz property follows.
\item[(iii)] Definition \ref{def:c11_M_times_wasserstein}(2) naturally encodes that $K\ni q\mapsto u(q,\mu)$ is of class $C^{1,\a}$, uniformly with respect to $\mu$.
\end{enumerate}
\end{remark}

\begin{lemma}\label{lem:c11_equiv}
$\cU\in C^{1,1}\big( \sP_2(\M)\big)$ if and only if $\tilde \cU\in C^{1,1}(\bH)$. 
\end{lemma}
\begin{proof} {\bf Part 1.} Suppose first that $\tilde \cU\in C^{1,1}(\bH)$ so that by Remark \ref{rem:C11onH} there exists a constant $C\ge0$ such that 
\begin{equation}\label{eq:proof_1}
|\tilde \cU(y)-\tilde \cU(x)-\nabla\tilde \cU(x)(y-x)|\le {C\over 2} \|x-y\|^2,\ \ \forall x, y\in\bH.
\end{equation}
This implies in particular that $\cU\in C^1(\sP_2(\M))$ and for any $x\in\bH$ such that $\sharp(x)=\mu\in\sP_2(\M),$ we have $\nabla\tilde \cU(x)=\nabla_w \cU(\mu)\circ x.$

{\it Claim.} For any $\mu\in\sP_2(\M)$, $q\mapsto\nabla_w \cU(\mu)(q)$ is Lipschitz continuous on $\spt(\mu)$ uniformly in $\mu$, with Lipschitz constant at most $C$.

{\it Proof of the claim.} Let $\mu\in\sP_2(\M)$ and consider $x, y\in\bH$, such that $\sharp(x)=\sharp(y)=\mu$ and $\|x-y\|>0$. Since $\nabla\tilde \cU$ is Lipschitz continuous, one has that 
$$\|\nabla\tilde \cU(x)-\nabla\tilde \cU(y)\|\le C\|x-y\|.$$
This reads off 
\be\label{ineq:lip}
\|\nabla_w \cU(\mu)(x)-\nabla_w \cU(\mu)(y)\|\le C\|x-y\|.
\ee


Suppose that $\spt(\mu)$ contains more than one element, otherwise the statement is trivial. Although $x$ is defined up to a set of measure zero, we are going to choose a representative which is Borel.  Set 
\[
\Omega_0:=\bigl\{\omega \in \Omega\; | \; \omega \;\; \text{is a Lebesgue point for}\;\; x, \nabla\tilde \cU(x) \bigr\} \cap x^{-1}({\rm spt\,}(\mu))
\]
Note that $\Omega_0$ is a set of full measure in $\Omega$ and so, $x(\Omega_0)$ is a set of full $\mu$--measure. In fact, we do not know that $x(\Omega_0)$ is Borel, but we can find a Borel set $A\subset x(\Omega_0)$ of full $\mu$--measure.

We suppose that $A$ has more than one element, otherwise the statement is trivial. Let $q_1, q_2\in A$ with $q_1\neq q_2$ and let $q_1^0,q_2^0\in\Om_0$ such that $x(q_1^0)=q_1$ and $x(q_2^0)=q_2$.  Let $r>0$ small such that $B_r(q_1^0)\cap B_r(q_2^0)=\emptyset$. Set

\begin{equation}\label{eq:FOMFG} 
S_r(\omega):=
\left\{
\begin{array}{ll}
\omega, &\hbox{if} \; \omega \in \Omega \setminus \bigl(B_r(q_1^0) \cup B_r(q_2^0) \bigr),\\
\omega -q_1^0+q_2^0,&\hbox{if} \; \omega \in  B_r(q_1^0),\\
\omega -q_2^0+q_1^0,&\hbox{if} \; \omega \in B_r(q_2^0). 
\end{array}
\right.
\end{equation}
Since $S_r$ preserves $\sL^d\mres\Om$, $x$ and $y:=x \circ S_r$ have the same law $\mu$. We notice that in particular
$$y=x\chi_{\M\setminus(B_r(q_1^0)\cup B_r(q_2^0))}+x(\cdot+q_2^0-q_1^0)\chi_{B_r(q_1^0)}+x(\cdot+q_1^0-q_2^0)\chi_{B_r(q_2^0)}.$$
Since $q_1$ and $q_2$ are distinct image points of $x$, for $r>0$ sufficiently small 
$$\|x-y\|^2=\int_{B_r(q_1^0)}|x(z)-x(z+q_2^0-q_1^0)|^2\dd z+\int_{B_r(q_2^0)}|x(z)-x(z+q_1^0-q_2^0)|^2\dd z>0.$$ 
Similarly, \eqref{ineq:lip} yields
\begin{align*}
&\|\nabla_w \cU(\mu)(x)-\nabla_w \cU(\mu)(y)\|^2=\int_{B_r(q_1^0)}|\nabla_w \cU(\mu)(x(z))-\nabla_w \cU(\mu)(x(z+q_2^0-q_1^0))|^2\dd z\\
&+\int_{B_r(q_2^0)}|\nabla_w \cU(\mu)(x(z))-\nabla_w \cU(\mu)(x(z+q_1^0-q_2^0))|^2\dd z\\
&\le C^2\left(\int_{B_r(q_1^0)}|x(z)-x(z+q_2^0-q_1^0)|^2\dd z+\int_{B_r(q_2^0)}|x(z)-x(z+q_1^0-q_2^0)|^2\dd z\right)
\end{align*}
Now, dividing the inequality by $\sL^d(B_r(q_1^0))$, and sending $r\da 0$, since $q_1^0$ and $q_2^0$ are Lebesgue point of $x$ with $x(q_1^0)=q_1$ and $x(q_2^0)=q_2$, one obtains that
$$|\nabla_w \cU(\mu)(q_1)-\nabla_w \cU(\mu)(q_2)|\le C|q_1-q_2|,$$
as desired. The claim follows.

Now, let $\mu,\nu\in\sP(\M)$ and $x,y\in\bH$ such that $\sharp(x)=\mu$, $\sharp(y)=\nu$ and $W_2(\mu,\nu)=\|x-y\|.$ Let us note that $\g:=\sharp(x,y) \in \Gamma_o(\mu, \nu).$ We have 
$$\nabla\tilde \cU(x)(y-x)=\int_\Om \nabla_w \cU(\mu)(x(\om))\cdot (y(\om)-x(\om))\dd \om=\int_{\M^2}\nabla_w \cU(\mu)(q_1)\cdot (q_2-q_1)\dd\g(q_1, q_2).$$
Thus, by \eqref{eq:proof_1} 
$$\left|\cU(\nu)-\cU(\mu)- \int_{\M^2}\nabla_w \cU(\mu)(q_1)\cdot(q_2-q_1)\dd\g(q_1, q_2)\right|\le {C \over 2}W_2^2(\mu,\nu),$$
which by the arbitrariness of $\mu,\nu$ implies the statement.

{\bf Part 2.} We now need to prove the reversed implication and start by assuming that $\cU$ is $C^{1,1}(\sP_2(\M))$. In particular $\nabla_w \cU(\mu)(\cdot)$ is $C$--Lipschitz continuous on $\spt(\mu)$ (uniformly in $\mu$) and increasing the value of $C$ if necessary, we assume the inequality in Definition \ref{def:c11_wasserstein}(2) to hold with the same constant $C$. Take $x,y\in\bH$ and set $\mu:=\sharp(x)$ and $\nu:=\sharp(y).$ Recall $\tilde \cU \in C^1(\bH)$ and $\nabla\tilde \cU (x)=\nabla_w \cU (\mu)\circ x$. Let $\g:=\sharp(x, y)$ and let $\g_0\in\Gamma_o(\mu,\nu)$. We have 
\begin{align*}
& \left|\tilde \cU(y)-\tilde \cU(x)-\nabla\tilde \cU(x)(y-x)\right| \\ 
= &\left|\cU(\nu)- \cU(\mu)- \int_{\M^2}\nabla_w\cU(\mu)(q_1)\cdot(q_2-q_1)\dd\g(q_1, q_2)\right|\\
\le &\left|\cU(\nu)- \cU(\mu)-\int_{\M^2}\nabla_w \cU(\mu)(q_1)\cdot(q_2-q_1)\dd\g_0(q_1, q_2)\right|\\
+&\left|\int_{\M^2}\nabla_w\cU(\mu)(q_1)\cdot(q_2-q_1)\dd(\g_0-\g)(q_1, q_2)\right|\\
\le & CW_2^2(\mu,\nu)+\frac12\left\|D_q\nabla_w \cU(\mu)\right\|_{L^\infty}\left(\int_{\M^2}|q_1-q_2|^2\dd\g(q_1, q_2)+\int_{\M^2}|q_1-q_2|^2\dd\g_0(q_1, q_2)\right)\\
\le &CW_2^2(\mu,\nu)+\frac12 C\left(\|x-y\|^2+W_2^2(\mu,\nu)\right)\le 2C\|x-y\|^2,
\end{align*}
where in the penultimate line we used an inequality from Lemma 3.3 \cite{GangboT2017}. Indeed, according to Lemma 3.3 \cite{GangboT2017} if $\g_1,\g_2\in\Gamma(\mu,\nu)$ and $\xi\in C_c^2(\M)$, then
$$\left|\int_{\M^2}D\xi(q_1)\cdot(q_2-q_1)\dd(\g_1-\g_2)(q_1, q_2)\right|\le\frac12\|D^2\xi\|_{L^\infty}\left(\int_{\M^2}|q_1-q_2|^2\dd(\g_1+\g_2)(q_1, q_2)\right).$$
Since $\nabla_w \cU(\mu)$ is the limit of $(D\xi_n)_{n\in\N}$ (where $(\xi_n)_{n\in\N}\in C_c^\infty(\M)$) in $L^2_{\mu}(\M;\R^d)$ and $\nabla_w\cU(\mu)$ has a global Lipschitz continuous extension to $\M$, it is easy to see that the previous inequality is still valid for $D\xi=\nabla_w \cU(\mu)$ (for which we use its Lipschitz continuous extension to $\M$). 

This completes the verification of the proof of the lemma.\end{proof}

\begin{remark}
\begin{itemize}
\item[(i)] It seems an interesting open problem whether the equivalence in Lemma \ref{lem:c11_equiv} hold for $C^{1,\alpha}$ functions for $\a\in(0,1)$.
\item[(ii)] The uniform Lipschitz continuity property of $q\mapsto\nabla_w \cU(\mu)(q)$, from the proof of Lemma \ref{lem:c11_equiv}, appeared already in \cite[Lemma 3.3]{CarmonaD2015} and in \cite[Proposition 5.36]{CarmonaD-I}. However, not only our proof is based on a different approach, it is considerably shorter and will be useful in the proof of Lemma \ref{lem:weak-Holder}.
\end{itemize}
\end{remark}

\begin{definition}\label{def:c21_wasserstein} 
Let $\cB\subseteq \sP_2(\M)$ be open and geodesically convex and let $\alpha \in (0,1].$ We say that $\cU\in C^{2, \alpha,w }(\cB)$, if $\cU\in C^{1,\a}(\cB)$, and if there exist a constant $C>0$, and functions 
$$\Lambda_0:\R^d\times\cB\to\R^{d\times d}, \;\; \Lambda_1:\M^2\times\cB\to\R^{d\times d}$$ such that 
$$\Lambda_0\in L^\infty(\M;\mu), \;\; \Lambda_1\in L^\infty(\M^2;\mu\otimes\mu)$$  
\begin{itemize}
\item[(1)] \small
$$\bigg{|}\nabla_w\cU(\nu)(\overline q_1)-\nabla_w\cU(\mu)(q_1)-\Lambda_0(q_1,\mu)(\overline q_1-q_1)-\int_{\M^2}\Lambda_1(q_1,a,\mu)(b-a)\dd\g(a,b)\bigg{|}\le C\left(|q_1-\overline q_1|^{1+\a}+W_2(\mu,\nu)^{1+\a}\right)
$$
\item[(2)] $\Lambda_0$ and $\Lambda_1$ are $\a$--H\"older continuous, i.e. 
$$|\Lambda_0(q_1,\mu)-\Lambda_0(\overline q_1,\nu)|_\infty\le C\big(|q_1-\overline q_1|^\alpha +W_2^\alpha (\mu,\nu)\big)$$
and
$$|\Lambda_1(q_1, q_2,\mu)-\Lambda_1(\overline q_1, \overline q_2,\nu)|_\infty\le C(|q_1-\overline q_1|^\alpha +|q_2-\overline q_2|^\alpha +W_2^\alpha (\mu,\nu)),$$
\end{itemize}
for any $\mu,\nu\in\cB,$ $(q_1,\overline  q_1),(q_2, \overline q_2)\in\spt(\mu)\times\spt(\nu)$ and $\g\in\Gamma_o(\mu,\nu).$

We say that $\cU\in C^{2,\alpha,w }_{\rm{loc}}(\sP_2(\M))$, if $\cU\in C^{2,\alpha,w }(\cB_r)$ for all $r>0$.
\end{definition}
\begin{remark}\label{rem:whessian} Let $\Lambda_0$ and $\Lambda_1$ be as above.
\begin{itemize} 
\item[(1)] By abuse of notation we write 
$$D_{q_1}\big(\nabla_w\cU(\mu)(q_1)\big):=\Lambda_0(q_1,\mu)\ \  {\rm{and}}\ \  \overline \nabla^2_{ww}\cU(\mu)(q_1,q_2):=\Lambda_1(q_1, q_2,\mu),$$
for all $\mu\in\sP_2(\M)$ and $x,y\in\spt(\mu).$ The bar is to recall that $ \Lambda_1$ is not exactly the second Wasserstein gradient as introduced in \cite{ChowGan}.
\item[(2)]  Note that if we choose any  matrix $\Lambda(a, \mu)$ such that  any of its rows $w$ is such that $\nabla\cdot (w\mu)=0$ and $w\in L^2(\mu)$, then the matrix defined as $\overline \Lambda_1(q, a, \mu):= \Lambda_1(q, a, \mu)+\Lambda(a, \mu)$ also satisfies Definition \ref{def:c21_wasserstein} (1). We could determine $\Lambda_1(q, \cdot, \mu)$ uniquely by imposing that the $i$-th row of $(\Lambda_0(q, \mu), \Lambda_1(q, \cdot, \mu))$ is the unique element of minimal norm of the subdifferential of $(q, \mu) \mapsto \nabla_w \cU(\mu)(q).$  The $i$-th row of the element of minimal norm belongs to $\M \times T_\mu \sP_2(\M)$ and  the new matrix will be denoted as $\nabla_{ww}^2 \cU(\mu).$ This new matrix is selected at the expense of giving up the property that $\Lambda_1$ is uniformly bounded.  Increasing $C$ if necessary, we can instead ensure  
\[\|\nabla_{ww}^2 \cU(\mu)(q_1, \cdot)\|_{L^2_\mu} \leq C(r) \qquad \forall \mu \in \cB, \forall q_1 \in\spt(\mu). \]

\item[(3)] In the spirit of the terminology used in \cite{ChowGan}, we refer to $\overline \nabla^2_{ww}\cU$ as an ``extended Wasserstein Hessian'' of $\cU$. In contrast with the assumptions in \cite{ChowGan}, in Definition \ref{def:c21_wasserstein} (1), we assume slightly different conditions: the expansion here is required only on $\spt(\mu)\times\spt(\nu)$, $\Lambda_0$ and $\Lambda_1$ are supposed to be essentially bounded only on $\spt(\mu)$, and in addition we require the H\"older/Lipschitz property in Definition \ref{def:c21_wasserstein} (2) to be fulfilled. 

\item[(4)] Let us compare our definition of $C^{2,\a,w}_{\rm{loc}}(\cP_2(\M))$ regularity of $\cU$ to $C^{2,\a}_{\rm{loc}}(\bH)$ regularity of $\tilde \cU$ (where $\tilde\cU(x)=\cU(\sharp(x))$). If $\tilde\cU\in C^{2,\a}_{\rm{loc}}(\bH)$, then $\tilde\cU$ is twice continuously differentiable in the Fr\'echet sense and for each $r>0$ there exists $C=C(r)$ such that
\begin{equation}\label{equ:C21-hil}
\|\nabla\tilde\cU(y)-\nabla\tilde\cU(x)-\nabla^2\tilde\cU(x)(y-x,\cdot)\|\le C\|x-y\|^{1+\a},\ \forall\ x,y\in\B_r.
\end{equation}
To heuristically compare this inequality to the setting of $\cP_2(\M)$ we proceed as follows. Let $\sharp(x)=\mu$ and $\sharp(y)=\nu$ with $\|x-y\|=W_2(\mu,\nu).$ Then we know (see \cite{GangboT2017}) that $\nabla\tilde\cU(x)=\nabla_w\cU(\mu)\circ x$, $\nabla\tilde\cU(y)=\nabla_w\cU(\nu)\circ y$ and 
\begin{align*}
\nabla^2\tilde U(x)(h,h_*)=
 \int_{\Omega} D_{q}\big(\nabla_w \cU(\mu) \big)\circ x\; h\cdot h_* d\omega+  \int_{\Omega^2} \nabla^2_{ww} \cU(\mu)\big(x(\omega), x(\omega_*)\big)h(\omega)\cdot h_*(\omega_*) d\omega d\omega_*,
\end{align*}
if $\xi, \xi_* \in T_\mu \mathcal P_2(\mathbb M)$ and $h=\xi \circ x$ and $h_*=\xi_* \circ x.$ Thus, \eqref{equ:C21-hil} would read as 
\begin{align}
&\sup_{\|h_*\|\le 1}\Bigg{|}\int_\Om\left[ \nabla_w\cU(\nu)(y(\omega))\cdot h_*(\omega) - \nabla_w\cU(\mu)(x(\omega))\cdot h_*(\omega)\right]d\omega\nonumber\\
& -  \int_{\Omega} D_{q}\big(\nabla_w \cU(\mu) \big)\circ x\; (y-x)\cdot h_* d\omega-  \int_{\Omega^2} \nabla^2_{ww} \cU(\mu)\big(x(\omega), x(\omega_*)\big)(y-x)(\omega)\cdot h_*(\omega_*) d\omega d\omega_*  \Bigg{|}\nonumber\\
&\le CW_2(\mu,\nu)^{1+\a}\label{eq:LinftyvsL2}.
\end{align}
From here we see, a necessary condition to obtain inequality (1) in Definition \ref{def:c21_wasserstein} is to have \eqref{eq:LinftyvsL2} hold when we maximize over the set of $h$ such that $\|h_*\|_{L^1}\le 1$ rather than maximizing over the set of $h$ such that $\|h_*\|\le 1.$ In other words, we have not been able to show that if $\tilde\cU\in C^{2,\alpha}_{\rm{loc}}(\bH)$ then $\cU\in C^{2,\alpha,w}_{\rm{loc}}(\cP_2(\M)).$ Moreover, in Appendix \ref{appendixE} we show that imposing $\cU\in C^{2,\a,w}_{\rm{loc}}(\cP_2(\M))$ in general does not imply that $\tilde\cU\in C^{2,\alpha}_{\rm{loc}}(\bH).$

\item[(5)] Let us point out that using an extrinsic approach, \cite{BuckdahnLPR} introduced  spaces of the type $C^{2,1}(\cP_2(\M))$ via the differentials of their lifts on a  Hilbert space. In this work, we define $C^{2,1,w}(\cP_2(\M))$, in an intrinsic way, i.e. directly via the differential calculus on the Wasserstein space. As a result, our derivatives are always defined on the supports of the corresponding measures, while in \cite{BuckdahnLPR} the authors work with global extensions. Similarly, we require essential boundedness of the Wasserstein Hessian only on the support of the corresponding measures, while \cite{BuckdahnLPR} requires boundedness of the global extensions. The work by \cite{GangboT2017}, allows to assert that both the intrinsic and extrinsic approaches is essentially the same. However, $C^{2,1,w}(\cP_2(\M))$ has the advantage that it can be seen  as an increasing `limit' of the spaces $C^{2,1}(\M^m)$, when $m\to+\infty$, as we show this in Subsection \ref{subsec:regularity_limit} below.

\item[(6)] \cite[Section 2]{BuckdahnLPR} constructs an example of $\cU\in C^{2,1}(\cP_2(\M))$ for which its lifted version $\tilde\cU$ fails to be twice Fr\'echet differentiable at any point. More discussions can be found in \cite{BuckdahnLPR,CardaliaguetDLL, CarmonaD-I,CarmonaD-II,ChassagneuxCD}.
\end{itemize}
\end{remark}

\subsection{Regularity of $\cU$ as a by-product of regularity estimates on $U^{(m)}$}\label{subsec:regularity_limit} This subsection infer regularity properties on functions $\cU$ defined on $\sP_2(\M)$, from estimates on their restrictions $U^{(m)}.$ Recall that for $r>0$ $\B_r^m$ is a ball in $\M^m$ while $\cB_r$ is a ball in $\sP_2(\M).$ We assume that we have at hand a constant $C=C(r)>0$. 

\begin{lemma}\label{lem:reg_finite_infinite}
Suppose for each $m\in\N$ fixed, $U^{(m)}:\M^m\to\R$ is permutation invariant with respect to its $m$-variables and $|U^{(m)}|$ is bounded on $\B^m_r$ by a constant which depends on $r>0$ but is independent of $m.$ Then there exists $C=C(r)>0$ such that the followings hold true.
\begin{itemize}
\item[(i)] If $U^{(m)}$ satisfies Property \ref{def:app_reg_estim} (1)-(b) then for any $q, \; b\in \B^m_r,$ we have
$$|U^{(m)}(q)-U^{(m)}(b)|\le C W_2(\mu^{(m)}_q,\mu^{(m)}_b).$$
\item[(ii)] If $U^{(m)}$ satisfies Property \ref{def:app_reg_estim} (2). Then for any $q, \; b\in \B^m_r,$ we have 
$$\Big{|}U^{(m)}(b)-U^{(m)}(q)-\sum_{i=1}^m D_{q_i}U^{(m)}(q)\cdot(b_i-q_i)\Big{|}\le C W_2^2(\mu^{(m)}_q,\mu^{(m)}_b).$$ 
\item[(iii)] The assumption in (ii) implies  for any $q, \; b\in \B^m_r,$ 
\begin{itemize}
\item[(a)]  
$$m |D_{q_i}U^{(m)}(q)-D_{q_i}U^{(m)}(b)|\le C\left(|q_i-b_i|+W_2(\mu^{(m)}_q,\mu^{(m)}_b)\right).$$
\item[(b)] We have 
$$m |D_{q_i}U^{(m)}(q)-D_{q_j}U^{(m)}(b)|\le C\left(|q_i-b_j|+W_2(\mu^{(m)}_q,\mu^{(m)}_b)+\frac{1}{\sqrt{m}}\right),\ i\neq j.$$
\end{itemize}
\item[(iv)] Suppose that $U^{(m)}$ satisfies Property \ref{def:app_reg_estim} (3). If $i \in \{1, \cdots, m\}$ and $q, \; b\in \B^m_r$ then  
$$
m \bigg{|}D_{q_i}U^{(m)}(b)-D_{q_i}U^{(m)}(q)-\sum_{j=1}^m D^2_{q_iq_j}U^{(m)}(q)(b_j-q_j)\bigg{|}\le C\left(|q_i-b_i|^2+ W_2^2(\mu^{(m)}_q,\mu^{(m)}_b)\right).
$$ 
\item[(v)] The assumption in (iv) implies, $q, \; b\in \B^m_r,$ 
\begin{itemize}
\item[(a)] If  $i\neq j$ then 
$$m^2 |D^2_{q_i q_j}U^{(m)}(q)-D^2_{q_iq_j}U^{(m)}(b)|\le C\left(|q_i-b_i|+|q_j-b_j|+W_2(\mu^{(m)}_q,\mu^{(m)}_b)\right).$$
\item[(b)] If $(i,j)\neq (k,l), i\neq j, k\neq l$  then  
\begin{align*}
&m^2 |D^2_{q_iq _j}U^{(m)}(q)-D^2_{q_kq_l}U^{(m)}(b)|\le C \left(|q_i-b_k|+|q_j-b_l|+W_2(\mu^{(m)}_q,\mu^{(m)}_b)+\frac{1}{\sqrt{m}}\right).\\
\end{align*}
\item[(c)]  We have 
$$m |D^2_{q_iq_i}U^{(m)}(q)-D^2_{q_iq_i}U^{(m)}(b)|\le C \left(|q_i-b_i|+W_2(\mu^{(m)}_q,\mu^{(m)}_b)\right).$$
\item[(d)] We have 
$$m |D^2_{q_iq_i}U^{(m)}(q)-D^2_{q_jq_j}U^{(m)}(b)|\le C \left(|q_i-b_j|+W_2(\mu^{(m)}_q,\mu^{(m)}_b)+\frac{1}{\sqrt{m}}\right).$$
\end{itemize}
\end{itemize}
\end{lemma}

\begin{proof} Since $U^{(m)}$ is permutation invariant reordering $q$ and $b$ if necessary, we may assume 
$$\g^{(m)}:=\frac{1}{m}\sum_{i=1}^m\d_{(q_i, b_i)}\in\Gamma_{o}(\mu^{(m)}_q,\mu^{(m)}_b).$$ 
Below, using Taylor's  expansion, we may find $\xi\in\B^m_r$ on the line segment connecting $q$ to $b$ such that (using the shorthand notation $\|\cdot\|_\infty$ to denote $\|\cdot\|_{L^\infty(\B_r^m)}$)

(i)  we have
$$ |U^{(m)}(b)-U^{(m)}(q)|\le \Big{|}\sum_{i=1}^m D_{q_i}U^{(m)}(\xi)\cdot(b_i-q_i)\Big{|}\le\left(\sum_{i=1}^m m|D_{q_i} U^{(m)}|^2\right)^{\frac12}\left(\sum_{i=1}^m\frac{1}{m}|q_i-b_i|^2\right)^{\frac12}.$$
Using the fact that 
$$\sum_{i=1}^m m|D_{q_i}U^{(m)}(q)|^2\le C^2\quad \text{and} \quad \sum_{i=1}^m\frac1m |q_i-b_i|^2=W_2^2(\mu^{(m)}_q,\mu^{(m)}_b),$$
we verify the statement in (i).

(ii) A second order Taylor expansion yields
\begin{align*}
U^{(m)}(b)&-U^{(m)}(q)-\sum_{i=1}^m D_{q_i}U^{(m)}(q)\cdot(b_i-q_i)=\frac12\sum_{i,j=1}^m \langle (b_i-q_i),D^2_{q_iq_j}U^{(m)}(\xi)(b_j-q_j)\rangle\\
&=\frac12\sum_{i=1}^m \langle (b_i-q_i),D^2_{q_iq_i}U^{(m)}(\xi)(b_i-q_i)\rangle+\frac12\sum_{i\neq j} \langle (b_i-q_i),D^2_{q_iq_j}U^{(m)}(\xi)(b_j-q_j)\rangle
\end{align*}
Thus, under the assumption in (ii), we have
\begin{align*}
\Big{|}U^{(m)}(b)&-U^{(m)}(q)-\sum_{i=1}^m D_{q_i}U^{(m)}(q)\cdot(b_i-q_i)\Big{|}\le \frac{C}{2m}\sum_{i=1}^m |q_i-b_i|^2+\frac14\sum_{i\neq j}\|D^2_{q_i q_j} U^{(m)}\|_{\infty}|q_i-b_i|^2\\
&+\frac14\sum_{i\neq j}\|D^2_{q_i q_j} U^{(m)}\|_{\infty}|q_j-b_j|^2\\
&\le \Bigl(\frac{C}{2}+\frac{C}{4} +\frac{C}{4}\Big) \int_{\M^2}|z-w|^2\dd\g^{(m)}(z,w)=CW_2^2(\mu^{(m)}_q,\mu^{(m)}_b).
\end{align*}

(iii)-(a) Performing again a first order Taylor expansion, we find
\begin{align*}
D_{q_i}U^{(m)}(q)-D_{q_i}U^{(m)}(b)&=\sum_{k=1}^m D^2_{q_k q_i} U^{(m)}(q)(q_k-b_k)\\
&=D^2_{q_i q_i} U^{(m)}(\xi)(q_i-b_i)+\sum_{k\neq i}D^2_{q_k q_i} U^{(m)}(\xi)(q_k-b_k).
\end{align*}
Thus using the assumptions, we find
\begin{align*}
\Big{|}D_{q_i}U^{(m)}(q)-D_{q_i}U^{(m)}(b)\Big{|}&\le \frac{C}{m}|q_i-b_i|+\left(\sum_{k\neq i}m^3\|D^2_{q_kq_i}U^{(m)}\|^2_{\infty}\right)^{\frac{1}{2}}\left(\sum_{k\neq i}\frac{1}{m^3}|q_k-b_k|^2\right)^{\frac{1}{2}}\\
&\le \frac{C}{m}\left(|q_i-b_i|+W_2(\mu^{(m)}_q,\mu^{(m)}_b)\right).
\end{align*}

(iii)-(b) Without loss of generality, let us suppose that $i<j$. By the permutation invariance of $U^{(m)}$, we observe that $D_{q_i}U^{(m)}(q)= D_{q_1}U^{(m)}(q^{ij})$ and a similar identity holds for $\quad D_{q_j}U^{(m)}(b)$ if we set 
\begin{equation}\label{eq:qij1} 
q^{ij}:= (q_i,q_j,q_1,\dots,q_{i-1},q_{i+1},\dots,q_{j-1},q_{j+1},\dots,q_m).
\end{equation} 
Using a similar identity for $D_{q_j}U^{(m)}(b)$ we obtain 
\begin{align*}
|D_{q_i}U^{(m)}(q)-D_{q_j}U^{(m)}(b)| & =|D_{q_1}U^{(m)}(q^{ij})- D_{q_1}U^{(m)}(b^{ij})|\\
&\le \|D^2_{q_1q_1}U^{(m)}\|_\infty|q_i-b_j|+\|D^2_{q_2q_1} U^{(m)}\|_\infty|q_j-b_i|\\
&+\sum_{k=1}^{i-1}\|D^2_{q_{k+2} \; q_1}U^{(m)}\|_\infty|q_{k}-b_{k}|+\sum_{k=i+1}^{j-1}\|D^2_{q_{k+1}\; q_1} U^{(m)}\|_\infty|q_{k}-b_{k}|\\
&+\sum_{k=j+1}^{m}\|D^2_{q_{k}q_1} U^{(m)}\|_\infty|q_{k}-b_{k}|.
\end{align*}
Thus, 
\begin{align*}
|D_{q_i}U^{(m)}(q)-D_{q_j}U^{(m)}(b)| 
&\le\frac{C}{m}|q_i-b_j|+\frac{C}{m^2}(|q_j|+|b_i|)+\frac{C}{m^2}\sum_{k=1}^m|q_k-b_k|\\
&\le \frac{C}{m}\left(|q_i-b_j|+W_2(\mu^{(m)}_q,\mu^{(m)}_b)+\frac{2r\sqrt{m}}{m}\right)\\
&\le \frac{C}{m}\left(|q_i-b_j|+W_2(\mu^{(m)}_q,\mu^{(m)}_b)+\frac{1}{\sqrt{m}}\right),
\end{align*}
where we have used the assumptions on $D^2_{q_iq_j}U^{(m)}$ and in the last two rows we used the facts that since $q,b\in\B_r^m$, we have that $|q_i|,|b_j|\le r\sqrt{m}$, for all $i,j\in\{1,\dots,m\}.$

(iv) Similarly to the previous points, we perform a Taylor expansion (of order two) to obtain
\begin{align*}
D_{q_i}U^{(m)}(b)-D_{q_i}U^{(m)}(x)-\sum_{j=1}^m D^2_{q_iq_j}U^{(m)}(q)(b_j-q_j)=\frac{1}{2} \sum_{j,k=1}^m \langle (b_k-q_k),D^3_{q_iq_jq_k}U^{(m)}(q)(b_j-q_j)\rangle,
\end{align*}
and thus
\begin{align*}
\Big{|}D_{q_i}U^{(m)}(b)&-D_{q_i}U^{(m)}(q)-\sum_{j=1}^m D^2_{q_iq_j}U^{(m)}(q)(b_j-q_j)\Big{|}\\
&\le\frac12\|D^3_{q_iq_iq_i} U^{(m)}\|_{\infty}|q_i-b_i|^2+\frac12\sum_{j\neq i}\|D^3_{q_iq_jq_j} U^{(m)}\|_{\infty}|q_j-b_j|^2\\
&+\frac12\sum_{j\neq k\neq i}\|D^3_{q_iq_jq_j} U^{(m)}\|_{\infty}|q_j-b_j|\cdot|q_k-b_k|. 
\end{align*}
We conclude 
\begin{align*} 
& \Big{|}D_{q_i}U^{(m)}(b)-D_{q_i}U^{(m)}(q)-\sum_{j=1}^m D^2_{q_iq_j}U^{(m)}(q)(b_j-q_j)\Big{|}\\
&\le\frac{C}{2m}|q_i-b_i|^2+\frac{C}{2m}\sum_{j=1}^m\frac{1}{m}|q_j-b_j|^2+\frac{C}{2m}\left(\sum_{j=1}^m\frac{1}{m}|q_j-b_j|\right)\left(\sum_{k=1}^m\frac{1}{m}|q_k-b_k|\right)\\
&\le\frac{C}{2m}\left(|q_i-b_i|^2+W_2^2(\mu^{(m)}_q,\mu^{(m)}_q)\right),
\end{align*}

(v) We write again 
\begin{align*}
D^2_{q_iq_j}U^{(m)}(q)-D^2_{q_iq_j}U^{(m)}(b)&=\sum_{k=1}^m D^3_{q_iq_jq_k}U^{(m)}(q)(q_k-b_k)\\
&=D^3_{q_iq_jq_i}U^{(m)}(q)(q_i-q_i)+D^3_{q_iq_jq_j}U^{(m)}(q)(q_j-b_j)\\
&+\sum_{k=1,k\neq i, k\neq j}^mD^3_{q_iq_jq_k}U^{(m)}(q)(q_k-q_k).
\end{align*}
Thus in the case of (a) using the assumptions, we find

\begin{align*}
\Big{|}D^2_{q_iq_j}U^{(m)}(q)-D^2_{q_iq_j}U^{(m)}(b)\Big{|}&\le \frac{C}{m^2}(|q_i-b_i|+|q_j-b_j|)+C\sum_{k=1}^m\frac{1}{m^3}|q_k-b_k| \\
&\le \frac{C}{m^2}\left(|q_i-b_i|+|q_j-b_j|+W_2(\mu^{(m)}_q,\mu^{(m)}_b)\right).
\end{align*}

In the case of (c), since $i=j$ in the above expansion, we find
\begin{align*}
|D^2_{q_iq_i}U^{(m)}(q)-D^2_{q_i q_i}U^{(m)}(b)| &\le \|D^3_{q_iq_iq_i}U^{(m)}\|_{\infty}|q_i-b_i|+\sum_{k\neq i}\|D^3_{q_iq_iq_k}U^{(m)}\|_{\infty}|q_k-b_k|\\
&\le \frac{C}{m}\left(|q_i-b_i|+W_2(\mu^{(m)}_q,\mu^{(m)}_b)\right).
\end{align*}

To show (b), let us suppose without loss of generality that $i<j< k<l$. By the permutation invariance of $U^{(m)}$ we have the identities 
$$
D^2_{q_iq_j}U^{(m)}(q)=D^2_{q_1q_2}U^{(m)}(q_i,q_j,q_k, q_l, \ov q) \quad \text{and} \quad 
D^2_{q_k q_l}U^{(m)}(b)=D^2_{q_1q_2}U^{(m)}(b_k,b_l,b_i,b_j,\ov b),
$$
where $\ov q,\ov b\in\R^{d\times(m-4)}$ obtained from $q$ and $b$, respectively, by deleting the vectors indexed by $i,j,k,l$. Therefore, using the local bounds on the third order derivatives of $U^{(m)}$, we have
\[
|D^2_{q_iq_j}U^{(m)}(q)-D^2_{q_k q_l}U^{(m)}(b)|=|D^2_{q_1q_2}U^{(m)}(q_i,q_j,q_k, q_l,\ov q)-D^2_{q_1q_2}U^{(m)}(q_k,q_l,q_i,q_j,\ov b)|
\]
and so, 
\begin{align*}
|D^2_{q_iq_j}U^{(m)}(q)-D^2_{q_k q_l}U^{(m)}(b)| 
&\le \|D^3_{q_1q_2q_1}U^{(m)}\|_\infty|q_i-b_k| + \|D^3_{q_1q_2q_2}U^{(m)}\|_\infty |q_j-b_l|\\
&+ \|D^3_{q_1q_2q_3}U^{(m)}\|_\infty|q_k-b_i| + \|D^3_{q_1q_2q_4}U^{(m)}\|_\infty|q_l-b_j|\\
&+\sum_{\a=1}^{i-1}\|D^3_{q_1q_2q_{\a+4}}U^{(m)}\|_\infty|q_\a-b_\a|+\sum_{\a=i+1}^{j-1}\|D^3_{q_1q_2q_{\a+3}}U^{(m)}\|_\infty|q_\a-b_\a|\\
&+\sum_{\a=j+1}^{k-1}\|D^3_{q_1q_2q_{\a+2}}U^{(m)}\|_\infty|q_\a-b_\a|+\sum_{\a=k+1}^{l-1}\|D^3_{q_1q_2q_{\a+1}}U^{(m)}\|_\infty|q_\a-b_\a|\\
&+\sum_{\a=l+1}^{m}\|D^3_{q_1q_2q_{\a}}U^{(m)}\|_\infty|q_\a-b_\a|.
\end{align*}
Thus, 
\begin{align*} 
& |D^2_{q_iq_j}U^{(m)}(q)-D^2_{q_k q_l}U^{(m)}(b)| \\
&\le \frac{C}{m^2}\left(|q_i-b_k| +  |q_j-b_l|\right) + \frac{C}{m^3}(|q_k|+|b_i|+|q_l|+|b_j|)+\frac{C}{m^3}\sum_{\a=1}^m|q_\a-b_\a|\\
&\le \frac{C}{m^2}\left(|q_i-b_k| +  |q_j-b_l|+W_2(\mu^{(m)}_q,\mu^{(m)}_b)+\frac{1}{\sqrt{m}}\right),
\end{align*}
where we have used again that since $q, b\in\B_r^m$, we have $|q_\a|, |b_\a|\le C\sqrt{m}$ for all $\a\in\{1,\dots,m\}.$

In the case of (d), we proceed similarly as for (b). Let us suppose without loss of generality that $i<j$. Then, by the permutation invariance of $U^{(m)}$, we use the expression in \eqref{eq:qij1} to obtain 
$$D^2_{q_iq_i}U^{(m)}(q)=D^2_{q_1q_1}U^{(m)}(q^{ij}),$$
Using the analogous identity with $D^2_{q_jq_j}U^{(m)}(b)$ we conclude 
\begin{align*}
|D^2_{q_iq_i}U^{(m)}(q)-D^2_{q_jq_j}U^{(m)}(b)|&=|D^2_{q_1q_1}U^{(m)}(q^{ij})-D^2_{q_1q_1}U^{(m)}(b^{ij})|\\
&\le \|D^3_{q_1q_1q_1}U^{(m)}\|_\infty|q_i-b_j| + \|D^3_{q_1q_1q_2}U^{(m)}\|_\infty|q_j-b_i|\\
&+\sum_{k=1}^{i-1}\|D^3_{q_1q_1q_{k+2}}\|_\infty|q_{k}-b_{k}|+\sum_{k=i+1}^{j-1}\|D^3_{q_1q_1q_{k+1}}\|_\infty|q_{k}-b_{k}|\\
&+\sum_{k=j+1}^{m}\|D^3_{q_{1}q_1q_k}\|_\infty|q_{k}-b_{k}|.
\end{align*}
Thus, 
\begin{align*}
|D^2_{q_iq_i}U^{(m)}(q)-D^2_{q_jq_j}U^{(m)}(b)|
&\le\frac{C}{m}|q_i-b_j|+\frac{C}{m^2}(|q_j|+|b_i|)+\frac{C}{m^2}\sum_{k=1}^m|q_k-b_k|\\
&\le \frac{C}{m}\left(|q_i-b_j|+W_2(\mu^{(m)}_q,\mu^{(m)}_b)+\frac{1}{\sqrt{m}}\right),
\end{align*}
where we have used again that since $q, b\in\B_r^m$, we have $|q_\a|,|b_\a|\le C\sqrt{m}$ for all $\a\in\{1,\dots,m\}.$
\end{proof}

The following two theorems show how the quantified regularity estimates on the restrictions of functions $u:\M\times\sP_2(\M)\to\R$ and $\cU:\sP_2(\M)\to\R$ to $\M\times\M^m$ and $\M^m$, respectively, will imply the corresponding regularity of the original functions.

\begin{theorem}\label{thm:finite_infinite_reg-scalar-mast} 
Let $u:\M\times\sP_2(\M)\to\R$ be a continuous function. For $m\in\N$, we define $u^{(m)}:\M\times(\M)^m\to\R$ as
$$u^{(m)}(q_0,q):=u(q_0,\mu^{(m+1)}_q),$$
where $(q_0,q)=(q_0,q_1,\dots,q_m)\in (\M)^{m+1}$ and $\mu_q^{(m+1)}=\frac{1}{m+1}\sum_{i=0}^m\d_{q_i}$. 
Suppose that $u^{(m)} \in C^{1,1}_{\rm{loc}}(\M\times(\M)^m)$ and that for $K\subset\M$ compact and $r>0$, $u^{(m)}(q_0,\cdot)$ satisfies the estimates of Property \ref{def:app_reg_estim}(1)-(a) and (2) for all $q_0\in K$, with a constant $C=C(K,r)>0$. Let us moreover assume that for any $K\subset\M$ compact and $r>0$, there exists $C=C(K,r)>0$ such that 
\begin{align}\label{def:x_0-constant}
&|D_{q_0}u^{(m)}(q_0,q)|\le C,\ \ |D^2_{q_0q_0}u^{(m)}(q_0,q)|_\infty\le C,\ \sum_{i=1}^m m|D^2_{q_iq_0}u^{(m)}(q_0,q)|^2_\infty\le C\\
\nonumber&{\rm{and}}\\ 
\nonumber&|D^2_{q_iq_j} u^{(m)}(q_0,q)|_\infty \le \left\{
\begin{array}{ll}
\ds\frac{C}{m}, & i=j,\ {\rm{and}}\ i>0,\\[5pt]
\ds\frac{C}{m^2}, & i\neq j,\ i,j>0,
\end{array}
\right.
\end{align}
for any $q_0\in K$ and $q=(q_1,\dots,q_m)\in\B_r^m$.

Then, there exists $\Phi_1:\M\times\cP_2(\M)\times\M\to\R^d$ locally Lipschitz continuous function such that for any $r>0$ and $K\subset\M$ compact, there exists $C=C(K,r)>0$ such that for any $q_0,y_0\in K$, any $\mu,\nu\in\sP_2(\M)$ and $\g\in\Gamma_o(\mu,\nu)$, $u$ satisfies 
\begin{align*} 
& \Big{|}u(y_0,\nu)-u(q_0,\mu)-D_{q_0}u(q_0,\mu)\cdot(y_0-q_0)-\int_{\M^2}\Phi_1(q_0,\mu,q)\cdot(y-q)\dd\g(q,y)\Big{|}\\
& \le C\left(|q_0-y_0|^2+W_2^2(\mu,\nu)\right).
\end{align*}
This implies in particular that $u\in C^{1,1}_{\rm{loc}}(\M\times\sP_2(\M))$,  $\nabla_w u(q_0,\mu)(\cdot)$ can be obtained as the projection of $\Phi_1(q_0,\mu,\cdot)$ onto $T_\mu\cP_2(\M)$ and
\begin{align*} 
& \Big{|}u(y_0,\nu)-u(q_0,\mu)-D_{q_0}u(q_0,\mu)\cdot(y_0-q_0)-\int_{\M^2}\nabla_w u(q_0,\mu)(q)\cdot(y-q)\dd\g(q,y)\Big{|}\\
& \le C\left(|q_0-y_0|^2+W_2^2(\mu,\nu)\right).
\end{align*}
\end{theorem}
\begin{proof}  Our construction is inspired by \cite[Lemma 8.10]{GangboS2015}. 

For $m\in\N$ we define $\Phi_0^{(m)}:\M\times \sP_{2}^{(m)}(\M)\to\R^d$ and $\Phi_1^{(m)}:\M\times\bigcup_{\mu\in\sP_{2}^{(m)}(\M)}\spt(\mu)\times \{\mu\}\to \R^d$ as
$$\Phi_0^{(m)}(q_0,\mu^{(m)}_q):=D_{q_0}u^{(m)}(q_0,q)$$
and
$$\Phi_1^{(m)}(q_0,q_i,\mu^{(m)}_q):=m D_{q_i}u^{(m)}(q_0,q),\quad \forall i\in\{1,\dots,m\}.$$ 
Here $$q=(q_1,\dots,q_m) \quad \text{and} \quad \mu^{(m)}_q:=\frac1m\sum_{i=1}^m\d_{q_i} \in \sP_{2}^{(m)}(\M).$$
From the assumptions of this theorem, as a consequence of Lemma \ref{lem:reg_finite_infinite}(i), when restricted to $K\times\sP_{2}^{(m)}(\M)\cap\cB_r$ where $K\subseteq\M$ is  compact and $r>0$, $\Phi_0^{(m)}$ is uniformly bounded and uniformly Lipschitz continuous, with respect to $m$ (and the Lipschitz constant depends solely on $K$ and $r$). 

Let $\mathcal K$ be the collection of compact sets in $\M$. We assume there exists a positive function $C$ defined $\mathcal K \times (0,\infty)$ such that $C(K, r) \leq C(K', r')$ $K \subset K'$ and $r \leq r'.$ 

We assume to be given a family of functions 
\[
f^{(m)}: \M \times \sP_2^{(m)}(\M) \rightarrow \R
\]
such that  for each $r>0$ and each $K \in \mathcal K,$  the restriction of $f^{(m)}$ to  $K \times \Big(\sP^{(m)}_2(\M) \cap \cB_r \Big)$ is $C(K, r)$--Lipschitz. We assume there exists a compact subset in the real line which contains all the $f^{(m)}(0, \delta_0).$

In what follows, we will perform Lipschitz extensions of various functions using the Kirszbraun extension formula. For $r>0$, $q_0\in \M$ and $K \in \mathcal K,$ we define the Kirszbraun--Valentine extension  $f^{(m)}_{K, r}(q_0, \cdot):\sP_2(\M) \rightarrow \R$ as 
\begin{equation}\label{lip:extension}
f^{(m)}_{K, r}(q_0, \mu)= \inf_{\nu} \Big\{ f^{(m)}(q_0, \nu)+C(K, r)W_2(\mu, \nu) \; : \; \nu \in \sP_2^{(m)}(\M) \cap \cB_r  \Big\}.
\end{equation}
We have that $f^{(m)}_{K, r}(q_0, \cdot)$ is $C(K, r)$--Lipschitz for all $q_0 \in \M$ and  $f^{(m)}_{K, r}$ coincides with $f^{(m)}$ on $K \times ( \sP_2^{(m)}(\M) \cap \cB_r).$ Furthermore, for any $K' \in \mathcal K$, $f^{(m)}_{K, r}(\cdot, \mu)$ is $C(K', r)$--Lipschitz on $K' \times \sP_2(\M)$. 

Let $\overline B_R(0)$ denote the closed ball of radius $R>0$, centered at the origin in $\M$ and let  $\sP_{\rm c}(\M)$ be the union of all the $\sP_2(\overline B_R(0)).$  Since  $\sP_2(\overline  B_R(0))$ is a compact subset of $\sP_2(\M)$, we apply the Ascoli--Arzel\`a theorem and use a diagonalization argument to obtain a function 
$$
f^\infty_{K, r}: \M \times \sP_{\rm c}(\M) \rightarrow \R 
$$ 
such that a subsequence of $(f^{(m)}_{K, r})_m$ converges locally uniformly to $f^\infty_{K, r}$ on compact sets. We have that $f^\infty_{K, r}(q_0, \cdot)$ is $C(K, r)$--Lipschitz on $\sP_{\rm c}(\M)$ for all $q_0 \in \M$ and  $f^\infty_{K, r}(\cdot, \mu)$ is $C(K', r)$--Lipschitz on $K'$ for $\mu\in \sP_{\rm c}(\M).$
In fact 
\begin{equation}\label{eq:important1}
\big|f^\infty_{K, r}(q_0, \mu)-f^\infty_{K, r}(a_0, \nu)\big| \leq C(K', r)\Big(|q_0-a_0|+W_2(\mu, \nu) \Big) 
\end{equation}
for all $q_0, a_0 \in K'$ and $\mu, \nu \in \cB_r.$

The function $f^\infty_{K, r}$ admits a unique $C(K, r)$--Lipschitz extension to $K \times \cB_r$ which we continue to denote as $f^\infty_{K, r}$. Using the construction \eqref{lip:extension} for each coordinate function of $\Phi_0^{(m)}$, we construct $$\Phi_{0, K, r}^{\infty}:\M\times\sP_2(\M)\to\R^d.$$

Similarly, assume we are given a family of functions $\Phi_1^{(m)}$ defined on 
\[
 \M \times  \bigg\{\Big(q_i, {1\over m} \sum_{j=1}^m \delta_{q_j}\Big)\; : \; q \in (\M)^m\bigg\}.
\]
As a consequence of the assumptions and Lemma \ref{lem:reg_finite_infinite}(iii)-(b) we assume for each $r>0$
and $K \in \mathcal K,$ 
\[
\Big|\Phi^{(m)}_1(q_0, q_1, \mu_q^{(m)})- \Phi^{(m)}_1( \qo_0, \qo_1, \mu_{\qo}^{(m)})\Big| \leq C(K, r) \Big( |q_0-\qo_0|+ |q_1-\qo_1|+ W_2(\mu_q^{(m)},  \mu_{\qo}^{(m)})+{1\over \sqrt m}\Big)
\]
for all $q_0, \qo_0 \in K$ and all $q, \qo \in \B_r^m.$ 

For each $k \in \{1, \cdots, d\}$,  $\Phi^{(m), k}_1$ and $q_0, q_* \in \M$, define 
\[
\Phi^{(m),k }_{1,K, r}(q_0, q_*, \mu):=  \inf_{\qo} \Big\{\Phi^{(m),k}_1(q_0, \qo_i, \mu_{\qo}^{(m)}))+C(K, r)\Big(|q_*-\qo_i|+W_2(\mu, \mu_{\qo}^{(m)}) \Big) \; : \; \qo \in \B_r^m  \Big\}
\]
Note 
\begin{equation}\label{eq:need-k-limit1}
\big|\Phi^{(m),k }_{1,K, r}\big(q_0,q_i,\mu_q^{(m)}\big)-\Phi^{(m),k}_1(q_0, q_i,\mu_q^{(m)}) \big| \le \frac{C}{\sqrt{m}},\quad \forall (q_0, q)\in K\times\B_r^m.
\end{equation}

As done earlier, there is a function 
$$ 
\Phi^{\infty,k }_{1,K, r}: \M \times \M \times \sP_2(\M) \rightarrow \R
$$ 
and a subsequence (which we may assume subsequence to be the same as the ones above) such that $\Big(\Phi^{(m),k }_{1,K, r}\Big)_m$ converges locally uniformly to $\Phi^{\infty,k }_{1,K, r}$ on compact sets. Increasing the value of $C(K', r)$ if necessary, we have 
\begin{equation}\label{eq:important2}
\Big|\Phi^{\infty }_{1,K, r}(q_0, q_1, \mu)-\Phi^{\infty }_{1,K, r}(\qo_0, \qo_1, \nu)\Big| \leq C(K', r)\Big(|q_0-\qo_0|+ |q_1-\qo_1|+ W_2(\mu, \nu) \Big) 
\end{equation}
if $q_0, q_1, \qo_0, \qo_1 \in K'$ and $\mu, \nu \in \cB_r.$

Let $q_0, \qo_0\in\M$ and let $K\subset\M$ be the closure of a bounded open set containing the line segment $[q_0,\qo_0]$. Let furthermore $q,\qo\in\B_r^m$. By the regularity assumptions  on $u^{(m)}$ one can write the following Taylor expansion 
\begin{align*}
u^{(m)}(\qo_0,  \qo )&-u^{(m)}(q_0, q)-D_{q_0} u^{(m)}(q_0,q)\cdot(\qo_0-q_0)-\sum_{i=1}^m D_{q_i} u^{(m)}(q_0,q)\cdot(\qo_i-q_i)\\
&=\frac{1}{2}(\qo_0-q_0)\cdot D^2_{q_0q_0}u^{(m)}(z_0,z)(\qo_0-q_0)+\sum_{i=1}^m(\qo_i-q_i)\cdot D^2_{q_iq_0}u^{(m)}(z_0,z)(\qo_0-q_0)\\
&+\frac12\sum_{i=1}^m(\qo_i-q_i)D^2_{q_iq_i} u^{(m)}(z_0,z)(\qo_i-q_i)+\frac12\sum_{i\neq j=1}^m(\qo_j-q_j)D^2_{q_iq_j} u^{(m)}(z_0,z)(\qo_i-q_i),
\end{align*}
where $(z_0,z)\in\M\times(\M)^m$ is a point on the line segment connecting $(q_0,q)$ to $(\qo_0,\qo)$. If 
$q,\qo \in\B_r^m$, by convexity, we also have that $z\in\B_r^m$. Now, using the uniform bounds on $D^2_{q_iq_j}u^{(m)}$ from the assumptions of this theorem, increasing the value of $C=C(K,r)>0$ if necessary, we have 
\begin{align}\label{eq:taylor-altern}
\Big{|}u^{(m)}(\qo_0,\overline y)&-u^{(m)}(q_0,q)-D_{q_0} u^{(m)}(q_0,q)\cdot(\qo_0-q_0)-\sum_{i=1}^m D_{q_i} u^{(m)}(q_0,q)\cdot(\qo_i-q_i)\Big{|}\\
\nonumber&\le C|\qo_0-q_0|^2 +C|\qo_0-q_0|\sum_{i=1}^m\frac{1}{\sqrt{m}}|\qo_i-q_i|\sqrt{m}|D_{q_iq_0}^2 u^{(m)}|\\
\nonumber&+\frac{C}{2m}\sum_{i=1}^m|\qo_i-q_i|^2+\frac{C}{2}\left(\sum_{j=1}^m\frac{1}{m}|\qo_j-q_j|^2\right)^{\frac12}\left(\sum_{i=1}^m\frac{1}{m}|\qo_i-q_i|^2\right)^{\frac12}\\
\nonumber&\le C\left(|q_0-\qo_0|^2+W_2^2(\mu_q^{(m)},\mu_{\qo}^{(m)})\right),
\end{align}
where in the last inequality we have used a Cauchy-Schwarz and a Young inequality, i.e.
\begin{align*}
|\qo_0-q_0|\sum_{i=1}^m\frac{1}{\sqrt{m}}|\qo_i-q_i|\sqrt{m}|D_{q_iq_0}^2 u^{(m)}|&\le|\qo_0-q_0|\left(\sum_{i=1}^m\frac{1}{m}|\qo_i-q_i|^2\right)^{\frac12}\left(m|D_{q_iq_0}^2 u^{(m)}|^2\right)^\frac{1}{2}\\
&\le\frac12|\qo_0-q_0|^2+\frac{C}{2}\sum_{i=1}^m\frac{1}{m}|\qo_i-q_i|^2
\end{align*}
Now, using the previous constructions, the first line in the chain of inequalities \eqref{eq:taylor-altern} can be rewritten as
\begin{align}\label{eq:identification}
&u^{(m)}(\qo_0,\qo)-u^{(m)}(q_0,q)-D_{q_0} u^{(m)}(q_0,q)\cdot(\qo_0-q_0)-\sum_{i=1}^m D_{q_i} u^{(m)}(q_0,q)\cdot(\qo_i-q_i)\nonumber \\
&=u(\qo_0,\mu_{\qo}^{(m+1)})-u(q_0,\mu_q^{(m+1)})-\Phi_0^{(m)}(q_0,\mu_q^{(m)})\cdot(\qo_0-q_0))\nonumber \\
&-\int_{\M^2}\Phi_1^{(m)}(q_0,q,\mu_q^{(m)})\cdot(\qo-q)\gamma^{(m)}(dq,d\qo),
\end{align}
where $(q_i)_{i=1}^m$ and $(\qo_i)_{i=1}^m$ are ordered in a way that 
$$
W_2^2(\mu_q^{(m)},\mu_{\qo}^{(m)})=\frac{1}{m}\sum_{i=1}^m|q_i-\qo_i|^2\quad \text{ and} \quad \gamma^{(m)}:=\frac{1}{m}\sum_{i=1}^m\delta_{(q_i,\qo_i)}\in\Gamma_o(\mu^{(m)}_q,\mu_{\qo}^{(m)}).
$$ 
In what follows, we pass to the limit all the terms in the previous line, keeping in mind that only the integral term needs some additional effort. We have
\begin{align}
&\int_{\M^2}\Phi_1^{(m)}(q_0, e,\mu_q^{(m)})\cdot(\overline e-e)\gamma^{(m)}(de,d\overline e)\nonumber\\
=&\int_{\M^2}\Phi^{(m)}_{1, K, r}(q_0, e,\mu_q^{(m)})\cdot(\overline e-e)\gamma^{(m)}(de,d\overline e)\nonumber\\
+&\int_{\M^2}\left(\Phi_1^{(m)}(q_0,e,\mu_q^{(m)})-\Phi^{(m)}_{1, K, r}(q_0,e,\mu_q^{(m)})\right)\cdot(\overline e-e)\gamma^{(m)}(de,d\overline e)\label{eq:int-Phi1-m1}
\end{align}
Let us observe that
\begin{align}
&\Bigg{|}\int_{\M^2}\left(\Phi_1^{(m)}(q_0,e,\mu_q^{(m)})-\Phi^{(m)}_{1, K, r}(q_0,e,\mu_q^{(m)})\right)\cdot(\overline e-e)\gamma^{(m)}(de,d\overline e)\Bigg{|}\nonumber\\
&\le\frac{C}{\sqrt{m}}\int_{\M^2} |e-\overline e|\gamma^{(m)}(de,d\overline e)\le\frac{2rC}{\sqrt{m}}\label{eq:int-Phi1-m2}.
\end{align}

The next step in our argument to pass to the limit in the remaining integral in the {first line of \eqref{eq:int-Phi1-m2}} works as follows. Fix a compact set $K \subset \M$, $R>0$ $q_0\in K$ and let $\mu,\nu\in\sP(\ov B_R(0))$ and $\g\in\Gamma_o(\mu,\nu)$. Let moreover $x, y\in\bH$ be such that  $\sharp(x,y)=\g$, which implies $\sharp(x)=\mu$, $\sharp(y)=\nu$. For $m\in \N$, recall $(\Omega_j^m)_{j=1}^m$ is the partition of introduced in Section \ref{sec:preliminaries}. Let us notice that for a.e. $\omega\in\Om$, $(x(\omega), y(\omega))\in\spt(\g).$ Let $(\omega)_{i=1}^m$ be Lebesgue points of $(x,y)$ such that $\omega_i\in\Om_i$ for all $i\in\{1,\dots,m\}$. Let us define 
$$ 
q_i:=x(\omega_i),\; \qo_i:=y(\omega_i), \; q:=(q_1,\dots,q_m), \qo:=(\qo_1,\dots,\qo_m)\in\B_r^m, \quad i\in\{1,\dots, m\}.
$$ 
We will assume we have chosen the Lebesgue points such that $M^q_{m}\to x$, $M^{\qo}_{m}\to y$ as $m\to+\infty$, strongly in $\bH$. We have that $\{(q_i,\qo_i)\}_{i=1}^m$ is contained in $\spt(\g)$ and so, it is cyclical monotone.  This implies that if we define $\g^{(m)}:=1/m\sum_{i=1}^m\d_{(q_i,\qo_i)}$ then monotonicity of the set of these points, one has that 
$$\g^{(m)}\in\Gamma_o(\mu^{(m)}_q,\mu^{(m)}_{\qo}).$$
Let us underline that in our construction it is very important that $\g^{(m)}$ is an optimal plan and a necessary and sufficient condition for this is the cyclical monotonicity of its support (cf. \cite{Santambrogio,Villani}).

Furthermore, as the supports of the measure involved are contained in the compact set $\overline B_R(0)$, we have the following narrow convergence 
$$\g^{(m)}\weakly \g,\ m\to+\infty,\quad \lim_{m\rightarrow \infty }W_2(\mu^{(m)}_q,\mu)=\lim_{m\rightarrow \infty } W_2(\mu^{(m)}_{\qo},\nu)= 0.$$
As, 
$$\sharp(M^q_m)=\mu_q^{(m)}, \; \sharp(M^{\qo}_m)=\mu_{\qo}^{(m)}\;\; \text{and}\;\; \sharp(M^q_m,M^{\qo}_m)=\g^{(m)},$$ 
we have in particular 
$$W_2^2(\mu^{(m)}_q,\mu^{(m)}_{\qo})=\sum_{i=1}^m\frac{1}{m}|q_i- \qo_i|^2=\|M^q_m-M^{\qo}_m\|^2.$$
By the uniform Lipschitz property of $\Phi^{(m)}_{1, K, r}$, we have  
 $$\lim_{m\rightarrow \infty }\Phi^{(m)}_{1, K, r}(q_0,M^q_m(\omega),\mu^{(m)}_q)=\Phi^\infty_{1, K, r}(q_0,x(\omega),\mu)$$ 
 and
 $$\lim_{m\rightarrow \infty } \Phi^{(m)}_{1, K, r}(q_0,M^{\qo}_m(\omega),\mu^{(m)}_{\qo}) =\Phi^\infty_{1, K, r}(q_0,y(\omega),\nu),$$
for a.e. $\omega$ in $\Om$. Also, since for a.e. $\omega\in\Om$, \eqref{eq:need-k-limit1} implies 
$$ \Phi^{(m)}_{1, K, r}(q_0,M^q_m(\omega),\mu^{(m)}_q)=m D_{q_i}u^{(m)}(q_0, q)+O(1/\sqrt{m}),$$ for some $i\in\{1,\dots, m\}$, by the assumption Property \ref{def:app_reg_estim}(1)(a), we have that $\left(\Phi^{m}_{1, K, r}(q_0,M^q_m(\cdot),\mu^{(m)}_q)\right)_m$ is a uniformly bounded sequence. Therefore, using all these facts, Lebesgue's dominated convergence theorem yields that up to passing to a suitable subsequence, that we do not relabel, we obtain
$$
\lim_{m\rightarrow \infty}\big\|\Phi^{(m)}_{1, K, r}(q_0,M^q_m,\mu^{(m)}_q)-\Phi^{\infty}_{1, K, r}(q_0,x,\mu)\big\|= \lim_{m\rightarrow \infty}\big\|\Phi^{(m)}_{1, K, r}(q_0,M^{\qo}_m,\mu^{(m)}_{\qo})-\Phi^{\infty}_{1, K, r}(q_0,y,\nu)\big\|=0.
$$
Now, using a suitable subsequence that we do not relabel, we conclude 
\begin{align*}
&\lim_{m\rightarrow \infty}\int_{\M^2} \Phi^{(m)}_{1, K, r}(q_0,q,\mu_q^{(m)})\cdot(\overline e-e)\g^{(m)}(de, d\overline e)\\
=& \lim_{m\rightarrow \infty}\int_\Om \Phi^{(m)}_{1, K, r} (q_0,M^q_m(\omega),\mu_q^{(m)})\cdot(M^{\qo}_m(\omega)-M^q_m(\omega))d\omega\\
= &\int_\Om\Phi_{1,K,r}^\infty(q_0,x(\omega),\mu)\cdot(y(\omega)-x(\omega))\dd\omega\\
=& \int_{\M^2}\Phi_{1,K,r}^\infty(q_0,e,\mu)\cdot(\overline e-e)\g(de,d \overline e).
\end{align*}
We combine \eqref{eq:taylor-altern} and \eqref{eq:identification}   to obtain
\begin{align*}
&\Big{|}u(\qo_0,\nu)-u(q_0,\mu)-\Phi^{\infty}_{0, K, r}(q_0,\mu)\cdot(\qo_0-q_0)-\int_{\M^2}\Phi^{\infty}_{1, K, r}(q_0,e,\mu)\cdot(\overline e-e)\g(de,d\overline e)\Big{|}\\
\le &C(K, r)\left(|q_0-\qo_0|^2+W_2^2(\mu,\nu)\right).
\end{align*}
We underline that the previous inequality has only been established under the condition that $\mu,\nu\in\cB_r$ have compact support. Since $u$ is continuous, we combine \eqref{eq:important1} and \eqref{eq:important2} to conclude 
\begin{align}\label{ineq:egy}
&\Big{|}u(\qo_0,\nu)-u(q_0,\mu)-\Phi^{\infty}_{0, K, r}(q_0,\mu)\cdot(\qo_0-q_0)-\int_{\M^2}\Phi^{\infty}_{1, K, r}(q_0,e,\mu)\cdot(\overline e-e)\g(de,d\overline e)\Big{|}\nonumber\\
\leq &C(K, r)\left(|q_0-\qo_0|^2+W_2^2(\mu,\nu)\right)
\end{align} 
for any $q_0, \qo_0 \in K$ and $\mu,\nu\in\cB_r.$

Note that in \eqref{ineq:egy}, $\Phi^{\infty}_{0, K, r}$ and $\Phi^{\infty}_{1, K, r}$ depend a priori on $K$ and $r.$ However since $K$ and $r$ are arbitrary, $u$ is differentiable at every $(q_0, \mu) \in \M \times \sP_2(\M).$ We have that $\Phi^{\infty}_{0, K, r}(q_0,\mu)$ must coincide with $D_{q_0} u(q_0, \mu)$ which is uniquely determined and so, it is independent of $K$ and $r$. Furthermore, the Wasserstein sub- and super-differentials of $u(q_0, \cdot)$ at $\mu$ coincide and contain a unique element of minimal norm $\nabla_w u(q_0, \mu)$. We do not know that $\Phi^{\infty}_{1, K, r}(q_0, \cdot,\mu)$ equals to $\nabla_w u(q_0, \mu)(\cdot)$,  however,  for $\gamma \in \Gamma_o(\mu, \nu)$, \eqref{ineq:egy} implies 
\begin{align}\label{ineq:egybis}
&\Big{|}u(\qo_0,\nu)-u(q_0,\mu)- D_{q_0} u(q_0, \mu) \cdot(\qo_0-q_0)-\int_{\M^2}\nabla_w u(q_0, \mu)(e)\cdot(\overline e-e)\g(de,d\overline e)\Big{|}\nonumber\\
\leq &C(K, r)\left(|q_0-\qo_0|^2+W_2^2(\mu,\nu)\right)
\end{align} 
for any $q_0, \qo_0 \in K$ and $\mu,\nu\in\cB_r.$ In fact $\nabla_w u(q_0, \mu)$ is the projection of $\Phi^{\infty}_{1, K, r}(q_0, \cdot,\mu)$ onto $T_\mu\cP_2(\R^d)$.
\end{proof}

Using the exact same steps as in the proof of Theorem \ref{thm:finite_infinite_reg-scalar-mast}, we can show an analogous result for functions depending on time as well. We formulate this in the following 

\begin{corollary}\label{cor:finite_infinite_reg-scalar-mast-time} 
Let $u:(0,+\infty)\times\M\times\sP_2(\M)\to\R$ be a continuous function. For $m\in\N$, we define $u^{(m)}:(0,+\infty)\times\M\times(\M)^m\to\R$ as
$$u^{(m)}(t_0,q_0,q):=u(t_0,q_0,\mu^{(m+1)}_q),$$
where $(q_0,q)=(q_0,q_1,\dots,q_m)\in (\M)^{m+1}$ and $\mu_q^{(m+1)}=\frac{1}{m+1}\sum_{i=0}^m\d_{q_i}$. 
Suppose that $u^{(m)} \in C^{1,1}_{\rm{loc}}((0,+\infty)\times\M\times(\M)^m)$ and that for $I\subset(0,+\infty)$ and $K\subset\M$ compacts and $r>0$, $u^{(m)}(t_0,q_0,\cdot)$ satisfies the estimates of Property \ref{def:app_reg_estim}(1)-(a) and (2) for all $(t_0,q_0)\in I\times K$, with a constant $C=C(I,K,r)>0$.We assume moreover that for any $I\subset(0,+\infty)$ and $K\subset\M$ compacts and $r>0$, there exists $C=C(I,K,r)>0$ such that 
\begin{align}\label{def:q_0-constant1}
&|D_{q_0}u^{(m)}(t_0,q_0,q)|\le C,\ \ |D^2_{q_0q_0}u^{(m)}(t_0,q_0,q)|_\infty\le C,\ \sum_{i=1}^m m|D^2_{q_iq_0}u^{(m)}(t_0,q_0,q)|^2_\infty\le C\\
\nonumber&|D^2_{q_iq_j} u^{(m)}(t_0,q_0,q)|_\infty \le \left\{
\begin{array}{ll}
\ds\frac{C}{m}, & i=j,\ {\rm{and}}\ i>0,\\[5pt]
\ds\frac{C}{m^2}, & i\neq j,\ i,j>0,
\end{array}
\right.
\end{align}
and
\begin{align}\label{def:q_0-constant2}
&|\partial_{t_0}u^{(m)}(t_0,q_0,q)|\le C,\ \ |\partial^2_{t_0t_0}u^{(m)}(t_0,q_0,q)|\le C,\ \ \ |\partial_{t_0}D_{q_0}u^{(m)}(t_0,q_0,q)|\le C,\\
\nonumber&\sum_{i=1}^m m|D_{q_i}\partial_{t_0}u^{(m)}(t_0,q_0,q)|^2\le C
\end{align}
for any $(t_0,q_0)\in I\times K$ and $q=(q_1,\dots,q_m)\in\B_r^m$.

Then, there exists $\Phi_1:(0,+\infty)\times\M\times\cP_2(\M)\times\M\to\R^d$ locally Lipschitz continuous function such that for any $r>0$ and $I\subset(0,+\infty)$ and $K\subset\M$ compacts, there exists $C=C(I,K,r)>0$ such that for any $s_0,t_0\in I$, $q_0,y_0\in K$, any $\mu,\nu\in\sP_2(\M)$ and $\g\in\Gamma_o(\mu,\nu)$, $u$ satisfies 
\begin{align*} 
 \Bigg{|}u(s_0,y_0,\nu)-u(t_0,q_0,\mu)&-D_{q_0}u(t_0,q_0,\mu)\cdot(y_0-q_0) -\partial_{t_0}u(t_0,q_0,\mu)(s_0-t_0)\\
 &-\int_{\M^2}\Phi_1(t_0,q_0,\mu,q)\cdot(y-q)\dd\g(q,y)\Bigg{|}\\
& \le C\left(|s_0-t_0|^2+|q_0-y_0|^2+W_2^2(\mu,\nu)\right).
\end{align*}
This implies in particular that $u\in C^{1,1}_{\rm{loc}}((0,+\infty)\times\M\times\sP_2(\M))$ and $\nabla_w u(t_0,q_0,\mu)(\cdot)$ is the projection of $\Phi_1(t_0,q_0,\mu,\cdot)$ onto $T_\mu\cP_2(\M)$ and
\begin{align*} 
 \Bigg{|}u(s_0,y_0,\nu)-u(t_0,q_0,\mu)&-D_{q_0}u(t_0,q_0,\mu)\cdot(y_0-q_0)\\
&-\partial_{t_0}u(t_0,q_0,\mu)(s_0-t_0)-\int_{\M^2}\nabla_w u(t_0,q_0,\mu)(q)\cdot(y-q)\dd\g(q,y)\Bigg{|}\\
& \le C\left(|s_0-t_0|^2+|q_0-y_0|^2+W_2^2(\mu,\nu)\right).
\end{align*}
\end{corollary}

\begin{theorem}\label{thm:finite_infinite_reg} 
Let $\cU\in C^{1,1}_{\rm{loc}}(\sP_2(\M))$. Let $U^{(m)}:(\M)^m\to\R$ be defined as $U^{(m)}(q):=\cU(\mu_q^{(m)})$ for $q \in \M^m$, such that Property \ref{def:app_reg_estim}(2--3) are satisfied. Then $\cU\in C^{2,1,w}_{\rm{loc}}(\sP_2(\M))$  in the sense of Definition \ref{def:c21_wasserstein}, such that the following hold. There exist  $C:(0,\infty) \to (0,\infty)$ monotone nondecreasing and 
\begin{enumerate}
\item[(i)]there are continuous maps 
$$ 
\Lambda_0:\M\times\sP_2(\M)\to\R^{d\times d} \quad \text{and} \quad \Lambda_1:\M\times \M\times\sP_2(\M)\to\R^{d\times d}
$$  such that  for $\mu\in\sP_2(\M)$ we have 
$$
\sup_{\mu \in \cB_r}\|\Lambda_0(\cdot,\mu)\|_{L^\infty(\mu)}, \quad 
\sup_{\mu \in \cB_r}\|\Lambda_1(\cdot, \cdot,\mu)\|_{L^\infty(\mu\otimes \mu)} \leq C(r).
$$ 
\item[(ii)] Let $\mu,\nu\in\cB_r$ and $\g\in\Gamma_o(\mu,\nu)$. We have  
\small
\begin{equation}\label{ineq:forHessian}
\Big{|}\nabla_w\cU(\nu)(\qo)-\nabla_w\cU(\mu)(q)-\Lambda_0(q,\mu)(\qo-q)-\int_{\M^2}\Lambda_1(q,a,\mu)(b-a)\dd\g(a,b)\Big{|}\le C\left(|q-\qo|^2+W_2^2(\mu,\nu)\right)
\end{equation}
and 
\begin{align}\label{ineq:Lip-grad}
\big{|}\nabla_w\cU(\mu)(q)-\nabla_w\cU(\nu)(\qo)\big{|}\le C\left(|q-\qo|+W_2(\mu,\nu)\right),\ \forall\mu,\nu\in\cB_r,\ .
\end{align}
for all $(q,\qo)\in\spt(\mu)\times\spt(\nu)$. 
\end{enumerate}
\end{theorem} 

\begin{proof}
We follow ideas similar to those presented in the proof of Theorem \ref{thm:finite_infinite_reg-scalar-mast}. Recall that for, $q\in \B^m_r$ we use the notation $\mu^{(m)}_q:=1/m\sum_{i=1}^m\d_{q_i}$ and use a similar notation for $\qo\in \B^m_r.$ Let us define the matrix valued functions 
$$\Lambda_{0}^{(m)}:\bigcup_{q\in\B_r^m}\spt(\mu^{(m)}_q)\times \{\mu^{(m)}_q\}\to\R^{d\times d}$$ 
and 
$$\Lambda_{1}^{(m)}:\bigcup_{q\in\B_r^m}\left(\big{(}\spt(\mu^{(m)}_q)\times\spt(\mu^{(m)}_q)\big{)}  \setminus \{(q_i,q_i)\, : \, i=1, \cdots, m\}\right)\times \{\mu^{(m)}_q\}\to\R^{d\times d}$$ 
as
$$\Lambda_0^{(m)}(q_i,\mu^{(m)}_q):=m D^2_{q_iq_i} U^{(m)}(q),\ \ {\rm{and}}\ \ \Lambda_{1}^{(m)}(q_i,q_j,\mu^{(m)}_q):=m^2 D^2_{q_iq_j} U^{(m)}(q),\ {\rm{if}}\ i\neq j.$$
Let us underline that we have not defined $\Lambda_{1}^{(m)}(q_i,q_i,\mu_q^{(m)})$  for $i=j.$ Because of this, later we will need special care when one passes to the limit the corresponding objects as $m\to+\infty$.

We observe that as a consequence of the assumptions and Lemma \ref{lem:reg_finite_infinite}(v)-(b,d), we have that for any $r>0$, there exists a constant $C=C(r)>0$ such that
$$|\Lambda_0^{(m)}(q_i,\mu^{(m)}_q)-\Lambda_0^{(m)}(\qo_j,\mu^{(m)}_{\qo})|\le C\left(|q_i-\qo_j|+W_2(\mu_q^{(m)},\mu_{\qo}^{(m)})+\frac{1}{\sqrt m}\right)$$
and
$$|\Lambda_1^{(m)}(q_i,q_k,\mu^{(m)}_q)-\Lambda_1^{(m)}(\qo_j,\qo_l,\mu^{(m)}_{\qo})|\le C\left(|q_i-\qo_j|+|q_k-\qo_l|+W_2(\mu_q^{(m)},\mu_{\qo}^{(m)})+\frac{1}{\sqrt m}\right)$$
for any $q, \qo\in\B_r^m$, and for any $i,j,k,l\in\{1,\dots,m\}$, $i\neq k$, $j\neq l$. For every coordinate function $(\Lambda_0^{(m)})_{\a\b}, (\Lambda_1^{(m)})_{\a\b}$ ($\a,\b\in\{1,\dots,d\}$), we define the extensions 
$$\Big(\Lambda_{0, r}^{(m)}\Big)_{\a\b}:\M\times\sP_2(\M)\to\R\quad \text{and} \quad \Big(\Lambda_{1,r}^{(m)}\Big)_{\a\b}:\M\times\M\times\sP_2(\M)\to\R$$ 
as follows. For $z,z_1,z_2\in\M$, $\mu\in\sP_2(\M)$  we set 
$$\Big( \Lambda_{0, r}^{(m)}\Big)_{\a\b}(z,\mu):=\inf\left\{ (\Lambda_{0}^{(m)})_{\a\b}(q_i,\mu_q^{(m)})+C\left(|q_i-z|+W_2(\mu_q^{(m)},\mu\right) \right\}$$
and
$$\Big(\Lambda_{1, r}^{(m)}\Big)_{\a\b}(z_1,z_2,\mu):=\inf\left\{ (\Lambda_{1}^{(m)})_{\a\b}(q_i,q_k,\mu_q^{(m)})+C\left(|q_i-z_1|+|q_k-z_2|+W_2(\mu_q^{(m)},\mu\right)\right\},$$
where both infima is taken over $q\in\B_r^m,\ i,k\in\{1,\dots, m\}, i\neq k$. 

Recall  $\Lambda_{0, r}^{(m)}$ and $\Lambda_{1, r}^{(m)}$ are $C(r)$--Lipschitz and we have
\begin{align}\label{eq:close1}
|\Lambda_{0, r}^{(m)}(q_i,\mu_q^{(m)})- \Lambda_{0}^{(m)}(q_i,\mu_q^{(m)})|_\infty\le \frac{C}{\sqrt{m}},\ \forall q\in \B_r^m, i\in\{1,\dots,m\}
\end{align}
and 
\begin{align}\label{eq:close2}
|\Lambda_{1, r}^{(m)}(q_i,q_k,\mu_q^{(m)})- \Lambda_{1}^{(m)}(q_i,q_k,\mu_q^{(m)})|_\infty\le \frac{C}{\sqrt{m}},\ \forall q\in \B_r^m, i,k\in\{1,\dots,m\}, i\neq k.
\end{align}
If $R>0$, $z_1, z_2 \in B_R(0)$ and $\mu$ is supported by $B_R(0)$ then  for all $\a,\b\in\{1,\dots, d\}$  
\[
-C \leq \left(\Lambda_{1, r}^{(m)}\right)_{\a\b}(z_1, z_2, \mu) \leq C+C \Big(|z_1|+|z_2|+W_2(0, \mu) \Big) \leq C(3R).
\]
We obtain a similar uniform bound on $\left( \Lambda_{0, r}^{(m)}\right)_m$. As in the proof of Theorem \ref{thm:finite_infinite_reg-scalar-mast}, there are $C$--Lipschitz functions 
$$
\Lambda_{0, r}:\M\times\sP_2(\M)\to\R^{d\times d}, \quad \Lambda_{1,r}:\M\times\M\times\sP_2(\M)\to\R^{d\times d}
$$ 
locally bounded respectively on $\M \times \sP_2(\M)$ and $\M^2 \times \sP_2(\M)$ by a constant depending only on $r$ and $R$. Up to a subsequence, as $m\to+\infty$, $\left( \Lambda_{0, r}^{(m)}\right)_m$ and $\left( \Lambda_{1,r}^{(m)}\right)_m$ converge to $\Lambda_{0, r}$ and $\Lambda_{1, r}$, uniformly on $\ov B_R(0) \times \sP(\ov B_R(0))$ and $\ov B_R(0)\times \ov B_R(0)\times \sP(\ov B_R(0))$, respectively.

Our next task is to show that 
$$\Lambda_{0, r}(\cdot,\mu)\in L^\infty(\M;\mu), \quad \Lambda_{1, r}(\cdot,\cdot,\mu)\in L^\infty(\M\times\M;\mu\otimes\mu), \qquad \forall \mu\in\cB_r\cap\sP(\ov B_R(0)).$$

{\it Claim 1.}  $\Lambda_{1,r}(\cdot,\cdot,\mu)\in L^\infty(\M^2;\mu\otimes \mu).$

{\it Proof of Claim 1.} Let $r>0$, $R>0$ and first let $\mu\in\cB_R\cap\sP(\ov B_R(0))$. Let $z_1,z_2 \in B_R(0).$ As we plan to let $m$ tend to $\infty$ it is not a loss of generality to assume $R \leq r \sqrt m.$ Since $q=(z_1, z_2, 0, \cdots, 0) \in \B^m_r$ we have 
\[
-C \leq \big(\Lambda^{(m)}_{1, r}\big)_{\a \b}(z_1, z_2, \mu) \leq (\Lambda^{(m)}_1)_{\alpha\beta}(z_1,z_2, \mu_q^{(m)}) + C(r) \Big(|z_1-z_1|+|z_2-z_2|+W_2\big(\mu_q^{(m)}, \mu \big) \Big) \leq C(r) + 2r C(r)
\]
Letting $m$ tend to $\infty$ we conclude $\big|\big(\Lambda_{1, r}\big)_{\a \b}(z_1, z_2, \mu)\big|  \leq C(r) + 2r C(r)$ first on $\M^2 \times \sP_{\rm c}(\M)$ and by continuity, this holds on $\M^2 \times \sP_2(\M)$. 
\hfill\break

{\it Claim 2.}  $\Lambda_{0,r}(\cdot,\mu)\in L^\infty(\M;\mu).$

{\it Proof of Claim 2.} The proof is similar but simpler than that of Claim 1. 
\hfill\break

For $q,\qo\in\B_r^m$ we have the expansion 
\begin{align}\label{eq:tay-grad}
& m D_{q_1}U^{(m)}(\qo)-m D_{q_1}U^{(m)}(q)-mD^2_{q_1q_1}U^{(m)}q)(\qo_1-q_1)- m\sum_{k=2}^m D^2_{q_{1}q_k}U^{(m)}(q)(\qo_k-q_k)\\
\nonumber&=\frac{m}{2}\sum_{k,l=1}^m (\qo_l-q_l) D^3_{q_1q_kq_l} U^{(m)}( z)(\qo_k-q_k)
\end{align}
where $ z$ is a point on the line segment connecting $q$ to $\qo$.  

Let $\mu,\nu\in\cB_r$, $\g\in\Gamma_o(\mu,\nu)$ and let $(q_1, \qo_1)\in\spt(\mu)\times\spt(\nu)$ (which is not necessarily in $\spt(\g)$). Suppose that both $\spt(\mu)$ and $\spt(\nu)$ contain more than one element. We choose $x, y \in \bH$ such that $\sharp(x, y)=\g$ and so, $\sharp(x)=\mu$, $\sharp(y)=\nu$. Let $(\Om_i^{m-1})_{i=1}^{m-1}$ be the partition of $\Om$ introduced in Section \ref{sec:preliminaries}.
We are going to choose special values of $m:=2^l+1$ and choose Lebesgue points $\omega_{i+1} \in \Om^{2^l}_i$ such that all the points in $\Om^{2^l}_i$ are kept in $\Om^{2^{l+1}}_i$. We set 
$q_i:=x(\omega_i), \quad \qo_i:=y(\omega_i)$ for $i=2, \cdots, m$
Set 
\[
\g^{(m-1)}:=\frac{1}{m-1}\sum_{i=2}^m\d_{(q_i,\qo_i)}, \qquad \mu_q^{(m-1)}:= \frac{1}{m-1}\sum_{i=2}^m\d_{q_i}, \qquad \mu_{\qo}^{(m-1)}:= \frac{1}{m-1}\sum_{i=2}^m\d_{\qo_i}.
\]
Since, $(q_i,\qo_i)_{i=2}^\infty$ is cyclically monotone 
\[
\g^{(m-1)} \in \Gamma_o\big(\mu_q^{(m-1)}, \mu_{\qo}^{(m-1)} \big).
\]
By construction $\big(\g^{(m-1)}\big)_m$ converges narrowly to $\g$. Let $M^q_{(m-1)}, M^{\qo}_{(m-1)}\in\bH$ the random variables corresponding to the previously chosen points $(q_2,\dots,q_m)$ and $(\qo_2,\dots,\qo_m)$, respectively. We have 
\begin{equation}\label{eq:nov28.2019.1} 
\lim_{m\to+\infty}W_2(\mu_q^{(m)},\mu)=\lim_{m\to+\infty}W_2(\mu_q^{(m-1)},\mu)=\lim_{m\to+\infty}W_2(\mu_{\qo}^{(m)},\nu)=\lim_{m\to+\infty}W_2(\mu_{\qo}^{(m-1)},\nu)=0.
\end{equation}
Furthermore, 
$$
\sharp\Big(M^q_{(m-1)},M^{\qo}_{(m-1)}\Big)=\g^{(m-1)},
$$
and \
$$ 
\lim_{m\to+\infty} \big\|M^q_{(m-1)} -x\big\|=  \lim_{m\to+\infty} \big\|M^{\qo}_{(m-1)} -y\big\|=0.
$$ 
Using the assumptions on $D^3_{q_jq_kq_l}U^{(m)}$, since $z\in\B_r^m$, increasing the value of $C$ if necessary, we have 
\begin{align*}
&\Big{|}m\sum_{k,l=1}^m (y_l-x_l) D^3_{q_1q_kq_l} U^{(m)}(z)(\qo_k-q_k)\Big{|}\\
\le &m|D^3_{q_1q_1q_1} U^{(m)}(z)|_\infty|\qo_1-q_1|^2+m\sum_{k=2}^m |D^3_{q_1q_kq_1} U^{(m)}(z)|_\infty|\qo_k-q_k| |\qo_1-q_1|\\
+ &m\sum_{l=2}^m |D^3_{q_1q_1q_l} U^{(m)}(z)|_\infty|\qo_1-q_1| |\qo_l-q_l|\\
+ &m\sum_{k=2}^m |D^3_{q_1q_kq_k} U^{(m)}(z)|_\infty|\qo_k-q_k|^2+m\sum_{k\neq l=2}^m |\qo_l-q_l| |D^3_{q_1q_kq_l} U^{(m)}(z)|_\infty|\qo_k-q_k|\\
\le &C\left(|\qo_1-q_1|^2+|\qo_1-q_1|\sum_{k=2}^m\frac{1}{m}|\qo_k-q_k|+\sum_{k=2}^m\frac{1}{m}|\qo_k-q_k|^2+\frac{1}{m^2}\sum_{k\neq l=2}^m |\qo_l-q_l| |\qo_k-q_k|\right)\\
\le &C\left(|\qo_1-q_1|^2+W_2^2(\mu_q^{(m-1)},\mu_{\qo}^{(m-1)})\right)
\end{align*}
Thus, this together with \eqref{eq:tay-grad} implies 
\begin{align*} 
&m  \Big{|}D_{q_{1}}U^{(m)}(\qo)- D_{q_{1}}U^{(m)}(q)-D^2_{q_{1}q_1}U^{(m)}(q)(\qo_1-q_1)-\sum_{k=2}^m D^2_{q_{1}q_k}U^{(m)}(q)(\qo_k-q_k)\Big{|}\\
\le & C\left(|\qo_1-q_1|^2+W_2^2(\mu_q^{(m-1)},\mu_{\qo}^{(m-1)})\right). 
\end{align*}
Using the definition of $\Lambda_0^{(m)}$ and $\Lambda_1^{(m)}$ we read off 
\begin{align}\label{eq:tay-grad-estim-3}
& \Big{|}\nabla_w\cU(\mu_\qo^{(m)})(\qo_1)-\nabla_w\cU(\mu_q^{(m)})(q_1)- \Lambda_0^{(m)}(q_1,\mu_q^{(m)})(\qo_1-q_1) \nonumber\\
& \qquad \qquad - {m-1 \over m}\int_{\M^2}\Lambda_1^{(m)}(q_1,a,\mu_q^{(m)})(b-a)\g^{(m-1)}(da,db)\Big{|}\\
\nonumber&\le C\left(|\qo_j-q_i|^2+W_2^2(\mu_q^{(m-1)},\mu_{\qo}^{(m-1)})\right),
\end{align}
Now, first by the continuity of $\nabla_w\cU$, \eqref{eq:nov28.2019.1} implies
$$\lim_{m\rightarrow \infty }\nabla_w\cU(\mu_q^{(m)})(q_1)= \nabla_w\cU(\mu)(q_1), \quad \text{and} \quad \lim_{m\rightarrow \infty }\nabla_w\cU(\mu_{\qo}^{(m)})(\qo_1)=\nabla_w\cU(\nu)(\qo_1).$$ 
Before passing to the limit in the other terms, let us {\it further} suppose that $\mu, \nu\in\sP(\ov B_R(0))$ for some $R>0$. In light of \eqref{eq:close1}, $\Lambda_0^{(m)}(q_1,\mu_q^{(m)})$ and $ \Lambda_{0,r}^{(m)}(q_1,\mu_q^{(m)})$ have the same limit. By the local uniform convergence property of $\Lambda_{0, r}^{(m)}$, we have that $\lim_{m\rightarrow \infty }\Lambda_0^{(m)}(q_1,\mu_q^{(m)})=\Lambda_{0, r}(q_1,\mu).$ 

To handle the limit in the last term on the left hand side of the inequality \eqref{eq:tay-grad-estim-3}, we observe that 
\begin{align*}
\int_{\M^2}\Lambda_1^{(m)}&(q_1,a,\mu_q^{(m)})(b-a)\g^{(m-1)}(da,db)=\int_{\M^2}\Lambda_{1, r}^{(m)}(q_1,a,\mu_q^{(m)})(b-a)\g^{(m-1)}(da,db)\\
&+\int_{\M^2}\left(\Lambda_1^{(m)}(q_1,a,\mu_q^{(m)})-\Lambda_{1, r}^{(m)}(q_1,a,\mu_q^{(m)})\right)(b-a)\g^{(m-1)}(da, db)
\end{align*}
and by \eqref{eq:close2}, increasing $C$ if necessary, we have that 
\begin{align*}
\Bigg{|}\int_{\M^2}\left(\Lambda_1^{(m)}(q_1,a,\mu_q^{(m)})-\Lambda_{1,r}^{(m)}(q_1,a,\mu_q^{(m)})\right)(b-a)\g^{(m-1)}(da,db)\Bigg{|}&\le\frac{C}{\sqrt{m}} \iint_{\M^2}|b-a|\g^{(m-1)}(da,db)\\
&\le\frac{C r}{\sqrt{m}}.
\end{align*}
Therefore, it is enough to study the limit of 
$$\int_{\M^2}\Lambda_{1, r}^{(m)}(q_1,a,\mu_q^{(m)})(b-a)\g^{(m-1)}(da, db).$$
Since 
$$ 
\Big|\Lambda_{1, r}^{(m)}(q_1,M^q_{(m-1)}(\omega),\mu_q^{(m)})-\Lambda_1^{(m)}(q_1,M^q_{(m-1)}(\omega),\mu_q^{(m)})\Big|\le \frac{C}{\sqrt{m}}
$$ 
and since $\Lambda_1^{(m)}(q_1,M^q_{(m-1)}(\omega),\mu_q^{(m)})=\Lambda_1^{(m)}(q_1,q_i,\mu_q^{(m)})$ for some $i\in\{2,\dots,m\}$ for a.e. $\omega\in\Om$, we have that 
$\omega\mapsto \Lambda_{1, r}^{(m)}(q_1,M^q_{(m-1)}(\omega),\mu_q^{(m)})$ is uniformly bounded with respect to $m\in\{2,3,\dots\}$. Thus by the previous convergences and by Lebesgue's dominated convergence theorem, up to passing to a subsequence that we do not relabel, we have that 
$$\lim_{m\to\infty}\Big\|\Lambda_{1, r}^{(m)}(q_1,M^q_{(m-1)},\mu_q^{(m)}) - \Lambda_1(q_1,x,\mu)\Big\|=0.$$
Thus, up to a subsequence, 
\begin{align*}
&\lim_{m\to\infty} \int_{\M^2}\Lambda_{1, r}^{(m)}(q_1,a,\mu_q^{(m)})(b-a)\g^{(m-1)}(da,bb)\\
=&\lim_{m\to\infty} \int_\Om  \Lambda_{1, r}^{(m)}\Big(q_1,M^q_{(m-1)}(\omega),\mu_q^{(m)}\Big)\Big(M^{\qo}_{(m-1)}(\omega)-M^q_{(m-1)}(\omega)\Big)\dd\omega\\
=&\int_\Om\Lambda_{1, r}(q_1,x(\omega),\mu)(y(\omega)-x(\omega))d\omega=\int_{\M^2}\Lambda_{1, r}(q_1,a,\mu)(b-a)\g(da,db)
\end{align*}
We have all the ingredients to conclude that up to subsequence  \eqref{eq:tay-grad-estim-3} implies
to obtain
\begin{align*}
&\Big{|}\nabla_w\cU(\nu)(\qo_1)-\nabla_w\cU(\mu)(q_1)-\Lambda_{0,r}(q_1,\mu)(\qo_1-q_1)-\int_{\M^2}\Lambda_{1, r}(q_1,a,\mu)(b-a)\g(da,db)\Big{|}\\
&\le C\left(|q_1-\qo_1|^2+W_2^2(\mu,\nu)\right).
\end{align*}
{As $C$ is independent of $R$, we extend the previous inequality to all $\mu,\nu\in\cB_r$ without imposing they lie in $\sP(B_R(0)).$}  We also notice that by the assumptions, i.e. Property \ref{def:app_reg_estim}(3), the map $q\mapsto\nabla_w\cU(\mu)(q)$  is Lipschitz continuous uniformly with respect to $\mu\in\cB_r$. More precisely, Lemma \ref{lem:reg_finite_infinite} (iii)-(b) yields that there exists $C=C(r)>0$ such that for all $\mu,\nu\in\cB_r$ and $(q_1, \qo_1)\in \spt(\mu)\times\spt(\nu)$ we have
\begin{align*}
|\nabla_w\cU(t,\mu)(q_1)-\nabla_w\cU(t,\nu)(\qo_1)|\le C(|q_1-\qo_1|+W_2(\mu,\nu)),
\end{align*}
so \eqref{ineq:Lip-grad} follows.
\end{proof}
\begin{remark}\label{rem:symmetricxx} Note that $\Lambda_0$ is a symmetric matrix, as limit of symmetric matrices.
\end{remark}
%

%
%
%
%
%

%
%
%
%
%
%
%
%

\section{Global well-posedness of master equations}\label{sec:master_equations}
 Throughout this section, we fix $T>0$ and impose \eqref{ass:F-U0-C11new1}-\eqref{ass:convexity-on-tildeL}. We further assume 
 \begin{equation}\label{eq:collectUF} 
 \cU_0,\cF \in C^{2,1,w}_{\rm{loc}}(\sP_2(\M))\quad \text{and} \quad  U_0^{(m)}, \; F^{(m)} \; \text{satisfy Property} \; \ref{def:app_reg_estim}(3).
\tag{H\arabic{hyp}} 
\end{equation}
\stepcounter{hyp}  
  
Let $\tilde \cU$ be the solution obtained in Proposition \ref{theorem:hilbert-smooth} and define $\cU:[0,T]\times\sP_2(\M)\to\R$ as $\cU(t, \mu):=\tilde \cU(t, x)$ where $\mu=\sharp(x).$  By Lemma \ref{lem:c11_equiv}, the regularity property obtained on $\tilde \cU$ in Proposition \ref{theorem:hilbert-smooth} ensures that $\cU(t, \cdot)$ is $C^{1,1}_{\rm{loc}}(\sP_2(\M)).$ We use Remark \ref{rem:factorization}  to obtain that $\cU \in C^{1,1}_{\rm{loc}}([0, T] \times \sP_2(\M))$ (in the sense of Definition \ref{def:c11_wasserstein}) and it is a classical solution to the Hamilton-Jacobi equation 
\begin{equation}\label{eq:HJB}
\left\{
\begin{array}{ll}
\partial_t\cU + \cH(\mu,\nabla_w\cU)=\cF(\mu),&  {\rm{in}}\ (0,T)\times\sP_2(\M),\\
\cU(0,\mu)=\cU_0(\mu), & {\rm{in}}\ \sP_2(\M).
\end{array}
\right.
\end{equation}

%
%
\subsection{The vectorial master equation} 
Let $\cV:\cP_2(\M)\times\M\to\R^d$ and define 
\[
\overline \cN_\mu\big[\cV, \nabla_{w}^\top \cV \big](t, \mu, q):=\int_\M\nabla_{w}^\top \cV(t, \mu,q)(b)  D_pH\big(b,  \cV(t, \mu, b)\big) \mu(db)
\]
We plan to obtain existence of $\cV:[0,T]\times\sP_2(\M)\times\M\to\R^d,$ solution to the so-called {\it vectorial master equation}
\begin{equation}\label{eq:master_vector}
\left\{
\begin{array}{r}
\ds\partial_t\cV+D_q H(q,\cV(t,\mu,q))+D_q\cV(t,\mu,q)\nabla_p H(q,\cV(t,\mu,q))+\overline \cN_\mu\big[\cV, \nabla_{w}^\top \cV \big](t, \mu, q)\\ 
\ds =\nabla_w\cF(\mu)(q)\\[5pt]
\ds\cV(0,\mu,\cdot)=\cV_0(\mu),
\end{array}
\right.
\end{equation}  
as a by--product of the regularity properties of the solution to \eqref{eq:HJB}. The lower order regularity results in the Hilbert setting are starting points to improve to higher order regularity results in the Wasserstein space. First, let us discuss about the existence and regularity of solutions of \eqref{eq:HJB}.

\begin{theorem}\label{thm:main-regularity} The equation \eqref{eq:HJB} has a unique classical solution $\cU\in C^{1,1}_{\rm{loc}}([0,T]\times\sP_2(\M))$ such that $\cU(t,\cdot)\in C^{2,1,w}_{\rm{loc}}(\cP_2(\M))$, which has to be understood in the sense of Definition \ref{def:c21_wasserstein}.
\end{theorem}

\begin{proof} 
First, we notice that Proposition \ref{theorem:hilbert-smooth} asserts existence and uniqueness of a solution $\cU\in C^{1,1}_{\rm{loc}}([0,T]\times\sP_2(\M))$. Then, Theorem \ref{thm:app_regularity-m} will imply that $U^{(m)}(t,q):=\cU(t,\mu_q^{(m)})$ for $t\in(0,T)$, $m\in\N$, $q\in (\M)^m$ satisfies the regularity estimates from Property \ref{def:app_reg_estim} in $\B_r^m(0)$ with constant $C(t,r)$. We apply Theorem \ref{thm:finite_infinite_reg} to infer $\cU(t,\cdot)$ is of class $C^{2,1,w}_{\rm{loc}}(\sP_2(\M))$. 
\end{proof}
 
 \begin{remark}
In this subsection we discuss existence of \emph{weak solutions} to \eqref{eq:master_vector}. The regularity of solutions $\cU$ to the Hamilton-Jacobi equation \eqref{eq:HJB} established in Theorem \ref{thm:main-regularity} are enough to differentiate this equation with respect to the measure variable. This procedure gives us a notion of weak solution to the vectorial master equation. Better regularity properties of this solution are subtle and we need additional effort to obtain these. We postpone this analysis to Subsection \ref{subsec:further_vectorial_eq}, where we point out a deep connection between the vectorial and the scalar master equations as well.
 \end{remark}
 
 \begin{definition}\label{def:master_weak}
 We say that $\cV:[0,T]\times \bigcup_{\mu\in\cP_2(\M)}\{\mu\}\times\spt(\mu)\to\R^d$ is a \emph{weak solution} to \eqref{eq:master_vector} if it is locally Lipschitz on its domain of definition, $\cV(\cdot,\mu,q)$ is differentiable on $(0,T)$ for all $\mu\in\cP_2(\M)$ and $q\in\spt(\mu)$, $\cV(t,\cdot,\cdot)\in C^{1,1}_{\rm{loc}}\left(\cup_{\mu\in\cP_2(\M)}\{\mu\}\times\spt(\mu)\right)$, $\cV(t,\mu,\cdot)$ is differentiable on $\spt(\mu)$ for all $t\in[0,T]$ and $\mu\in\cP_2(\M)$ and the equation \eqref{eq:master_vector} is satisfied pointwise on $[0,T]\times \bigcup_{\mu\in\cP_2(\M)}\{\mu\}\times\spt(\mu)$.
 \end{definition}

 \begin{theorem}\label{thm:main-vector}
Suppose $\cU(t,\cdot)\in C^{2,1,w}_{\rm{loc}}(\sP_2(\M))$ (in the sense of Definition \ref{def:c21_wasserstein}). Using the notation in Remark \ref{rem:whessian}, we have assumed 
$$
D_q\big(\nabla_w\cU(t,\mu)(\cdot)\big)\in L^\infty(\M;\mu), \;\; \overline \nabla^2_{ww}\cU(t,\mu)(\cdot,\cdot)\in L^\infty(\M\times\M;\mu\otimes\mu) \quad \forall \mu\in\sP_2(\M), \;\; \text{and a.e.} \; t\in(0,T).
$$ 
Then the vector field $\cV(t,\mu,q) := \nabla_w \cU(t,\mu)(q)$ defined on  $[0,T]\times \bigcup_{\mu\in\cP_2(\M)}\{\mu\}\times\spt(\mu)$, solves the vectorial master equation \eqref{eq:master_vector} with initial data $\cV_0=\nabla_w\cU_0$ in the sense of Definition \ref{def:master_weak}.
\end{theorem}


\begin{proof}[Proof of Theorem \ref{thm:main-vector}] Let $\mu\in\sP_2(\M)$, let $\varphi \in C_c^\infty(\M)$ be arbitrary and set $\xi:=D\varphi.$ Choose $\e>0$ be small enough such that for all  $s\in [0,\e]$, $X_s:=\id + s \xi$ is a diffeomorphism of $\M$ into $\M$ and $|\id|^2/2 +s \varphi$ is convex. For any $q \in \spt(\mu)$ we have 
\begin{align} 
\nabla_w \cU(t, \sigma_s)(X_s(q)) = & \nabla_w \cU(t, \mu)(q)+ s D_q\nabla_w  \cU(t, \mu)(q)\xi(q) +s\int_{\M} \nabla^2_{ww} \cU(t, \mu)(q , a) \xi(a)\mu(da) \nonumber\\
+& o(s).\label{eq:dec04.2019.1}
\end{align} 
Since 
\[
\int_{\M} H\big( z, \nabla_w \cU(t, \sigma_s)(z)\big)\sigma_s(dz)=\int_{\M} H\Big(X_s(q), \nabla_w \cU(t, \sigma_s)\big(X_s(q) \big)\Big)\mu(dq),
\]
\eqref{eq:dec04.2019.1} implies 
\begin{align} 
\cH\big(\sigma_s, \nabla_w \cU(t, \sigma_s)\big) = & \cH\big(\mu, \nabla_w \cU(t, \mu)\big)+ s \int_{\M}D_qH\big(q,  \nabla_w \cU(t, \mu(q))\big) \cdot \xi(q) \mu(dq)\nonumber\\
+ &s \int_{\M}D_pH\big(q,  \nabla_w \cU(t, \mu)(q)\big) \cdot  \Big(D_q  \nabla_w \cU(t, \mu)(q)\xi(q) \Big) \mu(dq)\nonumber\\
+ &s \int_{\M^2}D_pH\big(q,  \nabla_w \cU(t, \mu)(q)\big) \cdot \Big(  \nabla^2_{ww} \cU(t, \mu)(q , a) \xi(a)\mu(da) \Big) \mu(dq)\nonumber\\
- &  \cF(\mu)- s \int_{\M} \nabla_w \cF(\mu)(q) \cdot \xi(q) \mu(dq) + o(s).\label{eq:dec04.2019.2}
\end{align} 
Similarly, 
\begin{equation}\label{eq:dec04.2019.3}
\partial_t \cU(t, \sigma_s)= \partial_t \cU(t, \mu)+ s \int_{\M}\partial_t\nabla_w \cU(t, \mu)(q) \cdot \xi(q) \mu(dq) + o(s).
\end{equation}
Let us remark that since $\cU$ is a $C^{1,1}_{\rm{loc}}([0,T]\times\cP_2(\M))$ solution to \eqref{eq:HJB}, $\nabla_w\cU(\cdot,\mu)(q)$ is Lipschitz continuous on $[0,T]$. Moreover, from the equation \eqref{eq:HJB} and since $\cU(t,\cdot)\in C^{2,1,w}_{\rm{loc}}(\cP_2(\M))$,  we get that $\partial_t\cU(t,\cdot)$ is differentiable for all $t\in(0,T)$. Therefore, $\partial_t\nabla_w\cU(t,\mu)(q)=\nabla_w\partial_t\cU(t,\mu)(q)$ for all $(t,\mu)\in(0,T)\times\cP_2(\M)$ and $q\in\spt(\mu)$.

Since 
\[
\partial_t \cU(t, \sigma_s)+ \cH\big(\sigma_s, \nabla_w \cU(t, \sigma_s)\big)=0, 
\]
\eqref{eq:dec04.2019.2} and \eqref{eq:dec04.2019.3} imply  
\begin{align} 
&  \int_{\M}\Big( \partial_t\nabla_w \cU(t, \mu)(q)+ D_qH\big(q,  \nabla_w \cU(t, \mu)(q)\big) -\nabla_w \cF(\mu)(q)\Big) \cdot \xi(q) \mu(dq) 
\nonumber\\
+ & \int_{\M}D_pH\big(q,  \nabla_w \cU(t, \mu)(q)\big) \cdot  \Big(D_q  \nabla_w \cU(t, \mu)(q)\xi(q) \Big) \mu(dq)\nonumber\\
+ & \int_{\M^2}D_pH\big(q,  \nabla_w \cU(t, \mu)(q)\big) \cdot \Big(  \nabla^2_{ww} \cU(t, \mu)(q , a) \xi(a)\mu(da) \Big) \mu(dq) 
=0.\label{eq:dec04.2019.5}
\end{align} 
Since we asserted in Remark \ref{rem:symmetricxx} that  $D_q\nabla_w\cU(t,\mu)(\cdot)$ is symmetric, \eqref{eq:dec04.2019.5} can be rewritten as 
\begin{align*}
&  \int_{\M}\Big[ \partial_t\nabla_w \cU(t, \mu)(q)+ D_qH\big(q,  \nabla_w \cU(t, \mu)(q)\big) -\nabla_w \cF(\mu)(q)\Big] \cdot \xi(q) \mu(dq)\nonumber\\ 
+ & \int_{\M}D_q  \nabla_w \cU(t, \mu)(q) D_pH\big(q,  \nabla_w \cU(t, \mu)(q)\big) \cdot \xi(q) \mu(dq)\nonumber\\
+ &\int_{\M^2}\Big(  \nabla^2_{ww} \cU(t, \mu)(q , a)^\top  D_pH\big(q,  \nabla_w \cU(t, \mu)(q)\big)\Big) \mu(dq)\cdot\xi(a)\mu(da)
=0.
\end{align*}
Note 
\[
D_qH\big(\cdot,  \nabla_w \cU(t, \mu) \big)+ D_q  \nabla_w \cU(t, \mu) D_p H\big(\cdot,  \nabla_w \cU(t, \mu)\big)=D_q\Big( H\big(\cdot,  \nabla_w \cU(t, \mu)\big)\Big) \in T_\mu\sP_2(\M).
\]
Since the rows of $\nabla^2_{ww} \cU(t, \mu)(q , a)$ belong to $T_\mu\sP_2(\M)$, so does $ \nabla^2_{ww} \cU(t, \mu)(q , a)^\top  D_pH\big(q,  \nabla_w \cU(t, \mu)(q)\big)$ (as linear combinations of these rows). By the arbitrariness of $\xi$ and the previous claims, we  conclude 
\begin{align*}
\partial_t\nabla_w \cU(t, \mu)+ D_qH\big(\cdot,  \nabla_w \cU(t, \mu) \big)+ D_q  \nabla_w \cU(t, \mu) D_p H\big(\cdot,  \nabla_w \cU(t, \mu)\big)+\overline \cN_\mu\big[\cV, \nabla_{w}^T \cV \big](t, \mu,\cdot)=\nabla_w \cF(\mu),
\end{align*}
$\mu$--almost everywhere on $q\in\M$.
\end{proof}

\begin{remark}
At this point we do not know whether all the terms appearing in \eqref{eq:master_vector} could be extended to (at least $\sL^d$--a.e.) $q\in\M$. We have good pointwise continuity properties of $\overline \nabla_{ww}^\top \cU(t, \cdot)(\cdot,\cdot)$, but we do not know much about the continuity properties of $\nabla_{ww}^\top \cU(t, \cdot)(\cdot,\cdot).$ If we knew 
$$\overline \cN_\mu\big[\cV, \nabla_{ww}^\top \cU \big](t, \mu, q)=\overline \cN_\mu\big[\cV, \overline \nabla_{ww}^\top \cU \big](t, \mu, q)$$ we could deduce that $q \mapsto \overline \cN_\mu\big[\cV, \nabla_{ww}^\top \cU \big](t, q, \mu)$ is continuous. In the same time, we do not know whether $\ds\partial_t\cV$ admits a continuous extension. 

As a last remark, despite the fact that $\cV(t,\mu,\cdot)$ itself is defined only on $\spt(\mu)$, we know that it is Lipschitz continuous there, uniformly with respect to $t$ and $\mu$. But it is not clear at all whether any Lipschitz continuous extension of this in the same time would produce a valid extension for $\partial_t\cV$ and $\nabla_w^\top\cV$. As highlighted before, we revisit this question in Subsection \ref{subsec:further_vectorial_eq}, and in particular there we produce a solution to the vectorial master equation which is defined for (Lebesgue) a.e. $q\in\M$.

\end{remark}

%
%
%
%
\subsection{The scalar master equation}\label{subsec:scalar_main}

In this subsection we assume there exists a function $C$ which assume to each compact set $K\subset\M$ and each real number $r>0$, a positive value $C(K, r).$ We assume  to be given 
\begin{equation}\label{ass:C11onuandf} 
u_0,f\in C^{1,1}_{\rm{loc}}(\M\times\sP_2(\M))\tag{H\arabic{hyp}}
\end{equation} 
\stepcounter{hyp} 
such that 
\begin{align}\label{hyp:u_0-f}
\nabla_w\cU_0(\mu)(q)=D_q u_0(q,\mu),\nabla_w\cF(\mu)(q)=D_q f(q,\mu),\ \forall\ (q,\mu)\in\M\times\sP_2(\M).\tag{H\arabic{hyp}}
\end{align}
\stepcounter{hyp}

Since we can modify $L$ or $\tilde \cF$ as follows, 
\[
\tilde \cL(x, a)= \int_\Omega\bigl( L(x(\omega), a(\omega)) -r |x(\omega)|^2 \bigr)d\omega +\tilde \cF(x)+ r \|x\|^2, 
\]
we learn from Proposition \ref{prop:conv} that \eqref{ass:F-U0-C11new2} and \eqref{ass:convexity-on-tildeL} imply that
\begin{equation}\label{ass:convexity} 
\M\ni q\mapsto u_0(q,\mu)\ \ {\rm{is\ convex\ and}}\ \ \M\times\R^d\ni(q,v)\mapsto L(q,v)+f(q,\mu) \ \text{is strictly convex} \quad \forall \mu\in\sP_2(\M).
\end{equation}

Let us remark that by the fact that $u_0,f\in C^{1,1}_{\rm{loc}}(\M\times\cP_2(\M))$, we have that $u_0$ and $f$ are locally bounded, i.e.
$\forall K\subset\M\ {\rm{compact\ and\ }}r>0,\  \exists C=C(K, r):\ |u_0(q_0,\mu)|,|f(q_0,\mu)|\le C,\ \forall\ (q_0,\mu)\in K\times\cB_r.$

We are to find a function $u:[0,T]\times\M\times\sP_2(\M)\to\R$ that satisfies the {\it scalar master equation}
\begin{equation}\label{eq:master}
\left\{
\begin{array}{ll}
\ds\partial_t u(t, q, \mu)+H(q,D_q u(t, q, \mu))+\cN_{\mu}\big[D_q u(t, \cdot, \mu), \nabla_w {u}(t, q, \mu)(\cdot)\big]=f(q,\mu), & (0,T)\times\M\times\sP_2(\M),\\[5pt]
u(0, \cdot,\cdot)=u_0, & \M\times\sP_2(\M),
\end{array}
\right.
\end{equation} 
where the non--local operator $\cN_{\mu}$ is defined as in \eqref{eq:defineNmu}. We define the notion of classical solution to \eqref{eq:master} as follows.

\begin{definition}\label{def:classical_sol}
We say that $u$ is a classical solution to \eqref{eq:master}, if the following holds. It is continuously differentiable on $(0,T)\times\M\times\cP_2(\M)$, continuous up to the initial time 0 and the PDE is satisfied pointwise. The vector field $\M\ni q\mapsto D_q u(t,q,\nu)$ is Lipschitz, uniformly with respect to $(t,\nu)\in[0,T]\times\cB_r$ ($r>0$).

Furthermore, for all $\nu\in\cP_2(\M)$ and for $\cL^1\otimes\cL^d$--a.e. $(s,q)\in (0,T)\times\M$, $D_q\nabla_w u(s,q,\nu)(\cdot)$ and $\nabla_wD_q u(s,q,\nu)(\cdot)$ exist, belong to $L^2(\nu)$ and they satisfy additionally 
\begin{equation}\label{cond:uniqueness} 
\int_{\M}\Big(\big(D_q\nabla_w -\nabla_wD_q\big) u(s,q,\nu)(y)\Big) D_pH(y,D_q u(s,y,\nu))\nu(dy)= 0.
\end{equation}
\end{definition} 

\begin{remark}
The condition \eqref{cond:uniqueness} in the previous definitions needs some comments. In Theorem \ref{thm:main-scalar} we will actually show existence of $C^{1,1}_{\rm{loc}}([0,T]\times\M\times\cP_2(\M))$ solution to \eqref{eq:master}. Let us notice that for functions $w\in C^{1,1}_{\rm{loc}}(\M\times\cP_2(\M))$,  $D_q\nabla_w w(q,\nu)(\cdot)$ is meaningful for all $\nu\in\cP_2(\M)$ and for a.e. $q\in\M$ (see Subsection \ref{subsec:further_vectorial_eq}). But, since $D_q w$ is only Lipschitz continuous with respect to the measure variable, $\nabla_w D_qw(q,\nu)(\cdot)$ might not be meaningful in general (since, Rademacher-type theorems in $(\cP_2(\M),W_2)$ are more subtle, cf. \cite{Dello}). So the $C^{1,1}$ regularity in general is not enough to ensure \eqref{cond:uniqueness}.

Nevertheless, as the discussion in Subsection \ref{subsec:further_vectorial_eq} shows, the solution that we construct for the master equation \eqref{eq:master} naturally satisfies \eqref{cond:uniqueness}. This condition in particular will imply uniqueness of the solution as well.
\end{remark}

For $m\in\N$, we define 
$$u_0^{(m)},f^{(m)}:\M\times(\M)^m\to\R, \quad U_0^{(m)}, F^{(m)}:(\M)^m\to\R$$ as
$$u_0^{(m)}(y, q):=u_0\Big(y,\mu^{(m)}_q\Big),\quad f^{(m)}(y,q):=f\Big(y,\mu^{(m)}_q\Big),\quad U^{(m)}_0(q):=\cU_0\Big(\mu^{(m)}_q\Big),\quad F^{(m)}(q):=\cF\Big(\mu^{(m)}_q\Big),$$
where for $q=(q_1,\cdots,q_m)\in(\M)^m$, $\mu^{(m)}_q$ is defined as in \eqref{eq:average-q}.

We impose the following hypotheses on $u_0^{(m)}$ and $f^{(m)}.$

\begin{align}\label{hyp:u_0-f-m} 
u_0^{(m)}(y,\cdot),\ f^{(m)}(y,\cdot)\ {\rm{satisfy\ Properties\ }}\ref{def:app_reg_estim}{\rm{(1)(a)\ and}}\ \ref{def:app_reg_estim}(2),\ {\rm{locally\ uniformly\ w.r.t.\ }}y\in\M.
\tag{H\arabic{hyp}}
\end{align}
\stepcounter{hyp}
\begin{align}\label{hyp:D_yu_0-D_yf-m} 
D_yu_0^{(m)}(y,\cdot),\  D_yf^{(m)}(y,\cdot)\ {\rm{satisfy\ Property\ }}\ref{def:app_reg_estim}{\rm{(1)(a)}},\ {\rm{locally\ uniformly\ w.r.t.\ }}y\in\M.
\tag{H\arabic{hyp}}
\end{align}
\stepcounter{hyp}
Let us notice that based on the previous assumptions, we have that $D_yu_0^{(m)}$ and $D_yf^{(m)}$ are locally uniformly bounded, i.e.
$
\forall r>0, K\subset\M\ {\rm{compact}},\ \exists C=C(K,r):\ |D_{y}u_0^{(m)}(y,q)|, |D_{y}f^{(m)}(y,q)|\le C,\ {\rm{if}}\ (y,q)\in K\times\B_r^m.$ In the same time, by the assumption \eqref{ass:LipschitzH}, $D_qL$ and  $\partial_y^{a}\partial_v^{b}L$ (for all $a,b$ multi-indices with $|a|+|b| = 2$) are locally uniformly bounded.

We assume that there exists a constant $C>0$ such that 
\begin{align}\label{hyp:H_3} 
\|\partial_q^{a}\partial_p^{b}H\|_{L^\infty(\M\times\R^d)}\le C, \; \ {\rm{for}}\ a,b \ {\rm{multi-indices}\ with\ } |a|+|b| = 3.\tag{H\arabic{hyp}}
\end{align}
\stepcounter{hyp}

We assume there exists a locally bounded continuous function $\theta: \cP_2(\M) \rightarrow [0,\infty)$ such that 

\begin{equation}\label{hyp:f-theta} 
L(q, v)+f(q, \mu) \geq \lambda_1 |v|^2 -\theta(\mu) (|q|+1), \qquad \forall (q, v) \in \M \times \R^d, \; \forall \mu \in \cP_2(\M).\tag{H\arabic{hyp}}
\end{equation}\stepcounter{hyp}
Note that it suffices to impose that $f(\cdot, \mu)$ is convex to have that \eqref{ass:UpperboundonL} implies \eqref{hyp:f-theta}.
\hfill\break


Recall that Remark \ref{rem:new-bound-onHL} (iii) ensures there exists a constant $C$ such that 

We assume that there exists $C>0$ such that 
\begin{equation}\label{hyp:D_qH}
|D_q H(q,p)|\le C(1+|q|+|p|),\quad |D_q L(q,v)|\le C(1+|q|+|v|) \ \ \forall (q,p, v)\in \M\times\R^{2d}.\tag{H\arabic{hyp}}
\end{equation}\stepcounter{hyp}

\vskip0.40cm
\subsection{Examples of data functions} \label{ex:example} 
We pause for a moment to give examples of initial data  $\cU_0$ and $u_0$, which satisfy the standing assumptions of this manuscript. Similar examples can be constructed for $\cF$ and $f$ as well. 

Let $\phi_0,\phi_1:\M\to\R$ be smooth bounded functions with uniformly bounded derivatives up to order 3. For simplicity, we assume also that they are positive and  $\phi_1$ is even. Fix $\l>0$ and let $\phi:\M\to\R$ be defined as $\phi(q):=\frac{\l}{2}|q|^2+\phi_0(q)$ and assume $\l$ is large enough such that $D^2\phi+D^2\phi_1\ge0$  on $\M$. Then, let us define $\cU_0:\sP_2(\M)\to\R$ as
$$\cU_0(\mu):=\int_{\M}\phi(q)\mu(dq)+\frac{1}{2}\int_{\M}\phi_1*\mu(q)\mu(dq), \quad \tilde \cU_0(x)=\cU_0\big(x_\sharp \cL^d_\Om\big), \qquad \forall \mu \in \sP_2(\M), x \in \bH.$$ 
Then $\tilde \cU_0$ fulfills the  assumptions \eqref{ass:F-U0-C11new1} and \eqref{ass:F-U0-C11new2}. 

Set  
$$u_0(q_0,\mu)=\phi(q_0)+(\phi_1*\mu)(q_0).$$
For $q:=(q_1, \cdots, q_m) \in \M^m$ and $q_0 \in \M$ , we have 
$$u_0^{(m)}(q_0,q)=\phi(q_0)+\sum_{i=1}^m\frac{1}{m}\phi_1(q_0-q_i), \quad \text{and}\quad U_0^{(m)}(q)=\frac{1}{m}\sum_{i=1}^m\phi(q_i)+\frac{1}{2m^2}\sum_{i,j=1}^m\phi_1(q_i-q_j),$$
and so for $1\leq i \leq m$, 
$$D_{q_i}u_0^{(m)}(q_0, q)=\frac{1}{m}D\phi_1(q_0-q_i)\ \ {\rm{and}}\ \ {D^2_{q_0 q_i}}u_0^{(m)}(q_0, q))=\frac{1}{m}D^2\phi_1(y-x_i).$$
We have 
$$D_{q_0} u_0^{(m)}(q_0, q)=D\phi(y)+\sum_{i=1}^m\frac{1}{m}D\phi_1(q_0-q_i).$$
From these computations, one can easily verify that \eqref{ass:C11onuandf} through \eqref{hyp:D_yu_0-D_yf-m} are satisfied. 

Under appropriate conditions on functions $L_0,$ $l$ and $g,$ lagrangians of the form 
\[
L(q, v):= L_0(v) + l(q, v) + g(q)
\]
and Hamiltonian defined as $H(q,\cdot):=L^*(q,\cdot),$ satisfy \eqref{ass:Hamiltonian} through \eqref{ass:convexity-on-tildeL} and \eqref{hyp:H_3} through \eqref{hyp:D_qH}.

\medskip
We are ready now to define the candidate for the solution to the scalar master equation. Given $t \in [0,T]$, $q \in \M$ and $\mu \in \cP_2(\M)$ we define 
\begin{equation}\label{def:u_alt}
u(t, q, \mu):= \inf_{\gamma} \biggl\{u_0(\gamma_0, \sigma_0^t[\mu]) +\int_0^t \Big(L(\gamma_s, \dot \gamma_s) + f(\gamma_s, \sigma_s^t[\mu])\Big)ds\; : \; \gamma \in W^{1, 2}([0,t], \M), \gamma_t=q \biggr\}.
\end{equation}
Here the curve $(\s^t_s[\mu])_{s\in[0,t]}$ is defined in \eqref{def:sigma_optimal}.
Define 
\[
M_*(r):=\sup_{B_r(0) \times \cB_{e_T(r)}} |\theta|+ |u_0|+T(|f|+|L(0, \cdot)|), \qquad c_*(r):= \sup_{\ov B_1(0) \times \ov \cB_r} |u_0|
\]
%
\begin{remark}\label{rem:summary-bounds4} Let $r>0.$ 
\begin{enumerate}
\item[(i)] As $u_0(\cdot, \nu)$  is convex, if $D_qu(0, \nu)\not =0$ then  
\[
u_0\bigg({D_q u(0, \nu) \over |D_qu(0, \nu)|}, \nu \bigg) \geq u_0(0, \nu)+ {D_q u(0, \nu) \over |D_qu(0, \nu)|} \cdot D_q u(0, \nu)= u_0(0, \nu)+ {|D_qu(0, \nu)|^2 \over |D_qu(0, \nu)|}.
\]
Thus, if $\nu \in \cB_r,$ we conclude that  
\[
|D_qu(0, \nu)| \leq 2c_*(r).
\]
Clearly, the previous inequality still holds when $D_qu(0, \nu) =0.$ Consequently, 
\[
u_0(q, \nu) \geq u_0(0, \nu)+ D_qu(0, \nu) \cdot q \geq -c_*(r)(1+|q|).
\]
\item[(ii)] Suppose  $(t, q, \mu) \in [0, T] \times B_r(0) \times \cB_r$. Then 
\[
u(t, q, \mu) \leq M_*(r), 
\]
and so, if $\gamma$ is the unique minimizer in \eqref{def:u_alt}, we use \eqref{hyp:f-theta} and Remark \ref{rem:summary-bounds1} (ii) to obtain  
\[
M_*(r)\geq u(t, q, \mu) \geq-c_*\big(e_T(r)\big) (1+|\gamma(0)|)-M_*(r)T -M_*(r) \int_0^t |\gamma|ds +\lambda_1 \int_0^t |\dot \gamma|^2ds.
\]
We conclude there exists a constant $\ov M(r)$ independent of $t$ such that $$\int_0^t |\dot \gamma|^2ds \leq \ov M(r).$$
Hence, 
\begin{equation}\label{eq:holder-for-gamma}
|\gamma_{\tau_1}-\gamma_{\tau_2}|^2 \leq \ov M(r)|\tau_2-\tau_1| \qquad \text{if} \quad 0\leq \tau_1 \leq \tau_2 \leq t.
\end{equation}
\item[(iii)] By (ii), there is constant $M^*(r)$ such that 
\[
|u(t, q, \mu)| \leq M^*(r) \qquad (t, q, \mu) \in [0, T] \times B_r(0) \times \cB_r
\]
Since $$(q, v) \mapsto L_{s, t}(q, v) := L(q, v)+ f(q, \sigma_s^t[\mu]), \quad q \mapsto   u_0(q, \sigma_0^t[\mu])$$ 
are convex, we obtain that $u(t, \cdot, \mu)$ is a convex function and so as argued above, 
\[
\big|D_q u(t, q, \mu) \big| \leq u\bigg(t, q+ {D_q u(t, q, \mu) \over \big|D_q u(t, q, \mu) \big|} \bigg)-u(t, q, \mu) \leq M^*(r)+M^*(r+1).
\]
\end{enumerate}
\end{remark}

\begin{lemma}\label{lem:gamma_c11}
Let $(t, q, \mu) \in [0, T] \times B_r(0) \times \cB_r$ and let $\g:[0,t]\to\M$ be the unique optimizer in \eqref{def:u_alt}. Suppose that the assumptions \eqref{ass:Hessianv-v},\eqref{ass:LipschitzH}, \eqref{ass:UpperboundonL}, \eqref{hyp:u_0-f} and \eqref{hyp:D_qH} take place. Then $\g\in C^{1,1}([0,t])$.
\end{lemma}
\begin{proof}
The proof  follows the same lines as the one of \cite[Theorem 6.2.5]{CannarsaS}. 
\end{proof}

%
%
\begin{proposition}\label{prop:representu} Let $\mu \in \cP_2(\M)$ and $t \in [0,T]$.   Recall $[0,t]\ni s \mapsto \sigma_s^t[\mu]$ is defined in \eqref{def:sigma_optimal} in Lemma \ref{lem:unique-path}.
\begin{enumerate} 
\item[(i)] We have $u(t, \cdot, \mu) \in C^{1,1}_{\rm loc}(\M)$. Furthermore, there exists a unique $\gamma$ minimizer in \eqref{def:u_alt} which we denote as $s \mapsto S_s^t[\mu](q).$
\item[(ii)] If $\omega \in \Om$, $x \in \bH$,   $\mu=\sharp(x)$ and $q=x(\omega)$ (meaning in particular that $q\in\spt(\mu)$), then $\tilde S_s^t[x](\omega)= S_s^t[\mu](q).$ 
\item[(iii)] Under the assumptions in (ii) we have $D_q u(t, q, \mu)= \nabla_w \cU(t, \mu)(q).$
\item[(iv)] $[0,t]\ni s\mapsto D_q u(s, S^t_s[\mu](q), \s^t_s[\mu])$ is Lipschitz continuous, for all $(q,\mu)\in\M\times\cP_2(\M)$.
\item[(v)] We have that $u(\cdot,\cdot,\mu)\in C^{0,1}_{\rm{loc}}([0,T]\times\M)$, with Lipschitz constants depending on $r>0$, where $\mu\in\cB_r$.
\end{enumerate} 
\end{proposition}
\proof{} (i) By Remark \ref{rem:summary-bounds4} (iii), $u(t, \cdot, \mu)$ is a convex function. The fact that $u(t, \cdot, \mu)$ is locally semi--concave is a standard property. Thus,  $u(t, \cdot, \mu)$ is $C^{1,1}_{\rm loc}(\M).$ Since, the action 
$$ 
\gamma \mapsto A_t[\gamma]:= u_0(\gamma_0, \sigma_0^t[\mu]) +\int_0^t  L_{s, t}(\gamma_s, \dot \gamma_s) ds
$$ 
is strictly convex, $S_s^t[\mu](q)$ is uniquely defined. 

(ii) By the convexity of $A_t$, any critical point of $A_t$ on the set $\{ \gamma \in C^1([0,t], \M)\,: \, \gamma_t=q \}$ is a minimizer. Set 
\[
p_s:= P_s^t[\mu](q).
\]
The Hamiltonian associated to $L_{s, t}$ is $H_{s, t}(q,p):= H(q, p)-  f(q, \sigma_s^t[\mu]).$ Since $D_p H_{s, t}(q, p)\equiv D_p H(q, p)$, in light of Proposition \ref{prop:homeo1} (iv) we have 
\begin{equation}\label{eq:jan02.2020.1}
D_p H_{s, t}(\gamma_s, p_s) = D_p H\Big( \tilde S_s^t[x](\omega),  \tilde P_s^t[x](\omega)\Big)= \partial_s \tilde S_s^t[x](\omega)= \dot \gamma_s.
\end{equation} 
By \eqref{hyp:u_0-f}
\[
D_q H_{s, t}(q, p)=D_q H(q, p)- D_q f(q, \sigma_s^t[\mu])= D_q H(q, p)- \nabla_w \cF (\sigma_s^t[\mu])(q).
\]
Thus, by Remark \ref{rem:factorization} 
\begin{equation}\label{eq:jan02.2020.2}
D_q H_{s, t}(\gamma_s, p_s)= D_q H\Big( \tilde S_s^t[x](\omega),  \tilde P_s^t[x](\omega)\Big)- \nabla\tilde \cF (\tilde S_s^t[\mu])(\omega)=-\partial_s \tilde P_s^t[x](\omega)=-\dot p_s.
\end{equation} 
We use first use \eqref{hyp:u_0-f}, second Remark \ref{rem:factorization} and third the last identity in \eqref{eq:app_Hamiltonian_sys_1}, to obtain 
\[
D_q  u_0(\gamma_0, \sigma_0^t[\mu])= \nabla_w \cU_0\big( \sigma_0^t[\mu]\big)(\gamma_0)=\nabla\tilde \cU_0 (\tilde S_0^t[\mu])(\omega)=  \tilde P_0^t[x](\omega))= p_0.
\]
This, together with \eqref{eq:jan02.2020.1} and \eqref{eq:jan02.2020.2} implies $\gamma$ is a critical point of $A_t$ on the set $\{ \gamma \in C^1([0,t], \M)\,: \, \gamma_t=q \}.$ Hence, $\gamma$ is the unique minimizer, which verifies (ii). 

(iii) 
By the optimality property of $\gamma$, the standard Hamilton--Jacobi theory ensures that 
\begin{equation}\label{eq:jan02.2020.0}
\dot \gamma_s =D_p H(\gamma_s, D_q u(s, \gamma_s, \s^t_s[\mu])) \qquad \forall s \in (0, t).
\end{equation} 
First, by the strict convexity of $H$ in the second variable, we have that
\begin{align*}\label{eq:jan02.2020.0}
D_q u(s, \gamma_s, \s^t_s[\mu])=D_v L(\g_s,\dot \gamma_s ) \qquad \forall s \in (0, t),
\end{align*} 
from where, by Lemma \ref{lem:gamma_c11} and the by the regularity of $D_vL$ one obtains that $[0,t]\ni s\mapsto D_q u(s, \gamma_s, \s^t_s[\mu])$ is Lipschitz continuous. This shows (iv).

Then, by Proposition \ref{prop:homeo1} (iv), 
\[
\dot \gamma_s =D_p H\Big(\gamma_s, \nabla_w \cU(s, \sigma_s^t[\mu])(\gamma_s)\Big),
\]
which, together with \eqref{eq:jan02.2020.0} implies 
\[
D_p H\Big(\gamma_s, \nabla_w \cU(s, \sigma_s^t[\mu])(\gamma_s)\Big)= D_p H\big(\gamma_s, D_q u(s, \gamma_s, \s^t_s[\mu])\big) \qquad \forall s \in (0, t).
\]
Thus, by \eqref{ass:Hessianv-v}, one has
\[
 \nabla_w \cU(s, \sigma_s^t[\mu])(\gamma_s)=   D_q u(s, \gamma_s, \s^t_s[\mu])  \qquad \forall s \in (0, t).
\]
Letting $s$ increase to $t$ we verify (iii). 

(v) What remains to be shown is the Lipschitz regularity of $u$ with respect to the variable $t$. But, this follows from the dynamic programming principle and from the time Lipschitz continuity of $(\g_s)_{s\in[0,t]}$ and $(\s^t_s[\mu])_{s\in[0,t]}$ (see Lemma \ref{lem:summary-bounds2}(ii) and Lemma \ref{lem:gamma_c11}). \endproof

\begin{remark}\label{rem:representu} 

(i) Let $\mu \in \cP_2(\M)$, $t \in [0,T]$. Note that in Proposition \ref{prop:representu} $S_s^t[\mu]$ is defined on the whole set $\M$ and not just on the support of $\mu$. When $x \in \bH$ is such that  $\mu=\sharp(x)$, Proposition \ref{prop:representu} (ii) reads off 
\[
\tilde S_s^t[x]=S_s^t[\mu]\circ x.
\]
Also, 
\begin{equation}\label{eq:flow}
\left\{
\begin{array}{ll}
\partial_s S_s^t[\mu]=D_pH(S_s^t[\mu],\nabla_w\cU(s,\sigma_s^t[\mu])(S_s^t[\mu])), & s\in(0,t),\\
S_t^t[\mu]=\id.
\end{array}
\right.
\end{equation} 
\smallskip

(ii) It is very important to underline also the fact that by Proposition \ref{prop:representu}(iii) we have that for all $(t,\mu)\in(0,T)\times\cP_2(\M)$, $D_qu(t,\cdot,\mu)=\nabla_w\cU(t,\mu)(\cdot)$ on $\spt(\mu)$. Since $D_qu(t,\cdot,\mu)$ is defined on the whole $\M$ (and we will see below that it is locally Lipschitz continuous), this produces a very natural extension for $\nabla_w\cU(t,\mu)(\cdot)$ to the whole $\M$. This observation will also help us to improve the previous notion of weak solution to the vectorial master equation, as we will see in Subsection \ref{subsec:further_vectorial_eq}.

(iii) Since $\cU$ is of class $C^{1,1}_{\rm{loc}}$ (cf. Definition \ref{def:c11_wasserstein}) \cite[Corollary 3.38]{CarmonaD-I} yields the existence of a Lipschitz continuous extension of $\nabla_w\cU(t,\mu)(\cdot)$ to the whole $\M$, with a Lipschitz constant independent of $\mu$. This extension has the property that it is continuous at $(\mu,q)$ for $q\in\spt(\mu)$. Our result, as described above, because of the local Lipschitz continuity of $D_qu$ (cf. Lemma \ref{lem:u_alt-reg}) provides a slightly better extension.
\end{remark}

\begin{proposition}\label{prop:ext_u}  For all $t\in[0,T]$ and $q\in\M$, the function $u(t,q,\cdot)$ is continuous on $\cP_2(\M).$
\end{proposition}
We skip the proof of this proposition since it is obtained by standard arguments, similar to those appearing in the proof of Proposition \ref{prop:summary-bounds3}
%

%
%
%
\begin{lemma}\label{lem:u_alt-reg} When \eqref{ass:F-U0-C11new1} - \eqref{hyp:D_qH} hold, then $u$ defined in \eqref{def:u_alt} is of class $C^{1,1}_{\rm{loc}}([0,T]\times\M\times\cP_2(\M))$.
\end{lemma}
\begin{proof}
We proceed by a discretization approach. Let $\mu\in\sP_2(\M)$, $t>0$, $m\in\N$ and $q_0\in\spt(\mu)$ be fixed. Moreover given  $\{q_1,\dots,q_m\}\subset\spt(\mu)$ we shall use the notation of $q=(q_1,\dots,q_m)\in(\M)^m.$  We define 
$$
\mu_q^{(m+1)}=\frac{1}{m+1}\sum_{i=0}^m\delta_{q_i}, \quad \sigma^{(m+1)}_s:=\sigma_s^t\Big[ \mu_q^{(m+1)}\Big]
$$
so that $\sigma^{(m+1)}$ is the solution to the continuity equation \eqref{continuity} with $\mu_q^{(m+1)}$ as terminal condition. Note 
$$\sigma^{(m+1)}_s=\frac{1}{m+1}\sum_{i=0}^m\d_{S_s^t[  \mu_q^{(m+1)}](q_i)},\qquad \forall s\in(0,t).$$

We define 
$$ 
u_0^{(m+1)},f^{(m+1)}:\M\times(\M)^{(m+1)}\to\R, \qquad U^{(m+1)}, u^{(m)}:(0,T)\times(\M)^{(m+1)}\to\R
$$ 
as  
$$u_0^{(m+1)}(y_0,q_0,q):=u_0(y_0,\mu^{(m+1)}_q), \quad f^{(m+1)}(y_0,q_0,q):=f(y_0,\mu^{(m+1)}_q),$$ 
and 
\begin{equation}\label{eq:Um_and_u_m}
U^{(m+1)}(s,q_0,q):=\cU(s,\mu^{(m+1)}_q), \quad u^{(m)}(t,q_0,q):=u(t,q_0,\mu_q^{(m+1)}).
\end{equation}
Observe 
\begin{align}\label{def:u_alt_m}
u^{(m)}(t,q_0,q)&=u_0\Big(Q_0(0,q_0,q),\s^{(m+1)}_0\Big)\\
\nonumber+&\int_0^t L\Big(Q_0(s,q_0,q),D_pH\Big(Q_0(s,q_0,q),\nabla_w\cU\big(s,\sigma^{(m+1)}_s\big)\big(Q_0(s,q_0,q)\big)\Big)\Big)\dd s\\
\nonumber+&\int_0^tf\Big(Q_0(s,q_0,q),\sigma^{(m+1)}_s\Big) ds\\
\nonumber=&u_0^{(m+1)}\Big(Q_0(0,q_0,q),Q_0(0,q_0,q),Q(0,q_0,q))\Big)\\
\nonumber+&\int_0^t L\Big(Q_0(s,q_0,q),D_pH(Q_0(s,q_0,q),(m+1)D_{q_0}U^{(m+1)}(s,Q_0(s,q_0,q), Q(s,q_0,q)\Big) ds\\
\nonumber+&\int_0^tf^{(m+1)}\Big(Q_0(s,q_0,q),Q_0(s,q_0,q),Q(s,q_0,q)\Big) ds
\end{align}
where we have set 
\begin{equation}\label{eq:m-discrete-flow}
Q_i(s, q_0, q):= S_s^t\big[\mu_q^{(m+1)} \big](q_i) \quad \text{and}\quad Q(s, q_0, q):=(Q_1(s, q_0, q), \cdots, Q_m(s, q_0, q))
\end{equation}
Now our first goal is to obtain derivative estimates on $u^{(m)}$ with respect to the `distinguished' variable $q_0$ and second, with respect to all the other variables $q$. Finally, we also derive the necessary estimates involving the time variable $t$ as well. It is convenient to introduce the notation $$\tilde u_0^{(m+1)},\tilde f^{(m+1)},V^{(m+1)}:\M\times(\M)^m\to\R$$ defined as
\begin{equation}\label{new-quantities}
\begin{array}{l}
\tilde u_0^{(m+1)}(q_0,q):=u_0^{(m+1)}(Q_0(0,q_0,q),Q_0(0,q_0,q),Q(0,q_0,q)),\\[3pt]
\ds \tilde f^{(m+1)}(q_0,q):=\int_0^tf(Q_0(s,q_0,q),Q_0(s,q_0,q),Q(s,q_0,q))\dd s\\[3pt]
\ds V^{(m+1)}(q_0,q):=\int_0^t L(Q_0(s,q_0,q),D_pH(Q_0(s,q_0,q),(m+1)\nabla_{q_0}U^{(m+1)}(s,Q_0(s,q_0,q), Q(s,q_0,q))))\dd s.
\end{array}
\end{equation}
In Lemma \ref{lem:last-estimates} and Lemma \ref{lem:last-estimates_time} below we establish the necessary derivative estimates on these new quantities. These imply in particular that there exists a constant $C=C(T,r,K)>0$ such that for any $(q_0,q)\in\B^{(m+1)}_r$; $q_0\in K$ (where $K\subset\M$ is compact) and for all $t\in[0,T]$ and $i,j\in\{0,\dots,m\}$, we have 
\begin{align}\label{estim:space1}
|D_{q_i}u^{(m)}(t,q_0,q)|\le \left\{
\begin{array}{ll}
C, & i=0,\\
\ds\frac{C}{m+1}, & i>0, 
\end{array}
\right.
\end{align}
\begin{align}\label{estim:space2}
|D^2_{q_iq_j} u^{(m)}(t,q_0,q)|_\infty\le\left\{
\begin{array}{ll}
C, & i=j=0,\\[4pt]
\ds\frac{C}{m+1}, & (i=j,\ {\rm{and}}\ i>0),\ {\rm{or}}\ (i\cdot j=0\ {\rm{and}}\ \max\{i,j\}>0),\\[5pt]
\ds\frac{C}{(m+1)^2}, & i\neq j,\ i,j>0.
\end{array}
\right.
\end{align}
and
\begin{align}\label{estim:time_space}
|D_{q_0}\partial_t u^{(m)}(t,q_0,q)|\le C,\quad\quad \sum_{k=1}^m(m+1)|D_{q_k}\partial_t u^{(m)}|^2\le C,
\end{align}
and
\begin{equation}\label{estim:time2}
|\partial_{t} u^{(m)}(t,q_0,q)|\le C, \ \  |\partial^2_{tt} u^{(m)}(t,q_0,q)|\le C.
\end{equation}

Let us notice that by definition and the assumption \eqref{hyp:u_0-f}, $u$ is bounded on $[0,T]\times K\times\cB_r$ for any $K\subseteq\M$ compact and $r>0$. Therefore, $u^{(m)}$ is uniformly bounded (with respect to $m$) on $[0,T]\times K\times\B_r^m$.

Now, all these properties allow us to verify the assumptions of Corollary \ref{cor:finite_infinite_reg-scalar-mast-time} and conclude by this that there exists $\tilde u:[0,T]\times\M\times\sP_2(\M)\to\R$ such that after passing to a suitable subsequence $(u^{(m)})_{m\in\N}$ converges to $\tilde u$ in the sense as described in Corollary \ref{cor:finite_infinite_reg-scalar-mast-time}. Let us notice furthermore that $\tilde u(t,q_0,\mu)$ has to be also the limit of $u(t,q_0,\mu_q^{(m+1)})$ (since by Proposition \ref{prop:ext_u} $u(t,q_0,\cdot)$ is continuous) and therefore $\tilde u$ and $u$ must coincide. Thus, as a consequence of Corollary \ref{cor:finite_infinite_reg-scalar-mast-time} $u\in C^{1,1}_{\rm{loc}}([0,T]\times\M\times\sP_2(\M))$.
\end{proof}

\begin{corollary}\label{cor:D_qu-global-Lip}
Under the assumptions of Lemma \ref{lem:u_alt-reg}, we have that the vector field $\M\ni q\mapsto D_q u(t,q,\mu)$ is globally Lipschitz, uniformly with respect to $(t,\mu)\in[0,T]\times\cB_r$ for any $r>0$.
\end{corollary}

\begin{proof}
Let $r>0$, $t\in[0,T]$ and $\mu\in\cB_r$. Let $q_1,q_2\in\M$. Let $(\mu_n)_{n\in\N}$ be a sequence in $\cB_r$ such that $W_2(\mu_n,\mu)\to 0$ as $n\to+\infty$ and $\spt(\mu_n)=\M$ for all $n\in\N$. By Proposition \ref{prop:representu}(iii) we have $D_q u(t,q_i,\mu_n)=\nabla_w\cU(t,\mu_n)(q_i)$, $i=1,2$. In the light of Proposition \ref{prop:semi-convex1} and Lemma \ref{lem:c11_equiv} there exists $C=C(r,T)>0$ independent of $n$ such that 
$$|D_q u(t,q_1,\mu_n)-D_q u(t,q_2,\mu_n)|=|\nabla_w\cU(t,\mu_n)(q_1)-\nabla_w\cU(t,\mu_n)(q_2)|\le C|q_1-q_2|.$$
By the continuity of $D_qu(t,q_i,\cdot)$ provided in Lemma \ref{lem:u_alt-reg}, one can pass to the limit with $n\to+\infty$ to obtain
$$|D_q u(t,q_1,\mu_n)-D_q u(t,q_2,\mu_n)|\le C|q_1-q_2|.$$
The result follows.
\end{proof}

\begin{lemma}\label{lem:last-estimates}
Let $\tilde u_0^{(m+1)}, \tilde f^{(m+1)}$ and $V^{(m+1)}$ be defined in \eqref{new-quantities} and suppose the assumptions of Lemma \ref{lem:u_alt-reg} are fulfilled. Then, for $T,r>0$ and $K\subset\M$ compact, there exists a constant $C=C(T,r,K)>0$ such that for any $(q_0,q)\in\B^{(m+1)}_r$ with $q_0\in K$ and $i,j\in\{0,\dots,m\}$, we have 
\begin{itemize}
\item[(1)]
\begin{align*}
|D_{q_i}\tilde u_0^{(m+1)}(q_0,q)|\le \left\{
\begin{array}{ll}
C, & i=0,\\
\ds\frac{C}{m+1}, & i>0, 
\end{array}
\right.
\ \  {\rm{and}}\ \ 
|D_{q_i}\tilde f^{(m+1)}(q_0,q)|\le \left\{
\begin{array}{ll}
C, & i=0,\\
\ds\frac{C}{m+1}, & i>0 .
\end{array}
\right.
\end{align*}
\item[(2)]
\begin{align*}
|D^2_{q_iq_j}\tilde u_0^{(m+1)}(q_0,q)|_\infty\le\left\{
\begin{array}{ll}
C, & i=j=0,\\
\ds\frac{C}{m+1}, & (i=j,\ {\rm{and}}\ i>0),\ {\rm{or}}\ (i\cdot j=0\ {\rm{and}}\ \max\{i,j\}>0),\\
\ds\frac{C}{(m+1)^2}, & i\neq j,\ i,j>0,
\end{array}
\right.
\end{align*}
and
\item[(2)]
\begin{align*}
|D^2_{q_iq_j}\tilde f^{(m+1)}(q_0,q)|_\infty\le\left\{
\begin{array}{ll}
C, & i=j=0,\\
\ds\frac{C}{m+1}, & (i=j,\ {\rm{and}}\ i>0),\ {\rm{or}}\ (i\cdot j=0\ {\rm{and}}\ \max\{i,j\}>0),\\
\ds\frac{C}{(m+1)^2}, & i\neq j,\ i,j>0.
\end{array}
\right.
\end{align*}
\item[(3)]
\begin{align*}
|D_{q_i}V^{(m+1)}(q_0,q)|\le \left\{
\begin{array}{ll}
C, & {\rm{if}}\ i=0,\\
\ds\frac{C}{m+1}, & {\rm{if}}\ i>0.
\end{array}
\right.
\end{align*}
\item[(4)]
\begin{align*}
|D^2_{q_iq_j}V^{(m+1)}(q_0,q)|_\infty\le\left\{
\begin{array}{ll}
C, & i=j=0,\\
\ds\frac{C}{m+1}, & (i=j\ {\rm{and}}\ i>0) \ {\rm{or}}\ (i\cdot j=0 \ {\rm{and}}\ \max\{i,j\}>0),\\
\ds\frac{C}{(m+1)^2}, & i\neq j.
\end{array}
\right.
\end{align*}
\end{itemize}
As a consequence, $u^{(m)}$ defined in \eqref{def:u_alt_m} satisfied the estimates \eqref{estim:space1} and \eqref{estim:space2} from Lemma \ref{lem:u_alt-reg}.
\end{lemma}
\begin{proof}

As the computations to obtain the corresponding estimates in the case of $\tilde u_0^{(m+1)}$ and $\tilde f^{(m+1)}$ are completely parallel, we perform these only in the case of $\tilde u_0^{(m+1)}.$

(1) In the computations below, to facilitate the reading, we will display neither the time nor the space variables in $Q_i$. For $i\ge 0$, we have
\begin{align}\label{eq:tilde_deriv}
D_{q_i}\tilde u_0^{(m+1)}(q_0,q)&=D_y u_0^{(m+1)}(Q_0,Q_0,Q)D_{q_i}Q_0+D_{q_i} u_0^{(m+1)}(Q_0,Q_0,Q)D_{q_i}Q_i\\
\nonumber&+\sum_{k=0,k\neq i}^mD_{q_k} u_0^{(m+1)}(Q_0,Q_0,Q)D_{q_i}Q_k.
\end{align}
Now, let us observe recall that by assumption \eqref{hyp:u_0-f} we have
$$
D_y u_0(y,\mu)=\nabla_w\cU_0(\mu)(y), \quad u_0^{(m+1)}(y,q_0,q_1,\dots,q_m)=u_0(y,\mu^{(m+1)}_q),
$$ 
for all $\mu\in\sP_2(\M)$, all $y\in\spt(\mu)$ and all $q_0, q_1, \cdots, q_m \in \M.$ This implies  
$$
D_y u_0^{(m+1)}(y,q_0,q)=D_y u_0(y,\mu^{(m+1)}_q)=\nabla_w \cU_0(\mu^{(m+1)}_q)(y),
$$
ans so
\begin{align}\label{eq:correspondance}
D_y u_0^{(m+1)}(q_i,q_0,q)=D_y u_0(q_i,\mu^{(m+1)}_q)=\nabla_w \cU_0(\mu^{(m+1)}_q)(q_i)=(m+1)D_{q_i}U_0^{(m+1)}(q_0,q)
\end{align}
for all $i\in\{0,\dots,m\}.$ 

Let us notice  that by \eqref{hyp:u_0-f-m}-\eqref{hyp:D_yu_0-D_yf-m}, Lemma \ref{lem:reg-Q} and Lemma \ref{lem:Q_0-prop} provide precise regularity estimates on the discrete flow $(Q_i)_{i=0}^m$),  with a constant $C=C(T,r,K)$ such that
$$(m+1)|D_{q_k} u_0^{(m+1)}(Q_0,Q_0,Q_1,\dots,Q_m)|\le C, \quad |D_y u_0^{(m+1)}(Q_0,Q_0,Q_1,\dots,Q_m)|\le C.$$
so (1) follows by combining the previous arguments with Lemma \ref{lem:reg-Q}.

(2) Differentiating \eqref{eq:tilde_deriv} with respect to $q_j$ one obtains
\begin{align*}
D^2_{q_iq_j}\tilde u_0^{(m+1)}(q_0,q)&=D_{q_j}Q_0D^2_{yy} u_0^{(m+1)}(Q_0,Q_0,Q)D_{q_i}Q_0
+\sum_{k=0}^mD_{q_j}Q_k D^2_{yq_k}u_0^{(m+1)}(Q_0,Q_0,Q)D_{q_i}Q_0\\
&+D_y u_0^{(m+1)}(Q_0,Q_0,Q)D^2_{q_iq_j}Q_0+\sum_{k,l=0}^m D_{q_j}Q_l D^2_{q_kq_l} u_0^{(m+1)}(Q_0,Q_0,Q)D_{q_i}Q_k\\
&+\sum_{k=0}^m D_{q_k} u_0^{(m+1)}(Q_0,Q_0,Q)D^2_{q_iq_j}Q_k
\end{align*}

From \eqref{eq:correspondance} we observe again for any $ i\in\{0,\dots,m\}$, 
\[
D^2_{yy} u_0^{(m+1)}(q_i,q_0,q)=D^2_{yy} u_0(q_i,\mu^{(m+1)}_q)=D_y\nabla_w \cU_0(\mu^{(m+1)}_q)(q_i)
=(m+1)D^2_{q_iq_i}U_0^{(m+1)}(q_0,q).
\]
Thus, if $i,j>0$ and $i\neq j$
\begin{align*}
|D^2_{q_iq_j}\tilde u_0^{(m+1)}(q_0,q)|_\infty&\le\frac{C}{m+1}(m+1)|D^2_{q_0q_0}U_0^{(m+1)}(Q_0,Q)|_\infty\frac{C}{m+1}\\
&+\sum_{k=0}^m|D_{q_j}Q_k|_\infty |D^2_{yq_k}u_0^{(m+1)}(Q_0,Q_0,Q)|_\infty |D_{q_i}Q_0|_\infty\\
&+(m+1)|D_{q_0}U_0^{(m+1)}(Q_0,Q)|\frac{C}{(m+1)^2}\\
&+\sum_{k=0}^m |D_{q_j}Q_k|_\infty |D^2_{q_kq_k}u_0^{(m+1)}(Q_0,Q_0,Q)|_\infty |D_{q_i}Q_k|_\infty\\
&+\sum_{k\neq l}^m |D_{q_j}Q_l|_\infty |D^2_{q_kq_l} u_0^{(m+1)}(Q_0,Q_0,Q)|_\infty |D_{q_i}Q_k|_\infty\\
&+\sum_{k=0}^m |D_{q_k} u_0^{(m+1)}(Q_0,Q_0,Q)| |D^2_{q_iq_j}Q_k|_\infty
\end{align*}
Let us recall that by our assumptions, there exists $C=C(T,r,K)$ such that 
\begin{align*}
&|D^2_{q_0q_0}U_0^{(m+1)}(Q_0,Q)|_\infty\le\frac{C}{m+1},\ \ |D^2_{yq_k}u_0^{(m+1)}(Q_0,Q_0,Q)|_\infty\le\frac{C}{m+1},\\
&|D^2_{q_kq_l} u_0^{(m+1)}(Q_0,Q_0,Q)|_\infty\le\left\{
\begin{array}{ll}
\frac{C}{m+1}, & k=l,\\
\frac{C}{(m+1)^2}, & k\neq l,
\end{array}
\right.\\
&|D_{q_k} u_0^{(m+1)}(Q_0,Q_0,Q)|\le \frac{C}{m+1}
\end{align*}
and by Lemma \ref{lem:Q_0-prop} and by the assumptions on $U_0^{(m+1)}$, 
$$|D_{q_0}U_0^{(m+1)}(Q_0,Q)|\le\frac{C}{m+1}.$$
Therefore, combining the previous arguments and computations, we conclude that 
\begin{align*}
|D^2_{q_iq_j}\tilde u_0^{(m+1)}(q_0,q)|_\infty\le\frac{C}{(m+1)^2}.
\end{align*}
Similar arguments yield that if $i=j$, we have
\begin{align*}
|D^2_{q_iq_i}\tilde u_0^{(m+1)}(q_0,q)|_\infty\le\frac{C}{m+1}.
\end{align*}

Computations  and arguments to the one's above yield that 
$$|D^2_{q_0q_0}\tilde u_0^{(m+1)}(q_0,q)|_\infty\le C\ \ {\rm{and}}\ \ |D^2_{q_0q_k}\tilde u_0^{(m+1)}(q_0,q)|_\infty\le \frac{C}{m+1}, \ {\rm{if}}\ k>0,$$
and so the thesis of the claim follows.

(3) Let us set $v_0:=D_pH(Q_0,(m+1)\nabla_{x_0}U^{(m+1)}(s,Q_0,Q))$. First, we have
\begin{align}\label{eq:v_0-deriv}
D_{q_i}v_0&=D^2_{pq}H(Q_0,(m+1)\nabla_{x_0}U^{(m+1)}(s,Q_0, Q))D_{q_i}Q_0\\
\nonumber&+D^2_{pp}H(Q_0,(m+1)\nabla_{q_0}U^{(m+1)}(s,Q_0, Q))(m+1)\sum_{k=0}^mD^2_{q_0q_k}U^{(m+1)}(s,Q_0,Q))D_{q_i}Q_k,
\end{align}
from where using the assumptions \eqref{ass:Hamiltonian} and \eqref{ass:LipschitzH} on $H$, Lemma \ref{lem:reg-Q} and the properties of $D^2_{x_0x_k}U^{(m+1)}$, we obtain
\begin{align*}
|D_{q_i}v_0|_\infty&\le \frac{C}{m+1}
+\frac{C}{m+1}+ (m+1)\sum_{k=1}^m|D^2_{q_0q_k}U^{(m+1)}(s,Q_0,\dots,Q_m))|_\infty |D_{q_i}Q_k|_\infty\\
&\le\frac{C}{m+1}, \ \ {\rm{if}}\ i>0.
\end{align*}
The very same computation and arguments yield that $|D_{q_0}v_0|_\infty\le C.$

Now, we compute
\begin{align}\label{eq:D_x_iV}
D_{q_i}V^{(m+1)}(q_0,q)&=\int_0^t \big(D_yL(Q_0,v_0)D_{q_i}Q_0+D_v L(Q_0,v_0)D_{q_i}v_0\big) ds
\end{align}
Using the smoothness property and the assumptions \eqref{ass:Hamiltonian} and \eqref{ass:LipschitzH} on $L$, together with Lemma \ref{lem:Q_0-prop}, we have that there exists $C=C(T,r,K)$ such that $|Q_0(s,\cdot)|\le C$ and $|\dot Q_0(s,\cdot)|\le C$ for all $s\in(0,t)$, and so $|D_y L(Q_0,v_0)|\le C$ and  $|D_v L(Q_0,v_0)|\le C$. Therefore, by combining all the previous arguments, the thesis of the claim follows.

\medskip
(4) From \eqref{eq:D_x_iV} one obtains
\begin{align}\label{eq:V-double-deriv}
& D^2_{q_iq_j}V^{(m+1)}(q_0,q)\\
\nonumber&=\int_0^t \big(D_{q_j}Q_0D^2_{yy}L(Q_0,v_0)D_{q_i}Q_0+D_{q_j}v_0D^2_{yv}L(Q_0,v_0)D_{q_i}Q_0+D_yL(Q_0,v_0)D^2_{q_iq_j}Q_0\big)ds\\
\nonumber&+\int_0^t \big(D_{q_j}Q_0D^2_{vy} L(Q_0,v_0)D_{x_i}v_0+D_{q_j}v_0D^2_{vv} L(Q_0,v_0)D_{q_i}v_0+D_vL(Q_0,v_0)D^2_{q_iq_j}v_0\big)ds
\end{align}
We first notice that by the arguments from (3), we have that there exists a constant $C=C(T,r,K)$ such that $|Q_0(s,\cdot)|\le C$ and $|v_0(s,\cdot)|\le C$ for all $s\in(0,t)$, and so $|D^2_{yy} L(Q_0,v_0)|\le C$,  $|D^2_{yv} L(Q_0,v_0)|\le C$ and $|D^2_{vv} L(Q_0,v_0)|\le C$.

To conclude, from \eqref{eq:v_0-deriv} we compute
\begin{align*}
&D^2_{q_iq_j}v_0=D_{q_j}Q_0D^3_{pqq}H(Q_0,(m+1)D_{q_0}U^{(m+1)}(s,Q_0, Q))D_{q_i}Q_0\\
&+(m+1)\sum_{k=0}^mD^2_{q_0q_k}U^{(m+1)}(s,Q_0, Q))D_{q_j}Q_kD^3_{pqp}H(Q_0,(m+1)D_{q_0}U^{(m+1)}(s,Q_0,Q))D_{qx_i}Q_0\\
&+D^2_{pq}H(Q_0,(m+1)D_{q_0}U^{(m+1)}(s,Q_0, Q))D^2_{q_iq_j}Q_0\\
&+D_{q_j}Q_0D^3_{ppq}H(Q_0,(m+1)D_{q_0}U^{(m+1)}(s,Q_0, Q))(m+1)\sum_{k=0}^mD^2_{q_0q_k}U^{(m+1)}(s,Q_0, Q))D_{q_i}Q_k\\
&+D^2_{pp}H(Q_0,(m+1)D_{q_0}U^{(m+1)}(s,Q_0, Q ))(m+1)\sum_{k,l=0}^mD_{q_j}Q_lD^3_{q_0q_kq_l}U^{(m+1)}(s,Q_0, Q))D_{q_i}Q_k\\
&+D^2_{pp}H(Q_0,(m+1)D_{q_0}U^{(m+1)}(s,Q_0, Q ))(m+1)\sum_{k=0}^mD^2_{q_0q_k}U^{(m+1)}(s,Q_0,\dots,Q_m))D^2_{q_iq_j}Q_k,
\end{align*} 
From here, using the assumptions \eqref{ass:LipschitzH} and \eqref{hyp:H_3} on $H$, the estimates on $D^2_{q_0q_k}U^{(m+1)}$, on $D^3_{q_0q_kq_l}U^{(m+1)}$ and Lemma \ref{lem:reg-Q}, we obtain that there exists $C=C(T,r,K)>0$ such that 
\begin{align*}
|D^2_{q_iq_j}v_0(q_0,q)|_\infty\le\left\{
\begin{array}{ll}
C, & i=j=0,\\
\ds\frac{C}{m+1}, & (i=j\ {\rm{and}}\ i>0) \ {\rm{or}}\ (i\cdot j=0 \ {\rm{and}}\ \max\{i,j\}>0),\\
\ds\frac{C}{(m+1)^2}, & i\neq j.
\end{array}
\right.
\end{align*}
Combining this with the previous arguments and with \eqref{eq:V-double-deriv} the thesis of the claim follows.
\end{proof}

\begin{lemma}\label{lem:reg-Q} For $m \in \mathbb N$ and $q=(q_0,\dots,q_m) \in (\M)^{m+1}$, let 
\[ 
\mu^{(m+1)}_q:={1\over (m+1)} \sum_{i=0}^m \delta_{q_i}, \  Q_i(s, q):=S_s^t[\mu^{(m+1)}_q](q_i), \  P_i(s, q):={1\over (m+1)} P_s^t[\mu^{(m+1)}_q](q_i) \quad 0 \leq i \leq m.
\]
We set $U_0^{(m+1)}(q):=\cU_0(\mu^{(m+1)}_q)$ and $F^{(m+1)}(q):=\cF(\mu^{(m+1)}_q).$
Further assume $U_0^{(m+1)}$ and $F^{(m+1)}$ satisfy Property \ref{def:app_reg_estim}(3). Then (as in Theorem \ref{thm:app_regularity-m}) for $r>0$ and $t>0$,  there exists $C=C(t,r)$ such that for all $q\in\B_r^{(m+1)}$, $s\in(0,t)$ and $i,j\in\{0,\dots,m\}$ we have
\begin{equation}\label{estim:Q1}
|D_{q_j}Q_i(s,q)|_\infty\le
\left\{
\begin{array}{ll}
C, & i=j,\\[3pt]
\frac{C}{(m+1)}, & i\neq j.
\end{array}
\right.
\end{equation}
and
\begin{equation}\label{estim:Q2}
|D^2_{q_kq_j}Q_i(s,q)|_\infty\le 
\left\{
\begin{array}{ll}
C, & i=j=k,\\[3pt]
\frac{C}{(m+1)}, & i=j\neq k,\ i\neq j=k, \ i=k\neq j,\\[3pt]
\frac{C}{(m+1)^2}, & i\neq j\neq k.
\end{array}
\right.
\end{equation}
\end{lemma}
\begin{proof} 
Let $\xi(\cdot,z)=(\xi_0(\cdot,z),\dots,\xi_m(\cdot,z))$ be defined as in \eqref{eq:define-xi-eta} (see also the systems in \eqref{eq:app_Hamiltonian_sys-discrete} and \eqref{eq:app_Hamiltonian_sys-m}). By Proposition \ref{prop:homeo1} we first observe that that $\xi(t,\cdot)^{-1}=S_0^{t,m}$. To facilitate the writing, as it is done in Appendix \ref{sec:app-reg}, we denote  $\zeta(t,\cdot):=\xi^{-1}(t,\cdot)$ and so, we have
$$Q_i(s,q)=\xi_i(s,\zeta(t,q)).$$
Thus, by differentiating and using the estimates on $(\xi_0,\dots,\xi_m)$ and $(\zeta_0,\dots,\zeta_m)$ from Theorem \ref{thm:app_regularity-m}, by denoting $|\cdot|_\infty:=\|\cdot\|_{L^\infty(\B_r^{(m+1)}}$, we have that there exists $C=C(t,r)$ such that
\begin{align*}
|D_{q_j}Q_i(s,\cdot)|_{\infty}&\le\sum_{k=0}^m |D_{z_k}\xi_i(s,\zeta_0(t,\cdot),\dots,\zeta_m(t,\cdot))|_\infty | D_{q_j}\zeta_k(t,\cdot)|_\infty\\
&=|D_{z_i}\xi_i(s,\zeta_0(t,\cdot),\dots,\zeta_m(t,\cdot))|_\infty | D_{q_j}\zeta_i(t,\cdot)|_\infty\\
&+\sum_{k\neq i} |D_{z_k}\xi_i(s,\zeta_0(t,\cdot),\dots,\zeta_m(t,\cdot))|_\infty | D_{q_j}\zeta_k(t,\cdot)|_\infty
\leq 
\left\{
\begin{array}{ll}
C, & i=j,\\[3pt]
\frac{C}{m+1}, & i\neq j.
\end{array}
\right.
\end{align*}
Therefore, \eqref{estim:Q1} follows. Furthermore, since 
\begin{align*}
D^2_{q_kq_j}Q_i(s,\cdot)&=\sum_{l_1,l_2=0}^m D^2_{q_{l_2}q_{l_1}}\xi_i(s,\zeta_0(t,\cdot),\dots,\zeta_m(t,\cdot)D_{q_k}\zeta_{l_2}(t,\cdot) D_{q_j}\zeta_{l_1}(t,\cdot)\\
&+\sum_{l_1=0}^m D_{z_{l_1}}\xi_i(s,\zeta_0(t,\cdot),\dots,\zeta_m(t,\cdot)) D^2_{q_kq_j}\zeta_{l_1}(t,\cdot)\\
&=\sum_{l_1\neq l_2}^m D^2_{q_{l_2}q_{l_1}}\xi_i(s,\zeta_0(t,\cdot),\dots,\zeta_m(t,\cdot)D_{q_k}\zeta_{l_2}(t,\cdot) D_{q_j}\zeta_{l_1}(t,\cdot)\\
&+\sum_{l=0}^m D^2_{q_{l}q_{l}}\xi_i(s,\zeta_0(t,\cdot),\dots,\zeta_m(t,\cdot)D_{q_k}\zeta_{l}(t,\cdot) D_{q_j}\zeta_{l}(t,\cdot)\\
&+\sum_{l_1=0}^m D_{z_{l_1}}\xi_i(s,\zeta_0(t,\cdot),\dots,\zeta_m(t,\cdot)) D^2_{q_kq_j}\zeta_{l_1}(t,\cdot),\\
\end{align*}
we have that \eqref{estim:Q2} follows.
\end{proof}

\begin{lemma}\label{lem:Q_0-prop}
Let us suppose that we are in the setting of Lemma \ref{lem:u_alt-reg} and in particular all of its assumptions are in place.
Let $(Q_i)_{i=0}^{m}$ be defined in \eqref{eq:m-discrete-flow}. Let $(q_0,q)\in\M^{(m+1)}$. Then $(0,t)\ni s\mapsto Q_0(s,q_0,q)$ is Lipschitz continuous with a Lipschitz constant independent of $m$ and for all $r>0$ and $K\subset\M$ compact there exists $C=C(t,K,r)>0$ such that $|Q_0(s,q_0,q)|\le C$ for all $s\in(0,t)$, whenever $(q_0,q)\in\B_r^{(m+1)}$ and $q_0\in K$.
\end{lemma}

\begin{proof} 
Let us notice that $(Q_0(s,q_0,q))_{s\in(0,t)}$ solves \eqref{eq:flow}, with data $\s^t_s[\mu_q^{(m+1)}]$ and final condition $q_0$. Furthermore, since $(\s^t_s[\mu_q^{(m+1)}])_{s\in(0,t)}$ belongs to $\cB_{\b(t,r)}$, for some $\b(t,r)>0$, the velocity field $(0,t)\times\M\ni (s,y)\mapsto D_pH(y,\nabla_w\cU(s,\s^t_s[\mu_q^{(m+1)}](y)))$ is globally Lipschitz continuous (after a suitable extension of $\nabla_w\cU(s,\s^t_s[\mu_q^{(m+1)}](\cdot)$). Therefore, classical results in the theory of ODEs imply the thesis of the lemma and the bound on $Q_0(s,\cdot,\cdot)$ depends only on $t, K$ and on the Lipschitz constant of the previously mentioned velocity field (hence on $r$).
\end{proof}

\begin{lemma}\label{lem:last-estimates_time}
Under the assumptions of Theorem \ref{thm:main-scalar}, $u^{(m)}$ defined in \eqref{def:u_alt} satisfies the estimates \eqref{estim:time_space} and \eqref{estim:time2} from Lemma \ref{lem:u_alt-reg}.
\end{lemma}

\begin{proof}
In Lemma \ref{lem:last-estimates} we showed that $u^{(m)}(t,\cdot,\cdot)\in C^{1,1}_{\rm{loc}}(\M^{m+1})$ with the corresponding derivative estimates \eqref{estim:space1} and \eqref{estim:space2}, uniformly in with respect to $t\in[0,T]$. Furthermore, since by Proposition \ref{prop:representu}(v), $u(\cdot,q,\mu)$ is Lipschitz continuous for all $q,\mu\in\M\times\sP_2(\M)$, this property is inherited by $u^{(m)}$, and therefore $u^{(m)}(\cdot,q_0,q)$ is Lipschitz continuous on $[0,T]$ for all $(q_0,q)\in\M^{m+1}$. 

Let us recall now the representation formula \eqref{def:u_alt_m} of $u^{(m)}(t,q_0,q)$. We fix $K$ to be the closure of a bounded open set in $\M$ and $r>0$ such that $\mu^{(m+1)}_q\in\B^{m+1}_r$.  The regularity properties of $u^{(m)}$ and \eqref{def:u_alt_m} for almost every $t\in(0,T)$ and all $(q_0,q)\in\M^{m+1}$ yield
\begin{align}\label{eq:partial_t_um}
&\partial_t u^{(m)}(t,q_0,q) + D_{q_0}u^{(m)}(t,q_0,q)\cdot D_p H(q_0,(m+1)D_{q_0}U^{(m+1)}(t,q_0,q))\\
\nonumber&+\sum_{j=1}^mD_{q_j}u^{(m)}(t,q_0,q)\cdot D_p H(q_j,(m+1)D_{q_j}U^{(m+1)}(t,q_0,q))\\
\nonumber&=L(q_0,D_p H(q_0,(m+1)D_{q_0}U^{(m+1)}(t,q_0,q)))+f^{(m+1)}(q_0,q_0,q).
\end{align}
Proposition \ref{prop:representu}(iii) and \eqref{eq:Um_and_u_m} yield
$$(m+1)D_{q_0}U^{(m+1)}(t,q_0,q)=\nabla_w\cU(t,\mu_q^{(m+1)})(q_0)=D_{q_0}u(t,q_0,\mu_q^{(m+1)}).$$
Now, let us notice that by the definition of $u^{(m)}$, one has the identity
$$D_{q_0}u^{(m)}(t,q_0,q)=D_{q_0}u(t,q_0,\mu_q^{(m+1)})+\frac{1}{m+1}\nabla_w u(t,q_0,\mu^{(m+1)}_q)(q_0).$$
For an arbitrary $a\in\M$, if we set in $\hat u^{(m+1)}(t,a,q_0,q):=u(t,a,\mu^{(m+1)}_q)$, we have that 
$$\frac{1}{m+1}\nabla_w u(t,q_0,\mu^{(m+1)}_q)(q_0)=D_{q_0}\hat u^{(m+1)}(t,a,q_0,q)|_{a=q_0}$$
and so
\begin{align}\label{iden:important}
(m+1)D_{q_0}U^{(m+1)}(t,q_0,q)=D_{q_0}u(t,q_0,\mu_q^{(m+1)})= D_{q_0}u^{(m)}(t,q_0,q) - D_{q_0}\hat u^{(m+1)}(t,q_0,q_0,q).
\end{align}
We notice furthermore that $\hat u^{(m+1)}$ (with respect to the regularity and derivative estimates) essentially behaves as $u^{(m+1)}(t,q_0,q_0,q)$ and in particular by \eqref{estim:space1} and \eqref{estim:space2} there exists a constant $C=C(K,r)>0$ such that 
\[
|D_{q_0}\hat u^{(m+1)}(t,q_0,q_0,q)|\le \frac{C}{m+2}.
\]
All these arguments allow us conclude that 
\[
|(m+1)D_{q_0}U^{(m+1)}(t,q_0,q)|\le C.
\]
Now, we differentiate \eqref{eq:partial_t_um} with respect to the spacial variables.

Differentiating with respect to $q_0$, denoting the variables of $f^{(m+1)}$ as $(y_0,q_0,q)$, we find that there exists $C=C(T,K,r)$ such that if $(t,q_0,q)\in[0,T]\times\B^{(m+1)}_r$ with $q_0\in K$, then
\begin{align*}
|D_{q_0}\partial_t u^{(m)}|&\le |D^2_{q_0q_0}u^{(m)}| |D_pH(q_0,(m+1)D_{q_0}U^{(m+1)})|+ |D_{q_0}u^{(m)}| |D^2_{qp}H(q_0,(m+1)D_{q_0}U^{(m+1)})|\\
&+ (m+1)|D_{q_0}u^{(m)}| |D^2_{pp}H(q_0,(m+1)D_{q_0}U^{(m+1)})| |D^2_{q_0q_0}U^{(m+1)}|\\
&+\sum_{j=1}^m|D^2_{q_0q_j}u^{(m)}||D_pH(q_j,(m+1)D_{q_j}U^{(m+1)})| + I + II
\end{align*}
where, 
\begin{align*}
I&:=\sum_{j=1}^m|D_{q_j}u^{(m)}| |D^2_{pp}H(q_j,(m+1)D_{q_j}U^{(m+1)})|(m+1)|D^2_{q_0q_j}U^{(m+1)}|\\
&+|D_{q_0}L(q_0,D_p H(q_0,(m+1)D_{q_0}U^{(m+1)}))|+|D_{v}L(q_0,D_p H(q_0,(m+1)D_{q_0}U^{(m+1)}))| |D^2_{qp}H|
\end{align*}
and 
\begin{align*}
II&:=|D_{v}L(q_0,D_p H(q_0,(m+1)D_{q_0}U^{(m+1)}))| |D^2_{pp}H|(m+1)|D^2_{q_0q_0}U^{(m+1)}|+|D_{y_0}f^{(m+1)}(q_0,q_0,q)|\\
&+|D_{q_0}f^{(m+1)}(q_0,q_0,q)|
\end{align*}
Thus, using \eqref{estim:space1}, \eqref{estim:space2} and the estimates on $U^{(m+1)}$ from Theorem \ref{thm:app_regularity-m}, as well as the hypotheses on the data $H$ and $f^{(m+1)}$ we have 
\begin{align*}
|D_{q_0}\partial_t u^{(m)}| 
&\le C+C\left(\sum_{j=1}^m m|D^2_{q_0q_j}u^{(m)}|^2\right)^\frac{1}{2}\left(\sum_{j=1}^m\frac1m\Big{|}D_pH(q_j,(m+1)D_{q_j}U^{(m+1)})\Big{|}^2\right)^{\frac12}\\
&+C+\sum_{i=0}^m\frac{1}{\sqrt{m+1}}\sqrt{m+1}|D_{q_i}f^{(m+1)}|\\
&\le C +\left(\sum_{i=0}^m\frac{1}{m+1}\right)^{\frac12}\left(\sum_{i=1}^m(m+1)|D_{q_i}f^{(m+1)}|^2\right)^{\frac12}\le C,
\end{align*}
This yields the first part of \eqref{estim:time_space}, since 
$$ D_pH\Big(\cdot,\nabla\cU\big(t,\mu^{(m+1)}_q\big)(\cdot)\Big) \in L^2(\mu^{(m+1)}_q),$$ 
with and $L^2(\mu^{(m+1)}_q)$ uniformly bounded with respect to $m$.

If $k\in\{1,\dots,m\}$, completely parallel computation gives
\begin{align*}
|D_{q_k}\partial_t u^{(m)}|&\le |D^2_{q_kq_0}u^{(m)}| |D_pH(q_0,(m+1)D_{q_0}U^{(m+1)})|\\
&+(m+1)|D_{q_0}u^{(m)}| |D^2_{pp}H(q_0,(m+1)D_{q_0}U^{(m+1)})| |D^2_{q_kq_0}U^{(m+1)}|\\
&+\sum_{j=1}^m|D^2_{q_kq_j}u^{(m)}| |D_p H(q_j,(m+1)D_{q_j}U^{(m+1)})|\\
&+|D_{q_k}u^{(m)}||D^2_{qp} H(q_k,(m+1)D_{q_k}U^{(m+1)})|\\
&+\sum_{j=1}^m|D_{q_j}u^{(m)}|(m+1)|D^2_{pp}H(q_j,(m+1)D_{q_j}U^{(m+1)})| |D^2_{q_kq_j}U^{(m+1)}|\\
&+|D_vL(q_0,D_p H(q_0,(m+1)D_{q_0}U^{(m+1)}))| |D^2_{pp}H|(m+1)|D^2_{q_kq_0}U^{(m+1)}| +|D_{q_k}f^{(m+1)}|\\ 
&\le C |D^2_{q_kq_0}u^{(m)}| +\frac{C}{(m+1)}|D_p H(q_k,(m+1)D_{q_k}U^{(m+1)})| +\frac{C}{(m+1)}\\
&+ |D_{q_k}f^{(m+1)}|,
\end{align*}
from where, using the same arguments as for the conclusion of the first part of \eqref{estim:time_space}, we find $\sum_{k=1}^m(m+1)|D_{q_k}\partial_t u^{(m)}|^2\le C$, as desired.

\smallskip

To show \eqref{estim:time2}, we argue similarly.  First, from \eqref{eq:partial_t_um} we simply have
\begin{align*}
&|\partial_t u^{(m)}|\le |D_{q_0}u^{(m)}| |D_pH(q_0,(m+1)D_{q_0}U^{(m+1)})|+\sum_{j=1}^m|D_{q_j}u^{(m)}| |D_p H(q_j,(m+1)D_{q_j}U^{(m+1)})|\\
&+|L(q_0,D_pH(q_0,(m+1)D_{q_0}U^{(m+1)}))| +|f^{(m+1)}|\\
&\le C + \left(\sum_{j=1}^m m|D_{q_j}u^{(m)}|^2\right)^{\frac12} \left(\sum_{j=1}^m\frac{1}{m}|D_p H(q_j,(m+1)D_{q_j}U^{(m+1)})|^2\right)^{\frac12}\le C,
\end{align*}
where we used the previous estimates and the fact that $H(q_0,D_{q_0}u^{(m)})$ and $f^{(m+1)}$ are locally bounded.

Second, differentiating  \eqref{eq:partial_t_um} with respect to $t$, we find
\begin{align*}
|\partial_{tt}^2 u^{(m)}|&\le |\partial_t D_{q_0}u^{(m)}||D_pH(q_0,(m+1)D_{q_0}U^{(m+1)})|+ |D_{q_0}u^{(m)}| |D^2_{pp}H| (m+1)|\partial_t D_{q_0}U^{(m+1)})|\\
&+\sum_{j=1}^m|\partial_tD_{q_j}u^{(m)}||D_pH(q_j,(m+1)D_{q_j}U^{(m+1)})|\\
&+\sum_{j=1}^m|D_{q_j}u^{(m)}||D^2_{pp}H(q_j,(m+1)D_{q_j}U^{(m+1)})|(m+1)|\partial_t D_{q_j}U^{(m+1)}|\\
&+|(m+1)D_{q_0}U^{(m+1)}| |D^2_{pp}H|(m+1)|\partial_t D_{q_0}U^{(m+1)})|\\
&\le C+\left(\sum_{j=1}^m(m+1)|\partial_tD_{q_j}u^{(m)}|^2\right)^{\frac12}\left(\sum_{j=1}^m\frac{1}{(m+1)}|D_pH(q_j,(m+1)D_{q_j}U^{(m+1)})|^2\right)^{\frac12}\\
&+C(m+1)|\partial_t D_{q_0}U^{(m+1)})|\\
&+C\left(\sum_{j=1}^m (m+1)|\partial_t D_{q_j}U^{(m+1)}|^2\right)^{\frac12}
\end{align*}
Let us notice that by \eqref{iden:important} we have that 
\begin{align*}
(m+1)|\partial_tD_{q_0}U^{(m+1)}|\le |\partial_t D_{q_0}u^{(m)}|+|\partial_tD_{q_0}\hat u^{(m+1)}|\le C + \frac{C}{\sqrt{m+2}},
\end{align*}
where we have used that $\sum_{j=0}^{m}(m+2)|\partial_tD_{q_0}\hat u^{(m+1)}|^2\le C$. Relying on the previously obtained estimates and on the fact that by Theorem \ref{thm:app_regularity-m}(3), 
$$\sum_{j=1}^m (m+1)|\partial_t D_{q_j}U^{(m+1)}|^2\le C,$$ 
the claim in \eqref{estim:time2} follows.
\end{proof}

Recall that throughout this section, we have imposed that  \eqref{ass:F-U0-C11new1}-\eqref{ass:convexity-on-tildeL} and \eqref{eq:collectUF} hold. We are ready to state and prove the main theorem of this section.

\begin{theorem}\label{thm:main-scalar} Suppose the assumptions \eqref{ass:F-U0-C11new1} through \eqref{hyp:D_qH} are satisfied. Then, the scalar master equation \eqref{eq:master} has a unique global in time classical solution of class $C^{1,1}_{\rm{loc}}([0,+\infty)\times\M\times\cP_2(\M))$ in sense of Definition \ref{def:classical_sol}.

\end{theorem}

\begin{proof}
Let $T>0$ be a fixed time horizon. Notice that Theorem \ref{thm:main-regularity} yields that $u$ defined in \eqref{def:u_alt} is of class $C^{1,1}_{\rm{loc}}([0,T]\times\M\times\cP_2(\M))$.

Let $\mu\in\sP_2(\M)$, $q\in\M$ and $t\in (0,T)$. Using the representation formula \eqref{def:u_alt}, by the dynamic programming principle we have that for $s\in (0,t)$
\begin{align*}
u(t,q,\mu)
=&u\big(s,S_s^t[\mu](q),\sigma_s^t[\mu]\big)\\
+&\int_s^t L\Big(S_\tau^t[\mu](q),D_pH\big(S_\tau^t[\mu](q),D_q u(\t,S_\t^t[\mu](q),\sigma_\t^t[\mu])\big)\Big)+f\Big(S_\tau^t[\mu](q),\sigma_s^t[\mu]\Big) d\tau.
\end{align*}
Hence, 
\begin{align*}
&\lim_{s\to t}\frac{u(t,q,\mu) - u\big(s,S_s^t[\mu](q),\sigma_s^t[\mu]\big)}{t-s}\\
=& \lim_{s\to t}\int_s^t L\Big(S_\tau^t[\mu](q),D_pH\big(S_\tau^t[\mu](q),D_q u(\t,S_\t^t[\mu](q),\sigma_\t^t[\mu])\big)\Big)+f\Big(S_\tau^t[\mu](q),\sigma_s^t[\mu]\Big) d\tau,
\end{align*}
where both limits exist and are finite, due to the continuity of the integrand on the right hand side. 
Using the chain rule with respect to the measure variable (provided in Lemma \ref{lem:chain_rule}), this is equivalent to 
\begin{align*}
\partial_t u(t,q,\mu) +& D_q u(t,q,\mu) \cdot D_pH(q,D_q u(t,q,\mu)) + \int_{\M}\nabla_wu(t,q,\mu)(y)\cdot D_p H\big(y,\nabla_w \cU\big(s,\mu\big)(y)\big)\mu(dy)\\
=& L\big(q,D_pH\big(q,D_q u(t,q,\mu)\big)\big)+f(q,\mu)
\end{align*}
Here above we used that the optimal curve $\t\mapsto S^t_\t[\mu](q)$ satisfies \eqref{eq:jan02.2020.0}, while the curve $\t\mapsto \s^t_\t[\mu]$ solves the continuity equation \eqref{continuity}.

Using that by Proposition \ref{prop:representu}(ii)
$$D_q u(t,\cdot ,\mu)=\nabla_w\cU(t,\mu)(\cdot) \;\; \mu - \text{a.e.}, $$ 
one obtains
\begin{align*}
f(q,\mu)&=\partial_t u(t,q,\mu)+D_qu(t,q,\mu)\cdot D_pH(q,D_q u(t,q,\mu))\\
&+\int_{\M}\nabla_w u(t,q,\mu)(y)\cdot D_p H(y,D_q u(t,y,\mu))\dd\mu(y)-L(q,D_pH(q,D_q u(t,q,\mu)))\\
&=\partial_t u(t,q,\mu)+H(q,D_q u(t,q,\mu))+\int_{\M}\nabla_w u(t,q,\mu)(y)\cdot D_p H(y,D_q u(t,y,\mu))\mu(dy),
\end{align*}
where we have used the Legendre duality in the last equation. The arguments in Subsection \ref{subsec:further_vectorial_eq} imply in particular that $u$ also satisfies the condition \eqref{cond:uniqueness}. This completes the existence part of the theorem.

{\it Uniqueness.} Let $u\in C^{1,1}_{\rm{loc}}([0,T]\times\M\times\cP_2(\M))$ be a solution to \eqref{eq:master}. Let $t\in(0,T)$, $\mu\in\cP_2(\M)$ and $z\in\bH$ be fixed such that $\sharp(z)=\mu$. Using the vector field $D_pH(\cdot,D_q u(\cdot,\cdot,\cdot))$, let $(\s_s)_{s\in(0,t)}$ be the unique solution to the continuity equation
\begin{equation}\label{cont:uniqueness}
\left\{
\begin{array}{ll}
\partial_s \s_s+\nabla\cdot(\s_s D_pH(\cdot,D_q u(s,\cdot,\s_s)))=0, & {\rm{in}}\ \mathcal{D}'((0,t)\times\M),\\
\s_t=\mu. 
\end{array}
\right.
\end{equation}
Since $D_q u$ is locally Lipschitz on $[0,T]\times\M\times\cP_2(\M)$ and the vector field $\M\ni q\mapsto D_q u(t,q,\nu)$ is Lipschitz, uniformly with respect to $(t,\nu)\in[0,T]\times\cB_r$, the existence and uniqueness of $\s$ above follows from standard arguments and from the adaptation of Theorem 3.3 from \cite{GangboS2014}.

Then, in $\bH$ we consider the ODE 
\begin{align}\label{eq:uniqueness_ODE}
\left\{
\begin{array}{ll}
x_s' = D_p H(x_s, D_q u(s,x_s,\s_s)), & s\in (0,t),\\
x_t=z. 
\end{array}
\right.
\end{align}
This has a unique continuously differentiable solution $x:(0,t)\to\bH$. 

{\it Claim 1.} We have that $\sharp(x_s)=\s_s$. 

{\it Proof of Claim 1.} Indeed, let us denote $\ov \s_s:=\sharp(x_s)$ we have 
$$\partial_s \ov\s_s+\nabla\cdot(\ov\s_s D_pH(\cdot,D_q u(s,\cdot,\s_s)))=0,$$
in the sense of distributions. But the vector field $(s,q)\mapsto D_pH(q,D_q u(s,q,\s_s))$ induces a unique solution to the the continuity equation, therefore $\s$ and $\ov\s$ must coincide and the claim follows.

\medskip
{\it Claim 2.} The unique solution $x$ to \eqref{eq:uniqueness_ODE}, satisfies the Euler-Lagrange equations 
$$D_q L(x_s,x_s')+\nabla\tilde\cF(x_s)=\frac{d}{ds}D_vL(x_s,x_s')\quad \text{and}\quad D_v L(x(0),x'(0))=\nabla\tilde\cU_0(x(0))\quad \text{a.e. in}\quad\Om.$$

{\it Proof of Claim 2.} Let us notice first that by our assumptions $D_v L(q,\cdot)$ and $D_p H(q,\cdot)$ are inverses of each others for all $q\in\M$. Furthermore, we have
\begin{align*}
D_q L(q,D_pH(q,p))=-D_q H(q,p),\ \ \forall (q,p)\in\M\times\R^d.
\end{align*}
Indeed, this last equation is a consequence of the Legendre-Fenchel identity
$$H(q,p)=p\cdot D_pH(q,p)-L(q,D_pH(q,p)).$$ 
Now, from \eqref{eq:uniqueness_ODE} by continuity, by \eqref{hyp:u_0-f} and by the fact $\nabla_w\cU_0(\s_s)(x_s)=\nabla\tilde\cU_0(x_s)$, one can deduce that 
$$x'(0)=D_pH(x(0),D_qu_0(x(0),\sigma_0))=D_pH(x(0),\nabla_w\cU_0(\s_0)(x(0)))=D_pH(x(0),\nabla\tilde\cU_0(x(0))),$$
which by inversion of $D_pH(x(0),\cdot)$ is equivalent to $D_vL(x(0),x'(0))=\nabla\tilde\cU_0(x(0))$.

Then, from \eqref{eq:uniqueness_ODE}, again by inversion of $D_pH(x_s,\cdot)$ we have 
\begin{align*}
D_vL(x_s,x_s')=D_q u(s,x_s,\s_s).
\end{align*}
Since $u\in C^{1,1}_{\rm{loc}}([0,T]\times\M\times\cP_2(\M))$, for a.e. $s\in(0,t)$ we have
\begin{align}\label{eq:EL-uniqueness1}
\frac{d}{ds}D_vL(x_s,x_s')&=\partial_sD_q u(s,x_s,\s_s)+ D^2_{qq}u(s,x_s,\s_s)D_p H(x_s, D_q u(s,x_s,\s_s))\\
\nonumber&+\int_\M\nabla_w D_qu(s,x_s,\s_s)(a)\cdot D_p H(a, D_q u(s,a,\s_s))\s_s(da)\\
\nonumber&=\partial_sD_q u(s,x_s,\s_s)+ D^2_{qq}u(s,x_s,\s_s)D_p H(x_s, D_q u(s,x_s,\s_s))\\
\nonumber&+\int_\M D_q\nabla_wu(s,x_s,\s_s)(a)\cdot D_p H(a, D_q u(s,a,\s_s))\s_s(da),
\end{align}
a.e. in $\Om$, where we have used \eqref{cond:uniqueness} in the last equation. Let us note that the previous computation is meaningful. Indeed, by the regularity on $u$ (see also the arguments in Subsection \ref{subsec:further_vectorial_eq}), we can differentiate the master equation \eqref{eq:master} with respect to $q$, and so for $\cL^1\otimes\cL^d$--a.e. $(s,q)\in(0,t)\times\M$ and for all $\nu\in\cP_2(\M)$ we have
\begin{align}\label{eq:EL-uniqueness2}
&\partial_sD_q u(s,q,\nu)+D^2_{qq}u(s,q,\nu)D_pH(q,D_qu(s,q,\nu))\\
+&\int_{\M}D_q\nabla_w u(s,q,\nu)(a)D_p H(a,D_qu(s,a,\nu))\nu(da)=D_q f(q,\nu)-D_q H(q,D_qu(s,q,\nu)).\nonumber
\end{align}
We notice that \eqref{hyp:u_0-f} implies that $D_q f(q,\nu)=\nabla_w\cF(\nu)(q)$ and so, by combining \eqref{eq:EL-uniqueness1} and  \eqref{eq:EL-uniqueness2} one deduces
\begin{align*}
\frac{d}{ds}D_vL(x_s,x_s')&=D_q f(x_s,\s_s)-D_q H(x_s,D_qu(s,x_s,\s_s))=\nabla_w\cF(\s_s)(x_s)+D_q L(x_s,D_qu(s,x_s,\s_s))\\
&=\nabla\tilde\cF(x_s)+D_q L(x_s,D_qu(s,x_s,\s_s)),
\end{align*}
and so the claim follows.

\medskip
{\it Claim 3.} For each $t\in[0,T]$ and $\mu\in\cP_2(\M)$, $u(t,\cdot,\mu)$ is uniquely determined on $\spt(\mu)$.

{\it Proof of Claim 3.} By the strict convexity of the action, the previous claims show that $(x_s)_{s\in(0,t)}$ is the unique solution in the action minimization problem \eqref{eq:defineUtilde} for $\tilde\cU(t,z)$. But, since $\tilde\cU\in C^{1,1}_{\rm{loc}}([0,T]\times\bH)$ (as we showed in Proposition \ref{theorem:hilbert-smooth}(ii)), we have in the same time that the optimal velocity for this curve is $D_pH(x_s,\nabla\tilde\cU(s,x_s))$ and so, by the convexity of $H$ in the second variable, one deduces that 
$$D_q u(s,x_s(\omega),\s_s)=\nabla\tilde\cU(s,x_s)(\omega),$$
for a.e.  $\omega\in\Omega$. This further yields that the vector field $q\mapsto D_q u(s,q,\s_s)$ is unique (i.e. does not depend on the solution $u$) on $\spt(\s_s)$ for all $s\in[0,t]$. From here we also deduce that for each $\mu\in\cP_2(\M)$, the solution to the continuity equation \eqref{cont:uniqueness} is unique (independent of the solution $u$) and this corresponds to the unique minimizer in the action minimization problem, i.e. to the solution to \eqref{continuity}.

Now let $q_1\in\spt(\mu)$ and let $(q_s)_{s\in(0,t)}$ be the unique solution to 
\begin{align}\label{eq:uniqueness_ODE2}
\left\{
\begin{array}{ll}
q_s' = D_p H(q_s, D_q u(s,q_s,\s_s)), & s\in (0,t),\\
q_t=q_1. 
\end{array}
\right.
\end{align}
It is clear that $q_s\in\spt(\s_s)$ for all $s\in[0,t]$. Moreover, for each fixed $q_1$, the curve solving \eqref{eq:uniqueness_ODE2} is unique (independent of the solution $u$). 

Using the Legendre duality, the master equation for $u$ can be rewritten as 
\begin{align*}
\partial_s u(s,q,\nu) &+ D_q u(s,q,\nu)\cdot D_p H(q,D_q u(s,q,\nu))+\int_\M\nabla_w u(s,q,\nu)(a)\cdot D_p H(a,D_q u(s,a,\nu))\nu(da)\\
&=f(q,\nu)+ L(q,D_pH(q,D_q u(s,q,\nu)))
\end{align*}
and replacing in $(q,\nu)=(q_s,\s_s)$ the chain rule gives us 
\begin{align}\label{lag:identity}
\frac{d}{ds}\left(u(s,q_s,\s_s)\right)=f(q_s,\s_s)+ L(q_s,D_pH(q_s,D_q u(s,q_s,\s_s))).
\end{align}

\medskip

Now, let $\ov u\in C^{1,1}_{\rm{loc}}([0,T]\times\M\times\cP_2(\M))$ be another solution to \eqref{eq:master} in the sense of Definition \ref{def:classical_sol}. By the previous arguments one has $D_q\ov u(s,q,\s_s)=D_q u(s,q,\s_s)$ for all $s\in[0,t]$ and $q\in\spt(\s_s)$. Then, similarly to \eqref{lag:identity}, one has that
\begin{align}\label{lag:identity2}
\frac{d}{ds}\left(\ov u(s,q_s,\s_s)\right)=f(q_s,\s_s)+ L(q_s,D_pH(q_s,D_q u(s,q_s,\s_s))).
\end{align}
By defining now $w:[0,t]\to\R$ as $w(s):=u(s,q_s,\s_s)-\ov u(s,q_s,\s_s)$ we have that $w'(s)=0$ (by subtracting \eqref{lag:identity2} from \eqref{lag:identity}) and $w(0)=0$. Therefore one must have $w\equiv 0$ and so $u(s,q_s,\s_s)=\ov u(s,q_s,\s_s)$. By continuity one has also that 
$$u(t,q_1,\mu)=\ov u(t,q_1,\mu),\ \ \forall q_1\in\spt(\mu).$$ 

\medskip

{\it Claim 4.} $u$ is a unique solution to \eqref{eq:master}.

{\it Proof of Claim 4.} It remains to show that if $u$ and $\ov u$ are two solutions to \eqref{eq:master}, one has $u(t,q,\mu)=\ov u(t,q,\mu)$ for all $q\in\M\setminus\spt(\mu)$. Suppose that $\mu$ does not have full support, otherwise there is nothing to prove. Let $q_0\in\M\setminus\spt(\mu)$. For $\e>0$ let $\rho_\e$ stand for the heat kernel centered at $0$ with variance $\e>0$ and define $\mu_\e:=\mu*\rho_\e$. Then one obtained a fully supported smooth probability measure $\mu_\e$ such that $W_2(\mu,\mu_\e)\to 0$ as $\e\da 0$. Therefore, we have
$$u(t,q_0,\mu_\e)=\ov u(t,q_0,\mu_\e).$$
By the continuity of both $u$ and $\mu_\e$ with respect to the measure variable, one can pass to the limit as $\e\da 0$ to obtain that
$$u(t,q_0,\mu)=\ov u(t,q_0,\mu),$$
as desired.

\end{proof}

Despite the fact that  the velocity field $v(t,\cdot ):=D_p H\big(\cdot,\nabla_w \cU\big(t,\mu\big)\big)$ appearing in the continuity equation \eqref{continuity} typically does not belong to $T_\mu\cP_2(\M)$, we have the following chain rule (cf. e.g.  \cite{Mayorga} in the compact setting).

\begin{lemma}\label{lem:chain_rule} We assume that the hypotheses  of Theorem \ref{thm:main-scalar} take place. Let $T>0$, $t_0,t\in(0,T)$, $s\in(0,t)$, $q\in\M$ and $\mu\in\cP_2(\M)$ and let $(0,t)\ni s\mapsto\sigma^t_s[\mu]$ be the solution to the continuity equation \eqref{continuity}. Then
$$\lim_{s\to t}\frac{u(t_0,q,\mu) - u\big(t_0,q,\sigma_s^t[\mu]\big)}{t-s}=\int_{\M}\nabla_wu(t_0,q,\mu)(y)\cdot D_p H\big(y,\nabla_w \cU\big(t,\mu\big)(y)\big)\mu(dy).$$
\end{lemma}

\section{Further implications of the scalar master equation}\label{sec:further}

\subsection{Improvements on the notion of weak solution to the vectorial master equation}\label{subsec:further_vectorial_eq}
Let us recall that the first part of Theorem \ref{thm:main-scalar} asserts the existence of $u\in C^{1,1}_{\rm{loc}}([0,T]\times\M\times\cP_2(\M)),$ which satisfies the scalar master equation
\begin{align}\label{eq:master_2}
\partial_t u(t,q,\mu)+H(q,D_q u(t,q,\mu))+\int_{\M}\nabla_w u(t,q,\mu)(y)\cdot D_p H(y,D_q u(t,y,\mu))\mu(dy)=f(q,\mu).
\end{align} 
Let us observe that all the terms in the previous equation are locally Lipschitz continuous with respect to the $q$ variable. Indeed, except the nonlocal term, the Lipschitz continuity of the others is a consequence of the regularity of $u$ and the data. Setting $v(t,y):=D_p H\big(y,\nabla_w \cU\big(t,\mu\big)(y)\big)$ and denoting $\ov v(t,\cdot)$ the projection of $v(t,\cdot)$ onto $T_\mu\cP_2(\M)$, we have that
\begin{align*}
\int_{\M}\nabla_w u(t,q,\mu)(y)\cdot v(t,y)\mu(dy)=\int_{\M}\Phi_1(t,q,\mu,y)\cdot \ov v(t,y)\mu(dy),
\end{align*}
where $\Phi_1$ is defined in Corollary \ref{cor:finite_infinite_reg-scalar-mast-time}. This relationship holds because $\nabla_w u(t,q,\mu)(\cdot)$ is the projection of $\Phi_1(t,q,\mu,\cdot)$ onto $T_\mu\cP_2(\M)$. Since $\Phi_1\in C^{1,1}_{\rm{loc}}([0,T]\times\M\times\cP_2(\M)\times\M),$ the function $q\mapsto \int_{\M}\Phi_1(t,q,\mu,y)\cdot \ov v(t,y)\mu(dy)$ is locally Lipschitz continuous and for (Lebesgue) a.e. $q\in\M$ we have 
\begin{align*}
\int_{\M}D_q\nabla_w u(t,q,\mu)(y)\cdot v(t,y)\mu(dy)=\int_{\M}D_q\Phi_1(t,q,\mu,y)\cdot \ov v(t,y)\mu(dy),
\end{align*}
Therefore, we are allowed  to differentiate \eqref{eq:master_2} for (Lebesgue) a.e. $q\in\M$ to obtain
\begin{align*}
\partial_t D_qu(t,q,\mu)&+D_qH(q,D_q u(t,q,\mu))+D^2_{qq} u(t,q,\mu)D_pH(q,D_q u(t,q,\mu))\\
&+\int_{\M}D_q\nabla_w u(t,q,\mu)(y)\cdot D_p H(y,D_q u(t,y,\mu))\mu(dy)=D_qf(q,\mu).
\end{align*} 
By Proposition \ref{prop:representu}(iii) we know that for all $(t,\mu)\in(0,T)\times\cP_2(\M)$, $D_qu(t,\cdot,\mu)=\nabla_w\cU(t,\mu)(\cdot)$ on $\spt(\mu)$, where $\cU$ is the unique solution to \eqref{eq:HJB}. Since $D_q u$ is locally Lipschitz continuous with respect to all of its variables, it serves a very natural extension for $\nabla_w\cU(t,\mu)(\cdot)$ to the whole space, and so we have
\begin{align}\label{eq:correspondence}
\partial_t \nabla_w\cU(t,\mu)(q)&+D_qH(q,\nabla_w\cU(t,\mu)(q))+D_{q} \nabla_w\cU(t,\mu)(q)D_pH(q,\nabla_w\cU(t,\mu)(q))\\
\nonumber&+\int_{\M}D_q\nabla_w u(t,q,\mu)(y)\cdot D_p H(y,\nabla_w\cU(t,\mu)(y))\mu(dy)=D_qf(q,\mu)=\nabla_w\cF(\mu)(q),
\end{align} 
for all $(t,\mu)\in(0,T)\times\sP_2(\M)$ and for (Lebesgue) a.e. $q\in\M$.

In Theorem \ref{thm:main-vector} we have seen that $\cV:=\nabla_w\cU$ solves the vectorial master equation \eqref{eq:master_vector}, when the variable $q$ needs to be taken in $\spt(\mu)$. Since we have a correspondence between all terms in \eqref{eq:master_vector} and \eqref{eq:correspondence}, except the nonlocal ones, we can deduce that we must have
$$\overline \cN_\mu\big[\cV, \nabla_{w}^\top \cV \big](t, \mu, q)=\overline \cN_\mu\big[\nabla_w\cU, \nabla^2_{ww} \cU^\top \big](t, \mu, q) =\int_{\M}D_q\nabla_w u(t,q,\mu)(y)\cdot D_p H(y,\nabla_w\cU(t,\mu)(y))\mu(dy)$$
for $\cL^d$--a.e. $q\in\M$.

This fact implies furthermore that 
\begin{align}
\int_{\M}D_q\nabla_w u(t,q,\mu)(y)D_pH(y,D_qu(t,y,\mu))\mu(dy)&=\int_{\M}\nabla_w D_q u(t,q,\mu)(y)D_pH(y,D_qu(t,y,\mu))\mu(dy)\\
\nonumber&=\int_{\M}\nabla^2_{ww}\cU(t,\mu)(q,y)D_pH(y,D_qu(t,y,\mu))\mu(dy)
\end{align}
for all $\mu\in\cP_2(\M)$ and for $\cL^1\otimes\cL^d$--a.e. $(t,q)\in(0,T)\times\M$, which  shows in particular that the function $u$ constructed in the first part of the proof of Theorem \ref{thm:main-scalar} satisfies also \eqref{cond:uniqueness}.

All the previous arguments allow to formulate the following
\begin{proposition}
The weak solution $\cV$ to the vectorial master equation \eqref{eq:master_vector} provided in Theorem \ref{thm:main-vector} can be extended in a Lipschitz continuous way to $[0,T]\times\sP_2(\M)\times\M$ such that this extension still solves \eqref{eq:master_vector} at every $(t,\mu)\in(0,T)\times\cP_2(\M)$ and at $\cL^d$--a.e. $q\in\M$.
\end{proposition}

\begin{remark}
Relying on the very same procedure as in Theorems \ref{thm:app_regularity-m} and  \ref{thm:finite_infinite_reg-scalar-mast}, if we assume higher regularity properties on the data (as $H,L\in C^4$ with uniformly bounded fourth order derivatives, $\cF,\cU\in C^{3,1,w}_{\rm{loc}}$ and $f,u_0\in C^{2,1}_{\rm{loc}}$), one can improve further the regularity of both $u$ and $\cU$ (as $u\in C^{2,1}_{\rm{loc}}$ and $\cU\in C^{3,1,w}_{\rm{loc}}$). Such improvements would imply furthermore that one could have the vectorial master equation satisfied for all $q\in\M$ (rather than $\cL^d$--a.e.). We do not pursue the realistic goal of  improving the regularity of $u$ only to avoid writing a longer paper.
\end{remark}

%
%

\medskip

\appendix

\section{Hilbert regularity is too stringent for rearrangement invariant functions}\label{appendixE}
Let $\Phi \in C^2\big(\mathcal P_2(\mathbb M) \big)$ and let $\tilde \Phi \in C^2\big(\mathbb H \big)$ be such that $\Phi(\mu)=\tilde \Phi(x)$ if $\mu$ is the law of $x$. Recall that 
\begin{equation}\label{eq:jan19.2020.1}
\nabla^2 \tilde \Phi(x)(h, h_*)= 
 \int_{\Omega} D_{q}\big(\nabla_w \Phi(\mu) \big)\circ x\; h\cdot h_* d\omega+  \int_{\Omega^2} \nabla^2_{ww} \Phi(\mu)\big(x(\omega), x(\omega_*)\big)h(\omega)\cdot h_*(\omega_*) d\omega d\omega_*
\end{equation}
if $\xi, \xi_* \in T_\mu \mathcal P_2(\mathbb M)$ and $h=\xi \circ x$ and $h_*=\xi_* \circ x.$

For $k\in\N$ and $g \in C^2(\mathbb M^{k})$, we define 
\[
\tilde \Phi_g^{(k)}(x):= \int_{\Omega^k} g\big(x(\omega_1), \cdots, x(\omega_k) \big) d\omega_1 \cdots d\omega_k \qquad \forall x \in \mathbb H, 
\] 
and 
\[
 \Phi_g^{(k)}(\mu):= \int_{\M^k} g(q_1, \cdots, q_k) \mu(dq_1) \cdots \mu(dq_k) \qquad \forall \mu \in \mathcal P_2(\mathbb M) .
\] 
Let $P_k$ be the set of permutations of $k$ letters. Replacing $g$ by its symmetrization 
\[
\tilde g(x_1, \cdots, x_k)={1 \over k!} \sum_{\tau \in P_k}  g(x_{\tau(1)}, \cdots, x_{\tau(k)})
\]
we have $\tilde \Phi_g^{(k)}=\tilde \Phi_{\tilde g}^{(k)}$. Therefore, it is never a loss of generality to assume $g$ is symmetric.

We do not know how to write \eqref{eq:jan19.2020.1} for general $h, h_* \in \bH \setminus \{\xi \circ x \; : \; \xi \in T_\mu \mathcal P_2(\mathbb M)  \}.$ In some particular cases such as when $\tilde \Phi=\tilde \Phi_g^{(k)}$ for some smooth $g$, then \eqref{eq:jan19.2020.1} extends to $h, h_* \in \bH \setminus \{\xi \circ x \; : \; \xi \in T_\mu \mathcal P_2(\mathbb M)  \}.$ This can be checked by hand by writing the Taylor expansion of second order of $$g\big(x(\omega_1)+ h(\omega_1), \cdots, x(\omega_k)+h(\omega_k) \big).$$

Another example is when 
\begin{equation}\label{eq:jan19.2020.1.5}
\Phi(\mu)=\theta \bigg( {1\over 2} \int_{\mathbb M} |q|^2\mu(dq)\bigg) \quad \forall \mu \in \mathcal P_2(\mathbb M)\ \ {\rm{and\ so\ }}\  \tilde\Phi(x)=\theta \left( {1\over 2}\|x\|^2\right) \quad \forall x \in \bH.
\end{equation}
Writing the second order Taylor expansion, we have 
\[
\nabla \tilde \Phi(x)(h)=\theta' \bigg( {1\over 2} \|x\|^2\bigg) (x, h)
\]
and 
\begin{equation}\label{eq:jan19.2020.2}
\nabla^2 \tilde \Phi(x)(h, h)=\theta' \bigg( {1\over 2} \|x\|^2\bigg) \|h\|^2 + \theta'' \bigg( {1\over 2} \|x\|^2\bigg) (x, h)^2 \qquad \forall x , h \in \mathbb H.
\end{equation}
We conclude  
\begin{equation}\label{eq:jan19.2020.2.1}
D_q \big(\nabla_{w} \Phi(\mu) \big)=\theta' \bigg( {1\over 2} \int_{\mathbb M} |q|^2\mu(dq)\bigg) I_{d}  \qquad \forall \mu \in \mathcal P_2(\mathbb M) 
\end{equation}
and 
\[
\nabla^2_{ww} \Phi(\mu)(q, b)= \theta'' \bigg( {1\over 2} \int_{\mathbb M} |q|^2\mu(dq)\bigg) q \otimes b \qquad \forall \mu \in \mathcal P_2(\mathbb M) \quad \forall q, b \in {\rm spt}( \mu).
\]
Thus, when $\Phi$ is of the form \eqref{eq:jan19.2020.1.5},  \eqref{eq:jan19.2020.1} continues to hold for all $h, h_* \in \mathbb H.$ Note that the expression in \eqref{eq:jan19.2020.2.1} is constant on $\M.$ In fact, we shall see this is not a coincidence which is the aim of these notes.

Our goal is to show that if $\tilde \Phi \in C^{2,\alpha}_{\rm loc}\big(\bH \big)$, then $D_{q}\big(\nabla_w \Phi(\mu) \big)$ must be constant function on ${\rm spt}( \mu)$. This will allow us to make inference about the dimension of $C^{2,\alpha}_{\rm loc}\big(\bH \big) \cap \{ \tilde \Phi_g^{(k)}\}$ for any natural number $k$.  In conclusion, the set of $\tilde \Phi \in C^{2,\alpha}_{\rm loc}\big(\bH \big)$ maybe too small in some sense and a theory of mean field games for functions $\tilde \Phi \in C^{2,\alpha}_{\rm loc}\big(\bH \big)$ may be too restrictive. Hence, $C^{2, \alpha,w }_{\rm loc}\big(\mathcal P_2(\mathbb M) \big)$ (cf. Definition \ref{def:c21_wasserstein}) is a better space for a general theory. 
\begin{lemma}\label{lem:weak-Holder} Let $\alpha \in (0,1]$ and assume $\tilde \Phi \in C^{2,\alpha}_{\rm loc}\big(\bH \big)$ is rearrangement invariant so that it is the lift of a function $\Phi.$  If \eqref{eq:jan19.2020.1} holds for all $h, h_* \in \mathbb H$ then $D_{q}\big(\nabla_w \Phi(\mu) \big)$ is constant function on ${\rm spt}( \mu)$. 
\end{lemma}
\proof{} Let $x \in \bH$ and let $\mu$ be the law of $x$. Fix an open ball $\bB \subset \bH$ that contains $x$ and choose $\kappa_{\bB}>0$ such that  
\begin{equation}\label{eq:jan18.2020.1}
\Big(\nabla^2 \tilde {\Phi}(x)-\nabla^2 \tilde {\Phi}(y) \Big)(h, h_*) \leq \kappa_{\bB} \|x-y\|^{\alpha}
\end{equation}
for all $y \in \bB$ and all $h, h_* \in \bH$ such that $\|h\|, \|h_*\| \leq 1.$

Let $\varrho\in C_{ c}^\infty(\M)$ be a probability density function whose support is the unit ball in $\R^d.$  For $z, z_* \in \R^d$ unit vectors and for $\omega, o\in \Omega$ , we set 
\[
h^\epsilon=z \sqrt{\varrho_\epsilon^{o}}, \quad h_*^\epsilon=z_* \sqrt{\varrho^{o}_\epsilon}, \qquad \varrho_\epsilon^{ o}(\omega):=\epsilon^{-d}\varrho\Big({\omega-o\over \epsilon} \Big).
\]
Let $y \in \bH$ have the same law with $x$. We have 
\begin{align}
& \Big(\nabla^2 \tilde \Phi(y)-\nabla^2 \tilde \Phi(x)\Big)(h^\epsilon, h_*^\epsilon)\nonumber\\
= & 
 \int_{\Omega} \Big(D_{q}\big(\nabla_w \Phi(\mu) \big)\big(y(\omega)\big)-D_{q}\big(\nabla_w \Phi(\mu) \big)\big(x(\omega)\big) \Big)h(\omega)\cdot h_*(\omega)\nonumber\\
+ & 
\int_{\Omega^2} \Big(\nabla^2_{ww} \Phi(\mu)\big(y(\omega), y(\omega_*)\big)-\nabla^2_{ww} \Phi(\mu)\big(x(\omega), x(\omega_*)\big) \Big)h(\omega)\cdot h_*(\omega_*) d\omega d\omega_*\nonumber\\
= & 
 \int_{\Omega} \Big(D_{q}\big(\nabla_w\Phi(\mu) \big)\big(y(o+\epsilon a)\big)-D_{q}\big(\nabla_w \Phi(\mu) \big)\big(x(o+\epsilon a)\big) \Big)z \cdot z_* \varrho(a)  da\label{eq:jan18.2020.1.5}\\
+ & 
\epsilon^d\int_{\Omega^2} \nabla^2_{ww} \Phi(\mu)\big(y(o+\epsilon a), y(o+\epsilon b)\big)z \cdot z_* \sqrt{\varrho(a) \varrho(b)} da db\nonumber\\
-& \epsilon^d\int_{\Omega^2} \nabla^2_{ww} \Phi(\mu)\big(x(o+\epsilon a), x(o+\epsilon b)\big) z \cdot z_* .\nonumber
\sqrt{\varrho(a) \varrho(b)} da db.
\end{align}
Since $\tilde \Phi \in C^{1,1}\big(\bB \big),$ $\nabla^2_{ww} \Phi(\mu)$ is bounded, we use \eqref{eq:jan18.2020.1.5}  to obtain  that if $o$ is a Lebesgue point for $\big({D_q}\nabla_w \Phi(\mu) \big)\circ y$ and $\big({D_q}\nabla_w \Phi(\mu) \big)\circ x$ then 
\[
\lim_{\epsilon \rightarrow 0} \Big(\nabla^2 \Phi(y)-\nabla^2 \Phi(x)\Big)(h^\epsilon, h_*^\epsilon)=
 \Big(D_{q}\big(\nabla_w \Phi(\mu) \big)\big(y(o)\big)-D_{q}\big(\nabla_w \Phi(\mu) \big)\big(x(o)\big) \Big)z \cdot z_* 
\]
This, together with \eqref{eq:jan18.2020.1} implies that if $y \in \bB$ then 
\begin{equation}\label{eq:jan18.2020.2}
|D_{q}\big(\nabla_w \Phi(\mu) \big)\circ y(o)-D_{q}\big(\nabla_w \Phi(\mu) \big)\circ x(o)| \leq \kappa_{\bB} \|x-y\|^{\alpha}
\end{equation}

In the spirit of the proof of Lemma \ref{lem:c11_equiv}, set 
\[
\Omega_0:=\bigl\{\omega \in \Omega\; | \; \omega \;\; \text{is a Lebesgue point for}\;\; x, {D_q}\nabla_w \Phi(\mu)\circ x \bigr\} \cap x^{-1}({\rm spt\,}(\mu))
\]
Note that $\Omega_0$ is a set of full measure in $\Omega$ and so, $x(\Omega_0)$ is a set of full $\mu$--measure. In fact, we do not know that $x(\Omega_0)$ is Borel, but we can find a Borel set $A\subset x(\Omega_0)$ of full $\mu$--measure.

 Assume in the sequel that $o \in A$ and set $q_1:=x(o)$. Assume we can find $\overline o \in A$ such that $q_2=x(\overline o)\not =q_1.$  Let $r>0$ small such that $B_r(o)\cap B_r(\overline o)=\emptyset$. Set

\[
S_r(\omega):=
\left\{
\begin{array}{ll}
\omega, &\hbox{if} \; \omega \in \Omega \setminus \bigl(B_r(o) \cup B_r( \overline o) \bigr),\\
\omega -o+ \overline o,&\hbox{if} \; \omega \in  B_r(o),\\
\omega -\overline o+ o,&\hbox{if} \; \omega \in B_r(\overline o). 
\end{array}
\right.
\]
Since $S_r$ preserves Lebesgue measure, $x$ and $y:=x \circ S_r$ have the same law $\mu$. We notice 
$$\|x-y\|^2=2\int_{B_r(o)}|x(\omega)-x(\omega+\overline o-o)|^2dz$$ 
and so, for $r$ small enough, $y \in \bB.$ By \eqref{eq:jan18.2020.2} implies
\begin{align}
\Big| D_{q}\big(\nabla_w \Phi(\mu) \big)(q_2) -D_{q}\big(\nabla_w \Phi(\mu) \big)(q_1)\Big|
=&\Big| D_{q}\big(\nabla_w \Phi(\mu) \big)\circ y(o) -D_{q}\big(\nabla_w \Phi(\mu) \big)\circ x(o)\Big| \nonumber\\
\leq & \kappa_{\bB} \bigg(2\int_{B_r(o)}|x(z)-x(z+\overline o-o)|^2dz\bigg)^{\alpha \over 2}\nonumber.
\end{align}
We let $r$ tend to $0$ to conclude the proof. 
\endproof

%
%
%
\begin{proposition}\label{prop:locald} For any $\alpha \in (0, 1]$ and 
$k\in\N$, we have
\[
{\rm dim}\left(C^{2,\alpha}_{\rm loc}(\bH) \cap \big\{ \tilde \Phi_g \; : \; g\in C^{2, \alpha}_{\rm loc}(\M^k), \;\; \|D^2 g\|_{L^\infty}<\infty \big\}\right)<\infty.
\]
\end{proposition}
\proof{} 
We aim to use Lemma \ref{lem:weak-Holder}, since this asserts that $D_q\nabla_w\Phi_g(\mu)(q)$ is a constant matrix $C(\mu)$ which depends only on $\mu$.

In particular, in the case of $k=1$, we have $D_q\nabla_w\Phi_g(\mu)(q)=D^2g(q)$ and this being constant implies that $g$ is a polynomial of degree $2,$ os the claim follows. 

For $k\in\N$ general we have  
\begin{align*}
D_q\nabla_w\Phi_g(\mu)(q)&=\int_{\M^{k-1}}D^2_{q_1q_1}g(q,q_2,\dots,q_k)\mu(dq_2)\dots\mu(dq_k)\\
&+\dots\\
&+\int_{\M^{k-1}}D^2_{q_kq_k}g(q_1,q_2,\dots,q_{k-1},q)\mu(dq_1)\dots\mu(dq_{k-1})
\end{align*}

In fact by  \cite{ChowGan} 
\begin{align}
C(\mu)=D_q\nabla_w\Phi_g(\mu)(q)&=k\int_{\M^{k-1}}D^2_{q_1q_1}g(q,q_2,\dots,q_k)\mu(dq_2)\dots\mu(dq_k)\label{eq:on-the-support1}\\
&=\dots\nonumber\\
&=k\int_{\M^{k-1}}D^2_{q_kq_k}g(q_1,q_2,\dots,q_{k-1},q)\mu(dq_1)\dots\mu(dq_{k-1}) \quad \mu-\text{a.e.}\label{eq:on-the-support2}
\end{align}
For simplicity, let us set $k=2$ (the proof of the result for general $k\in\N$ follows the same lines). Let $a \in \M$ and $\varrho \in C_b(\M)$ has $\M$ as its support is a probability density and $\varrho_\epsilon$ is its standard rescaled function. The measures $\varrho_\epsilon(q-a)$ have the whole $\M$ as their support and so,
\[
\int_{\M}D^2_{q_1q_1}g(q,q_2)\varrho_\epsilon(q_2-a) dq_2= \int_{\M}D^2_{q_1q_1}g(\overline q,q_2)\varrho_\epsilon(q_2-a) dq_2 \qquad \forall q, \overline q \in \M.
\]
Letting $\epsilon$ tend to $0$ we conclude 
\[
D^2_{q_1q_1}g(q, a)= D^2_{q_1q_1}g(\overline q,a).
\]
In fact, 
\[
D^2_{q_1q_1}g(q, a)= D^2_{q_1q_1}g(\overline q,a)=D^2_{q_2q_2}g(a, q)= D^2_{q_2q_2}g(a, \overline q)=C(a).
\]

From these arguments, one can conclude that both $q_1\mapsto D^2_{q_1q_1}g(q_1,a)$ and $q_2\mapsto D^2_{q_2q_2}g(a,q_2)$ are constants for all $a\in\M$, therefore the $q_1\mapsto g(q_1,a)$ and $q_2\mapsto g(a,q_2)$ are polynomials of degree at most two for all $a\in\M$. By an adaptation of the result of \cite{Carroll} we conclude that $g$ needs to be a polynomial of degree at most two. The result follows.

\endproof

\begin{corollary}
Similarly, for the example in \eqref{eq:jan19.2020.1.5}, if $\tilde\Phi\in C^{2,\alpha}_{\rm{loc}}(\bH)$, then by Lemma \ref{lem:weak-Holder} and \eqref{eq:jan19.2020.2.1} we have that $\theta(t)=c_0 t$ for some $c_0\in\R$.
\end{corollary}

The result from Proposition \ref{prop:locald} in case of $k=1$ is the consequence of the proposition below, where we show that assuming even only $C^2$ regularity (instead of $C^{2,\alpha}$) for functionals on $\bH$ having local representations might result in trivialities.

\begin{proposition}\label{prop:finite-d-regularity}
$$C^{2}(\bH) \cap \big\{ \tilde \Phi_g \; : \; g\in C^{3}(\M), \;\; \|D^2 g\|_{L^\infty}<\infty,\ \|D^3 g\|_{L^\infty}<\infty,\ D^3g\not\equiv 0\big\}=\emptyset.$$
and so, 
\[
C^{2}(\bH) \cap \big\{ \tilde \Phi_g \; : \; g\in C^{3}(\M), \;\; \|D^2 g\|_{L^\infty}<\infty,\ \|D^3 g\|_{L^\infty}<\infty\big\}
\]
is a finite dimensional space.
\end{proposition}

\begin{proof}
For simplicity, let us suppose that $d=1$ and so $\Om=[0,1].$ The result in higher dimensions follows from similar arguments.

For $x,y\in\bH$ we can write the following expansion for $\tilde\Phi_g$
\begin{align}\label{equ:exp}
\int_\Om g(y(\omega))d\omega-\int_\Om g(x(\omega))&d\omega - \int_\Om g'(x(\omega))(y(\omega)-x(\omega))d\omega-\frac{1}{2}\int_\Om g''(x(\omega))(y(\omega)-x(\omega))^2d\omega\\
\nonumber&=\int_\Om\int_0^1\int_0^1\int_0^1t^2s g'''(x(\omega)+ts\tau(y(\omega)-x(\omega)))(y(\omega)-x(\omega))^3d\tau ds dt d\omega.
\end{align}
By the assumptions on $g'''$, there exist constants $c_0,c_1$, having the same sign, such that on a bounded open interval $c_0\le g'''\le c_1$. Without loss of generality, let us suppose that this open interval is $(-1,1)$ and $0<c_0<c_1$.   

\smallskip

{\it Claim.} The right hand side of \eqref{equ:exp} is not of order $o(\|x-y\|^2)$ when $x \equiv 0.$

\smallskip

{\it Proof of the Claim.} Let $x(\omega)=0$ and $y_n(\omega)=\omega^n$ for $\omega\in\Om$ and $n\in\N$. Then clearly $\|y_n\|^2=\frac{1}{2n+1}\to 0$, as $n\to+\infty.$ We write the previous expansion for $y_n$ and $x$. In particular, the remainder satisfies 
\begin{equation}\label{equ:contradiction}
\frac{c_0}{6}\int_\Om y_n^3(\omega)d\omega\le \int_\Om\int_0^1\int_0^1\int_0^1t^2s g'''(ts\tau y_n(\omega))y_n^3(\omega)d\tau ds dt d\omega\le \frac{c_1}{6}\int_\Omega y_n^3(\omega)d\omega.
\end{equation}
We easily find $\int_0^1 y_n^3(\omega)d\omega=\frac{1}{3n+1}$. Therefore dividing \eqref{equ:contradiction} by $\|y_n\|^2$ and taking $n\to+\infty$ we find
$$\frac{2c_0}{18}\le \lim_{n\to+\infty}\frac{1}{\|y_n\|^2}\int_\Om\int_0^1\int_0^1\int_0^1t^2s g'''(ts\tau y_n(\omega))y_n^3(\omega)d\tau ds dt d\omega \le \frac{2c_1}{18}.$$ 
The claim follows and so does the thesis of the proposition.
\end{proof}

\medskip

%
%
\section{Convexity versus displacement convexity}\label{app:convex_dispalcement}

%
%
%
%
%
%
%
%
%
%
\vskip0.40cm
\subsection{Displacement convexity versus classical convexity}\label{subsec:conv_displ}

Using the terminology of \cite{CardaliaguetDLL}, in the remaining of this section will consider weakly  Fr\'echet continuously differentiable functions $\cV: \cP_2(\M) \to \R$ and denote their weak Fr\'echet differentials as ${\delta \cV \over \delta \mu}: \R^d \times \cP_2(\M) \rightarrow \R.$ Let $\phi_1, \phi \in C^2(\M)$ be  functions of bounded second derivatives such that $\phi_1$ is even. Set 
\[
\cV_1(\mu):={1\over 2} \int_{\R^d} \phi_1 * \mu(q)\mu(dq),  \qquad \mu \in \cP_2(\M).
\]
and 
\[
 \cV(\mu):=\cV_1(\mu)+ \int_{\R^d} \phi(q)\mu(dq) , \qquad \mu \in \cP_2(\M).
\]
%
%
%
\begin{remark}\label{rem:warning} Recall from \cite{CardaliaguetDLL} that ${\delta \cV \over \delta \mu}$ is monotone if and only if  $\cV$ is convex in the classical sense. Furthermore, the function $\cV_1$ is  twice weakly  Fr\'echet continuously differentiable function, and 
\[
{\delta \cV_1 \over \delta \mu}(q, \mu)= (\phi_1 *\mu)(q), \ \ {\delta \cV \over \delta \mu}(q, \mu)= (\phi_1 *\mu)(q) + \phi(q),
\]
and 
\[ 
{\delta^2 \cV_1 \over \delta \mu^2}(q, y, \mu)={\delta^2 \cV \over \delta \mu^2}(q, y, \mu) =\phi_1(q-y). 
\]
\end{remark}
\begin{lemma}\label{lem:fourier-of-phi1} If we further assume $\phi_1 \in L^1(\M)$ then ${\delta \cV \over \delta \mu}$ is monotone if and only if  the Fourier transform $\phi_1$ is nonnegative.
\end{lemma}
\proof{}   Denote the Fourier transform of $\phi_1$ as $\hat \phi_1$. Note that for any $f \in L^2(\M)$ by Young's inequality we have $\phi_1 * f \in L^2(\M)$ and so $f (\phi_1 * f) \in L^1(\M).$ By the Riemann-Lebesgue lemma $\hat \phi_1 \in C_0(\M)$. Furthermore, $\hat \phi_1$ is even and has its range contained in the set of real numbers. By Remark \ref{rem:warning}  ${\delta \cV \over \delta \mu}$ is monotone if and only if $\cV_1$ is convex. Thus, using the expression of ${\delta^2 \cV_1 \over \delta \mu^2}$ in Remark \ref{rem:warning}  we conclude that ${\delta \cV \over \delta \mu}$ is monotone if and only if for any $ f \in C(\M) \cap L^2(\M)$ such that $\int_{\M} f(q)dq=0$ we have $0\leq \int_{\mathbb R^{d}} (\phi_1 \ast f) (q) f(q)dq.$ Thanks to Plancherel theorem, ${\delta \cV \over \delta \mu}$ is monotone if and only if 
\[
0\leq \int_{\mathbb R^{d}} \widehat{\phi_1 \ast f} (\xi) \hat f^*(\xi)d\xi=   \int_{\mathbb R^{d}} \hat{\phi_1}(\xi) \hat f(\xi) \hat f^*(\xi)d\xi= \int_{\mathbb R^{d}} \hat{\phi_1}(\xi) |\hat f(\xi)|^2d\xi.
\]
This concludes the proof of the lemma. \endproof

\begin{lemma}\label{lem:fourier-of-phi2} Assume $\lambda>0$, $\lambda_1 \in (-\lambda/2, \lambda/2),$ $\phi$ is $\lambda$--convex and $\phi_1$ is $\lambda_1$--convex. Then
\begin{enumerate} 
\item[(i)] $\cV$ is $\kappa$--displacement convex, hence  displacement convex, where $\kappa:=\lambda-2|\lambda_1|>0 .$
\item[(ii)] If we further assume $\phi_1$ is nonnegative, $\phi_1 \equiv 1$ on the unit ball, and $\phi_1 \equiv 0$ outside the ball of radius $2$, centered at the origin, then $\cV$ fails to be convex in the classical sense.
\end{enumerate}
\end{lemma}
\proof{} (i) As above, denote the Fourier transform of $\phi_1$ as $\hat \phi_1$. Let $\sigma \in AC_2(0,1; \cP_2(\M))$ be a geodesic such that its velocity $v$ is not identically null. Since $ \|v_t\|_{\sigma_t}$ is independent of $t$, it is then positive.  We have 
\begin{align*}
{d^2 \over dt^2} \cV(\sigma_t) &=\int_{\M} D^2 \phi(q) v_t(q) \cdot v_t(q) \sigma_t(dq) + \int_{\M^{2}} D^2 \phi_1(q-w) v_t(q) \cdot v_t(q) \sigma_t(dq) \sigma_t(dw)\\
& + \int_{\M^{2}} D^2 \phi_1(q-w) v_t(q) \cdot v_t(w) \sigma_t(dq) \sigma_t(dw)\\
&\geq \lambda \|v_t\|^2_{\sigma_t} +\lambda_1 \|v_t\|^2_{\sigma_t} -|\lambda_1| \|v_t\|^2_{\sigma_t} \geq \kappa \|v_t\|^2_{\sigma_t}.
\end{align*} 
This completes the verification of (i).

(ii) Since $\phi_1$ is even the range of its Fourier transform is contained in the set of real numbers (including negative ones).  Assume on the contrary that the range of $\hat \phi_1$ is contained in $[0,\infty)$. By Fourier inversion theorem we have for $x \in \M$, 
\[
|\phi_1(x)|=\Bigg{|}\int_{\M}\hat\phi_1(\xi)e^{2\pi i x\cdot\xi} d\xi\Bigg{|}\le \int_{\M}|\hat\phi_1(\xi)| d\xi= \int_{\M}\hat\phi_1(\xi) d\xi=\phi_1(0).
\]
Since $\phi_1(x) \equiv 1=\phi_1(0)$ on $B_1(0)$, the ball of center $0$ and radius $1$ we must have 
\begin{equation}\label{eq:oct18.2019.1}
\hat\phi_1(\xi)\cos({2\pi x\cdot\xi}) \equiv |\hat\phi_1(\xi)| \equiv \hat\phi_1(\xi) \qquad \forall (x, \xi) \in B_1(0) \times \M.
\end{equation}
Since $\phi_1$ is not the null function, $\hat \phi_1$ cannot be the null function. Choose $\xi_0$ such that $\hat \phi_1(\xi_0)>0$ and since $\hat\phi_1$ is continuous, assume without loss of generality that $\xi_0 \not =0.$ By \eqref{eq:oct18.2019.1}, $\cos({2\pi x\cdot\xi_0})=1$ for all $x \in B_1(0)$ which yields a contradiction. One concludes the proof of (ii) by Lemma \ref{lem:fourier-of-phi1}.
\endproof

\vskip0.40cm
\subsection{Convexity versus displacement convexity of the action}
 
Here we would like to emphasize the fact that imposing the joint convexity assumption on the Lagrangian action, as in \eqref{ass:convexity-on-tildeL} comes as a natural assumption for displacement convex potential mean field games, which are considered in this manuscript. We compare this to the more standard monotonicity assumption in potential MFG.

Assume $L, H \in C^1(\M\times\R^{d})$ are such that $H(q, \cdot)$ and $L(q, \cdot)$ are Legendre transform of each other. We consider the actions 
\[
\cA_0^T(\sigma, v):= \int_0^T \left( \int_{\mathbb M} L(q, v_t(q))\sigma_t(dq) + \cF(\sigma_t)\right) dt 
\]
over the set of pairs $(\sigma, v)$ such that 
\begin{equation}\label{eq:jan25.2020.1}
\partial_t \sigma +\nabla \cdot (\sigma v)=0 \qquad \mathcal D'\big((0,T) \times \mathbb M \big)
\end{equation}

Recall that if we set $\nabla_q f(q, \mu):=\nabla_w F(\mu)(q)$ then $f$ monotone means $\cF$ is convex. 

We can rewrite  $\cA_0^T(\sigma, v)$ in terms of the momentum by setting 
\[
\ov \cA_0^T(\sigma, \eta):= \int_0^T \biggl( \int_{\mathbb M} L\Big(q, {d\eta_t \over d\sigma_t}(q)\Big)\sigma_t(dq) + \cF(\sigma_t)\biggr) dt 
\]
over the set of pairs $(\sigma, \eta)$ such that $|\eta_t|\ll\sigma_t$ and 
\begin{equation}\label{eq:jan25.2020.2}
\partial_t \sigma +\nabla \cdot \eta=0 \qquad \mathcal D'\big((0,T) \times \mathbb M \big).
\end{equation}

In fact, for each $q \in \M$ we  introduce the function $\ov L_q: \R \times \R^{d} \rightarrow \R \cup\{\infty\}$ defined as 
\begin{equation}\label{eq:lagrangian-q}
\ov L_q(\rho, e):=
\left\{
\begin{array}{cl}
\rho L(q, \frac{e}{\rho})   &\text{if } \rho>0 \\
0   &\text{if } \rho=0, e=\vec 0 \\
+\infty   &\text{otherwise .}
\end{array}
\right.
\end{equation}
Here, $\vec 0:=(0, \ldots, 0).$ Since $L_q$ is homogeneous of degree $1$, whenever $\mu$ is a probability measure and $\xi_1, \cdots, \xi_d$ are signed Borel measures, the following function is well defined
\[
(\mu, \xi) \mapsto A(\mu, \xi):= 
\left\{
\begin{array}{ll}
\ds\int_{\M} L_q\Big(\mu(dq), {d\xi}\Big)    &\text{if } \; |\xi|\ll\mu  \\
+\infty   &\text{if } |\xi| \not\ll \mu 
\end{array}
\right.
\]
We can now extend the definition of  $\ov \cA_0^T$ over $\mathcal C$ to obtain 
\[
\ov \cA_0^T(\sigma, \eta):= \int_0^T \Bigl( A(\sigma_t, \eta_t) + \cF(\sigma_t)\Bigr) dt.
\]

\begin{lemma}\label{lem:convex-action} If $\cF$ is convex on  $\mathcal P_2(\mathbb M)$ then  $(\mu, \xi) \mapsto A(\mu, \xi)+ \cF(\mu)$ is convex (we do not assume $L$ is jointly convex).
\end{lemma}
\proof{} It suffices to show that $(\mu, \xi) \mapsto A(\mu, \xi)$ is convex. The proof of this well-known fact can be found for instance in \cite{Santambrogio}, Proposition 5.18.

%
 \endproof
 
Let $\mathcal C$ be the set of $(\sigma, \eta)$ such that $\sigma \in AC_2(0,T; \mathcal P_2(\mathbb M))$ and $t\mapsto \eta_t\in  \mathcal M(\mathbb M) \times \cdots \times \mathcal M(\mathbb M)$ is a Borel path of vector fields such that each one of its $d$ components is a signed Borel measure on $\mathbb M$ and 
\begin{equation}\label{eq:jan25.2020.3}
\partial_t \sigma +\nabla \cdot \eta=0 \qquad \mathcal D'\big((0,T) \times \mathbb M \big).
\end{equation}
 
\begin{remark}\label{rem:convex-action} 

(i) Note that the classical theory of potential mean field games which consists in assuming that $f$ is monotone and $L, H \in C^1(\M\times\R^{d})$ are such that $H(q, \cdot)$ and $L(q, \cdot)$ are Legendre transform of each other ensures that $(\mu, \xi) \mapsto A(\mu, \xi)+ \cF(\mu)$ is a convex function. Therefore, if we extend the definition of  $\ov \cA_0^T$ to obtain 
\[
\ov \cA_0^T(\sigma, \eta):= \int_0^T \Bigl( A(\sigma_t, \eta_t) + \cF(\sigma_t)\Bigr) dt 
\]
over $\mathcal C$, the action $\ov \cA_0^T$ is a convex function in the variables $(\sigma, \eta)$.

(ii) When replacing the assumption of convexity on the action by an assumption of displacement convexity, as it is done in this manuscipt, it seems natural to impose that $\cA_0^T(\sigma, v)$ is displacement convex on the set of pairs $(\sigma, v)$ satisfying \eqref{eq:jan25.2020.1}. This means that 
\[
\bH \times \bH\ni (X, V) \mapsto \int_{\Omega} L(X, V)d\omega +\tilde \cF(X) \quad {\rm{is\ convex}},
\]
and thus the Lagrangian $L$ is assumed to be jointly convex on $\M\times\R^d$. 
\end{remark}

\smallskip

\subsection{Convexity of $f(\cdot,\mu)$ is a consequence of the displacement convexity of $\cF$}

To study the scalar master equation, among others we have imposed the assumptions \eqref{ass:convexity} and \eqref{hyp:u_0-f} on the functions $f$ and $\cF$. As we have detailed in the previous couple of lines, in our setting it is natural for the Lagrangian $L$ to impose joint $\lambda$--convexity, and we impose that $\cF$ is displacement $\lambda$--convex. We show below that in this sense, imposing \eqref{ass:convexity}, i.e. that $f(\cdot,\mu)$ is $\lambda$--convex, is also natural, and it is a consequence of the displacement $\lambda$--convexity of $\cF$. 

\begin{proposition}\label{prop:conv}
Let $\cF:\sP_2(\M)\to\R$ and $f:\M\times\sP_2(\M)\to \R$ be of class $C^2$ such that they are related via \eqref{hyp:u_0-f}. We assume that $\cF$ is is displacement $\lambda$-convex; $\M\times\sP_2(\M)\ni (q,\mu)\mapsto D_q\nabla_w\cF(\mu)(q)=D^2_{qq}f(q,\mu)$ is continuous and that for any $\cK\subset\sP_2(\M)$ compact, there exists $C=C(\cK)>0$ such that $|D^2_{ww}\cF(\mu)(q_1,q_2)|\le C$ for any $\mu\in\cK$ and for any $q_1,q_2\in\spt(\mu)$.

Then, for any $\mu\in\sP_2(\M)$, the function $\spt(\mu)\ni q\mapsto f(q,\mu)$ is $\l$-convex, i.e. 
$$D^2_{qq}f(x,\mu)\ge\l I_d,\ \forall\ q\in\spt(\mu).$$
\end{proposition}

\begin{proof}[Proof of Proposition \ref{prop:conv}]
Let $m\in\N$ and we define $F^{(m)}:(\M)^m\to\R$ as $F^{(m)}(q_1,\dots,q_m):=\cF(\mu^{(m)}_q)$. By the assumptions on $\cF$, we have that $F^{(m)}$ is twice differentiable on $(\M)^m$ and by Lemma \ref{lem:con_conv}, it is $\frac{\l}{m}$-convex on $(\M)^m$. This means in particular that 
$$D^2F^{(m)}(q_1,\dots,q_m)\ge\frac{\l}{m}I_{md},\ \forall\ (q_1,\dots, q_m)\in(\M)^m$$
or equivalently 
$$a^\top D^2F^{(m)}(q_1,\dots,q_m)a\ge\frac{\l}{m}|a|^2_{md},\ \forall\ a\in\M^{m},(q_1,\dots, q_m)\in(\M)^m,$$
where $|\cdot|_{md}$ stands for the standard Euclidean norm on $\M^m$. For $i\in\{1,\dots,m\}$, let us choose the vector $a\in\M^{m}$ such that its coordinates between the indices $d(i-1)+1$ and $di$ are not all zero, while all the others are zero. Then, the previous inequality implies that
\begin{equation}\label{convexity}
D^2_{q_{i}q_{i}}F^{(m)}(q_1,\dots,q_m)\ge\frac{\l}{m} I_d, \forall\ (q_1,\dots,q_m)\in(\M)^m.
\end{equation}

\medskip

We also have (see for instance in \cite{ChowGan} Remark 3.5(iv)) that
$$mD^2_{q_iq_i}F^{(m)}(q_1,\dots,q_m)=D_q\nabla_w\cF(\mu^{(m)}_q)(q_i)+\frac{1}{m}\nabla^2_{ww}\cF(\mu^{(m)}_q)(q_i,q_i),$$
$\forall m\in\N,\ \{q_1,\dots,q_m\}\subseteq\spt(\mu^{m}_q)$.

Let $b\in\M$. By \eqref{convexity}, one has that 
$$b^\top D_q\nabla_w\cF(\mu^{(m)}_q)(q_i)b+\frac{1}{m}b^\top \nabla^2_{ww}\cF(\mu^{(m)}_q)(q_i,q_i)b\ge \l |b|^2_{d}, \ \forall m\in\N,\ \{q_1,\dots,q_m\}\subseteq\spt(\mu^{m}_q).$$
Now let us fix $\mu\in\sP_2(\M)$ and $q_1\in\spt(\mu)$. For $m\ge 2$ natural number, let $q_i\in\spt(\mu)$, $i\in\{2,\dots,m\}$, and let us build $\mu^{(m)}_q:=\sum_{i=1}^m\d_{q_i}$, as an approximation of $\mu$. 

We have that 
$$b^\top D_q\nabla_w\cF(\mu^{(m)}_q)(q_1)b+\frac{1}{m}b^\top \nabla^2_{ww}\cF(\mu^{(m)}_1)(q_1,q_1)b\ge \l |b|^2_{d}.$$
Since $\cK:=\left\{\mu^{(m)}_q:m\in\N\right\}\cup\{\mu\}$ is a compact set, by the assumptions, $\nabla^2_{ww}\cF(\mu^{(m)}_q)(q_1,q_1)$ is uniformly bounded by a constant $C=C(\cK)>0$ independent of $m$. By the continuity of $D_q\nabla_w\cF$, one can pass to the limit in the previous inequality to obtain
$$b^\top D_q\nabla_w\cF(\mu)(q_1)b\ge \l |b|^2_{d},$$
and equivalently 
$$b^\top D^2_{qq}f(q_1,\mu) b\ge \l |b|^2_{d}.$$
By the arbitrariness of $b\in\R^d$ and $q_1\in\spt(\mu)$, the thesis of the proposition follows.
\end{proof}

\smallskip

\subsection{Failure of smoothness of solutions to Hamilton-Jacobi equation for monotone initial data}

It is well-known in the theory of Hamilton-Jacobi equations on finite dimensional spaces that typically one cannot expect global existence of smooth solutions. This led to the development of the notion of viscosity solution by Crandall-Lions and Evans. We emphasize below that this phenomenon of existence of non-smooth solutions to Hamilton-Jacobi equations is also present on $\cP_2(\M).$

Let us consider $d=1$. Let $L:\R\times\R\to\R$ and $\phi:\R\to\R$ be defined as  
\[
L(q, v):={|v|^2 \over 2}, \quad \phi(q):= -\sqrt{1+ q^2}.
\]
Set 
\[
\cU_*(\mu):= \int_{\R} \phi(q)\mu(dq), \quad u_*(q, \mu)=\phi(q), \qquad \cL(\mu, \xi):= \int_{\R} L(q, \xi(q))\mu(dq).
\]
Note that $\cU_*$ is convex and so, $u_*$ is monotone. 

Let $\cU:[0,\infty) \times \mathcal P_2(\R)$ be the unique viscosity solution to the Hamilton--Jacobi equation
\begin{equation}\label{eq:jan26.2020.1}
\partial_t \cU +{1\over 2} \int_{\R} |\nabla_w \cU|^2\mu(dq)=0, \qquad \cU(0, \cdot)=\cU_*.
\end{equation}
Assume on the contrary that $\cU$ is of class $C^1$. Then $\cU$ must satisfy \eqref{eq:jan26.2020.1} pointwise and so,  its restriction defined as 
\[
u(t, q)=\cU(t, \delta_q)
\]
must be a $C^1$ function satisfying 
\begin{equation}\label{eq:jan26.2020.1b}
\partial_t u +{1\over 2}  |\partial_q u|^2=0, \qquad u(0, \cdot)=\phi.
\end{equation}
Thus, 
\begin{equation}\label{eq:jan26.2020.2}
u(t, q)= \min_y \Big\{{|y-q|^2 \over 2t} + \phi(y)\; : \; y \in \R \Big\}.
\end{equation}
Given $q$ the minimum in \eqref{eq:jan26.2020.2} is attained by $y$ such that 
\begin{equation}\label{eq:jan26.2020.3}
{y-q \over t} -{y\over \sqrt{1+y^2}}=0.
\end{equation} 
When $q=0$, \eqref{eq:jan26.2020.3} has three solutions which are
\[
y_0=0,\quad y_1= \sqrt{t^2-1}, \quad y_2= -\sqrt{t^2-1}.
\]
They produce in \eqref{eq:jan26.2020.2} the values 
\[
-1 \quad {\rm{and}}\quad -{t\over 2} -{1 \over 2t}.
\]
Therefore for $t>1$, we have 
\[
u(t, 0)=-{t\over 2} -{1 \over 2t}.
\]
Since 
\[
u(t, q)-u(t, 0) \leq {|y_i-q|^2 \over 2t} + \phi(y_i)- \Big({|y_i|^2 \over 2t} + \phi(y_i)\Big)=  {- y_i\cdot q \over t}+{ |q|^2 \over 2t},
\]
$\pm y_i/t$  belong to the super--differential of $u(t, \cdot)$ at $q=0$. Thus, $u(t, \cdot)$ is not differentiable at $0.$

\medskip

%
%
\section{Hamiltonian Flows and minimizers of the Lagrangian action}\label{appendixC}

Most of the results of this section are expected to be known in some communities. We include them here for the sake of completeness and because of a lack of a precise reference.

\subsection{Hamiltonian Flows on  the Hilbert space}
Throughout this subsection, we impose \eqref{ass:F-U0-C11new1}-\eqref{ass:UpperboundonL}. Showing that the value value function of our Hamilton--Jacobi equation is of class $C^{1,1}$ on the Hilbert space is the starting point before improving regularity property via a discretization method. We underline that in Subsection \ref{subsec:standardresults}, using `direct techniques' relying on the convexity of the Lagrangian action, we have shown already that the value function $\tilde\cU$ is of class $C^{1,1}_{\rm{loc}}$. In this section, we discuss the regularity properties of the infinite dimensional Hamiltonian flow \eqref{eq:Hamiltonian-Flow}, which could also be transferred to the value function.

Let $\tilde\xi,\tilde\eta:[0,\infty) \times \bH \rightarrow \bH$ be given by \eqref{eq:Hamiltonian-Flow2}. Using \eqref{eq:Hamiltonian-Flow3} and the last inequality in Remark \ref{rem:new-bound-onHL} (iii), we have 
\begin{equation}\label{eq:Hamiltonian-Flow3bis}
\|\big(\tilde \xi(t, x), \tilde \eta(t, x)\big)\|+1 \leq \Big( \sqrt{\|x\|^2 + \overline \kappa^2 (\|x\|^2+1)}  +1 \Big) e^{\tilde \kappa t} 
\end{equation} 
for any $t>0$ and $x \in \bH$. 
We can formulate the following result.

%
%
%

\begin{proposition}\label{prop:summary-bounds3} Let $t \in (0, T),$ $\mu \in \cP_2(\M)$ and $q \in \M.$ Suppose $(t_n)_n \subset [0, T]$ converges to $t$, $(\mu_n)_n \subset \cP_2(\M)$ converges to $\mu$ and $(q_n)_n \subset \M$ converges to $q$. Then for every compact set $K \subset [0,t)$ we have  
\[
\lim_{n \rightarrow \infty}   \Big\|S_s^{t_n}[\mu_n] (q_n) - S_s^{t}[\mu] (q)\Big\|_{C( K)} =0.
\]
\end{proposition}
\proof{} To alleviate the notation, we set $\gamma^n(s):= S_s^{t_n}[\mu_n] (q_n).$ It is characterized by the property that 
\begin{equation}\label{eq:jan04.2020.9}
u(t_n, q_n, \mu_n)=u_0\big( \gamma^n_0, \sigma_0^{t_n}[\mu_n]\big) + \int_0^{t_n} \Big( L\big(\gamma^n_\tau, \dot \gamma^n_\tau \big)+f\big(\gamma^n_\tau, \sigma_\tau^{t_n}[\mu_n]\big) \Big) d\tau, \quad \gamma^n_{t_n}=q_n.
\end{equation}

We assume without loss of generality that there exists $r>0$ such that $(\mu_n)_n \subset \cB_r$ and $(q_n) \subset B_r(0).$   
By Remark \ref{rem:summary-bounds1} (ii) 
$$\big\{\sigma_s^{t_n}[\mu_n]\, : \, n \in \mathbb N, s \in [0, t_n] \big\} \subset \cB_{e_T(r)}.$$
In light of Remark \ref{rem:summary-bounds4} (ii), we may apply the Ascoli--Arzel\`a lemma to obtain a subsequence which we continue to denote as $(\gamma^n)_n$ which converges uniformly in $C([0,t-\delta]; \M)$ for every $\delta \in (0,t).$ We have $\gamma \in W^{1,2}(0, t; \M)$ and may also assume $(\gamma^n)_n$ converges weakly to $\gamma$ in $ W^{1,2}(0, t; \M)$. We use  \eqref{eq:holder-for-gamma} to obtain that $\gamma_t=q.$ We would like to replace $t_n$ by $t-\delta$. Since the integrand there is not known to be non negative, we use  \eqref{hyp:f-theta} to write 
\begin{align*} 
u(t_n, q_n, \mu_n)=&u_0\big( \gamma^n_0, \sigma_0^{t_n}[\mu_n]\big) +
\int_0^{t_n} \theta(\sigma_\tau^{t_n}[\mu_n]) (|\gamma^n_\tau|+1) d\tau\\
+& \int_0^{t_n} \Big( L\big(\gamma^n_\tau, \dot \gamma^n_\tau \big)+f\big(\gamma^n_\tau, \sigma_\tau^{t_n}[\mu_n]\big)-\theta(\sigma_\tau^{t_n}[\mu_n]) (|\gamma^n_\tau|+1) \Big) d\tau.
\end{align*}
Thus, since all the  integrands are non negative, we have 
\begin{align*} 
\liminf_{n \rightarrow \infty}u(t_n, q_n, \mu_n)\geq & \liminf_{n \rightarrow \infty} u_0\big( \gamma^n_0, \sigma_0^{t_n}[\mu_n]\big) 
+ \liminf_{n \rightarrow \infty}\int_0^{t-\delta} \theta(\sigma_\tau^{t_n}[\mu_n]) (|\gamma^n_\tau|+1) d\tau\\
+ & \liminf_{n \rightarrow \infty}\int_0^{t-\delta} \Big( L\big(\gamma^n_\tau, \dot \gamma^n_\tau \big)+f\big(\gamma^n_\tau, \sigma_\tau^{t_n}[\mu_n]\big)-\theta(\sigma_\tau^{t_n}[\mu_n]) (|\gamma^n_\tau|+1) \Big) d\tau.
\end{align*}
We invoke the uniform convergence of $(\gamma^n)_n$, the pointwise convergence of $(\sigma_\tau^{t_n}[\mu_n])_n$ provided in \eqref{lem:summary-bounds2} and the convexity of the functions in \eqref{ass:convexity} to conclude that 
\begin{align*} 
\liminf_{n \rightarrow \infty}u(t_n, q_n, \mu_n)\geq & u_0\big( \gamma_0, \sigma_0^{t}[\mu]\big) 
+  \int_0^{t-\delta} \Big( L\big(\gamma_\tau, \dot \gamma_\tau \big)+f\big(\gamma_\tau, \sigma_\tau^{t}[\mu]\big)-\theta(\sigma_\tau^{t}[\mu]) (|\gamma_\tau|+1) \Big) d\tau \\
+&  \int_0^{t-\delta} \theta(\sigma_\tau^{t}[\mu]) (|\gamma_\tau|+1) d\tau.
\end{align*}
We let $\delta$ tend to $0$ to conclude that 
\[
\liminf_{n \rightarrow \infty}u(t_n, q_n, \mu_n)\geq 
u_0\big( \gamma_0, \sigma_0^{t}[\mu]\big)+ \int_0^{t} \Big( L\big(\gamma_\tau, \dot \gamma_\tau \big)+f\big(\gamma_\tau, \sigma_\tau^{t}[\mu]\big) \Big)
\geq  u(t, q, \mu).
\]
Since Proposition \ref{prop:ext_u}  asserts that $u$ is continuous, we infer 
\[
u(t, q, \mu)= u_0\big( \gamma_0, \sigma_0^{t}[\mu]\big)+ \int_0^{t} \Big( L\big(\gamma_\tau, \dot \gamma_\tau \big)+f\big(\gamma_\tau, \sigma_\tau^{t}[\mu]\big) \Big) d\tau.
\]
and so, $\gamma_s \equiv S_s^t[\mu](q).$ 

In conclusion, we have proven that every subsequence of $\big( S_s^t[\mu_n](q_n)\big)_n$ admits itself a subsequence which converges uniformly on every compact subset of $[0, t).$ This is enough to conclude the proof. \endproof

\begin{proposition}\label{prop:homeo1} Let $t>0$. Then the following hold. 
\begin{enumerate} 
\item[(i)] $\Sigma(t, \cdot)$ given in \eqref{eq:Hamiltonian-Flow} is of class $C^{0,1}_{\rm{loc}}.$ 
\item[(ii)] $\tilde \xi_t:\bH \rightarrow \bH$ is a bijection and its inverse is $\tilde S_0^t.$ For each natural number $m$,  $\tilde \xi_t$ is a homeomorphism $\{M^q\; : \; q \in \M^m\}$ onto $\{M^q\; : \; q \in \M^m\}$. This means $S_s^{t, m}: \M^m \rightarrow \M^m$ is a homeomorphism. 
\item[(iii)] $\tilde S_s^t \circ \tilde \xi_t= \tilde  \xi_s$ and $\tilde P_s^t \circ \tilde \xi_t= \tilde \eta_s$ for $s \in [0, t]$. 
\item[(iv)] We have $\nabla \tilde \cU(t, \tilde\xi(t, \cdot))= \tilde \eta(t, \cdot).$ Furthermore, the vector field $B$ in \eqref{eq:define-velocity} is a velocity for the flow $\tilde\xi$ in the sense that $\dot {\tilde \xi}=\nabla_b \tilde \cH(\tilde \xi, \tilde \nabla \cU(\cdot, \tilde \xi))$
\end{enumerate}
\end{proposition} 
\begin{remark}
Although  $\tilde\xi_t$ is a homeomorphism, let us underline that in Proposition \ref{prop:homeo1}(ii) we state that the image of $\{M^q\; : \; q \in \M^m\}$ through $\tilde\xi_t$ is not an arbitrarily closed space but is exactly $\{M^q\; : \; q \in \M^m\}$. Such special vector spaces are mapped onto themselves. Otherwise, we would not be able to conclude that the finite dimensional ODEs are restrictions of the infinite dimensional ones.
\end{remark}
\proof[Proof of Proposition \ref{prop:homeo1}] (i) Since $\tilde \cH$ is of class $C^{1,1}$, $\Sigma$ is Lipschitz continuous. Let $\kappa^*$ be the Lipschitz constant of $\nabla \tilde \cH$. We have 
\[
Lip( \Sigma(t, \cdot)) \leq Lip( \Sigma(0, \cdot)) e^{t \kappa^* }
\]
for all $t>0$. Here, $Lip( \Sigma(t, \cdot))$ stands for the Lipschitz constant of $\Sigma(t, \cdot)$.  

Since $\Sigma$ satisfies \eqref{eq:Hamiltonian-Flow}, we conclude that $\Sigma$ is of class  $C^{0,1}_{\rm loc}.$ 

(ii) {\em Surjectivity.} Given any $x \in \bH.$ Set $z:=\tilde S_0^t[x]$ and define 
\[
\gamma(s)=\tilde S_s^t[x], \quad b(s)=\nabla_{a} \tilde\cL(\gamma(s), \dot \gamma(s)).
\]
We have that $(\gamma, b)$ satisfies the same system of differential equations as $(\tilde \xi, \tilde \eta)$ on $(0, t)$. Furthermore, $\gamma(0)=z$ and 
\[
b(0)=\nabla_b \cL(\tilde S_s^t[x], \partial_s \tilde S_s^t[x]|_{s=0}) =\nabla \tilde \cU_0(z).
\]
Thus, $(\gamma, b)$ have the same initial conditions as  $(\tilde \xi, \tilde \eta)$. Hence, conclude that $\gamma \equiv \tilde\xi(\cdot, z)$ on $[0, t]$. In particular, $x=\tilde S_t^t[x]= \tilde\xi(t, z)= \tilde\xi(t,\tilde S_0^t[x])$. This  shows the surjectivity property.

{\em  Injectivity.} The above show that $\tilde S_0^t$ is injective and $\tilde\xi(t, \cdot)$ is its inverse. To show that $\tilde\xi(t, \cdot)$ is injective, it suffices to show that $\bH$ is the range of $\tilde S_0^t$. Let $z_0 \in \bH$. Set $x_0:=\tilde\xi(t, z_0)$ set 
\[
\gamma(s)=\tilde\xi(s, z_0), \quad g(s)=\tilde \eta(s, z_0).
\]
Then $(\gamma, g)$ satisfies the same system of differential equations as $[0, t]\ni s \mapsto (\tilde S_s^t[x_0], \tilde P_s^t[x_0])$ on $(0, t)$. We have $\gamma(t)=x_0$ and 
$$g(0)=\tilde \eta(0, z_0)= \nabla \tilde \cU_0(z_0)= \nabla \tilde \cU_0(\gamma(0)).$$ 
Thus,  $(\gamma, g)(s)\equiv (\tilde S_s^t[x_0], \tilde P_s^t[x_0])$ on $[0, t]$. In particular, $z_0=\gamma(0)=\tilde S_0^t[x_0].$  Thus, $\tilde S_0^t$ is  surjective. 

{\em  Continuity.} Since $\tilde \xi_t$ is a bijection of $\bH$ onto $\bH$, \eqref{eq:relation-m-infty} and the Invariance of Domain theorem imply that $\tilde \xi_t$ is a homeomorphism of $\{M^q\; : \; q \in \M^m\}$ onto $\{M^q\; : \; q \in \M^m\}$. 

(iii) By (ii) 
$$ \tilde S_0^t \circ \tilde \xi_t=\id_{\bH} =\tilde\xi_0\quad \text{and}\quad \tilde P_0^t \circ \tilde\xi_t= \nabla \tilde \cU_0\big( \tilde S_0^t \circ \xi_t\big) = \nabla \tilde \cU_0=\tilde \eta_0.$$
Since $s\mapsto (\tilde S_s^t \circ \tilde\xi_t, \tilde P_s^t \circ \tilde \xi_t)$ and  $s\mapsto (\tilde \xi_s, \tilde \eta_s)$ satisfy the same system of differential equations on $(0, t)$, we obtain the assertions in (iii).

(iv) We use first Proposition \ref{theorem:hilbert-smooth} (iv) and then (i) of the current Proposition to obtain that $\nabla \tilde \cU(t, \tilde \xi(t, \cdot))= \tilde \eta(t, \cdot).$  We use the identity 
$\dot {\tilde \xi}=\nabla_b \tilde \cH(\tilde \xi, \tilde \eta)$ to conclude the proof. 
\endproof

\begin{remark}
(i) We notice that Proposition \ref{prop:summary-bounds3}, which imposes \eqref{ass:convexity}, allows to improve the continuity property of $\tilde\xi_t$ and its inverse to the infinite dimensional space, i.e. this implies that $\tilde\xi_t$ is a homeomorphism of $\bH$ onto itself.

(ii) We observe that by Proposition \ref{prop:homeo1}(iv) we have that $\nabla\tilde\cU(t,\cdot)=\tilde\eta(t,\tilde S^t_0[\cdot])$, and since both $\tilde\eta$ and $\tilde S^t_0$ are locally Lipschitz continuous (by (i) of the previous proposition and Lemma \ref{lem:summary-bounds2}, respectively) we have that $\nabla\tilde\cU(t,\cdot)$ is locally Lipschitz continuous, just as in Subsection \ref{subsec:standardresults}, by a different perspective one obtains that $\tilde\cU(t,\cdot)\in C^{1,1}_{\rm{loc}}(\bH)$.
\end{remark}

\subsection{Flows on $\bH$, on $\cP_2(\M)$ and their properties}\label{subsec:flows}

%
\begin{lemma}\label{lem:unique-path} Let $x, y \in \bH$ be such that $\sharp(x)=\sharp(y).$ Then for $0\leq s \leq t$, we have $\sharp\Big(\tilde S_s^t[x]\Big)=  \sharp\Big(\tilde S_s^t[y]\Big).$ As a consequence, given $\mu \in \cP_2(\M)$ the following measures are well--defined 
\begin{equation}\label{def:sigma_optimal}
\sigma_s^t[\mu]:= \sharp\Big(\tilde S_s^t[x]\Big)
\end{equation}
where $\sharp(x)=\mu$, depends only on $\mu$ and is independent of the choice of $x.$
\end{lemma}
\begin{proof} Since $\sharp(x)=\sharp(y),$ there exist Borel bijective maps $S_n: \Omega \rightarrow \Omega$ such that (cf. e.g. \cite{Cardaliaguet} \cite{GangboT2017})
$$\sharp(S_n)=\sharp(S_n^{-1})=\cL^d_\Om, \quad \lim_{n \rightarrow \infty} \|y-x\circ S_n\|=0.$$ 
Thus, 
\[
\lim_{n \rightarrow \infty} \Big\|\tilde S_s^t[y]-\tilde S_s^t[x]\circ S_n\Big\|=\lim_{n \rightarrow \infty} \Big\|\tilde S_s^t[y]-\tilde S_s^t[x \circ S_n]\Big\|= 0.
\]
This proves 
\[
W_2\Big(\sharp\Big(\tilde S_s^t[y]\Big), \sharp\Big(\tilde S_s^t[x]\Big)\Big)= \lim_{n \rightarrow \infty} W_2\Big(\sharp\Big(\tilde S_s^t[x]\circ S_n\Big), \sharp\Big(\tilde S_s^t[x]\Big)\Big)=0.
\]
\end{proof}

\begin{remark}\label{rem:summary-bounds1} The following hold.       
\begin{enumerate}
\item[(i)] By Proposition \ref{theorem:hilbert-smooth}, there exists $e_T:[0,\infty) \rightarrow [0,\infty)$, monotone non--decreasing such that 
\[
\|\tilde S_s^t[x]\|,\; \|\partial_s \tilde S_s^t[x]\|
\leq e_T\bigl(\|x\|\bigr)  \qquad \forall s \in [0,t], \forall t \in [0,T].
\]
\item[(ii)] By (i) 
\[
\{\sigma_s^t[\mu]\, : \, \mu \in \cB_r, 0\leq s \leq t \leq T \} \subset  \cB_{e_T(r)}
\]
\item[(iii)] By Proposition \ref{theorem:hilbert-smooth} again, there exists $C_T:(0,\infty) \rightarrow (0, \infty)$ monotone non--decreasing such that  
\[
\|\nabla \tilde \cU(t, x)\| \leq C_T(r)(1+ \|x\|), \qquad \forall x \in \bB_r(0), \forall t \in [0,T].
\]
\item[(iv)] By Lemma \ref{lem:c11_equiv}, the regularity property obtained on $\tilde \cU$ in Proposition \ref{theorem:hilbert-smooth}, we have that $\cU$ is differentiable. We use Proposition \ref{prop:homeo1} (iv) to conclude that $(s, q) \mapsto D_pH\Big(q, \nabla_w \cU\big(s, \sigma_s^t[\mu]\big)(q)\Big)$ is a velocity for $s \mapsto \sigma_s^t[\mu].$ In other words 
\begin{equation}\label{continuity}
\partial_s  \sigma_s^t[\mu]+\nabla\cdot\Big(D_pH\big(\cdot,\nabla_w\cU(s, \sigma_s^t[\mu])\big)\sigma_s^t[\mu]\Big) =0,\ {\rm{in}}\ \mathcal D'((0,t)\times\M),\qquad  \sigma_t^t[\mu]=\mu.
\end{equation}
\end{enumerate}
\end{remark}

%
%
\begin{lemma}\label{lem:summary-bounds2} Suppose $0<t\leq \ov t\leq T$ and $r>0.$ Then there exists a constant $C(r, T)$ monotone increasing in $r$ such that the following hold.         
\begin{enumerate}
\item[(i)] If $x, y \in \bB_r(0)$ then 
\[
\|\tilde S_s^{\ov t}[x]-\tilde S_s^t[y]\| \leq e^{C(r, T)( t-s)}\Big( |\ov t-t| e_T(\|x\|) + \|x-y\|  \Big) \qquad \forall s \in [0, t].
\]
and 
\[
\|\tilde S_s^{\ov t}[x]-\tilde S_t^{\ov t}[x]\| \leq (s-t) e_T(r)  \qquad \forall s \in [t, \ov t].
\]
\item[(ii)]  If $\mu, \nu \in \cB_r$ then 
\begin{equation}\label{eq:jan03.01.2020.3}
W_2\Big(\sigma^{\ov t}_s[\mu], \sigma^t_s[\nu] \Big) \leq      e^{C(r, T)( t-s)}\Big( |\ov t-t| e_T(r) + W_2(\mu, \nu) \Big) \qquad \forall s \in [0, t].
\end{equation}
and 
\[
W_2\Big(\sigma^{\ov t}_s[\mu], \sigma^{\ov t}_t[\mu] \Big) \leq (s-t) e_T(r) \qquad \forall s \in [t, \ov t].
\]
\end{enumerate}
\end{lemma}
\proof{} (i) Let $x, y \in \bB_r(0).$ 

We have 
\[
\big\|x-\tilde S_t^{\ov t}[x]\big\| = \Big\|\int_t^{\ov t} \partial_s  \tilde S_s^{\ov t}[x] ds \Big\| \leq \int_t^{\ov t} \big\| \partial_s  \tilde S_s^{\ov t}[x]  \big\| ds
\]
We use Remark \ref{rem:summary-bounds1} (i)  to infer 
\begin{equation}\label{eq:jan03.01.2020.4}
\big\|x-\tilde S_t^{\ov t}[x]\big\| \leq |{\ov t}-t| e_T(\|x\|). 
\end{equation}

Set
\[
h(s):= {1\over 2} \big\|\tilde S_s^{\ov t}[x]-\tilde S_s^t[x]\big\| \qquad \forall s \in [0, t].
\]
We have
\begin{align}
h'(s)=& \int_\Om \big(\tilde S_s^{\ov t}[x]-\tilde S_s^t[x]\big) \cdot 
\Big( D_pH\big(\tilde S_s^{\ov t}[x], \nabla \tilde U(s, \tilde S_s^{\ov t}[x]) \big)- D_pH\big( \tilde S_s^t[x], \nabla \tilde U(s, \tilde S_s^t[x]) \big) \Big) d\omega. \nonumber
\end{align}
By the fact that $DH$ is Lipschitz we have  
\begin{align*}
& \bigg| D_pH\big(\tilde S_s^{\ov t}[x], \nabla \tilde U(s, \tilde S_s^{\ov t}[x]) \big)- D_pH\big( \tilde S_s^t[x], \nabla \tilde U(s, \tilde S_s^t[x]) \big) \bigg|^2\\
\leq 
& \kappa_0^2 
\Big( \big|\tilde S_s^{\ov t}[x]- \tilde S_s^t[x]\big|^2 + \big|  \nabla \tilde U(s, \tilde S_s^{\ov t}[x])-\nabla \tilde U(s, \tilde S_s^t[x]) \big|^2\Big).
\end{align*}
We use Proposition \ref{theorem:hilbert-smooth}  to obtain a constant $C(r, T)$ which increases in $r$ and such that 

\[
 \bigg\| D_pH\big(\tilde S_s^{\ov t}[x], \nabla \tilde U(s, \tilde S_s^{\ov t}[x]) \big)- D_pH\big( \tilde S_s^t[x], \nabla \tilde U(s, \tilde S_s^t[x]) \big) \bigg\|
\leq    C(r, T) \|\tilde S_s^{\ov t}[x]-\tilde S_s^t[x]\|.
\]
This implies $h' \geq -2C(r, T)h$ and so, Gr\"onwall's inequality yields
\[ 
h(s) \leq     e^{2C(r, T)( t-s)} h(t) \qquad \forall s \in [0, t].
\] 
Thus, 
\[
\|\tilde S_s^{\ov t}[x]-\tilde S_s^t[x]\| \leq e^{C(r, T)( t-s)} \|\tilde S_t^{\ov t}[x]-\tilde S_t^t[x]\|= e^{C(r, T)( t-s)} \|\tilde S_t^{\ov t}[x]-x\|
\]
This, together with \eqref{eq:jan03.01.2020.4} implies 
\begin{equation}\label{eq:jan03.01.2020.5}
\|\tilde S_s^{\ov t}[x]-\tilde S_s^t[x]\| \leq e^{C(r, T)( t-s)} |{\ov t}-t| e_T(\|x\|). 
\end{equation}
We use arguments similar to the ones above to obtain 
\begin{equation}\label{eq:jan03.01.2020.2}
\|\tilde S_s^t[x]-\tilde S_s^t[y]\| \leq     e^{C(r, T)(t-s)} \|x-y\| \qquad \forall s \in [0, t].
\end{equation}
We combine \eqref{eq:jan03.01.2020.5} and \eqref{eq:jan03.01.2020.2} to verify the first identity in (i). The second identity follows from direct integration.

(ii) Let $\mu, \nu \in \cB_r$ and choose $x, y \in \bH$ such that $\sharp(x)=\mu$ and $\sharp(y)=\nu$ and $W_2(\mu,\nu)=\|x-y\|$. Since $\sharp\big(\tilde S_s^{\ov t}[x]\big)=\sigma^{\ov t}_s[\mu]$ and $\sharp\big(\tilde S_s^{ t}[y]\big)=\sigma^{t}_s[\nu]$, (i) implies (ii). \endproof

\subsection{Proof of Proposition \ref{theorem:hilbert-smooth}}\label{subsec:theorem:hilbert-smooth2} Let $y \in \bB_r(0)$.

(i) By Remark \ref{rem:discrete-cont}, $U^{(m)}$ is a viscosity solution to  \eqref{eq:discrte-visc1} and so, the standard theory of Hamilton--Jacobi equations in finite dimensional spaces yields  the pointwise identity
\[
U^{(m)}(t_2, q)-  U^{(m)}(t_1, q)=- \int_{t_1}^{t_2}  \cH^m\bigl(q, D_q  U^{(m)}(\tau, q) \bigr) d\tau
\]
for $q \in \M^m.$ We use \eqref{eq:discretegra1b} to infer 
\[
\tilde \cU(t_2, M^q)- \tilde \cU(t_1, M^q)=- \int_{t_1}^{t_2} \tilde \cH\bigl(M^q, \nabla \tilde \cU(\tau, M^q) \bigr) d\tau
\]
By Proposition \ref{prop:semi-convex1}(ii), when $r>1$, $\nabla \tilde \cU$ is bounded on $[t_1, t_2] \times \bB_r(y)$. Observe that $\nabla \tilde \cU(\tau, \cdot)$ is continuous when $\tau \in [t_1, t_2]$ and $\tilde \cH$ is continuous. Since $\{M^q : q\in \M^m, m \in \mathbb N \}$ is dense in $\bH$, (i) holds. 

(ii) On first obtain a finite number $c(r, T)$ increasing in the variables $r$ and $T$ such that 
\begin{equation}\label{eq:nabla-time-Lip1}
\Bigl|\tilde \nabla \cU(t_2, y)- \tilde \nabla \cU(t_1, y)\Bigr|    \leq 2 c(r, T) |t_2-t_1|.
\end{equation}
This together with the space Lipschitz property of $\nabla \tilde \cU$ implies $\nabla \tilde \cU$ is Lipschitz on $[0, T] \times \bB_r(0)$. As a composition of locally--Lipschitz functions,  $(\tau, x) \mapsto \tilde \cH(x, \nabla \tilde \cU(\tau, x))$ is Lipschitz on $[0, T] \times \bB_r(0)$. Hence since by (i) $\partial_t \tilde  \cU=-\tilde \cH\bigl(\cdot, \nabla \tilde \cU\bigr)$, we conclude $\partial_t \tilde \cU$ is Lipschitz on  $[0,T] \times \bB_r(0)$. 

(iii--v) We refer the reader to \cite{GomesN2014}. \endproof

%

\end{document}